\documentclass{article}

\usepackage{graphicx}
\usepackage{xcolor}
\usepackage{amsthm}
\usepackage{amsmath}
\usepackage{amssymb}
\usepackage{fullpage}
\usepackage{subcaption}
\usepackage{hyperref}

\usepackage[backend=biber,url=false,eprint=false,maxbibnames=99]{biblatex}
\newtheorem{definition}{Definition}[section]
\newtheorem{lemma}[definition]{Lemma}
\newtheorem{proposition}[definition]{Proposition}
\newtheorem{theorem}[definition]{Theorem}
\newtheorem{remark}[definition]{Remark}

\DeclareMathOperator{\sign}{sgn}
\DeclareMathOperator{\real}{Re}
\DeclareMathOperator{\imag}{Im}

\DeclareMathOperator{\polylog}{Li}

\newcommand{\HT}{\mathbf{H}} 
\newcommand{\Hop}[1]{\mathcal{H}^{#1}}

\newcommand{\HopGeneric}{\mathcal{H}}

\newcommand{\inter}[1]{\textbf{\textsl{#1}}}
\newcommand{\lo}[1]{\underline{#1}}
\newcommand{\hi}[1]{\overline{#1}}

\graphicspath{{figures/}}

\title{Highest Cusped Waves for the Fractional KdV Equations}
\author{Joel Dahne}

\begin{document}
\maketitle

\begin{abstract}
  In this paper we prove the existence of highest, cusped, traveling
  wave solutions for the fractional KdV equations
  \(f_t + f f_x = |D|^{\alpha} f_x\) for all \(\alpha \in (-1,0)\) and
  give their exact leading asymptotic behavior at zero. The proof
  combines careful asymptotic analysis and a computer-assisted
  approach.
\end{abstract}

\section{Introduction}
\label{sec:introduction}
This paper is concerned with the existence and regularity of highest,
cusped, periodic traveling-wave solutions to the fractional
Korteweg-de Vries (KdV) equations, in the periodic setting given by
\begin{equation}
  \label{eq:fkdv}
  f_{t} + f f_{x} = |D|^{\alpha}f_{x},\quad \text{ for } (x, t) \in \mathbb{T} \times \mathbb{R}.
\end{equation}
Here \(|D|^{\alpha}\) is the Fourier multiplier operator given by
\begin{equation*}
  \widehat{|D|^{\alpha}f}(\xi) = |\xi|^{\alpha}\widehat{f}(\xi)
\end{equation*}
where the parameter \(\alpha\) may in general take any real value. It
can serve as a model for investigating the balance between nonlinear
and dispersive effects~\cite{doi:10.1137:130912001}. For
\(\alpha = 2\) and \(\alpha = 1\) it reduces to the classical KdV and
Benjamin-Ono equations, for \(\alpha = -2\) one gets the reduced
Ostrovsky equation. For \(\alpha = -1\) it reduces to the
Burgers-Hilbert equation~\cite{hunter18:burger-hilbert}
\begin{equation}
  \label{eq:BH}
  f_{t} + f f_{x} = \HT[f],\quad \text{ for } (x, t) \in \mathbb{T} \times \mathbb{R}.
\end{equation}
Here \(\HT\) is the Hilbert transform which, for
\(f: \mathbb{T} \to \mathbb{R}\), is defined by
\begin{equation*}
  \HT[f](x) = \frac{1}{2\pi}p.v. \int_{-\pi}^{\pi}\cot\left(\frac{x - y}{2}\right)f(y)\ dy,\quad
  \widehat{\HT[f]}(k) = -i \sign(k)\widehat{f}(k).
\end{equation*}
The results in this paper, see Theorem~\ref{thm:main}, are for
\(\alpha \in (-1, 0)\), though the analysis of the Burgers-Hilbert
case, \(\alpha = -1\), is also important for the proof.

For \(\alpha \in (-1, 0)\) and small initial data in
\(H^{N}(\mathbb{R})\) with \(N \geq 3\) estimates for the life span
were proved by Ehrnström and Wang~\cite{ehrnstromwang19:enhanced}, see
also~\cite{hunterifrim12:enhanced} and~\cite{hunter2015long} for the
Burgers-Hilbert equation. Well-posedness in \(H^{s}(\mathbb{R})\) with
\(s > 3 / 2\) was established by Ria\~no~\cite{Riao2021}.

For \(\alpha \in (-1, 0)\) the equation exhibits finite time blow
\cite{Castro2010,Hur2012}. The characterization as wave breaking
(i.e.\ the functions stay bounded, but its gradient blows up) was
proved for \(\alpha \in (-1, -1 / 3)\) by Hur and Tao
\cite{Hur2014,Hur2017} and for \(\alpha \in (-1, 0)\) by Oh and
Pasqualotto \cite{oh21:_gradien_burger}. See \cite{Klein2015} for a
numerical study and \cite{chickering21:_asymp_burger} where they give
a precise characterization of a chock forming in finite time for
\(0 < \alpha < \frac{1}{3}\).

The study of traveling waves is an important topic in fluid dynamics,
see e.g.~\cite{Haziot2022} for a recent overview of traveling water
waves. The traveling wave assumption \(f(x, t) = \varphi(x - ct)\),
where \(c > 0\) denotes the wave speed, gives us
\begin{equation}
  \label{eq:fkdv-wave}
  -c \varphi' + \varphi\varphi' = |D|^{\alpha} \varphi'.
\end{equation}
For the fractional KdV equation there is a branch of even, zero-mean,
\(2\pi\)-periodic, smooth traveling wave solutions bifurcating from
constant solutions. For \(\alpha < -1\) and \(\alpha \in (-1, 0)\)
this branch has been studied and proved to end in a highest cusped
wave, for \(\alpha < -1\) by Bruell and Dhara~\cite{bruell18:_waves}
and for \(\alpha \in (-1, 0)\) by Hildrum and
Xue~\cite{Hildrum23PeriodicWaves}. For \(\alpha = -1\) the branch has
been studied by Castro, Córdoba and Zheng~\cite{castro2021stability}.

The notion of a highest traveling wave goes back to Stokes. For the
free boundary Euler equation Stokes argued that if there exists a
singular solution with a steady profile it must have an interior angle
of \(120^{\circ}\) at the crest~\cite{stokes2018mathematical}. This is
known as the Stokes conjecture and was proved in
1982~\cite{Amick1982}. For the Whitham equation~\cite{Whitham1967} the
existence of a highest cusped traveling wave was conjectured by
Whitham in~\cite{Whitham1999}. Its existence, together with its
\(C^{1 / 2}\) regularity, was proved by Ehrnström and
Wahlén~\cite{Ehrnstrm2019}. For the family of fractional KdV equation
and variants thereof there has recently been much progress related to
highest waves. Bruell and Dhara proved existence of highest traveling
waves which are Lipschitz at their cusp for
\(\alpha < -1\)~\cite{bruell18:_waves}. Hildrum and Xue proved
existence and the optimal \(-\alpha\)-Hölder regularity for a family
of equations including the fractional KdV equation with
\(\alpha \in (-1, 0)\)~\cite{Hildrum23PeriodicWaves}. Ørke proved
their existence for \(\alpha \in (-1, 0)\) for the inhomogeneous
fractional KdV equations as well as the fractional Degasperis-Procesi
equations, together with their optimal \(-\alpha\)-Hölder
regularity~\cite{orke22:_highes_kortew}.

The results in~\cite{Ehrnstrm2019, bruell18:_waves,
  orke22:_highes_kortew, Hildrum23PeriodicWaves} are all based on
global bifurcation arguments, bifurcating from the constant solution
and proving that the branch must end in a highest wave which is not
smooth at its crest. For the Burgers-Hilbert equation Dahne and
Gómez-Serrano proved the existence of a highest wave which at its cusp
behaves like \(|x|\log|x|\)~\cite{dahne2022burgershilbert}. The proof
uses a different method where the problem is first reduced to a fixed
point problem. The wave is therefore not directly tied to a branch of
solutions as in the other results. This method of rewriting the
problem into a fixed point problem was first used by Enciso,
Gómez-Serrano and Vergara for proving the convexity and the precise
asymptotic behavior of a highest wave solution to the Whitham
equation~\cite{enciso2018convexity}, answering a conjecture made by
Ehrnström and Wahlén~\cite{Ehrnstrm2019}. In this paper we use a
similar approach.

Similar to Hildrum and Xue we prove the existence of a highest wave
for the fractional KdV equation with \(\alpha \in (-1, 0)\). Our main
contribution is a more precise description of the asymptotic behavior
of the wave at the cusp. The description is in line with earlier
results for the Whitham and Burgers-Hilbert equation. Recently
Ehrnström, Mæhlen and Varholm have obtained similar results for the
asymptotic behavior of two families of equations including the uni-
and bidirectional Whitham equations using different types of
methods~\cite{Ehrnstrom2023precisecuspedbehaviour}.

We prove the following theorem:
\begin{theorem}
  \label{thm:main}
  There is a \(2\pi\)-periodic traveling wave \(\varphi\)
  of~\eqref{eq:fkdv-wave} for every \(\alpha \in (-1, 0)\), which
  behaves asymptotically at \(x = 0\) as
  \begin{equation*}
    \varphi(x) = c - \nu_{\alpha} |x|^{-\alpha} + \mathcal{O}(|x|^{p}).
  \end{equation*}
  for \(\nu_{\alpha} > 0\) as given in
  Lemma~\ref{lemma:asymptotic-coefficient} and some
  \(-\alpha < p \leq 1\) to be made explicit later on.
\end{theorem}
\begin{remark}
  The remainder term \(\mathcal{O}(|x|^{p})\) in
  Theorem~\ref{thm:main} is in general not sharp. The estimate follows
  from the choice of our space.
\end{remark}
\begin{remark}
  Both our results and Hildrum and Xue's results assert the existence
  of a highest wave. While, a priori, they don't necessarily
  correspond to the same wave, we do believe this is the case.
\end{remark}

The ansatz \(\varphi(x) = c - u(x)\) allows us to rewrite
\eqref{eq:fkdv-wave} as an equation that does not explicitly depend on
the wave speed \(c\). Proving the existence of a solution \(u\) can be
rewritten as a fixed point problem by considering the ansatz
\begin{equation*}
  u(x) = u_{\alpha}(x) + w_{\alpha}(x)v(x)
\end{equation*}
where \(u_{\alpha}(x)\) is an explicit, carefully chosen, approximate
solution and \(w_{\alpha}(x)\) is an explicit weight factor. Proving
the existence of a fixed point can be reduced to checking an
inequality involving three constants, \(D_{\alpha}\),
\(\delta_{\alpha}\) and \(n_{\alpha}\), that only depend on the choice
of \(u_{\alpha}\) and \(w_{\alpha}\), see
Proposition~\ref{prop:contraction}. This inequality is checked by
bounding \(D_{\alpha}\), \(\delta_{\alpha}\) and \(n_{\alpha}\) using
a computer-assisted proof, see Section~\ref{sec:bounds-for-values}.

One of the key difficulties compared to the Whitham
equation~\cite{enciso2018convexity} and the Burgers-Hilbert
equation~\cite{dahne2022burgershilbert} is that we are treating a
family of equations, instead of one fixed equation. To handle the full
family we need to understand how the equation changes with respect to
\(\alpha\). In particular the endpoints of the interval, \(\alpha\)
near \(-1\) and \(0\), require adapting the method to those cases. To
handle this we split the interval \((-1, 0)\) into three parts,
\begin{equation*}
  (-1, 0) = (-1, -1 + \delta_{1}) \cup [-1 + \delta_{1}, -\delta_{2}] \cup (-\delta_{2}, 0)
          = I_{1} \cup I_{2} \cup I_{3},
\end{equation*}
and adapt the methods used for \(I_{1}\) and \(I_{3}\). For
\(\alpha \in I_{1}\) the main complication comes from that as
\(\alpha \to -1\) the coefficient \(\nu_{\alpha}\) in
Theorem~\ref{thm:main} tends to infinity. For \(\alpha \in I_{3}\) the
issue is that both sides of the inequality that needs to be verified
for the fixed point argument tend to zero as \(\alpha \to 0\) and to
assert that it holds arbitrarily close to \(\alpha = 0\) we need an
understanding of the rate at which the two sides tend to zero.

An important part of the work is the construction of the approximate
solution \(u_{\alpha}\). Due to the singularity at \(x = 0\) it is not
possible to use a trigonometric polynomial alone, it would converge
very slowly and have the wrong asymptotic behavior. Pure products of
powers and logarithms, \(|x|^{a}\log^{b} |x|\), have the issue that
they are not periodic and do not interact well with the operator
\(|D|^{\alpha}\). Instead we take inspiration from the construction
in~\cite{enciso2018convexity} and consider a combination of
trigonometric polynomials and Clausen functions of different orders,
defined as
\begin{equation*}
  C_{s}(x) = \sum_{n = 1}^{\infty}\frac{\cos(nx)}{n^{s}},\quad
  S_{s}(x) = \sum_{n = 1}^{\infty}\frac{\sin(nx)}{n^{s}},
\end{equation*}
for \(s > 1\) and by analytic continuation otherwise. We also make use
of their derivatives with respect to the order, for which we use the
notation
\begin{equation*}
  C_{s}^{(\beta)}(x) := \frac{d^{\beta}}{ds^{\beta}} C_{s}(x),\quad
  S_{s}^{(\beta)}(x) := \frac{d^{\beta}}{ds^{\beta}} S_{s}(x).
\end{equation*}
The Clausen functions are \(2\pi\)-periodic, non-analytic at \(x = 0\)
and behave well with respect to \(|D|^{\alpha}\). In particular
\(C_{s}(x) - C_{s}(0) \sim |x|^{s - 1}\), which corresponds to the
behavior we expect in Theorem~\ref{thm:main}. See
Section~\ref{sec:clausen-functions} and
Appendix~\ref{sec:computing-clausen} for more details about the
Clausen functions. The main idea for the construction is the same as
in~\cite{enciso2018convexity}, to choose the coefficients for the
Clausen functions according to the asymptotic expansion of the
solution at \(x = 0\). However, in our case we also have to handle the
limits \(\alpha \to -1\) and \(\alpha \to 0\), which requires a good
understanding of how the approximations depend on \(\alpha\). In
particular understanding the limit \(\alpha \to -1\) is important not
only for this work, but is also used to handle the Burgers-Hilbert
case in~\cite{dahne2022burgershilbert}.

An essential part of our work is the interplay between traditional
mathematical tools and rigorous computer calculations. Traditional
numerical methods typically only compute approximate results, to be
able to use the results in a proof we need the results to be
rigorously verified. The basis for rigorous calculations is interval
arithmetic, pioneered by Moore in the 1970s~\cite{Moore1979}. Due to
improvements in both computational power and great improvements in
software it is getting practical to use computer-assisted tools in
more and more complicated settings. The main idea with interval
arithmetic is to do arithmetic not directly on real numbers but on
intervals with computer representable endpoints. Given a function
\(f:\mathbb{R} \to \mathbb{R}\), an interval extension of \(f\) is an
extension to intervals satisfying that for an interval
\(\inter{x} = [\lo{x}, \hi{x}]\), \(f(\inter{x})\) is an interval
satisfying \(f(x) \in f(\inter{x})\) for all \(x \in \inter{x}\). In
particular this allows us to prove inequalities for the function
\(f\), for example the right endpoint of \(f(\inter{x})\) gives an
upper bound of \(f\) on the interval \(\inter{x}\). For an
introduction to interval arithmetic and rigorous numerics we refer the
reader to the books~\cite{Moore1979,Tucker2011} and to the
survey~\cite{GomezSerrano2019} for a specific treatment of
computer-assisted proofs in PDE. For computer-assisted proofs in fluid
mechanics in particular some recent results are:
\cite{vandenBerg2021spontaneousperiodicorbits,Arioli2021uniquenessbifurcationbranches}
for the Navier-Stokes equation, \cite{Figueras2017, Figueras2017II,
  Gameiro2017} for the Kuramoto-Shivasinsky equation,
\cite{Chen2022HouLuomodel} for the Hou-Luo model,
\cite{Buckmaster2022} for the compressible Euler and Navier-Stokes
equations and \cite{Chen2022blowup3deuler, Chen2023blowup3dEulerII}
for blowup of the 2D Boussinesq and 3D Euler equation.

For all the calculations in this paper we make use of the Arb
library~\cite{Johansson2017arb} for ball (intervals represented as a
midpoint and radius) arithmetic. It has good support for many of the
special functions we use~\cite{Johansson2014hurwitz,
  Johansson2014thesis, Johansson2019}, Taylor arithmetic (see
e.g.~\cite{Johansson2015reversion}) as well as rigorous
integration~\cite{Johansson2018numerical}.

The paper is organized as follows. In
Section~\ref{sec:reduct-fixed-point} we reduce the proof of
Theorem~\ref{thm:main} to a fixed point problem. In
Section~\ref{sec:clausen-functions} we give a brief overview of
properties of the Clausen functions that are relevant for the
construction of \(u_{\alpha}\), in Section~\ref{sec:construction} we
give the construction of \(u_{\alpha}\) and in
Section~\ref{sec:choice-weight} we discuss the choice of the weight
function \(w_{\alpha}\) used. Section~\ref{sec:proof-strategy} gives
the general strategy for bounding \(n_{\alpha}\), \(\delta_{\alpha}\)
and \(D_{\alpha}\), Section~\ref{sec:evaluation} and~\ref{sec:inv-u0}
discusses evaluation of \(u_{\alpha}\). Section~\ref{sec:evaluation-N}
is devoted to the approach for bounding \(n_{\alpha}\),
Section~\ref{sec:evaluation-F} to bounding \(\delta_{\alpha}\) and
Section~\ref{sec:analysis-T} to studying the linear operator that
appears in the construction of the fixed point problem and bounding
\(D_{\alpha}\). The computer-assisted proofs giving bounds for
\(n_{\alpha}\), \(\delta_{\alpha}\) and \(D_{\alpha}\) are given in
Section~\ref{sec:bounds-for-values}. Finally we give the proof of
Theorem~\ref{thm:main} in Section~\ref{sec:proof-main-theorem}.

Six appendices are given at the end of the paper.
Appendix~\ref{sec:removable-singularities} gives some technical
details for how to compute enclosures of functions around removable
singularities. Appendix~\ref{sec:taylor-models} gives a brief
introduction to Taylor models which are used in some parts of the
proof. Appendix~\ref{sec:computing-clausen} is concerned with
computing enclosures of the Clausen functions and
Appendix~\ref{sec:rigorous-integration} with the rigorous numerical
integration needed for bounding \(D_{0}\). Finally
Appendix~\ref{sec:evaluation-F-I-1-asymptotic},
\ref{sec:evaluation-T-I-1-asymptotic}
and~\ref{sec:evaluation-T-hybrid-asymptotic} contains some of the
details needed for \(\alpha\) close to \(-1\).

\section{Reduction to a fixed point problem}
\label{sec:reduct-fixed-point}
In this section we reduce the problem of proving
Theorem~\ref{thm:main} to proving the existence of a fixed point for a
certain operator.

From~\cite[Theorem 12]{Hildrum23PeriodicWaves} we have the following
characterization of even, nondecreasing solutions of \eqref{eq:fkdv}.
\begin{lemma}
  Let \(\varphi \in C^{1}\) be a nonconstant, even solution of
  \eqref{eq:fkdv} which is nondecreasing on \((-\pi, 0)\), then
  \begin{equation*}
    \varphi' > 0 \quad \text{ and }\quad \varphi < c
  \end{equation*}
  on \((-\pi, 0)\).
\end{lemma}
As a consequence, any continuous, nonconstant, even function which is
nondecreasing on \((-\pi, 0)\) that satisfy \eqref{eq:fkdv} almost
everywhere must satisfy \(\varphi \leq c\). The maximal possible
height is thus given by \(c\) and due to the function being even and
nondecreasing on \((-\pi, 0)\) the maximal height has to be attained
at \(x = 0\).

Now, the ansatz \(\varphi(x) = c - u(x)\) inserted in
\eqref{eq:fkdv-wave} gives an equation which does not explicitly
depend on the wave speed \(c\). Indeed, inserting this gives us
\begin{equation}
  \label{eq:main-derivative}
  uu' = -|D|^{\alpha}u'.
\end{equation}
Note that a solution of this equation gives a solution of
\eqref{eq:fkdv-wave} for any wave speed \(c\). This is to be expected
due to the Galilean change of variables
\begin{equation*}
  \varphi \mapsto \varphi + \gamma,\ c \mapsto c + \gamma
\end{equation*}
which leaves \eqref{eq:fkdv-wave} invariant. In particular, taking
\(c\) equal to the mean of \(u\) gives a zero mean solution. For a
highest wave we expect to have \(\varphi(0) = c\), giving us
\(u(0) = 0\). Integrating \eqref{eq:main-derivative} gives us
\begin{equation}
  \label{eq:main}
  \frac{1}{2}u^{2} = -\Hop{\alpha}[u].
\end{equation}
Here \(\HopGeneric\) is the operator
\begin{equation}
  \label{eq:H}
  \Hop{\alpha}[u](x) = |D|^{\alpha}u(x) - |D|^{\alpha}u(0).
\end{equation}
It is the integral of the right-hand side of
\eqref{eq:main-derivative} with the constant of integration taken such
that \(\HopGeneric^{\alpha}[f](0) = 0\), this ensures that any
solution of \eqref{eq:main} satisfies \(u(0) = 0\). Note that any
solution of \eqref{eq:main} is a solution of
\eqref{eq:main-derivative} and hence gives a solution to
\eqref{eq:fkdv-wave}.

To reduce the problem to a fixed point problem the idea is to write
\(u\) as one, explicit, approximate solution of Equation
\eqref{eq:main} and one unknown term. More precisely we make the
ansatz
\begin{equation}
  \label{eq:main-ansatz}
  u(x) = u_{\alpha}(x) + w_{\alpha}(x)v(x)
\end{equation}
where \(u_{\alpha}(x)\) is an explicit, carefully chosen, approximate
solution of Equation \eqref{eq:main}, see
Section~\ref{sec:construction}, and \(w_{\alpha}(x)\) is an explicit
weight, see Section~\ref{sec:choice-weight}, both of them depending on
\(\alpha\). By taking \(u_{\alpha}(x) \sim \nu_{\alpha}|x|^{-\alpha}\)
and \(w_{\alpha} = \mathcal{O}(|x|^{p})\), proving
Theorem~\ref{thm:main} reduces to proving existence of
\(v \in L^{\infty}(\mathbb{T})\) such that the given ansatz is a
solution of Equation \eqref{eq:main}.

Inserting the ansatz \eqref{eq:main-ansatz} into Equation
\eqref{eq:main} gives us
\begin{equation*}
  \frac{1}{2}(u_{\alpha} + w_{\alpha}v)^{2} = -\Hop{\alpha}[u_{\alpha} + w_{\alpha}v]
  \iff \frac{1}{2}u_{\alpha}^{2} + u_{\alpha}w_{\alpha}v + \frac{1}{2}w_{\alpha}^{2}v^{2} = -\Hop{\alpha}[u_{\alpha}] - \Hop{\alpha}[w_{\alpha}v]
\end{equation*}
By collecting all the linear terms in \(v\) we can write this as
\begin{equation*}
  u_{\alpha}w_{\alpha}v + \Hop{\alpha}[w_{\alpha}v] = -\Hop{\alpha}u_{\alpha} - \frac{1}{2}u_{\alpha}^{2} - \frac{1}{2}w_{\alpha}^{2}v^{2}
  \iff v + \frac{1}{w_{\alpha}u_{\alpha}}\Hop{\alpha}[v] = -\frac{1}{w_{\alpha}u_{\alpha}}\left(\Hop{\alpha}[u_{\alpha}] + u_{\alpha}^{2}\right) - \frac{w_{\alpha}}{2u_{\alpha}}v^{2}.
\end{equation*}
Now let \(T_{\alpha}\) denote the operator
\begin{equation}
  \label{eq:T_alpha}
  T_{\alpha}[v] = -\frac{1}{w_{\alpha}u_{\alpha}}\Hop{\alpha}[w_{\alpha}v].
\end{equation}
Denote the weighted defect of the approximate solution
\(u_{\alpha}(x)\) by
\begin{equation}
  \label{eq:F}
  F_{\alpha}(x) = \frac{1}{w_{\alpha}(x)u_{\alpha}(x)}\left(\Hop{\alpha}[u_{\alpha}](x) + \frac{1}{2}u_{\alpha}(x)^{2}\right),
\end{equation}
and let
\begin{equation}
  \label{eq:N}
  N_{\alpha}(x) = \frac{w_{\alpha}(x)}{2u_{\alpha}(x)}.
\end{equation}
Then we can write the above as
\begin{equation*}
  (I - T_{\alpha})v = -F_{\alpha} - N_{\alpha}v^{2}.
\end{equation*}
Assuming that \(I - T_{\alpha}\) is invertible we rewrite this as
\begin{equation}
  \label{eq:G}
  v = (I - T_{\alpha})^{-1}\left(-F_{\alpha} - N_{\alpha}v^{2}\right) =: G_{\alpha}[v].
\end{equation}
Hence proving the existence of \(v\) such that
\(u_{\alpha} + w_{\alpha}v\) is a solution to Equation \eqref{eq:main}
reduces to proving existence of a fixed point of the operator
\(G_{\alpha}\).

Next we reduce the problem of proving that \(G_{\alpha}\) has a fixed
point to checking an inequality for three numbers (depending on
\(\alpha\)) that depend only on the choice of \(u_{\alpha}\) and
\(w_{\alpha}\). We let \(\|T\|\) denote the
\(L^{\infty}(\mathbb{T}) \to L^{\infty}(\mathbb{T})\) norm of a linear
operator \(T\).
\begin{proposition}
  \label{prop:contraction}
  Let \(D_{\alpha} = \|T_{\alpha}\|\),
  \(\delta_{\alpha} = \|F_{\alpha}\|_{L^{\infty}(\mathbb{T})}\) and
  \(n_{\alpha} = \|N_{\alpha}\|_{L^{\infty}(\mathbb{T})}\). If
  \(D_{\alpha} < 1\) and they satisfy the inequality
  \begin{equation*}
    \delta_{\alpha} < \frac{(1 - D_{\alpha})^{2}}{4n_{\alpha}}
  \end{equation*}
  then for
  \begin{equation*}
    \epsilon_{\alpha} = \frac{1 - D_{\alpha} - \sqrt{(1 - D_{\alpha})^{2} - 4\delta_{\alpha}n_{\alpha}}}{2n_{\alpha}}
  \end{equation*}
  and
  \begin{equation*}
    X_{\epsilon} = \{v \in L^{\infty}(\mathbb{T}): v(x) = v(-x), \|v\|_{L^{\infty}(\mathbb{T})} \leq \epsilon\}
  \end{equation*}
  we have
  \begin{enumerate}
  \item \(G_{\alpha}(X_{\epsilon_{\alpha}}) \subseteq X_{\epsilon_{\alpha}}\);
  \item
    \(\|G_{\alpha}[v] - G_{\alpha}[w]\|_{L^{\infty}(\mathbb{T})} \leq
    k_{\alpha}\|v - w\|_{L^{\infty}(\mathbb{T})}\) with
    \(k_{\alpha} < 1\) for all \(v, w \in X_{\epsilon_{\alpha}}\).
  \end{enumerate}
\end{proposition}
\begin{proof}
  Using that \(N_{\alpha}\) and \(F_{\alpha}\) are even it can be checked that
  \begin{equation*}
    G_{\alpha}(X_{\epsilon_{\alpha}}) \subseteq (I - T)^{-1}X_{\delta_{\alpha} + n_{\alpha}\epsilon_{\alpha}^{2}}.
  \end{equation*}
  Since \(\|T_{\alpha}\| < 1\) the operator \(I - T_{\alpha}\) is
  invertible and an upper bound of the norm of the inverse is given by
  \(\frac{1}{1 - D_{\alpha}}\), moreover \(T_{\alpha}\) takes even
  functions to even functions and hence so will
  \((I - T_{\alpha})^{-1}\). This gives us
  \begin{equation*}
    G_{\alpha}(X_{\epsilon_{\alpha}}) \subseteq (I - T)^{-1}X_{\delta_{\alpha} + n_{\alpha}\epsilon_{\alpha}^{2}} \subseteq X_{\frac{\delta_{\alpha} + n_{\alpha}\epsilon_{\alpha}^{2}}{1 - D_{\alpha}}}.
  \end{equation*}
  The choice of \(\epsilon_{\alpha}\) then gives
  \begin{equation*}
    \frac{\delta_{\alpha} + n_{\alpha}\epsilon_{\alpha}^{2}}{1 - D_{\alpha}} = \epsilon_{\alpha}.
  \end{equation*}

  Next we have
  \(G_{\alpha}[v] - G_{\alpha}[w] = (I -
  T_{\alpha})^{-1}(-N_{\alpha}(v^{2} - w^{2}))\) and hence
  \begin{equation*}
    \|G_{\alpha}[v] - G_{\alpha}[w]\|_{L^{\infty}(\mathbb{T})}
    \leq \frac{n_{\alpha}}{1 - D_{\alpha}}\|v^{2} - w^{2}\|_{L^{\infty}(\mathbb{T})}
    \leq \frac{2n_{\alpha}\epsilon_{\alpha}}{1 - D_{\alpha}}\|v - w\|_{L^{\infty}(\mathbb{T})}.
  \end{equation*}
  Where
  \(k_{\alpha} = \frac{2n_{\alpha}\epsilon_{\alpha}}{1 - D_{\alpha}} <
  1\) since
  \(\epsilon_{\alpha} < \frac{1 - D_{\alpha}}{2n_{\alpha}}\).
\end{proof}

\section{Clausen functions}
\label{sec:clausen-functions}
We here give definitions and properties of the Clausen functions that
are used in the construction of \(u_{\alpha}\) and when bounding
\(n_{\alpha}\), \(\delta_{\alpha}\) and \(D_{\alpha}\). For more
details about the Clausen functions see
Appendix~\ref{sec:computing-clausen}.

For \(s > 1\) the Clausen functions can be defined through their
Fourier expansions as
\begin{align*}
  C_{s}(x) &= \sum_{n = 1}^{\infty}\frac{\cos(nx)}{n^{s}},\\
  S_{s}(x) &= \sum_{n = 1}^{\infty}\frac{\sin(nx)}{n^{s}}.
\end{align*}
For general \(s\) they are more conveniently defined by their relation
to the polylogarithm through
\begin{align*}
  C_{s}(x) = \frac{1}{2}\left(\polylog_{s}(e^{ix}) + \polylog_{s}(e^{-ix})\right) = \real\left(\polylog_{s}(e^{ix})\right),\\
  S_{s}(x) = \frac{1}{2}\left(\polylog_{s}(e^{ix}) - \polylog_{s}(e^{-ix})\right) = \imag\left(\polylog_{s}(e^{ix})\right).
\end{align*}
They behave nicely with respect to the operator \(|D|^{\alpha}\), for
which we have
\begin{equation*}
  |D|^{\alpha}C_{s} = -C_{s - \alpha}(x),\quad |D|^{\alpha}S_{s} = -S_{s - \alpha}(x).
\end{equation*}
In many cases we want to work with functions which are normalized to
be zero at \(x = 0\), for which we use the notation
\begin{equation*}
  \tilde{C}_{s}(x) = C_{s}(x) - C_{s}(0),\quad
  \tilde{C}_{s}^{(\beta)}(x) = C_{s}^{(\beta)}(x) - C_{s}^{(\beta)}(0).
\end{equation*}
Note that \(C_{s}(0)\) is in general only finite for \(s > 1\), in
which case we get directly from the Fourier expansion that
\(C_{s}(0) = \zeta(s)\) and
\(C_{s}^{(\beta)}(0) = \zeta^{(\beta)}(s)\). With this notation we get
for the operator \(\Hop{\alpha}\),
\begin{equation*}
  \Hop{\alpha}[\tilde{C}_{s}](x) = -\tilde{C}_{s - \alpha}(x),\quad \Hop{\alpha}[S_{s}](x) = -S_{s - \alpha}(x).
\end{equation*}

From~\cite{enciso2018convexity} we have the following expansion for
\(C_{s}(x)\) and \(S_{s}(x)\), valid for \(|x| < 2\pi\),
\begin{align*}
  C_{s}(x) &= \Gamma(1 - s)\sin\left(\frac{\pi}{2}s\right)|x|^{s - 1}
             + \sum_{m = 0}^{\infty} (-1)^{m}\zeta(s - 2m)\frac{x^{2m}}{(2m)!};\\
  S_{s}(x) &= \Gamma(1 - s)\cos\left(\frac{\pi}{2}s\right)\sign(x)|x|^{s - 1}
             + \sum_{m = 0}^{\infty} (-1)^{m}\zeta(s - 2m - 1)\frac{x^{2m + 1}}{(2m + 1)!}.
\end{align*}
For the functions \(C_{s}^{(\beta)}\) and \(S_{s}^{(\beta)}\) we will
mainly make use of \(C_{2}^{(1)}(x)\) and \(C_{3}^{(1)}(x)\), for
which we have the following expansions~\cite[Eq.~16]{Bailey2015},
valid for \(|x| < 2\pi\),
\begin{align*}
  C_{2}^{(1)}(x) =& \zeta^{(1)}(2) - \frac{\pi}{2}|x|\log|x| - (\gamma - 1)\frac{\pi}{2}|x|
                   + \sum_{m = 1}^{\infty}(-1)^{m}\zeta^{(1)}(2 - 2m)\frac{x^{2m}}{(2m)!};\\
  C_{3}^{(1)}(x) =& \zeta^{(1)}(3) - \frac{1}{4}x^{2}\log^2|x|
                    + \frac{3 - 2\gamma}{4}x^2\log|x|
                    - \frac{36\gamma - 12\gamma^2 - 24\gamma_{1} - 42 + \pi^{2}}{48}x^2\\
                 &+ \sum_{m = 2}^{\infty}(-1)^{m}\zeta^{(1)}(3 - 2m)\frac{x^{2m}}{(2m)!}.
\end{align*}
Where \(\gamma_{n}\) is the Stieltjes constant and
\(\gamma = \gamma_{0}\). Bounds for the tails are given in
Lemmas~\ref{lemma:clausen-tails}
and~\ref{lemma:clausen-derivative-tails}.

\section{Construction of \(u_{\alpha}\)}
\label{sec:construction}

In this section we give the construction of the approximate solution
\(u_{\alpha}\) for \(\alpha \in (-1, 0)\). As a first step we
determine the coefficient for the leading term in the asymptotic
expansion.

\begin{lemma}
  \label{lemma:asymptotic-coefficient}
  Let \(\alpha \in (-1, 0)\) and assume that \(u\) is a bounded, even
  solution of Equation \eqref{eq:main} with the asymptotic behavior
  \begin{equation*}
    u(x) = \nu_{a}|x|^{-\alpha} + o(|x|^{-\alpha})
  \end{equation*}
  close to zero, with \(\nu_{\alpha} \not= 0\). Then the coefficient
  is given by
  \begin{equation*}
    \nu_{\alpha} = \frac{
      2\Gamma(2\alpha)\cos\left(\pi\alpha\right)
    }{
      \Gamma(\alpha)\cos\left(\frac{\pi}{2}\alpha\right)
    }.
  \end{equation*}
\end{lemma}
\begin{proof}
  We directly get
  \begin{equation*}
    \frac{1}{2}u(x)^{2} = \frac{\nu_{\alpha}^{2}}{2}|x|^{-2\alpha} + o(|x|^{-2\alpha}).
  \end{equation*}
  To get the asymptotic behavior of \(\Hop{\alpha}[u]\) hand side we
  go through the Clausen function \(\tilde{C}_{1 - \alpha}(x)\). Based
  on the asymptotic behavior of \(\tilde{C}_{1 - \alpha}(x)\) we can
  write \(u\) as
  \begin{equation*}
    u(x) = \frac{\nu_{\alpha}}{\Gamma(\alpha)\cos\left(\frac{\pi}{2}\alpha\right)}\tilde{C}_{1 - \alpha}(x) + o(|x|^{-\alpha}).
  \end{equation*}
  This gives us
  \begin{align*}
    \Hop{\alpha}[u](x)
    &= -\frac{\nu_{\alpha}}{\Gamma(\alpha)\cos\left(\frac{\pi}{2}\alpha\right)}\tilde{C}_{1 - 2\alpha}(x) + \Hop{\alpha}[o(|x|^{-\alpha})](x)\\
    &= -\nu_{\alpha}\frac{\Gamma(2\alpha)\cos\left(\pi\alpha\right)}{\Gamma(\alpha)\cos\left(\frac{\pi}{2}\alpha\right)}|x|^{-2\alpha} + o(|x|^{-2\alpha}) + \Hop{\alpha}[o(|x|^{-\alpha})](x).
  \end{align*}
  In a similar way as in~\cite[Lemma 4.1]{dahne2022burgershilbert} it
  can be shown that
  \(\Hop{\alpha}[o(|x|^{-\alpha})](x) = o(|x|^{-2\alpha})\), which
  implies the lemma.
\end{proof}

In addition to having the correct asymptotic behavior we want
\(u_{\alpha}\) to be a good approximate solution of \eqref{eq:main},
in the sense that we want the defect,
\begin{equation*}
  F_{\alpha}(x) = \frac{1}{w_{\alpha}(x)u_{\alpha}(x)}\left(\Hop{\alpha}[u_{\alpha}](x) + \frac{1}{2}u_{\alpha}(x)^{2}\right),
\end{equation*}
to be small for \(x \in [0, \pi]\). The hardest part is to make
\(F_{\alpha}(x)\) small locally near the singularity at \(x = 0\),
this is done by studying the asymptotic behavior of
\(\Hop{\alpha}[u_{\alpha}](x) + \frac{1}{2}u_{\alpha}(x)^{2}\). Once
the defect is sufficiently small near \(x = 0\) it can be made small
globally by adding a suitable trigonometric polynomial to
\(u_{\alpha}\).

The construction is similar to that in~\cite{enciso2018convexity}, but
requires more work to handle the limits \(\alpha \to -1\) and
\(\alpha \to 0\). We take \(u_{\alpha}\) to be a combination of three
parts:
\begin{enumerate}
\item The first term is \(a_{\alpha,0}\tilde{C}_{1 - \alpha}\), where
  the coefficient is chosen to give the right asymptotic behavior
  according to Lemma~\ref{lemma:asymptotic-coefficient}
\item The second part is chosen to make the defect small near
  \(x = 0\), similarly to in~\cite{enciso2018convexity}, it is given
  by a sum of Clausen functions.
\item The third part is chosen to make the defect small globally and
  is given by a trigonometric polynomial.
\end{enumerate}
More precisely the approximation is given by
\begin{equation}
  \label{eq:u0}
  u_{\alpha}(x) = \sum_{j = 0}^{N_{\alpha,0}} a_{\alpha,j}\tilde{C}_{1 - \alpha + jp_{\alpha}}(x)
  + \sum_{n = 1}^{N_{\alpha,1}} b_{\alpha,n}(\cos(nx) - 1),
\end{equation}
The values of \(a_{\alpha,j}\) for \(j \geq 1\) and \(p_{\alpha}\)
will be taken to make the defect small near \(x = 0\) and
\(b_{\alpha,n}\) is taken to make the defect small globally.

While the general form of the approximation is the same for all values
of \(\alpha\) the limits \(\alpha \to -1\) and \(\alpha \to 0\)
requires a slightly different approach. For that reason we split it
into three cases
\begin{enumerate}
\item \(\alpha \in (-1, -1 + \delta_{1}) = I_{1}\);
\item \(\alpha \in [-1 + \delta_{1}, -\delta_{2}] = I_{2}\);
\item \(\alpha \in (-\delta_{2}, 0) = I_{3}\),
\end{enumerate}
where the precise values of \(\delta_{1}\) and \(\delta_{2}\) will be
determined later. We start with the approximation for the interval
\(I_{2}\) which has the least amount of technical details. We then
discuss the alterations required to handle \(I_{1}\) and \(I_{3}\).

\subsection{Construction of \(u_{\alpha}\) for \(I_{2}\)}
\label{sec:construction-i_2}
The asymptotic behavior of the approximation \eqref{eq:u0} is
determined by \(a_{\alpha,0}\) and for it to agree with
Lemma~\ref{lemma:asymptotic-coefficient} we have to take
\begin{equation}
  \label{eq:a0}
  a_{\alpha,0} = \frac{2\Gamma(2\alpha)\cos\left(\pi\alpha\right)}{\Gamma(\alpha)^{2}\cos\left(\frac{\pi}{2}\alpha\right)^{2}}.
\end{equation}
The behavior of \(a_{\alpha,0}\) as a function in \(\alpha\) is shown
in Figure~\ref{fig:a0}.

To choose the other parameters we have to study the defect of the
approximation. The defect is given by
\(\Hop{\alpha}[u_{\alpha}] + \frac{1}{2}u_{\alpha}^{2}\), where we
have
\begin{equation}
  \label{eq:Hu0}
  \Hop{\alpha}[u_{\alpha}](x) = -\sum_{j = 0}^{N_{\alpha,0}}a_{\alpha,j}\tilde{C}_{1 - 2\alpha + jp_{\alpha}}(x) - \sum_{n = 1}^{N_{\alpha,1}}b_{\alpha,n}n^{\alpha}(\cos(nx) - 1).
\end{equation}

For the defect close to \(x = 0\) we want to study the asymptotic
expansion. We get the asymptotic behavior of \(u_{\alpha}\) and
\(\Hop{\alpha}[u_{\alpha}]\) in the following lemma, whose proof we
omit since it follows directly from the expansions of the Clausen and
cosine functions.
\begin{lemma}
  \label{lemma:u0-asymptotic}
  Let \(u_{\alpha}\) be as in \eqref{eq:u0}. Then the following
  asymptotic expansions hold near \(x = 0\):
  \begin{equation}
    \label{eq:u0-asymptotic}
    u_{\alpha}(x) = \sum_{j = 0}^{N_{\alpha,0}}a_{\alpha,j}^{0}|x|^{-\alpha + jp_{\alpha}}
    + \sum_{m = 1}^{\infty}\frac{(-1)^{m}}{(2m)!}\left(
      \sum_{j = 0}^{N_{\alpha,0}}a_{\alpha,j}\zeta(1 - \alpha + jp_{\alpha} - 2m)
      + \sum_{m = 1}^{N_{\alpha,1}}b_{\alpha,j}n^{2m}
      \right)x^{2m}
  \end{equation}
  \begin{equation}
    \label{eq:Hu0-asymptotic}
    \Hop{\alpha}[u_{\alpha}](x) = -\sum_{j = 0}^{N_{\alpha,0}}A_{\alpha,j}^{0}|x|^{-2\alpha + jp_{\alpha}}
    - \sum_{m = 1}^{\infty}\frac{(-1)^{m}}{(2m)!}\left(
      \sum_{j = 0}^{N_{\alpha,0}}a_{\alpha,j}\zeta(1 - 2\alpha + jp_{\alpha} - 2m)
      + \sum_{m = 1}^{N_{\alpha,1}}b_{\alpha,j}n^{2m + \alpha}
      \right)x^{2m}
  \end{equation}
  where
  \begin{align*}
    a_{\alpha,j}^{0} &= \Gamma(\alpha - jp_{\alpha})\cos\left(\frac{\pi}{2}(\alpha - jp_{\alpha})\right)a_{\alpha,j};\\
    A_{\alpha,j}^{0} &= \Gamma(2\alpha - jp_{\alpha})\cos\left(\frac{\pi}{2}(2\alpha - jp_{\alpha})\right)a_{\alpha,j}.
  \end{align*}
\end{lemma}
From this we can compute the expansion of
\(\Hop{\alpha}[u_{\alpha}] + \frac{1}{2}u_{\alpha}^{2}\), given in the
following Lemma. Again we omit the proof which only involves standard
calculations.
\begin{lemma}
  \label{lemma:Du0-asymptotic}
  Let \(u_{\alpha}\) be as in \eqref{eq:u0}. Then we have the
  following asymptotic expansion near \(x = 0\)
  \begin{multline*}
    \Hop{\alpha}[u_{\alpha}] + \frac{1}{2}u_{\alpha}^{2} =
    \left(\frac{1}{2}(a_{\alpha,0}^{0})^{2} - A_{\alpha,0}^{0}\right)|x|^{-2\alpha}
    + \left(a_{\alpha,0}^{0}a_{\alpha,1}^{0} - A_{\alpha,1}^{0}\right)|x|^{-2\alpha + p_{\alpha}}\\
        +\sum_{k = 2}^{N_{\alpha,0}}\left(\frac{1}{4}((-1)^{k} + 1)(a_{\alpha,\lfloor\frac{k}{2}\rfloor}^{0})^{2}
      + \sum_{j = 0}^{\lfloor\frac{k - 1}{2}\rfloor}a_{\alpha,j}^{0}a_{\alpha,k - j}^{0} - A_{\alpha,k}^{0}\right)|x|^{-2\alpha + kp_{\alpha}}\\
    +\sum_{k = N_{\alpha,0} + 1}^{2N_{\alpha,0}}\left(\frac{1}{4}((-1)^{k} + 1)(a_{\alpha,\lfloor\frac{k}{2}\rfloor}^{0})^{2}
      + \sum_{j = 0}^{max(0, \lfloor\frac{k - 1}{2}\rfloor)}a_{\alpha,j}^{0}a_{\alpha,k - j}^{0}\right)|x|^{-2\alpha + kp_{\alpha}}\\
    + \left(\sum_{j = 0}^{N_{\alpha,0}}a_{\alpha,j}^{0}|x|^{-\alpha + jp_{\alpha}}\right)S_{1}
    + \frac{1}{2}S_{1}^{2}
    - S_{2}
  \end{multline*}
  where
  \begin{align*}
    S_{1} &= \sum_{m = 1}^{\infty}\frac{(-1)^{m}}{(2m)!}\left(
            \sum_{j = 0}^{N_{\alpha0}}a_{\alpha,j}\zeta(1 - \alpha + jp_{\alpha} - 2m)
            + \sum_{n = 1}^{N_{\alpha,1}}b_{\alpha,n}n^{2m}
            \right)x^{2m},\\
    S_{2} &= \sum_{m = 1}^{\infty}\frac{(-1)^{m}}{(2m)!}\left(
            \sum_{j = 0}^{N_{\alpha,0}}a_{\alpha,j}\zeta(1 - 2\alpha + jp_{\alpha} - 2m)
            + \sum_{n = 1}^{N_{\alpha,1}}b_{\alpha,n}n^{2m + \alpha}
            \right)x^{2m}.
  \end{align*}
\end{lemma}
Taking \(a_{\alpha,0}\) as in Equation \eqref{eq:a0} makes the leading
coefficient \(\frac{1}{2}(a_{\alpha,0}^{0})^{2} - A_{\alpha,0}^{0}\)
zero. If \(p_{\alpha}\) is not too large,
\(-\alpha + p_{\alpha} < 2\), the next leading term is
\begin{equation*}
  \left(a_{\alpha,0}^{0}a_{\alpha,1}^{0} - A_{\alpha,1}^{0}\right)|x|^{-2\alpha + p_{\alpha}}
\end{equation*}
where the coefficient is given by
\begin{equation*}
  a_{\alpha,0}^{0}a_{\alpha,1}^{0} - A_{\alpha,1}^{0} =
  \left(
    \Gamma(\alpha)\cos\left(\frac{\pi}{2}\alpha\right)
    \Gamma(\alpha - p_{\alpha})\cos\left(\frac{\pi}{2}(\alpha - p_{\alpha})\right)a_{\alpha,0}
    - \Gamma(2\alpha - p_{\alpha})\cos\left(\frac{\pi}{2}(2\alpha - p_{\alpha})\right)
    \right)a_{\alpha,1}.
\end{equation*}
and the only way for this to be zero is to either take \(a_{\alpha,1} = 0\)
or to choose \(p_{\alpha}\) such that
\begin{equation*}
  \Gamma(\alpha)\cos\left(\frac{\pi}{2}\alpha\right)
  \Gamma(\alpha - p_{\alpha})\cos\left(\frac{\pi}{2}(\alpha - p_{\alpha})\right)a_{\alpha,0}
  - \Gamma(2\alpha - p_{\alpha})\cos\left(\frac{\pi}{2}(2\alpha - p_{\alpha})\right) = 0.
\end{equation*}
Inserting the value of \(a_{\alpha,0}\) and rearranging the terms, this can
be rewritten as
\begin{equation}
  \label{eq:p-alpha}
  \frac{
    \Gamma(2\alpha - p_{\alpha})\cos\left(\frac{\pi}{2}(2\alpha - p_{\alpha})\right)
  }{
    \Gamma(\alpha - p_{\alpha})\cos\left(\frac{\pi}{2}(\alpha - p_{\alpha})\right)
  }
  = \frac{
    2\Gamma(2\alpha)\cos(\pi\alpha)
  }{
    \Gamma(\alpha)\cos\left(\frac{\pi}{2}\alpha\right)
  }.
\end{equation}
This tells us how to pick the parameter \(p_{\alpha}\), we take it to
be the smallest positive solution of \eqref{eq:p-alpha}. A plot of
\(p_{\alpha}\) as a function of \(\alpha\) is given in
Figure~\ref{fig:p-alpha}.

\begin{figure}
  \centering
  \begin{subfigure}[t]{0.45\textwidth}
    \includegraphics[width=\textwidth]{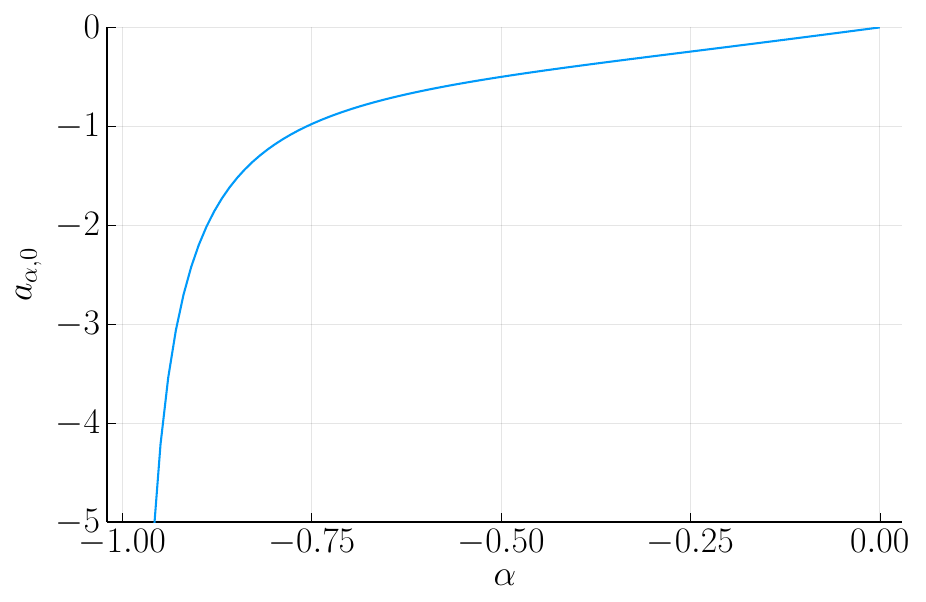}
    \caption{}
    \label{fig:a0}
  \end{subfigure}
  \begin{subfigure}[t]{0.45\textwidth}
    \includegraphics[width=\textwidth]{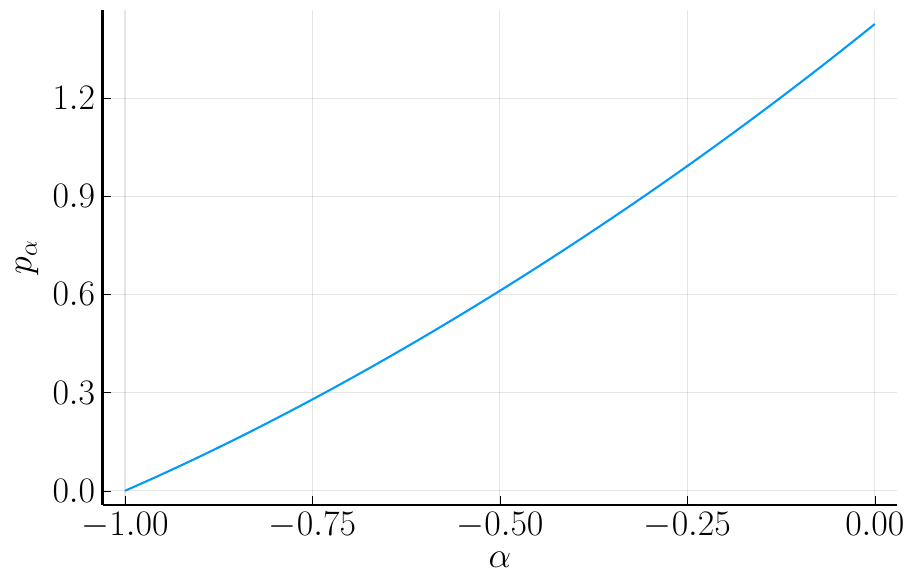}
    \caption{}
    \label{fig:p-alpha}
  \end{subfigure}
  \caption{Values of \(a_{\alpha,0}\), defined by Equation
    \eqref{eq:a0}, and \(p_{\alpha}\), defined by Equation
    \eqref{eq:p-alpha}, as functions of \(\alpha\).}
\end{figure}

The expansion in Lemma~\ref{lemma:Du0-asymptotic} also tells us how to
choose \(\{a_{\alpha,j}\}_{1 \leq j \leq N_{\alpha,0}}\): identify the
\(N_{\alpha,0}\) leading coefficients in the expansion, ignoring
\(|x|^{-2\alpha}\) and \(|x|^{-2\alpha + p_{\alpha}}\) whose
coefficients are zero from the choice of \(a_{\alpha,0}\) and
\(p_{\alpha}\), and take
\(\{a_{\alpha,j}\}_{1 \leq j \leq N_{\alpha,0}}\) to make them zero.
There are three potential issues with this approach for choosing
\(\{a_{\alpha,j}\}_{1 \leq j \leq N_{\alpha,0}}\)
\begin{enumerate}
\item The coefficients in the expansion might also depend on
  \(b_{\alpha,n}\) which we have yet to determine;
\item It is not clear if we indeed can pick \(a_{\alpha,j}\) to make
  these coefficients zero;
\item Even if we can pick \(a_{\alpha,j}\) to make the coefficients
  zero it is not clear how we would find these values.
\end{enumerate}
The first point is handled by simply ignoring the values of
\(b_{\alpha,n}\) at this stage, taking them as zero. This means that
some of the coefficients in the expansion might not be zero in the
end, once the values for \(b_{\alpha,n}\) are taken into account. This
however turns out to not be an issue, which is related to how to
handle the other two points. We summarize it in the following remark

\begin{remark}
  \label{remark:approximate-solution}
  Recall that \(u_{\alpha}\) is only an approximate solution to the
  equation. In the end the only important part is that the defect is
  sufficiently small. This means that we don't need to take exact
  values for \(p_{\alpha}\),
  \(\{a_{\alpha,j}\}_{1 \leq j \leq N_{\alpha,0}}\) or
  \(\{b_{\alpha,n}\}_{1 \leq n \leq N_{\alpha,1}}\). We are free to
  use any means, in particular non-rigorous numerical methods, for
  finding values which gives us a small defect. That the defect indeed
  is small is then later rigorously verified.

  Note that this does not apply to the parameter \(a_{\alpha,0}\). For
  the defect to be bounded near \(x = 0\) we need the leading
  coefficient in the expansion to be exactly zero, it is therefore not
  enough to use a numerical approximation of \(a_{\alpha,0}\).
\end{remark}

I light of this remark, points 2 and 3 above can be dealt with by
setting up the problem as a non-linear system of equations and finding
a solution using non-rigorous numerical methods. The numerical
procedure might or might not converge and the solution it gives might
or might not correspond to a true solution, but since we only need
approximate value this doesn't matter. That the solutions we get are
good enough will be verified when bounding the defect. Similarly, we
don't need to take \(p_{\alpha}\) to be an exact solution of
\eqref{eq:p-alpha}, a numerical approximation of the solution
suffices.

Once \(p_{\alpha}\) and
\(\{a_{\alpha,j}\}_{1 \leq j \leq N_{\alpha,0}}\) have been picked in
the above way, to make the asymptotic defect small, we wish to find
\(\{b_{\alpha,n}\}_{1 \leq n \leq N_{\alpha,1}}\) to make the defect
small globally. This is done by taking \(N_{\alpha,1}\) equally spaced
points \(\{x_{n}\}_{1 \leq n \leq N_{\alpha,1}}\) on the interval
\((0, \pi)\) and numerically solving the non-linear system
\begin{equation*}
  \Hop{\alpha}[u_{\alpha}](x_{n}) + \frac{1}{2}u_{\alpha}(x_{n})^{2} = 0 \quad \text{ for } 1 \leq n \leq N_{\alpha,1}.
\end{equation*}

So far we have not said anything about the values for \(N_{\alpha,0}\)
and \(N_{\alpha,1}\). These are tuning parameters, higher values will
typically lead to a smaller defect but also to a higher computational
cost. For \(N_{\alpha,0}\) this is only true up to some point where
the system for determining the coefficients starts to become
ill-conditioned. The values are chosen to make sure that the
inequality in Proposition~\ref{prop:contraction} holds and depends on
\(\alpha\). Some more details are given in
Section~\ref{sec:bounds-I-2}.

\subsection{Construction of \(u_{\alpha}\) for \(I_1\)}
\label{sec:construction-i_1}
The approximation in the previous section encounters several issues as
\(\alpha\) moves close to \(-1\). To begin with \(a_{\alpha,0}\)
diverges towards negative infinity, as can be seen in
Figure~\ref{fig:a0}. Furthermore \(p_{\alpha}\) tends towards zero, as
can be seen in Figure~\ref{fig:p-alpha}, which means that the
parameters for all the Clausen functions converge towards the same
value.

We can get an idea for how to handle these issues by numerically
studying the behavior of
\begin{equation}
  \label{eq:u0-i_1-clausen_sum}
  \sum_{j = 0}^{N_{\alpha,0}}a_{\alpha,j}\tilde{C}_{1 - \alpha + jp_{\alpha}}(x),
\end{equation}
as \(\alpha\) goes to \(-1\). The first important observation is that
as \(\alpha \to -1\) the parameter \(a_{\alpha,1}\) goes to infinity
in such a way that \(a_{\alpha,0} + a_{\alpha,1}\) remains bounded,
while \(\{a_{\alpha,j}\}_{2 \leq j \leq N_{\alpha,0}}\) remains
individually bounded. The second important observation is that as
\(\alpha \to -1\) the parameter \(p_{\alpha}\) behaves like
\(1 + \alpha + (1 + \alpha)^{2} / 2 + \mathcal{O}((1 + \alpha)^{3})\).
This hints at a solution to the problem that
\(a_{\alpha,0}\tilde{C}_{1 - \alpha}(x)\) doesn't converge: by taking
\(a_{\alpha,1} = -a_{\alpha,0}\) and
\(p_{\alpha} = 1 + \alpha + (1 + \alpha)^{2} / 2\) the first two terms
in the above sum become
\begin{equation*}
  a_{\alpha,0}(\tilde{C}_{1 - \alpha}(x) - \tilde{C}_{2 + (1 + \alpha)^{2} / 2}(x)).
\end{equation*}
For which we have the following lemma.
\begin{lemma}
  \label{lemma:u0-i_1-main-limit}
  The function
  \begin{equation*}
    a_{\alpha,0}(\tilde{C}_{1 - \alpha}(x) - \tilde{C}_{2 + (1 + \alpha)^{2} / 2}(x)),
  \end{equation*}
  with \(a_{\alpha,0}\) as in \eqref{eq:a0}, converges to
  \(\frac{2}{\pi^{2}}\tilde{C}_{2}^{(1)}(x)\) pointwise in \(x\) as
  \(\alpha \to -1\). Moreover, for
  \(\alpha \in \left(-1, -\frac{3}{4}\right)\) it satisfies the
  inequality
  \begin{equation*}
    a_{\alpha,0}(\tilde{C}_{1 - \alpha}(x) - \tilde{C}_{2 + (1 + \alpha)^{2} / 2}(x)) \geq \frac{2}{\pi^{2}}\tilde{C}_{2}^{(1)}(x).
  \end{equation*}
\end{lemma}
\begin{proof}
  We have
  \begin{equation*}
    a_{\alpha,0}(\tilde{C}_{1 - \alpha}(x) - \tilde{C}_{2 + (1 + \alpha)^{2} / 2}(x))
    = -(1 + \alpha)a_{\alpha,0} \cdot
    \frac{\tilde{C}_{2 + (1 + \alpha)^{2} / 2}(x) - \tilde{C}_{1 - \alpha}(x)}{1 + \alpha}.
  \end{equation*}
  Taking the limit as \(\alpha \to -1\) for the first factor we get
  \(\lim_{\alpha \to - 1} -(1 + \alpha)a_{\alpha,0} =
  \frac{2}{\pi^{2}}.\) The second factor converges to the derivative
  \(\tilde{C}_{1 - \alpha}^{(1)}(x)\). This gives us the pointwise
  convergence.

  To prove the inequality it is enough to prove that the left-hand
  side is increasing in \(\alpha\). For that it is enough to prove
  that both factors are positive and increasing in \(\alpha\).

  The proof that \(-(1 + \alpha)a_{\alpha,0}\) is positive and
  increasing in \(\alpha\) is computer-assisted. It is done by
  rigorously computing a lower bound of the value and the derivative
  for \(-1 < \alpha < -\frac{3}{4}\) and asserting that these lower
  bounds are positive.

  For the Clausen-factor we let \(h = 1 + \alpha\) and expand the
  Clausen functions at \(s = 2\), giving us
  \begin{align*}
    \tilde{C}_{1 - \alpha}(x) &= \sum_{n = 0}^{\infty}\frac{\tilde{C}_{2}^{(n)}(x)}{n!}(-h)^{n},\\
    \tilde{C}_{1 + (1 + \alpha)^{2} / 2}(x) &= \sum_{n = 0}^{\infty}\frac{\tilde{C}_{2}^{(n)}(x)}{n!}
                                              \left(\frac{h^{2}}{2}\right)^{n}.
  \end{align*}
  From which we get
  \begin{equation}
    \label{eq:clausen-diff-expansion}
    \frac{\tilde{C}_{2 + (1 + \alpha)^{2} / 2}(x) - \tilde{C}_{1 - \alpha}(x)}{1 + \alpha}
    = \sum_{n = 1}^{\infty}\frac{\tilde{C}_{2}^{(n)}(x)}{n!}h^{n - 1}
    \left(\frac{h^{n}}{2^{n}} + (-1)^{n - 1}\right).
  \end{equation}
  For \(-1 < \alpha < 0\) we have
  \(\left|\frac{h^{n}}{2^{n}}\right| < 1\) and the sign of the factor
  \(\frac{h}{2^{n}} + (-1)^{n - 1}\) is therefore the same as for
  \((-1)^{n - 1}\). Furthermore, we have
  \begin{equation*}
    \tilde{C}_{2}^{(n)}(x) = (-1)^{n}\sum_{k = 1}^{\infty}\frac{\cos(kx) - 1}{k^{2}}\log k
  \end{equation*}
  and since \(\cos(kx) - 1 \leq 0\) this has the same sign as
  \((-1)^{n - 1}\). It follows that all the terms in
  \eqref{eq:clausen-diff-expansion} are positive, hence the sum is
  positive and increasing in \(\alpha\) for \(-1 < \alpha < 0\).
\end{proof}

If we subtract
\(a_{\alpha,0}(\tilde{C}_{1 - \alpha}(x) - \tilde{C}_{2 + (1 +
  \alpha)^{2} / 2}(x))\) from \eqref{eq:u0-i_1-clausen_sum} the
remaining part of
\begin{equation*}
  (a_{\alpha,0} + a_{\alpha,1})\tilde{C}_{1 - \alpha + jp_{\alpha}}(x) + \sum_{j = 2}^{N_{\alpha,0}}a_{\alpha,j}\tilde{C}_{1 - \alpha + jp_{\alpha}}(x),
\end{equation*}
Numerically we can see that this converges to a bounded function as
\(\alpha\) goes towards \(-1\). Since we only need an approximation of
the limit we can simply fix some \(\alpha > -1\) and use the values
for this \(\alpha\).

More precisely we construct the approximation for \(I_{1}\) in the
following way. First fix \(\hat{\alpha}\) close to \(-1\) but such
that \(\alpha < \hat{\alpha}\) for all \(\alpha \in I_{1}\). Then
compute \(\{a_{\hat{\alpha},j}\}_{0 \leq j \leq N_{\hat{\alpha},1}}\)
and \(p_{\hat{\alpha}}\) using the approach in
Section~\ref{sec:construction-i_2}. We then take
\begin{multline}
  \label{eq:u0-I-1}
  u_{\alpha}(x) = a_{\alpha,0}(\tilde{C}_{1 - \alpha}(x) - \tilde{C}_{2 + (1 + \alpha)^{2} / 2}(x))
  + (a_{\hat{\alpha},0} + a_{\hat{\alpha},1})\tilde{C}_{1 - \hat{\alpha} + p_{\hat{\alpha}}}(x)\\
  + \sum_{j = 2}^{N_{\hat{\alpha},0}}a_{\hat{\alpha},j}\tilde{C}_{1 - \hat{\alpha} + jp_{\hat{\alpha}}}(x)
  + \sum_{n = 1}^{N_{-1,1}}b_{-1,n}(\cos(nx) - 1)
\end{multline}
The parameters \(\{b_{-1,n}\}_{1 \leq n \leq N_{-1,1}}\) are
determined by considering the defect at the limit \(\alpha \to -1\),
where we from Lemma~\ref{lemma:u0-i_1-main-limit} have
\begin{equation}
  \label{eq:u0-bh}
  u_{-1}(x) = \frac{2}{\pi^{2}}\tilde{C}_{2}^{(1)}(x)
  + (a_{\hat{\alpha},0} + a_{\hat{\alpha},1})\tilde{C}_{1 - \hat{\alpha} + p_{\hat{\alpha}}}(x)
  + \sum_{j = 2}^{N_{\hat{\alpha},0}}a_{\hat{\alpha},j}\tilde{C}_{1 - \hat{\alpha} + jp_{\hat{\alpha}}}(x)
  + \sum_{n = 1}^{N_{-1,1}}b_{-1,n}(\cos(nx) - 1).
\end{equation}
In the same way as in Section~\ref{sec:construction-i_2} we take
\(\{b_{-1,n}\}_{1 \leq n \leq N_{-1,1}}\) so that the defect is
minimized on \(N_{-1,1}\) equally spaced points on the interval
\((0, \pi)\). Note that by Lemma~\ref{lemma:u0-i_1-main-limit}
\(u_{-1}(x)\) gives a lower bound of \(u_{\alpha}(x)\).

With this approximation we get
\begin{multline}
  \label{eq:Hu0-I-1}
  \Hop{\alpha}[u_{\alpha}](x) = -a_{\alpha,0}(\tilde{C}_{1 - 2\alpha}(x) - \tilde{C}_{2 - \alpha + (1 + \alpha)^{2} / 2}(x))
  + (a_{\hat{\alpha},0} + a_{\hat{\alpha},1})\tilde{C}_{1 - \alpha - \hat{\alpha} + p_{\hat{\alpha}}}(x)\\
  - \sum_{j = 2}^{N_{\hat{\alpha},0}}a_{\hat{\alpha}, j}\tilde{C}_{1 - \alpha - \hat{\alpha} + jp_{\hat{\alpha}}}(x)
  - \sum_{n = 1}^{N_{-1,1}}b_{-1,n}n^{-\alpha}(\cos(nx) - 1).
\end{multline}
Note that since \(a_{\alpha,0} \to -\infty\) as \(\alpha \to -1\) we
cannot compute a finite enclosure of \(a_{\alpha,0}\) valid for an
interval of the form \((-1, -1 + \delta)\). When computing enclosures
we therefore have to treat
\begin{equation*}
  a_{\alpha,0}(\tilde{C}_{1 - \alpha}(x) - \tilde{C}_{2 + (1 + \alpha)^{2} / 2}(x))
  \quad\text{and}\quad
  a_{\alpha,0}(\tilde{C}_{1 - 2\alpha}(x) - \tilde{C}_{2 - \alpha + (1 + \alpha)^{2} / 2}(x))
\end{equation*}
as one part, and explicitly handle the removable singularity.

The asymptotic behavior are given in the following lemma, again we
omit the proof since it follows directly from the expansions of the
Clausen functions.
\begin{lemma}
  \label{lemma:u0-asymptotic-I-1}
  The approximation for \(\alpha \in I_{1}\), with \(u_{\alpha}\)
  given by \eqref{eq:u0-I-1} and \(\Hop{\alpha}[u_{\alpha}]\) given by
  \eqref{eq:Hu0-I-1} has the asymptotic expansions
  \begin{align*}
    u_{\alpha}(x) =& a_{\alpha,0}\left(
                     \Gamma(\alpha)\cos\left(\frac{\pi}{2}\alpha\right)
                     - \Gamma(-1 - (1 + \alpha)^{2} / 2)\cos\left(\frac{\pi}{2}(1 + (1 + \alpha)^{2} / 2)\right)|x|^{1 + \alpha + (1 + \alpha)^{2} / 2}
                     \right)|x|^{-\alpha}\\
                   &+ a_{\alpha,0}\sum_{m = 1}^{\infty}\frac{(-1)^{m}}{(2m)!}\left(
                     \zeta(1 - \alpha - 2m) - \zeta(2 + (1 + \alpha)^{2} / 2 - 2m)
                     \right)x^{2m}\\
                   &+ \sum_{j = 1}^{N_{\hat{\alpha},0}}\hat{a}_{\alpha,j}^{0}|x|^{-\hat{\alpha} + jp_{\hat{\alpha}}}
                     + \sum_{m = 1}^{\infty}\frac{(-1)^{m}}{(2m)!}\left(
                     \sum_{j = 1}^{N_{\hat{\alpha},0}}a_{\hat{\alpha},j}\zeta(1 - \hat{\alpha} + jp_{\hat{\alpha}} - 2m)
                     + \sum_{m = 1}^{N_{-1,1}}b_{\hat{\alpha},j}n^{2m}
                     \right)x^{2m}
  \end{align*}
  and
    \begin{align*}
      \Hop{\alpha}u_{\alpha}(x) =& -a_{\alpha,0}\left(
                                   \Gamma(2\alpha)\cos\left(\pi\alpha\right)
                                   - \Gamma(-1 + \alpha - (1 + \alpha)^{2} / 2)\cos\left(\frac{\pi}{2}(1 - \alpha + (1 + \alpha)^{2} / 2)\right)|x|^{1 + \alpha + (1 + \alpha)^{2} / 2}
                                   \right)|x|^{-2\alpha}\\
                   &- a_{\alpha,0}\sum_{m = 1}^{\infty}\frac{(-1)^{m}}{(2m)!}\left(
                     \zeta(1 - 2\alpha - 2m) - \zeta(2 - \alpha + (1 + \alpha)^{2} / 2 - 2m)
                     \right)x^{2m}\\
                   &- \sum_{j = 1}^{N_{\hat{\alpha},0}}\hat{A}_{\alpha,j}^{0}|x|^{-\alpha - \hat{\alpha} + jp_{\hat{\alpha}}}
                     - \sum_{m = 1}^{\infty}\frac{(-1)^{m}}{(2m)!}\left(
                     \sum_{j = 1}^{N_{\hat{\alpha},0}}a_{\hat{\alpha},j}\zeta(1 - \alpha - \hat{\alpha} + jp_{\hat{\alpha}} - 2m)
                     + \sum_{m = 1}^{N_{-1,1}}b_{\hat{\alpha},j}n^{2m + \alpha}
                     \right)x^{2m},
  \end{align*}
  where
  \begin{align*}
    \hat{a}_{\alpha,j}^{0} &= \Gamma(\hat{\alpha} - jp_{\hat{\alpha}})\cos\left(\frac{\pi}{2}(\hat{\alpha} - jp_{\hat{\alpha}})\right)a_{\hat{\alpha},j};\\
    \hat{A}_{j}^{0} &= \Gamma(\alpha + \hat{\alpha} - jp_{\hat{\alpha}})\cos\left(\frac{\pi}{2}(\alpha + \hat{\alpha} - jp_{\hat{\alpha}})\right)a_{\hat{\alpha},j}.
  \end{align*}
\end{lemma}
There are a couple of things that make these expansions tricky to
handle. To begin with the terms
\(a_{\alpha,0}(\zeta(1 - \alpha - 2m) - \zeta(2 + (1 + \alpha)^{2} / 2
- 2m))\) in \(u_{\alpha}\) have removable singularities at
\(\alpha = -1\), this is handled using the approach described in
Appendix~\ref{sec:removable-singularities}. For
\(\Hop{\alpha}[u_{\alpha}]\) the exponent for the term
\begin{equation*}
  -a_{\alpha,0}\left(
    \Gamma(2\alpha)\cos\left(\pi\alpha\right)
    - \Gamma(-1 + \alpha - (1 + \alpha)^{2} / 2)\cos\left(\frac{\pi}{2}(1 - \alpha + (1 + \alpha)^{2} / 2)\right)|x|^{1 + \alpha + (1 + \alpha)^{2} / 2}
  \right)|x|^{-2\alpha}
\end{equation*}
approaches \(2\) and its interaction with the \(x^{2}\) term
\begin{equation*}
  -\frac{1}{2}a_{\alpha,0}\left(
    \zeta(1 - 2\alpha - 2) - \zeta(2 - \alpha + (1 + \alpha)^{2} / 2 - 2)
  \right)x^{2}
\end{equation*}
needs special care. A similar thing happens for
\(\hat{A}_{\alpha,j}^{0}|x|^{-\alpha - \hat{\alpha} +
  jp_{\hat{\alpha}}}\) and the corresponding \(x^{2}\) terms. How this
is handled is discussed in detail in
Section~\ref{sec:evaluation-F-I-1} and
Appendix~\ref{sec:evaluation-F-I-1-asymptotic}. Bounding the tails of
the sum
\begin{equation*}
  \sum_{m = 1}^{\infty}\frac{(-1)^{m}}{(2m)!}a_{\alpha,0}\left(
    \zeta(1 - \alpha - 2m) - \zeta(2 + (1 + \alpha)^{2} / 2 - 2m)
  \right)x^{2m}
\end{equation*}
and the corresponding one for \(\Hop{\alpha}[u_{\alpha}]\) requires a
bit extra work, this is done in Lemma~\ref{lemma:bound-tail-I-1}
and~\ref{lemma:clausen-derivative-tails-alt}.

For the precise values of \(\hat{\alpha}\), \(N_{\hat{\alpha},0}\) and
\(N_{-1,1}\) see Section~\ref{sec:bounds-I-1}.

\subsection{Construction of \(u_{\alpha}\) for \(I_3\)}
\label{sec:construction-i_3}
Using the approximation from Section~\ref{sec:construction-i_2} and
letting \(\alpha\) tend to \(0\) it turns out that we need fewer and
fewer terms to get a small defect. For \(\alpha\) close to \(0\)
(around \(-1 / 6\) or closer) it is enough to take
\(N_{\alpha,0} = 2\) and \(N_{\alpha,1} = 0\), giving us the
approximation
\begin{equation}
  \label{eq:u0-I-3}
  u_{\alpha}(x) = \sum_{j = 0}^{2}a_{\alpha,j}\tilde{C}_{1 - \alpha + jp_{\alpha}}(x)
\end{equation}
and
\begin{equation}
  \label{eq:Hu0-I-3}
  \Hop{\alpha}[u_{\alpha}](x) = -\sum_{j = 0}^{2}a_{\alpha,j}\tilde{C}_{1 - 2\alpha + jp_{\alpha}}(x).
\end{equation}
With this approximation the defect tends to zero as \(\alpha \to 0\),
which is good. However, in Section~\ref{sec:evaluation-T-I-3} we will
see that the value of \(D_{\alpha}\) in
Proposition~\ref{prop:contraction} tends to \(1\), meaning that
\((1 - D_{\alpha})^{2} / 4n_{\alpha}\) also tends to zero. This means
that we can never compute an enclosure for \(D_{\alpha}\) and
\(\delta_{\alpha}\) on any interval \((-\delta_{2}, 0)\) such that the
inequality in Proposition~\ref{prop:contraction} holds. Instead, we
compute expansions in \(\alpha\) at \(\alpha = 0\) of both
\(D_{\alpha}\) and \(\delta_{\alpha}\), using these expansions we can
then prove that the inequality holds on the interval \(I_{3}\).

In this case it is not enough to use numerical approximations of
\(a_{\alpha,1}\), \(a_{\alpha,2}\) and \(p_{\alpha}\), we need to take
into account their dependence on \(\alpha\). For \(p_{\alpha}\) we use
the defining equation \eqref{eq:p-alpha}. As \(\alpha \to 0\) we have
\(p_{\alpha} \to 1.426\dots\), this gives us that the four leading
terms (in \(x\)) in the expansion from
Lemma~\ref{lemma:Du0-asymptotic} are given by, ordered by their
exponents,
\begin{gather*}
  \left(\frac{1}{2}(a_{\alpha,0}^{0})^{2} - A_{\alpha,0}^{0}\right)|x|^{-2\alpha},\\
  \left(a_{\alpha,0}^{0}a_{\alpha,1}^{0} - A_{\alpha,1}^{0}\right)|x|^{-2\alpha + p_{\alpha}},\\
  \frac{1}{2}\left(\sum_{j = 0}^{2}a_{\alpha,j}\zeta(-1 - 2\alpha + jp_{\alpha})\right)x^{2},\\
  -\frac{1}{2}a_{\alpha,0}^{0}\left(\sum_{j = 0}^{2}a_{\alpha,j}\zeta(-1 - \alpha + jp_{\alpha})\right)x^{2 - \alpha}.
\end{gather*}
The choice of \(a_{\alpha,0}\) makes the first term zero and the
choice of \(p_{\alpha}\) makes the second term zero. We then want to
take \(a_{\alpha,1}\) and \(a_{\alpha,2}\) to make the third and
fourth term zero. Solving for \(a_{\alpha,1}\) and \(a_{\alpha,2}\) we
get the linear system
\begin{equation*}
  \begin{cases}
    a_{\alpha,1}\zeta(-1 - 2\alpha + p_{\alpha}) + a_{\alpha,2}\zeta(-1 - 2\alpha + 2p_{\alpha}) &= -a_{\alpha,0}\zeta(-1 - 2\alpha)\\
    a_{\alpha,1}\zeta(-1 - \alpha + p_{\alpha}) + a_{\alpha,2}\zeta(-1 - \alpha + 2p_{\alpha}) &= -a_{\alpha,0}\zeta(-1 - \alpha)
  \end{cases},
\end{equation*}
with the solution
\begin{equation}
  \label{eq:a1-a2-I-3}
  \begin{cases}
    a_{\alpha,1} &= a_{\alpha,0}\frac{
      \zeta(-1 - \alpha)\zeta(-1 - 2\alpha + 2p_{\alpha}) - \zeta(-1 - 2\alpha)\zeta(-1 - \alpha + 2p_{\alpha})
    }{
      \zeta(-1 - 2\alpha + p_{\alpha})\zeta(-1 - \alpha + 2p_{\alpha}) - \zeta(-1 - 2\alpha + 2p_{\alpha})\zeta(-1 - \alpha + p_{\alpha})
    }\\
    a_{\alpha,2} &= -a_{\alpha,0}\frac{
      \zeta(-1 - \alpha)\zeta(-1 - 2\alpha + p_{\alpha}) - \zeta(-1 - 2\alpha)\zeta(-1 - \alpha + p_{\alpha})
    }{
      \zeta(-1 - 2\alpha + p_{\alpha})\zeta(-1 - \alpha + 2p_{\alpha}) - \zeta(-1 - 2\alpha + 2p_{\alpha})\zeta(-1 - \alpha + p_{\alpha})
    }
  \end{cases}.
\end{equation}

The following lemma gives us information about the asymptotic
behavior of \(a_{\alpha,j}\).
\begin{lemma}
  \label{lemma:a-expansion-alpha-I-3}
  The expansion of \(a_{\alpha,0}\) at \(\alpha = 0\) is given by
  \begin{equation*}
    a_{\alpha,0} = \alpha + \mathcal{O}(\alpha^{3}).
  \end{equation*}
  For \(a_{\alpha,1}\) and \(a_{\alpha,2}\) it is given by
  \begin{equation*}
    a_{\alpha,j} = a_{\alpha,j,1}\alpha + \mathcal{O}(\alpha^{2})
  \end{equation*}
  with
  \begin{equation*}
    a_{\alpha,1,1} = \lim_{\alpha \to 0} \frac{
      \zeta(-1 - \alpha)\zeta(-1 - 2\alpha + 2p_{\alpha}) - \zeta(-1 - 2\alpha)\zeta(-1 - \alpha + 2p_{\alpha})
    }{
      \zeta(-1 - 2\alpha + p_{\alpha})\zeta(-1 - \alpha + 2p_{\alpha}) - \zeta(-1 - 2\alpha + 2p_{\alpha})\zeta(-1 - \alpha + p_{\alpha})
    }
  \end{equation*}
  and
  \begin{equation*}
    a_{\alpha,2,1} = \lim_{\alpha \to 0} \frac{
      \zeta(-1 - 2\alpha)\zeta(-1 - \alpha + p_{\alpha}) - \zeta(-1 - \alpha)\zeta(-1 - \alpha + 2p_{\alpha})
    }{
      \zeta(-1 - 2\alpha + p_{\alpha})\zeta(-1 - \alpha + 2p_{\alpha}) - \zeta(-1 - 2\alpha + 2p_{\alpha})\zeta(-1 - \alpha + p_{\alpha})
    }.
  \end{equation*}
\end{lemma}
\begin{proof}
  The expansion for \(a_{\alpha,0}\) follows from
  Equation~\eqref{eq:a0} by using that
  \(\Gamma(\alpha) = \alpha^{-1} + \gamma + \mathcal{O}(\alpha)\),
  giving us
  \begin{equation*}
    a_{\alpha,0} = \frac{
      2\left((2\alpha)^{-1} + \gamma + \mathcal{O}(\alpha)\right)(1 + \mathcal{O}(\alpha^{2}))
    }{
      (\alpha^{-1} + \gamma + \mathcal{O}(\alpha))^{2}(1 + \mathcal{O}(\alpha^{2}))
    }
    = \frac{\alpha^{-1} + 2\gamma + \mathcal{O}(\alpha)}{\alpha^{-2} + 2\gamma\alpha^{-1} + \mathcal{O}(1)}
    = \alpha\frac{1 + 2\gamma\alpha + \mathcal{O}(\alpha^{2})}{1 + 2\gamma\alpha + \mathcal{O}(\alpha^{2})}
    = \alpha + \mathcal{O}(\alpha^{3}).
  \end{equation*}

  The expansions for \(a_{\alpha,1}\) and \(a_{\alpha,2}\) follows
  directly from Equation~\eqref{eq:a1-a2-I-3}.
\end{proof}
Using this we can get information about the asymptotic behavior of
\(u_{\alpha}\) and \(\Hop{\alpha}[u_{\alpha}]\).
\begin{lemma}
  \label{lemma:u0-expansion-alpha-I-3}
  For \(x \in (0, \pi]\) we have the following expansions at
  \(\alpha = 0\):
  \begin{align*}
    u_{\alpha}(x) &= 1 + b(x)\alpha + \mathcal{O}(\alpha^{2}),\\
    \Hop{\alpha}[u_{\alpha}](x) &= -\frac{1}{2} - b(x)\alpha + \mathcal{O}(\alpha^{2}),
  \end{align*}
  with
  \begin{equation*}
    b(x) = C_{1}(x) - \gamma + a_{\alpha,1,1}\tilde{C}_{1 + p_{0}}(x) + a_{\alpha,2,1}\tilde{C}_{1 + 2p_{0}}(x).
  \end{equation*}
\end{lemma}
\begin{proof}
  For \(j \geq 1\) the Clausen functions are all finite at
  \(\alpha = 0\) and from Lemma~\ref{lemma:a-expansion-alpha-I-3} we
  therefore get
  \begin{align*}
    a_{\alpha,1}\tilde{C}_{1 - \alpha + p_{\alpha}}(x) &= a_{\alpha,1,1}\tilde{C}_{1 + p_{0}}(x)\alpha
                                                         + \mathcal{O}(\alpha^{2}),\\
    a_{\alpha,2}\tilde{C}_{1 - \alpha + 2p_{\alpha}}(x) &= a_{\alpha,2,1}\tilde{C}_{1 + 2p_{0}}(x)\alpha
                                                         + \mathcal{O}(\alpha^{2}),\\
    a_{\alpha,1}\tilde{C}_{1 - 2\alpha + p_{\alpha}}(x) &= a_{\alpha,1,1}\tilde{C}_{1 + p_{0}}(x)\alpha
                                                         + \mathcal{O}(\alpha^{2}),\\
    a_{\alpha,2}\tilde{C}_{1 - 2\alpha + p_{\alpha}}(x) &= a_{\alpha,2,1}\tilde{C}_{1 + 2p_{0}}(x)\alpha
                                                         + \mathcal{O}(\alpha^{2}).
  \end{align*}

  For \(j = 0\) the function \(\tilde{C}_{1 - \alpha}(x)\) has a
  singularity at \(\alpha = 0\). More precisely
  \begin{equation*}
    \tilde{C}_{1 - \alpha}(x) = C_{1 - \alpha}(x) - \zeta(1 - \alpha),
  \end{equation*}
  where for \(0 < x < 2\pi\) the factor \(C_{1 - \alpha}(x)\) is
  finite, but \(\zeta(1 - \alpha)\) has a singularity at
  \(\alpha = 0\), similarly for \(\tilde{C}_{1 - 2\alpha}(x)\). After
  multiplying with \(a_{\alpha,0}\) the singularity is removable. To
  begin with we have
  \begin{equation*}
    a_{\alpha,0}C_{1 - \alpha}(x) = C_{1}(x)\alpha + \mathcal{O}(\alpha^{2}).
  \end{equation*}
  Using the Laurent series of the zeta function we have
  \begin{equation*}
    \zeta(1 - \alpha) = -\alpha^{-1} + \sum_{n = 0}^{\infty} \frac{1}{n!}\gamma_{n}\alpha^{n}
    = -\alpha^{-1} + \gamma + \mathcal{O}(\alpha),
  \end{equation*}
  where \(\gamma_{n}\) is the Stieltjes constant and
  \(\gamma = \gamma_{0}\) is Euler's constant, from which we get
  \begin{equation*}
    a_{\alpha,0}\zeta(1 - \alpha)
    = (\alpha + \mathcal{O}(\alpha^{3}))(-\alpha^{-1} + \gamma + \mathcal{O}(\alpha))
    = -1 + \gamma\alpha + \mathcal{O}(\alpha^{2}).
  \end{equation*}
  Giving us
  \begin{equation*}
    a_{\alpha,0}\tilde{C}_{1 - \alpha}(x)
    = 1 + (C_{1}(x) - \gamma)\alpha + \mathcal{O}(\alpha^{2}).
  \end{equation*}
  In the same way we get
  \begin{equation*}
    a_{\alpha,0}\tilde{C}_{1 - 2\alpha}(x)
    = \frac{1}{2} + (C_{1}(x) - \gamma)\alpha + \mathcal{O}(\alpha^{2}).
  \end{equation*}

  Combining all the above gives us
  \begin{equation*}
    u_{\alpha}(x) = 1
    + (C_{1}(x) - \gamma + \tilde{C}_{1 + p_{0}}(x) + \tilde{C}_{1 + 2p_{0}}(x))\alpha
    + \mathcal{O}(\alpha^{2})
  \end{equation*}
  and
  \begin{equation*}
    \Hop{\alpha}[u_{\alpha}](x) = -\frac{1}{2}
    - (C_{1}(x) - \gamma + \tilde{C}_{1 + p_{0}}(x) + \tilde{C}_{1 + 2p_{0}}(x))\alpha
    + \mathcal{O}(\alpha^{2}),
  \end{equation*}
  from which the result follows.
\end{proof}

\subsection{Hybrid cases}
\label{sec:hybrid-cases}

While it is technically possible to use the construction from
Section~\ref{sec:construction-i_2} to handle the entire interval
\(I_{2}\) it is not computationally feasible. As \(\alpha\) gets close
to either \(-1\) or \(0\) the bounds that needs to be computed becomes
harder and harder to handle.

For \(\alpha\) near \(0\) it is mainly bounding \(\delta_{\alpha}\)
that becomes hard to deal with. It goes to zero as \(\alpha \to 0\),
which means that absolute errors in the bound that are negligible when
\(\delta_{\alpha}\) is larger quickly grow to become very large
relative errors.

Near \(\alpha = -1\) it is mainly the behavior of \(F_{\alpha}(x)\)
and \(\mathcal{T}_{\alpha}(x)\) (from Equation~\eqref{eq:mathcal-T})
around \(x = 0\) that become problematic. This is related to the
powers for the terms \(\nu_{\alpha}|x|^{-\alpha}\) and
\(\mathcal{O}(|x|^{p})\) in Theorem~\ref{thm:main} getting closer and
closer to each other, giving us less space to work with. Another issue
is, as mentioned in Section~\ref{sec:construction-i_1}, that the two
Clausen functions \(a_{\alpha,0}\tilde{C}_{1 - \alpha}\) and
\(a_{\alpha,1}\tilde{C}_{1 - \alpha + p_{\alpha}}\) grow quickly in
opposite directions as \(\alpha\) gets closer to \(-1\) and as a
result there are large cancellations to handle.

For \(\alpha\) near \(0\) we use the approximation from
Equation~\eqref{eq:u0-I-3}, but instead of using expansions centered
at \(\alpha = 0\) we use expansions centered at \(\alpha\) (or the
midpoint, when computing with intervals). Compared to the case
\(\alpha = 0\) there are no removable singularities to handle, which
makes the process simpler.

For \(\alpha\) near \(-1\) we use the ideas from
Section~\ref{sec:construction-i_1} for handling \(\alpha \in I_{1}\).
Instead of using
\begin{equation*}
  u_{\alpha}(x) = \sum_{j = 0}^{N_{\alpha,0}} a_{\alpha,j}\tilde{C}_{1 - \alpha + jp_{\alpha}}(x)
  + \sum_{n = 1}^{N_{\alpha,1}} b_{\alpha,n}(\cos(nx) - 1),
\end{equation*}
from Equation~\eqref{eq:u0} we modify it in the same spirit as
Equation~\eqref{eq:u0-I-1}. More precisely we use
\begin{equation*}
  u_{\alpha}(x) = a_{\alpha,0}(\tilde{C}_{1 - \alpha}(x) - \tilde{C}_{1 - \alpha + p_{\alpha}}(x))
  + \hat{a}_{\alpha,1}\tilde{C}_{1 - \alpha + p_{\alpha}}(x)
  + \sum_{j = 2}^{N_{\alpha,0}} a_{\alpha,j}\tilde{C}_{1 - \alpha + jp_{\alpha}}(x)
  + \sum_{n = 1}^{N_{\alpha,1}} b_{\alpha,n}(\cos(nx) - 1).
\end{equation*}
If we take \(\hat{a}_{\alpha,1} = a_{\alpha,0} + a_{\alpha,1}\) this
is equivalent to Equation~\eqref{eq:u0-I-1} The benefit with this
approach is that we don't have to use an enclosure of \(a_{\alpha,0}\)
in \(\hat{a}_{\alpha,1}\), but can take a numerical approximation
instead.

In both cases we make adjustments to how the bounds are computed, the
details are given in the relevant sections.

\section{Choice of \(w_{\alpha}\)}
\label{sec:choice-weight}

In the above section we discussed how to construct the approximation
\(u_{\alpha}\) for the ansatz \eqref{eq:main-ansatz}. We here discuss
the choice of the weight \(w_{\alpha}\) in the ansatz. To get the
statement in Theorem~\ref{thm:main} we need
\(w_{\alpha}(x) = \mathcal{O}(|x|^{p})\) for some \(p\) satisfying
\(-\alpha < p \leq 1\). The natural choice is \(w_{\alpha}(x) = |x|\),
however this turns out to not work for all values of \(\alpha\).

The main obstruction when choosing \(w_{\alpha}\) is its effect on the
value of \(D_{\alpha}\). For Proposition~\ref{prop:contraction} to
apply we require that \(D_{\alpha} < 1\), if we take
\(w_{\alpha}(x) = |x|\) this is not always the case. Using the
procedure described in Section~\ref{sec:analysis-T}, in particular
lemma~\ref{lemma:evaluation-U-I-1}, we can compute a lower bound of
\(D_{\alpha}\) by computing \(\mathcal{T}_{\alpha}(0)\), introduced in
\eqref{eq:mathcal-T} and \eqref{eq:D}. A plot of
\(\mathcal{T}_{\alpha}(0)\) as a function of \(\alpha\), using
\(w_{\alpha}(x) = |x|\), is given in Figure~\ref{fig:T-alpha-0}. From
this plot we can see that with the weight \(w_{\alpha}(x) = |x|\) we
have \(D_{\alpha} > 1\) for \(\alpha\) between \(-1\) and some value
slightly smaller than \(-0.5\), in which case
Proposition~\ref{prop:contraction} won't apply.

\begin{figure}
  \centering
  \includegraphics[width=0.5\textwidth]{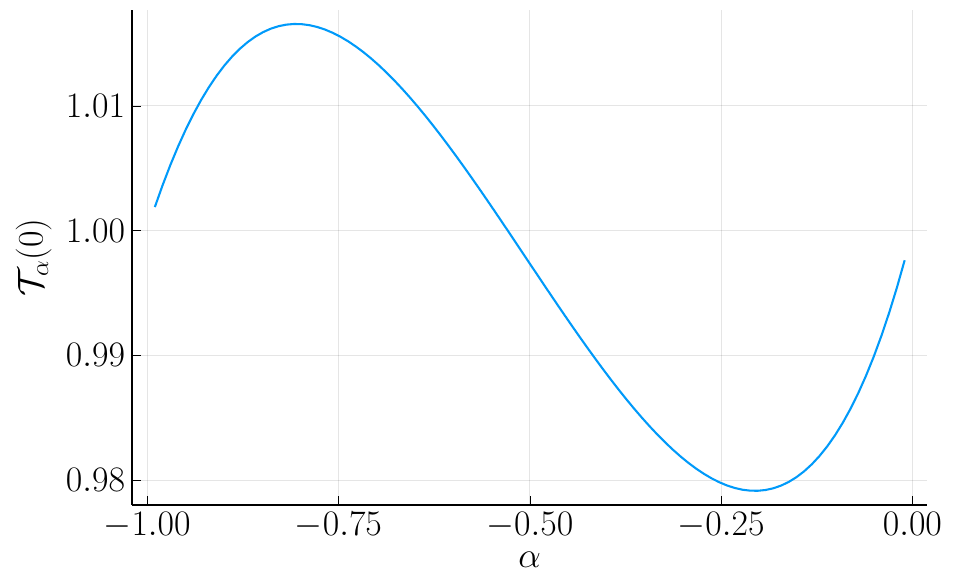}
  \caption{Value of \(\mathcal{T}_{\alpha}(0)\), defined by
    \eqref{eq:mathcal-T}, as a function of \(\alpha\) using the weight
    \(w_{\alpha}(x) = |x|\). This gives a lower bound of
    \(D_{\alpha}\) and since we need \(D_{\alpha} < 1\) it means that
    this choice of weight does not work for all values of \(\alpha\).}
  \label{fig:T-alpha-0}
\end{figure}

By taking \(w_{\alpha}(x) = |x|^{p}\) with \(-\alpha < p < 1\) the
value of \(D_{\alpha}\) can be made smaller than \(1\), so that
Proposition~\ref{prop:contraction} applies. The precise value of \(p\)
to use depends on \(\alpha\) and is made more precise in
Section~\ref{sec:bounds-I-2}.

The limits \(\alpha \to -1\) and \(\alpha \to 0\) both require extra
attention. For \(\alpha \to 0\) with the weight
\(w_{\alpha}(x) = |x|\) the value of \(D_{\alpha}\) tends to \(1\).
However, in this case \(\delta_{\alpha}\) tends to \(0\) and by
controlling the rate of convergence of \(D_{\alpha}\) and
\(\delta_{\alpha}\) it is still possible to apply
Proposition~\ref{prop:contraction}, see Section~\ref{sec:bounds-I-3}
for more details.

For \(\alpha \to -1\) the situation is more complicated. It is not
possible to use the weight \(w_{\alpha}(x) = |x|\) because
\(D_{\alpha}\) would be strictly greater than \(1\) close to
\(\alpha = -1\). It is also not possible to use the weight
\(w_{\alpha}(x) = |x|^{p}\) with \(p\) satisfying \(-\alpha < p < 1\)
since in the limit we would have need to have \(p \to 1\), and again
we find that \(D_{\alpha}\) is greater than \(1\). A natural choice
would be to add a log-factor,
\(w_{\alpha}(x) = |x|\log(1 + 1 / |x|)\). Lengthy calculations however
shows that this would not work either. Instead, we have to take
\(w_{\alpha}(x)\) to behave like \(|x|^{p}\log(1 + 1 / |x|)\) with
\(p < 1\) and \(p \to 1\) as \(\alpha \to -1\). More precisely we take
\(w_{\alpha}(x) = |x|^{(1 - \alpha) / 2}\log(2e + 1 / |x|)\), the
constant \(2e\) does not change the asymptotic behavior close to
\(x = 0\) but does change the non-asymptotic behavior in a way that
turns out to make it slightly easier to satisfy the required
inequality in Proposition~\ref{prop:contraction}. That we have to take
such a complicated weight, compared to just
\(w_{\alpha}(x) = |x|^{p}\), makes the analysis of \(D_{\alpha}\) in
Section~\ref{sec:analysis-T} significantly harder.

For the hybrid cases we mimic what happens at the endpoints.
Summarizing we have the following choice for \(w_{\alpha}\):
\begin{itemize}
\item For \(\alpha \in I_{1}\) we take
  \(w_{\alpha}(x) = |x|^{(1 - \alpha) / 2}\log(2e + 1 / |x|)\).
\item For \(\alpha \in I_{2}\) we have three cases
  \begin{itemize}
  \item Near \(\alpha = -1\) we take
    \(w_{\alpha}(x) = |x|^{p}\log(2e + 1 / |x|)\) with
    \(-\alpha < p < 1\).
  \item Near \(\alpha = 0\) we take \(w_{\alpha}(x) = |x|\)
  \item In between we take \(w_{\alpha}(x) = |x|^{p}\) with
    \(-\alpha < p < 1\).
  \end{itemize}
  The precise choice of \(p\) as a function of \(\alpha\) is given in
  Section~\ref{sec:bounds-I-2}.
\item For \(\alpha \in I_{3}\) we take \(w_{\alpha}(x) = |x|\).
\end{itemize}

\section{Proof strategy}
\label{sec:proof-strategy}
The main part of the proof of Theorem~\ref{thm:main} is the
computation of the bounds of \(n_{\alpha}\), \(\delta_{\alpha}\) and
\(D_{\alpha}\) required to apply Proposition~\ref{prop:contraction}.
Explicit bounds of these values are given in
Section~\ref{sec:bounds-for-values}. In this section we describe the
general procedure for computing these bounds. The details for the
different cases are given in Sections~\ref{sec:evaluation}
to~\ref{sec:analysis-T}.

The values we want to bound are given by
\begin{equation*}
  n_{\alpha} = \sup_{x \in [0, \pi]} |N_{\alpha}(x)|,\quad
  \delta_{\alpha} = \sup_{x \in [0, \pi]} |F_{\alpha}(x)|,\quad
  D_{\alpha} = \sup_{x \in [0, \pi]} |\mathcal{T}_{\alpha}(x)|.
\end{equation*}
With \(N_{\alpha}\) and \(F_{\alpha}\) as in \eqref{eq:N} and
\eqref{eq:F} in Section~\ref{sec:reduct-fixed-point} and
\(\mathcal{T}_{\alpha}\) as in \eqref{eq:mathcal-T} in
Section~\ref{sec:analysis-T}. At \(x = 0\) the functions all have
removable singularities, for small values of \(x\) the functions
therefore need extra care when evaluating them. For that reason the
computation of the supremum is split into two parts, one on the
interval \([0, \epsilon]\) and one on the interval
\([\epsilon, \pi]\), for some \(\epsilon\) depending on \(\alpha\) and
the function. On the interval \([0, \epsilon]\) we use a method of
evaluation which takes into account the removable singularity.

To enclose the supremum on the intervals \([0, \epsilon]\) and
\([\epsilon, \pi]\) we make use of methods from interval arithmetic
and rigorous numerics. In this section we give a brief description of
these methods and also discuss what we require of \(N_{\alpha}\),
\(F_{\alpha}\) and \(\mathcal{T}_{\alpha}\) to be able to apply these
methods.

\subsection{Enclosing the supremum}
\label{sec:enclosing-supremum}
Consider the general problem of enclosing the maximum of \(f\) on some
interval \(I\) to some predetermined tolerance. Assume that we can
evaluate \(f\) using interval arithmetic, i.e.\ given any interval
\(\inter{x} \subseteq I\) we can compute an interval enclosing the set
\(f(\inter{x}) = \{f(x): x \in \inter{x}\}\). The main idea is to
iteratively bisect the interval \(I\) into smaller and smaller
subintervals. At every iteration we compute an enclosure of \(f\) on
each subinterval. From these enclosures a lower bound of the maximum
can be computed. We then discard all subintervals for which the
enclosure is less than the lower bound of the maximum, the maximum
cannot be attained there. For the remaining subintervals we check if
their enclosure satisfies the required tolerance, in that case we
don't bisect them further. If there are any subintervals left we
bisect them and continue with the next iteration. In the end, either
when there are no subintervals left to bisect or we have reached some
maximum number of iterations (to guarantee that the procedure
terminates), we return the maximum of all subintervals that were not
discarded. This is guaranteed to give an enclosure of the maximum of
\(f\) on the interval.

If we are able to compute Taylor expansions of the function \(f\) we
can improve the performance of this procedure significantly (see
e.g.~\cite{dahne20_comput_tight_enclos_laplac_eigen,
  dahne21_count_to_paynes_nodal_line} where a similar approach is
used). Consider a subinterval \(I_{i}\), instead of computing an
enclosure of \(f(I_{i})\) we compute a Taylor polynomial \(P\) at the
midpoint and an enclosure \(R\) of the remainder term such that
\(f(x) \in P(x) + R\) for \(x \in I_{i}\). We then have
\begin{equation}
  \label{eq:taylor-bound}
  \sup_{x \in I_{i}} f(x) \in \sup_{x \in I_{i}} P(x) + R.
\end{equation}
To compute \(\sup_{x \in I_{i}} P(x)\) we isolate the roots of \(P'\)
on \(I_{i}\) and evaluate \(P\) on the roots as well as the endpoints
of the interval. In practice the computation of \(R\) involves
computing an enclosure of the Taylor expansion of \(f\) on the full
interval \(I_{i}\). Since this includes the derivative we can as an
extra optimization check if the derivative is non-zero, in which case
\(f\) is monotone, and it is enough to evaluate \(f\) on either the
left or the right endpoint of \(I_{i}\), depending on the sign of the
derivative.

The above procedures can easily be adapted to instead compute the
minimum of \(f\) on the interval, joining them together we can thus
compute the extrema on the interval. In some cases we don't care about
computing an enclosure of the maximum, but only to prove that it is
bounded by some value. Instead of using a tolerance we then discard
any subintervals for which the enclosure of the maximum is less than
the bound.

\subsection{Evaluation of \(N_{\alpha}\), \(F_{\alpha}\) and \(\mathcal{T}_{\alpha}\) on intervals}
\label{sec:evaluation-all}
To be able to use the above method we need to be able to evaluate
\(N_{\alpha}\), \(F_{\alpha}\) and \(\mathcal{T}_{\alpha}\) using
interval arithmetic. That is, given some interval
\(\inter{x} \subseteq [0, \pi]\) we need to be able to compute
enclosures of \(N_{\alpha}(\inter{x})\), \(F_{\alpha}(\inter{x})\) and
\(\mathcal{T}_{\alpha}(\inter{x})\) respectively.

In Sections~\ref{sec:evaluation} to~\ref{sec:analysis-T} we discuss
how to evaluate these functions for (non-interval) \(x \in [0, \pi]\).
This then needs to be extended to also work for intervals
\(\inter{x} \subseteq (0, \pi)\). In general this is straight forward
and handled directly by just replacing the point \(x \in [0, \pi]\)
with the interval \(\inter{x} \subseteq (0, \pi)\) in the
computations, the operations on the interval are then done using
Arb~\cite{Johansson2013arb}. In some cases we can make use of
monotonicity properties of the function. For example when computing
the root \(r_{\alpha,x}\) introduced in Lemma~\ref{lemma:I-0-x} for an
interval \(\inter{x} = [\lo{x}, \hi{x}]\) we can use that it is
decreasing in \(x\), giving us
\(r_{\alpha,\inter{x}} = [r_{\alpha,\hi{x}}, r_{\alpha,\lo{x}}].\).

To cover all values of \(\alpha\) in \((-1, 0)\) we need to use
interval values also for \(\alpha\), and not only for \(x\). Similar
to \(x\) this is in general done by just replacing
\(\alpha \in (-1, 0)\) with \(\boldsymbol{\alpha} \subseteq (-1, 0)\)
in the computations.

To be able to use the improved version of the method for enclosing the
supremum discussed in Section~\ref{sec:enclosing-supremum} we need to
be able to compute Taylor expansions of \(N_{\alpha}\), \(F_{\alpha}\)
and \(\mathcal{T}_{\alpha}\). This is possible in many, but not all
cases. For \(x \in [\epsilon, \pi]\) this is in general straight
forward for \(N_{\alpha}\) and \(F_{\alpha}\), for example the methods
in Appendix~\ref{sec:expansion-x} can be used for computing Taylor
expansions of the involved Clausen functions, but not for
\(\mathcal{T}_{\alpha}\).

For \(x \in [0, \epsilon]\) the functions are in general expanded at
\(x = 0\) and you have to be slightly careful when computing Taylor
expansions from this expansion. For example, for
\(\frac{\tilde{C}_{1 - \alpha}(x)}{x^{\alpha}}\) we have the expansion
\begin{equation*}
  \frac{\tilde{C}_{1 - \alpha}(x)}{x^{\alpha}}
  = \Gamma(\alpha)\sin\left(\frac{\pi}{2}(1 - \alpha)\right)
  + \sum_{m = 1}^{\infty} (-1)^{m}\zeta(1 - \alpha - 2m)\frac{x^{2m + \alpha}}{(2m)!}.
\end{equation*}
When computing it we need to bound the tail. If we let \(R\) be an
interval bounding the remainder term given in
Lemma~\ref{lemma:clausen-tails} on the interval \([0, \epsilon]\) then
for \(x \in [0, \epsilon]\) we have
\begin{equation*}
  \frac{\tilde{C}_{1 - \alpha}(x)}{x^{\alpha}}
  \in \Gamma(\alpha)\sin\left(\frac{\pi}{2}(1 - \alpha)\right)
  + \sum_{m = 1}^{M - 1} (-1)^{m}\zeta(1 - \alpha - 2m)\frac{x^{2m + \alpha}}{(2m)!}
  + Rx^{2M + \alpha}.
\end{equation*}
If we consider \(R\) to be fixed then it is straight forward to
compute Taylor expansions in \(x\) from this representation. This will
in general not give us a Taylor expansion of
\(\frac{\tilde{C}_{1 - \alpha}(x)}{x^{\alpha}}\) since we have lost
the remainder terms dependence on \(x\). When enclosing the supremum
this is however not an issue, we can apply the method on
\begin{equation*}
  \Gamma(\alpha)\sin\left(\frac{\pi}{2}(1 - \alpha)\right)
  + \sum_{m = 1}^{M - 1} (-1)^{m}\zeta(1 - \alpha - 2m)\frac{x^{2m + \alpha}}{(2m)!}
  + Rx^{2M + \alpha}
\end{equation*}
directly.

\section{Evaluation of \(u_{\alpha}\) and \(\Hop{\alpha}[u_{\alpha}]\)}
\label{sec:evaluation}
To compute bounds for \(D_{\alpha}\), \(\delta_{\alpha}\) and
\(n_{\alpha}\) we need to be able to compute enclosures of
\(u_{\alpha}(x)\) and \(\Hop{\alpha}[u_{\alpha}](x)\). In this section
we describe how this is done for the approximations used on the
intervals \(I_{1}\), \(I_{2}\) and \(I_{3}\) respectively.

\subsection{Evaluation of \(u_{\alpha}\) and \(\Hop{\alpha}[u_{\alpha}]\) for \(I_2\)}
\label{sec:evaluation-I-2}
In this case the expressions for \(u_{\alpha}\) and
\(\Hop{\alpha}[u_{\alpha}]\) are given by \eqref{eq:u0} and
\eqref{eq:Hu0}. The parameter \(a_{\alpha,0}\) is given by
\eqref{eq:a0}, which is easily enclosed for a given \(\alpha\), and
\(p_{\alpha}\), \(a_{\alpha,j}\) for \(j \geq 1\) and \(b_{\alpha,n}\)
are fixed, numerically determined, numbers. Evaluation is straight
forward, using the approach in Appendix~\ref{sec:computing-clausen} to
enclose the Clausen functions.

\subsection{Evaluation of \(u_{\alpha}\) and \(\Hop{\alpha}[u_{\alpha}]\) for \(I_1\)}
\label{sec:evaluation-I-1}
In this case the expressions for \(u_{\alpha}\) and
\(\Hop{\alpha}[u_{\alpha}]\) are given by \eqref{eq:u0-I-1} and
\eqref{eq:Hu0-I-1}. Most of the parameters are fixed, numerically
determined, numbers, these parts of the functions are straight forward
to enclose. The only problematic part are the leading terms,
\begin{equation*}
  a_{\alpha,0}(\tilde{C}_{1 - \alpha}(x) - \tilde{C}_{2 + (1 + \alpha)^{2} / 2}(x))
  \text{ and }
  -a_{\alpha,0}(\tilde{C}_{1 - 2\alpha}(x) - \tilde{C}_{2 - \alpha + (1 + \alpha)^{2} / 2}(x)),
\end{equation*}
which both have removable singularities at \(\alpha = -1\). They are
handled by rewriting them as
\begin{equation*}
  a_{\alpha,0}(\tilde{C}_{1 - \alpha}(x) - \tilde{C}_{2 + (1 + \alpha)^{2} / 2}(x))
  = ((\alpha + 1)a_{\alpha,0}) \cdot \frac{\tilde{C}_{1 - \alpha}(x) - \tilde{C}_{2 + (1 + \alpha)^{2} / 2}(x)}{\alpha + 1},
\end{equation*}
as well as
\begin{equation*}
  -a_{\alpha,0}(\tilde{C}_{1 - 2\alpha}(x) - \tilde{C}_{2 - \alpha + (1 + \alpha)^{2} / 2}(x))
  = -((\alpha + 1)a_{\alpha,0}) \cdot \frac{\tilde{C}_{1 - 2\alpha}(x) - \tilde{C}_{2 - \alpha + (1 + \alpha)^{2} / 2}(x)}{\alpha + 1}
\end{equation*}
and using the approach in Appendix~\ref{sec:removable-singularities}
to handle the removable singularities.

\subsection{Evaluation of \(u_{\alpha}\) and \(\Hop{\alpha}[u_{\alpha}]\) for \(I_3\)}
\label{sec:evaluation-I-3}
In this case the expressions for \(u_{\alpha}\) and
\(\Hop{\alpha}[u_{\alpha}]\) are, as for \(I_{2}\), given by
\eqref{eq:u0-I-3} and \eqref{eq:Hu0-I-3}. The parameters
\(p_{\alpha}\) and \(a_{\alpha,j}\) are not fixed numbers, but all
depend on \(\alpha\). Moreover, it is not enough to just compute
enclosures of \(u_{\alpha}(x)\) and \(\Hop{\alpha}[u_{\alpha}]\), we
need their expansions in \(\alpha\). For this we make use of Taylor
models~\cite{Joldes2011}, which we give a brief introduction of in
Appendix~\ref{sec:taylor-models}, centered at \(\alpha = 0\). For
\(a_{\alpha,0}\), \(a_{\alpha,1}\) and \(a_{\alpha,2}\) we have
explicit formulas in \eqref{eq:a0} and \eqref{eq:a1-a2-I-3} and Taylor
models can be computed directly from them, there are several removable
singularities which are handled as described in
Appendix~\ref{sec:removable-singularities}. The value of
\(p_{\alpha}\) is given implicitly as a function of \(\alpha\) by
\eqref{eq:p-alpha}, from this we can compute a Taylor model using the
approach described in Appendix~\ref{sec:taylor-models-p-alpha}. Once
these Taylor models are computed it is straight forward to compute
Taylor models of \eqref{eq:u0-I-3} and \eqref{eq:Hu0-I-3}.

\subsection{Evaluation of \(u_{\alpha}\) and \(\Hop{\alpha}[u_{\alpha}]\) in hybrid cases}
\label{sec:evaluation-hybrid}
For \(\alpha\) near \(-1\) the only part that needs special care is
the evaluation of
\(\tilde{C}_{1 - \alpha}(x) - \tilde{C}_{1 - \alpha + p_{\alpha}}(x)\)
for \(u_{\alpha}\) and
\(\tilde{C}_{1 - 2\alpha}(x) - \tilde{C}_{1 - 2\alpha +
  p_{\alpha}}(x)\) for \(\Hop{\alpha}[u_{\alpha}]\), in both cases due
to the large cancellations between the two terms. For interval
arguments \(\inter{x}\) and \(\inter{s}\) with midpoints \(x_{0}\) and
\(s_{0}\) respectively we compute
\(\tilde{C}_{\inter{s}}(\inter{x}) - \tilde{C}_{\inter{s} +
  p_{\alpha}}(\inter{x})\) using a midpoint approximation as
\begin{equation*}
  \tilde{C}_{\inter{s}}(\inter{x}) - \tilde{C}_{\inter{s} + p_{\alpha}}(\inter{x})
  = (\tilde{C}_{s_{0}}(x_{0}) - \tilde{C}_{s_{0} + p_{\alpha}}(x_{0}))
  + (\inter{s} - s_{0})(\tilde{C}_{\inter{s}}^{(1)}(x_{0}) - \tilde{C}_{\inter{s} + p_{\alpha}}^{(1)}(x_{0}))
  - (\inter{x} - x_{0})(\tilde{S}_{\inter{s} - 1}(\inter{x}) - \tilde{S}_{\inter{s} + p_{\alpha} - 1}(\inter{x})).
\end{equation*}

For \(\alpha\) near \(0\) we follow exactly the same approach as for
\(\alpha \in I_{3}\), the only difference being that the Taylor models
are centered at \(\alpha\) (or the midpoint of \(\boldsymbol{\alpha}\)
in the case of intervals) instead of at \(0\).

\section{Division by \(u_{\alpha}\)}
\label{sec:inv-u0}
The computations of the values \(n_{\alpha}\), \(\delta_{\alpha}\) and
\(D_{\alpha}\) all involve functions given by fractions where the
denominator contains \(u_{\alpha}\). For \(x > 0\) this is not an
issue since the denominators are non-zero, we can enclose the
numerator and denominator separately and then perform the division.
For \(x = 0\) we get a removable singularity and to handle that we
need to understand the asymptotic behavior of \(u_{\alpha}(x)\) at
\(x = 0\). For this we start of with the following, straight forward,
lemma
\begin{lemma}
  \label{lemma:asymptotic-inv-u0}
  For \(-1 < \alpha < 0\) the function
  \(\frac{|x|^{-\alpha}}{u_{\alpha}(x)}\) is non-zero and bounded at
  \(x = 0\).
\end{lemma}
\begin{proof}
  From the asymptotic expansion of \(u_{\alpha}(x)\) at \(x = 0\) we
  have
  \begin{equation*}
    u_{\alpha}(x) =
    a_{\alpha,0}\Gamma(\alpha)\cos\left(\frac{\pi}{2}\alpha\right)|x|^{-\alpha} + o(|x|^{-\alpha}).
  \end{equation*}
  This gives us
  \begin{equation*}
    \frac{|x|^{-\alpha}}{u_{\alpha}(x)} =
    \frac{1}{a_{\alpha,0}\Gamma(\alpha)\cos\left(\frac{\pi}{2}\alpha\right) + o(1)}.
  \end{equation*}
  This is non-zero and bounded as long as
  \begin{equation*}
    a_{\alpha,0}\Gamma(\alpha)\cos\left(\frac{\pi}{2}\alpha\right)
  \end{equation*}
  is non-zero and bounded. We get
  \begin{equation*}
    a_{\alpha,0}\Gamma(\alpha)\cos\left(\frac{\pi}{2}\alpha\right) = \frac{2\Gamma(2\alpha)\cos\left(\pi\alpha\right)}{\Gamma(\alpha)\cos\left(\frac{\pi}{2}\alpha\right)}
  \end{equation*}
  The denominator is easily seen to be non-zero and bounded for
  \(\alpha \in (-1, 0)\). For the numerator we can rewrite it as
  \begin{equation*}
    2\Gamma(2\alpha)\cos\left(\pi\alpha\right)
    = \frac{2\Gamma(2\alpha + 2)\cos\left(\pi\alpha\right)}{2\alpha(2\alpha + 1)}
    = \frac{\Gamma(2\alpha + 2)}{\alpha}\frac{\sin\left(\frac{\pi}{2}(2\alpha + 1)\right)}{(2\alpha + 1)},
  \end{equation*}
  which is bounded and non-zero.
\end{proof}
Computing \(\frac{|x|^{-\alpha}}{u_{\alpha}(x)}\) near \(x = 0\) is
straight forward. Compute the expansion of \(u_{\alpha}(x)\) at
\(x = 0\) and subtract \(-\alpha\) from the exponents of all terms in
the expansion, this gives an expansion where all exponents are
non-negative. Evaluate the expansion at \(x\) and then compute the
inverse.

While Lemma~\ref{lemma:asymptotic-inv-u0} is valid for the full
interval \((-1, 0)\) the bound is not uniform in \(\alpha\). For
\(\alpha \to 0\) it converges to \(1\), and we are fine, but for
\(\alpha \to -1\) the value of \(\frac{x^{-\alpha}}{u_{\alpha}(x)}\)
goes to zero, and we need an understanding of the rate in which it
does so. For this we use the following, modified version of the lemma
\begin{lemma}
  \label{lemma:asymptotic-inv-u0-I-1}
  For \(\alpha \in I_{1}\) the function
  \begin{equation*}
    \frac{\Gamma(1 + \alpha)|x|^{-\alpha}(1 - |x|^{1 + \alpha + (1 + \alpha)^{2} / 2})}{u_{\alpha}(x)}
  \end{equation*}
  is non-zero and uniformly bounded in \(\alpha\) at \(x = 0\).
\end{lemma}
\begin{proof}
  Since the function is even with respect to \(x\) we can take
  \(x \geq 0\).

  From Lemma~\ref{lemma:u0-asymptotic-I-1} we have that the leading
  term in the asymptotic expansion at \(x = 0\) for \(u_{\alpha}\) is
  given by
  \begin{equation*}
    a_{\alpha,0}\left(
      \Gamma(\alpha)\cos\left(\frac{\pi}{2}\alpha\right)
      - \Gamma(-1 - (1 + \alpha)^{2} / 2)\cos\left(\frac{\pi}{2}(1 + (1 + \alpha)^{2} / 2)\right)x^{1 + \alpha + (1 + \alpha)^{2} / 2}
    \right)x^{-\alpha}.
  \end{equation*}
  We therefore study the asymptotic behavior of
  \begin{equation}
    \label{eq:u-I-1-asym-behavior-upper-bound}
    \frac{
      \Gamma(1 + \alpha)x^{-\alpha}(1 - x^{1 + \alpha + (1 + \alpha)^{2} / 2})
    }{
      a_{\alpha,0}\left(
        \Gamma(\alpha)\cos\left(\frac{\pi}{2}\alpha\right)
        - \Gamma(-1 - (1 + \alpha)^{2} / 2)\cos\left(\frac{\pi}{2}(1 + (1 + \alpha)^{2} / 2)\right)x^{1 + \alpha + (1 + \alpha)^{2} / 2}
      \right)x^{-\alpha}
    }.
  \end{equation}
  We can cancel the \(x^{-\alpha}\) factor and split it as
  \begin{equation*}
    \frac{\Gamma(1 + \alpha)}{a_{\alpha,0}}
    \frac{1 - x^{1 + \alpha + (1 + \alpha)^{2} / 2}}{
      \Gamma(\alpha)\cos\left(\frac{\pi}{2}\alpha\right)
      - \Gamma(-1 - (1 + \alpha)^{2} / 2)\cos\left(\frac{\pi}{2}(1 + (1 + \alpha)^{2} / 2)\right)x^{1 + \alpha + (1 + \alpha)^{2} / 2}
    }.
  \end{equation*}

  From the definition of \(a_{\alpha,0}\) as given in Equation
  \eqref{eq:a0} we get that the first factor is given by
  \begin{equation*}
    \frac{\Gamma(1 + \alpha)\Gamma(\alpha)^{2}\cos\left(\frac{\pi}{2}\alpha\right)^{2}}{2\Gamma(2\alpha)\cos\left(\pi\alpha\right)}
    = \frac{\alpha\Gamma(\alpha)^{3}\cos\left(\frac{\pi}{2}\alpha\right)^{2}}{2\Gamma(2\alpha)\cos\left(\pi\alpha\right)}.
  \end{equation*}
  This has a removable singularity at \(\alpha = -1\). It is then
  easily checked to be non-zero.

  For the second factor we let
  \(c(\alpha) = \Gamma(\alpha)\cos\left(\frac{\pi}{2}\alpha\right)\)
  and \(\tilde{p}_{\alpha} = 1 + \alpha + (1 + \alpha)^{2} / 2\),
  allowing us to write it as
  \begin{equation*}
    \frac{1 - x^{\tilde{p}_{\alpha}}}{c(\alpha) - c(\alpha - \tilde{p}_{\alpha})x^{\tilde{p}_{\alpha}}}
    = \frac{1}{c(\alpha)}\frac{1 - x^{\tilde{p}_{\alpha}}}{1 - \frac{c(\alpha - \tilde{p}_{\alpha})}{c(\alpha)}x^{\tilde{p}_{\alpha}}}
  \end{equation*}
  Here \(c(\alpha)\) has a removable singularity, which can be handled
  in the same way as the one for \(\frac{\Gamma(1 + \alpha)}{a_{0}}\).
  Again it can easily be checked to be non-zero, so that
  \(\frac{1}{c(\alpha)}\) is bounded. What remains to handle is thus
  \begin{equation*}
    \frac{1 - x^{\tilde{p}_{\alpha}}}{1 - \frac{c(\alpha - \tilde{p}_{\alpha})}{c(\alpha)}x^{\tilde{p}_{\alpha}}}.
  \end{equation*}
  Computing an enclosure of the derivative of
  \(\frac{c(\alpha - \tilde{p}_{\alpha})}{c(\alpha)}\) allows us to
  check that it is negative. This together with the fact that
  \(\lim_{\alpha \to -1} \frac{c(\alpha -
    \tilde{p}_{\alpha})}{c(\alpha)} = 1\) means that an upper bound is
  given by
  \begin{equation*}
    \frac{1 - x^{\tilde{p}_{\alpha}}}{1 - x^{\tilde{p}_{\alpha}}} = 1.
  \end{equation*}
  For a lower bound we note that since
  \(\frac{c(\alpha - \tilde{p}_{\alpha})}{c(\alpha)} < 1\) it is
  decreasing in \(x\) as long as \(x < 1\). It is therefore enough to
  check the value at some \(0 < x_{0} < 1\) to get a lower bound. For
  this fixed \(x_{0} > 0\) we can handle the removable singularity of
  \begin{equation*}
    \frac{1 - x_{0}^{\tilde{p}_{\alpha}}}{1 - \frac{c(\alpha - \tilde{p}_{\alpha})}{c(\alpha)}x_{0}^{\tilde{p}_{\alpha}}}
  \end{equation*}
  at \(\alpha = -1\) to compute an enclosure of the lower bound.
\end{proof}
The proof of this lemma also tells us how to compute an upper bound of
\begin{equation*}
  \frac{\Gamma(1 + \alpha)|x|^{-\alpha}(1 - |x|^{1 + \alpha + (1 + \alpha)^{2} / 2})}{u_{\alpha}(x)}.
\end{equation*}
It is enough to check that the terms in the asymptotic expansion of
\(u_{\alpha}\) that we skipped are positive, in that case Equation
\eqref{eq:u-I-1-asym-behavior-upper-bound} gives an upper bound. For
\(x\) not overlapping zero we instead split it as
\begin{equation*}
  \left(\Gamma(1 + \alpha)(1 - |x|^{1 + \alpha + (1 + \alpha)^{2} / 2})\right) \cdot
  \frac{|x|^{-\alpha}}{u_{\alpha}(x)}
\end{equation*}
and handle the removable singularity for the first factor and enclose
the second factor directly using the asymptotic expansion.

\section{Evaluation of \(N_{\alpha}\)}
\label{sec:evaluation-N}
Recall that
\begin{equation*}
  n_{\alpha} := \|N_{\alpha}\|_{L^{\infty}(\mathbb{T})} = \sup_{x \in [0, \pi]} \left|N_{\alpha}(x)\right|
\end{equation*}
where
\begin{equation*}
  N_{\alpha}(x) = \frac{w_{\alpha}(x)}{2u_{\alpha}(x)}.
\end{equation*}
In this section we discuss how to evaluate \(N_{\alpha}\). The
interval \([0, \pi]\) is split into two subintervals,
\([0, \epsilon]\) and \([\epsilon, \pi]\). On the interval
\([0, \epsilon]\) we make use of asymptotic expansions, on
\([\epsilon, \pi]\) the function is evaluated directly. The value of
\(\epsilon\) depends on \(\alpha\) and is discussed more in
Section~\ref{sec:bounds-for-values}.

We start by describe the procedure for evaluation for
\(\alpha \in I_{2}\) and then discuss the alterations required to
handle \(I_{1}\) and \(I_{3}\).

\subsection{Evaluation of \(N_{\alpha}\) for \(I_{2}\)}
\label{sec:evaluation-N-I-2}
For \(x \in [\epsilon, \pi]\) both \(w_{\alpha}(x)\),
\(u_{\alpha}(x)\) and the resulting quotient are straight forward to
evaluate.

For \(x \in [0, \epsilon]\) we use that the weight is given by
\(w_{\alpha}(x) = x^{p}\) and write the function as
\begin{equation*}
  N_{\alpha}(x) = \frac{1}{2}x^{p + \alpha} \cdot \frac{x^{-\alpha}}{u_{\alpha}(x)}.
\end{equation*}
The first factor is easily enclosed and we use
Lemma~\ref{lemma:asymptotic-inv-u0} for enclosing the second factor.

\subsection{Evaluation of \(N_{\alpha}\) for \(I_{1}\)}
\label{sec:evaluation-N-I-1}
For \(x \in [\epsilon, \pi]\) we make one optimization compared to
\(I_{2}\), instead of computing an enclosure we only compute an upper
bound. For this we use that \(u_{\alpha}(x) \geq u_{-1}(x)\) for
\(\alpha \in I_{1}\) and hence
\begin{equation*}
  N_{\alpha}(x) \leq \frac{w_{\alpha}(x)}{2u_{-1}(x)},
\end{equation*}
as long as \(u_{-1}(x)\) is positive.

For \(x \in [0, \epsilon]\) we use that the weight is given by
\(w_{\alpha}(x) = x^{(1 - \alpha) / 2}\log(2e + 1 / x)\) and write the
function as
\begin{align*}
  N_{\alpha}(x) &= \frac{1}{2}\frac{w_{\alpha}(x)}{\Gamma(1 + \alpha)x^{-\alpha}(1 - x^{1 + \alpha + (1 + \alpha)^{2} / 2})} \cdot
                  \frac{\Gamma(1 + \alpha)x^{-\alpha}(1 - x^{1 + \alpha + (1 + \alpha)^{2} / 2})}{u_{\alpha}(x)}\\
                &= \frac{1}{\Gamma(2 + \alpha)}\cdot
                  \frac{\log(2e + 1 / x)}{2\log(1 / x)}\cdot
                  \frac{x^{(1 + \alpha) / 2}\log(1 / x)(1 + \alpha)}{1 - x^{1 + \alpha + (1 + \alpha)^{2} / 2}} \cdot
                  \frac{\Gamma(1 + \alpha)x^{-\alpha}(1 - x^{1 + \alpha + (1 + \alpha)^{2} / 2})}{u_{\alpha}(x)}.
\end{align*}
The first factor is well-behaved, the second factor can be enclosed
using that it is increasing in \(x\) and the limit as \(x \to 0\) is
\(\frac{1}{2}\). The fourth factor can be enclosed using
Lemma~\ref{lemma:asymptotic-inv-u0-I-1}. For the third factor we get a
bound from the following lemma
\begin{lemma}
  Let \(-1 < \alpha < 0\) and \(0 < x < 1\), then
  \begin{equation*}
    0 < \frac{x^{(1 + \alpha) / 2}\log(1 / x)(1 + \alpha)}{1 - x^{1 + \alpha + (1 + \alpha)^{2} / 2}} < 1.
  \end{equation*}
\end{lemma}
\begin{proof}
  The lower bound follows directly from that all the factors are
  positive. What remains is proving the upper bound.

  If we let \(s = (1 + \alpha)\log x\) we can write the quotient as
  \begin{equation*}
    \frac{-se^{s / 2}}{1 - e^{s(1 + (1 + \alpha) / 2)}}.
  \end{equation*}
  From \(-1 < \alpha < 0\) and \(0 < x < 1\) we get that \(s < 0\). We
  want to prove that
  \begin{equation*}
    -se^{s / 2} < 1 - e^{s(1 + (1 + \alpha) / 2)},
  \end{equation*}
  which is equivalent to proving that
  \begin{equation*}
    f(s) = 1 - e^{s(1 + (1 + \alpha) / 2)} + se^{s / 2}
  \end{equation*}
  is positive for \(s < 0\). We have \(\lim_{s \to 0} f(s) = 0\), and
  it is hence enough to prove that \(f\) is decreasing for \(s < 0\).
  The derivative is given by
  \begin{equation*}
    f'(s) = -(1 + (1 + \alpha) / 2)e^{s(1 + (1 + \alpha) / 2)} + \left(1 + \frac{s}{2}\right)e^{s / 2}
    = e^{s / 2}\left(1 + \frac{s}{2} - (1 + (1 + \alpha) / 2)e^{s(1 + (1 + \alpha)) / 2}\right)
  \end{equation*}
  so it is enough to prove that
  \begin{equation*}
    g(s) = 1 + \frac{s}{2} - (1 + (1 + \alpha) / 2)e^{s(1 + (1 + \alpha)) / 2}
  \end{equation*}
  is negative for \(s < 0\). We have
  \(g(0) = -\frac{1 + \alpha}{2} < 0\) and
  \(\lim_{s \to -\infty} g(s) = -\infty\). The unique critical point
  with respect to \(s\) is
  \begin{equation*}
    -2\frac{\log\left((1 + (1 + \alpha) / 2)(1 + (1 + \alpha))\right)}{1 + (1 + \alpha)},
  \end{equation*}
  which is negative. Inserting this into \(g\) gives us
  \begin{multline*}
    1 - \frac{\log\left((1 + (1 + \alpha) / 2)(1 + (1 + \alpha))\right)}{1 + (1 + \alpha)}
    - (1 + (1 + \alpha) / 2)e^{-\log\left((1 + (1 + \alpha) / 2)(1 + (1 + \alpha))\right)}\\
    = 1 - \frac{1 + \log\left((1 + (1 + \alpha) / 2)(1 + (1 + \alpha))\right)}{1 + (1 + \alpha)}\\
    = 1 - \frac{1 + \log(1 + (1 + \alpha) / 2) + \log(1 + (1 + \alpha))}{1 + (1 + \alpha)}
  \end{multline*}
  For this to be negative we need
  \begin{equation*}
    \log(1 + (1 + \alpha) / 2) + \log(1 + (1 + \alpha)) > 1 + \alpha,
  \end{equation*}
  which holds for \(-1 < \alpha < 0\). It follows that \(g(s)\) is
  negative for \(s < 0\) and hence \(f(s)\) is decreasing in \(s\) and
  from \(\lim_{s \to 0} f(s) = 0\) we then get that \(f(s)\) is
  positive for \(s < 0\) and the result follows.
\end{proof}

\subsection{Evaluation of \(N_{\alpha}\) for \(I_{3}\)}
\label{sec:evaluation-N-I-3}
Compared to \(\delta_{\alpha}\) and \(D_{\alpha}\) we don't need to
compute an expansion in \(\alpha\) of \(n_{\alpha}\) for \(I_{3}\), it
is therefore enough to just compute enclosures of \(N_{\alpha}(x)\).
The asymptotic expansions perform very well in this case and it is
possible to take \(\epsilon = \pi\), we hence only have to consider
the interval \([0, \epsilon]\). The evaluation is done in the same way
as for \(I_{2}\) (with \(p = 1\)), by rewriting it as
\begin{equation*}
  N_{\alpha}(x) = \frac{1}{2}x^{1 + \alpha} \cdot \frac{x^{-\alpha}}{u_{\alpha}(x)}.
\end{equation*}
and using Lemma~\ref{lemma:asymptotic-inv-u0}.

\subsection{Evaluation of \(N_{\alpha}\) in hybrid cases}
\label{sec:evaluation-N-hybrid}
For \(\alpha\) near \(-1\) we don't have to do anything special and
follow the same approach as in Section~\ref{sec:evaluation-N-I-2}. For
\(\alpha\) near \(0\) the approach is the same as for
\(\alpha \in I_{3}\).

\section{Evaluation of \(F_{\alpha}\)}
\label{sec:evaluation-F}
Recall that
\begin{equation*}
  \delta_{\alpha} = \|F_{\alpha}\|_{L^{\infty}(\mathbb{T})} = \sup_{x \in [0, \pi]} |F_{\alpha}(x)|,
\end{equation*}
where
\begin{equation*}
  F_{\alpha}(x) = \frac{\Hop{\alpha}[u_{\alpha}](x) + \frac{1}{2}u_{\alpha}(x)^{2}}{w_{\alpha}(x)u_{\alpha}(x)}.
\end{equation*}
In this section we discuss how to evaluate \(F_{\alpha}\). The
interval \([0, \pi]\) is split into two subintervals,
\([0, \epsilon]\) and \([\epsilon, \pi]\). On the interval
\([0, \epsilon]\) we make use of asymptotic expansions, on
\([\epsilon, \pi]\) the function is evaluated directly. The value of
\(\epsilon\) depends on \(\alpha\) and is discussed more in
Section~\ref{sec:bounds-for-values}.

We start by describe the procedure for evaluation for
\(\alpha \in I_{2}\) and then discuss the alterations required to
handle \(I_{1}\) and \(I_{3}\).

\subsection{Evaluation of \(F_{\alpha}\) for \(I_{2}\)}
\label{sec:evaluation-F-I-2}
For \(x \in [\epsilon, \pi]\) we evaluate each part separately. That
is, we compute \(w_{\alpha}(x)\), \(u_{\alpha}(x)\) and
\(\Hop{\alpha}[u_{\alpha}](x)\), it is then straight forward to
compute \(F_{\alpha}(x)\) from these values.

For \(x \in [0, \epsilon]\) we use that the weight is given by
\(w_{\alpha}(x) = x^{p}\) and write the function as
\begin{equation*}
  F_{\alpha}(x) = \frac{x^{-\alpha}}{u_{\alpha}(x)} \cdot \frac{\Hop{\alpha}[u_{\alpha}](x) + \frac{1}{2}u_{\alpha}(x)^{2}}{x^{p - \alpha}}.
\end{equation*}
The first factor is enclosed using
Lemma~\ref{lemma:asymptotic-inv-u0}. For the second factor we use the
expansion from Lemma~\ref{lemma:Du0-asymptotic} and explicitly cancel
the \(x^{p - \alpha}\) from the denominator. Note that since
\(-2\alpha < p - \alpha\) we need the leading term in the expansion to
be identically zero, hence our choice of \(a_{\alpha,0}\). For the
second term in the expansion we have
\(p - \alpha < -2\alpha + p_{\alpha}\), so it is bounded at \(x = 0\),
as required.

\subsection{Evaluation of \(F_{\alpha}\) for \(I_{1}\)}
\label{sec:evaluation-F-I-1}
For \(x \in [\epsilon, \pi]\) we make one optimization compared to
\(I_{2}\), instead of computing an enclosure we only compute an upper
bound. For this we use that \(u_{\alpha}(x) \geq u_{-1}(x)\) for
\(\alpha \in I_{1}\) and hence
\begin{equation*}
  |F_{\alpha}(x)| \leq \left|\frac{\Hop{\alpha}[u_{\alpha}](x) + \frac{1}{2}u_{\alpha}(x)^{2}}{w_{\alpha}(x)u_{-1}(x)}\right|,
\end{equation*}
as long as \(u_{-1}(x)\) is positive.

For \(x \in [0, \epsilon]\) we use that the weight is given by
\(w_{\alpha}(x) = x^{(1 - \alpha) / 2}\log(2e + 1 / x)\) and write the
function as
\begin{equation*}
  F_{\alpha}(x) = \frac{\log(1 / x)}{\log(2e + 1 / x)} \cdot
  \frac{\Gamma(1 + \alpha)x^{-\alpha}(1 - x^{1 + \alpha + (1 + \alpha)^{2} / 2})}{u_{\alpha}(x)} \cdot
  \frac{\Hop{\alpha}[u_{\alpha}](x) + \frac{1}{2}u_{\alpha}(x)^{2}}{\Gamma(1 + \alpha)\log(1 / x)(1 - x^{1 + \alpha + (1 + \alpha)^{2} / 2})x^{(1 - \alpha) / 2 - \alpha}}.
\end{equation*}
The first two factors are handled in the same way as when handling
\(N_{\alpha}\) in Section~\ref{sec:evaluation-N-I-1}. For the third
factor we use that \((1 - \alpha) / 2 < 1\) and hence, if we require
\(\epsilon \leq 1\), an upper bound, in absolute value, is given by
the absolute value of
\begin{equation}
  \label{eq:evaluation-F-I-1-factor}
  \frac{\Hop{\alpha}[u_{\alpha}](x) + \frac{1}{2}u_{\alpha}(x)^{2}}{\Gamma(1 + \alpha)\log(1 / x)(1 - x^{1 + \alpha + (1 + \alpha)^{2} / 2})x^{1 - \alpha}}.
\end{equation}
A bound for this is given in
Appendix~\ref{sec:evaluation-F-I-1-asymptotic}.

\subsection{Evaluation of \(F_{\alpha}\) for \(I_{3}\)}
\label{sec:evaluation-F-I-3}
In this case we need to not only compute an enclosure of
\(F_{\alpha}(x)\), but understand it behavior in \(\alpha\). We
therefore compute Taylor models of degree \(1\) centered at
\(\alpha = 0\).

We start with the following lemma that gives us information about the
first two terms in the Taylor model.
\begin{lemma}
  \label{lemma:evaluation-F-expansion-alpha-I-3}
  For \(x \in [0, \pi]\) the constant and linear terms in the
  expansion at \(\alpha = 0\) of \(F_{\alpha}(x)\) are both zero.
\end{lemma}
\begin{proof}
  From lemma~\ref{lemma:u0-expansion-alpha-I-3} we get
  \begin{equation*}
    F_{\alpha}(x) = \frac{
      \left(-\frac{1}{2} - b(x)\alpha + \mathcal{O}(\alpha^{2})\right) + \frac{1}{2}\left(1 + b(x)\alpha + \mathcal{O}(\alpha^{2})\right)^{2}
    }{
      x(1 + b(x)\alpha + \mathcal{O}(\alpha^{2}))
    }
    = \frac{\mathcal{O}(\alpha^{2})}{x(1 + b(x)\alpha + \mathcal{O}(\alpha^{2}))}
    = \mathcal{O}(\alpha^{2}),
  \end{equation*}
  which gives us the result.
\end{proof}

For \(x \in [\epsilon, \pi]\) we compute the Taylor models of
\(u_{\alpha}(x)\) and \(\Hop{\alpha}[u_{\alpha}](x)\) using the
approach described in Section~\ref{sec:evaluation-I-3}. For this case
we have \(w_{\alpha}(x) = |x|\) and hence \(w_{\alpha}\) does not
depend on \(\alpha\) and its corresponding Taylor model is just a
constant.

For \(x \in [0, \epsilon]\) we, similarly to for \(I_{2}\), write the
function as
\begin{equation*}
  F_{\alpha}(x) = \frac{x^{-\alpha}}{u_{\alpha}(x)} \cdot \frac{\Hop{\alpha}[u_{\alpha}](x) + \frac{1}{2}u_{\alpha}(x)^{2}}{x^{1 - \alpha}}
\end{equation*}
and compute Taylor models of the two factors independently, which are
then multiplied. See Appendix~\ref{sec:taylor-models-expansions-x} for
how to compute Taylor models of expansions in \(x\).

\subsection{Evaluation of \(F_{\alpha}\) in hybrid cases}
\label{sec:evaluation-F-hybrid}
Near \(\alpha = -1\) we don't have to do anything special for
\(x \in [\epsilon, \pi]\), the approach is thus the same as described
in Section~\ref{sec:evaluation-F-I-2}. For \(x \in [0, \epsilon]\) it
is in principle the same as well, but to get good enclosures it
requires slightly more care. The reasons this needs to be done is that
for \(\alpha\) near \(-1\) the factor
\begin{equation*}
  \frac{\Hop{\alpha}[u_{\alpha}](x) + \frac{1}{2}u_{\alpha}(x)^{2}}{x^{p - \alpha}}
\end{equation*}
tends to zero very slowly and there are several terms in the
asymptotic expansions between which there are large cancellations. The
main adjustment is to more carefully handle the cancellations between
the leading terms of \(\Hop{\alpha}[u_{\alpha}](x)\) and
\(\frac{1}{2}u_{\alpha}(x)^{2}\). We also make heavy use of
Lemma~\ref{lemma:clausenc-expansion-singular-term} when evaluating the
expansion of \(\Hop{\alpha}[u_{\alpha}](x)\).

Near \(\alpha = 0\) the approach is similar to that used for
\(\alpha \in I_{3}\). Instead of the Taylor models being centered at
zero they are centered at (the midpoint) of \(\alpha\). Since we in
the end don't need an expansion in \(\alpha\) we enclose the computed
Taylor models over the interval, giving us an enclosure of
\(F_{\alpha}(x)\).

\section{Analysis of \(T_{\alpha}\)}
\label{sec:analysis-T}
In this section we give more details about the operator \(T_{\alpha}\)
defined by
\begin{equation*}
  T_{\alpha}v = -\frac{1}{w_{\alpha}u_{\alpha}}\Hop{\alpha}[w_{\alpha}v]
\end{equation*}
and discuss how to bound \(D_{\alpha} := \|T_{\alpha}\|\).

The operator \(\Hop{\alpha}\) is defined by
\begin{equation*}
  \Hop{\alpha}[v](x) = |D|^{\alpha}[v](x) - |D|^{\alpha}[v](0).
\end{equation*}
For even functions \(v(x)\) and \(0 < x < \pi\) it has the integral
representation
\begin{equation*}
  \Hop{\alpha}[v](x) = -\frac{1}{\pi}\int_{0}^{\pi}(C_{-\alpha}(x - y) + C_{-\alpha}(x + y) - 2C_{-\alpha}(y))v(y)\ dy.
\end{equation*}
This gives us
\begin{equation*}
  T_{\alpha}[v](x) = \frac{1}{\pi w_{\alpha}(x) u_{\alpha}(x)}
  \int_{0}^{\pi}(C_{-\alpha}(x - y) + C_{-\alpha}(x + y) - 2C_{-\alpha}(y))w_{\alpha}(y)v(y)\ dy.
\end{equation*}

Using the above expressions it is standard that the norm of
\(T_{\alpha}\) is given by
\begin{equation}
  \label{eq:C_alpha}
  D_{\alpha} = \|T_{\alpha}\| =
  \sup_{0 < x < \pi} \frac{1}{\pi |w_{\alpha}(x)| |u_{\alpha}(x)|}
  \int_{0}^{\pi}|C_{-\alpha}(x - y) + C_{-\alpha}(x + y) - 2C_{-\alpha}(y)||w_{\alpha}(y)|\ dy.
\end{equation}
Let
\begin{equation*}
  I_{\alpha}(x, y) = C_{-\alpha}(x - y) + C_{-\alpha}(x + y) - 2C_{-\alpha}(y)
\end{equation*}
and
\begin{equation*}
  U_{\alpha}(x) = \int_{0}^{\pi}|I_{\alpha}(x, y)|w_{\alpha}(y)\ dt,
\end{equation*}
where we have removed the absolute value around \(w_{\alpha}\) since
it is positive. We are then interested in computing
\begin{equation}
  \label{eq:D}
  D_{\alpha} = \|T_{\alpha}\| = \sup_{0 < x < \pi} \frac{U_{\alpha}(x)}{\pi |w_{\alpha}(x)| |u_{\alpha}(x)|}.
\end{equation}
We use the notation
\begin{equation}
  \label{eq:mathcal-T}
  \mathcal{T}_{\alpha}(x) = \frac{U_{\alpha}(x)}{\pi |w_{\alpha}(x)| |u_{\alpha}(x)|}.
\end{equation}

\subsection{Properties of \(U_{\alpha}\)}
\label{sec:properties-T}
Before discussing how to evaluate \(\mathcal{T}_{\alpha}\) we give
some general properties about \(U_{\alpha}\) that will be useful.

The integrand of \(U_{\alpha}(x)\) has a singularity at \(y = x\). It
is therefore natural to split the integral into two parts
\begin{equation*}
  U_{\alpha,1}(x) = \int_{0}^{x}|I_{\alpha}(x, y)|w_{\alpha}(y)\ dy,\quad
  U_{\alpha,2}(x) = \int_{x}^{\pi}|I_{\alpha}(x, y)|w_{\alpha}(y)\ dy.
\end{equation*}
In some cases it will be beneficial to make the change of variables
\(y = tx\), giving us
\begin{equation*}
  U_{\alpha,1}(x) = x\int_{0}^{1}|\hat{I}_{\alpha}(x, t)|w_{\alpha}(tx)\ dt,\quad
  U_{\alpha,2}(x) = x\int_{1}^{\pi / x}|\hat{I}_{\alpha}(x, t)|w_{\alpha}(tx)\ dt,\quad
\end{equation*}
where
\begin{equation*}
  \hat{I}_{\alpha}(x, t) = C_{-\alpha}(x(1 - t)) + C_{-\alpha}(x(1 + t)) - 2C_{-\alpha}(xt).
\end{equation*}
The following lemmas give information about the sign of \(I(x, y)\)
and \(\hat{I}(x, t)\), allowing us to remove the absolute value.
\begin{lemma}
  \label{lemma:I-0-x-helper}
  For all \(\alpha \in (-1, 0)\) the function
  \begin{equation*}
    (1 - t)^{-\alpha - 1} + (1 + t)^{-\alpha - 1} - 2t^{-\alpha - 1}
  \end{equation*}
  is increasing and continuous in \(t\) for \(t \in (0, 1)\) and has a
  unique root \(r_{\alpha,0}\). For \(x \in (0, \pi)\) and
  \(t \in (0, r_{\alpha,0})\) the function \(\hat{I}_{\alpha}(x, t)\)
  is increasing in \(x\).
\end{lemma}
\begin{proof}
  Differentiating the function with respect to \(t\) gives us
  \begin{equation*}
    (-\alpha - 1)(-(1 - t)^{-\alpha - 2} + (1 + t)^{-\alpha - 2} - 2t^{-\alpha - 2}),
  \end{equation*}
  we want to prove that this is positive. Since \(-\alpha - 1 < 0\)
  and \(t^{-\alpha - 2} > 0\) it is enough to prove that
  \begin{equation*}
    -(1 - t)^{-\alpha - 2} + (1 + t)^{-\alpha - 2} < 0,
  \end{equation*}
  which follows immediately from that \(-\alpha - 2 < 0\) and
  \(1 - t < 1 + t\). This proves that the function is increasing.
  Noticing that the limit \(t \to 0\) is negative and \(t \to 1\) is
  positive gives us the existence of a unique root \(r_{\alpha,0}\).

  To get the monotonicity for \(\hat{I}_{\alpha}\) we differentiate
  with respect to \(x\),
  \begin{equation*}
    \frac{d}{dx}\hat{I}_{\alpha}(x, t) = -(1 - t)S_{-\alpha - 1}(x(1 - t)) - (1 + t)S_{-\alpha - 1}(x(1 + t)) + 2tS_{\alpha - 1}(xt),
  \end{equation*}
  and expand,
  \begin{multline*}
    \frac{d}{dx}\hat{I}_{\alpha}(x, t) =
    -\Gamma(2 + \alpha)\cos\left(-\frac{\pi}{2}(\alpha + 1)\right)x^{-\alpha - 2}((1 - t)^{-\alpha - 1} + (1 + t)^{-\alpha - 1} - 2t^{\alpha - 1})\\
    - \sum_{m = 0}^{\infty}(-1)^{m}\zeta(-\alpha - 1 - 2m)\frac{x^{2m}}{(2m)!}((1 - t)^{2m + 1} + (1 + t)^{2m + 1} - 2t^{2m + 1}).
  \end{multline*}
  We have
  \(-\Gamma(2 + \alpha)\cos\left(-\frac{\pi}{2}(\alpha +
    1)\right)x^{-\alpha - 2} < 0\) and since \(0 < t < r_{\alpha,0}\)
  we have from the above that
  \((1 - t)^{-\alpha - 1} + (1 + t)^{-\alpha - 1} - 2t^{\alpha - 1} <
  0\), the first term is hence positive. Due to the location of the
  zeros of the zeta function on the negative real axis the factor
  \((-1)^{m}\zeta(-\alpha - 1 - 2m)\) is negative. Hence all terms in
  the sum are negative and the sum is decreasing in \(x\). Taking into
  account the minus sign in front of the sum gives us something
  increasing.

  All terms in the sum are, due
  to the location of the zeta function on the negative real axis,
  adding the minus sign in front means that the whole expression is
  positive. It follows that \(\hat{I}_{\alpha}(x, t)\) is increasing
  in \(x\).
\end{proof}

\begin{lemma}
  \label{lemma:I-0-x}
  For all \(\alpha \in (-1, 0)\) and \(x \in (0, \pi)\) the function
  \(\hat{I}_{\alpha}(x, t)\) is increasing and continuous in \(t\) for
  \(t \in (0, 1)\) and has the limits
  \begin{align*}
    \lim_{t \to 0^{+}} \hat{I}_{\alpha}(x, y) &= -\infty,\\
    \lim_{t \to 1^{-}} \hat{I}_{\alpha}(x, y) &= \infty.
  \end{align*}
  Moreover, the unique root, \(r_{\alpha,x}\), in \(t\) is decreasing
  in \(x\) and satisfies the inequality
  \begin{equation*}
    \frac{1}{2} < r_{\alpha,x} < r_{\alpha,0},
  \end{equation*}
  with \(r_{\alpha,0}\) as in lemma~\ref{lemma:I-0-x-helper}.
\end{lemma}
\begin{proof}
  The continuity in \(t\) follows from that \(C_{\alpha}\) is
  continuous on the interval \((0, 2\pi)\) and that the arguments all
  stay in this range. For the left limit we note that
  \(C_{-\alpha}(x(1 - t))\) and \(C_{-\alpha}(x(1 + t))\) both remain
  finite, while \(C_{-\alpha}(xt)\) diverges towards \(\infty\). For
  the right limit \(C_{-\alpha}(x(1 + t))\) and \(C_{-\alpha}(xt)\)
  are finite while \(C_{-\alpha}(x(1 - t))\) diverges towards
  \(\infty\).

  To show that it is increasing in \(t\) we differentiate, giving us
  \begin{equation*}
    \frac{d}{dt}\hat{I}_{\alpha}(x, t) = x(S_{-\alpha - 1}(x(1 - t)) - S_{-\alpha - 1}(x(1 + t)) + 2S_{-\alpha - 1}(xt)).
  \end{equation*}
  We want to prove that this is positive. Note that \(0 < xt < \pi\)
  and from Lemma~\ref{lemma:clausenc-monotone} we have that
  \(S_{-\alpha - 1}\) is positive on the interval \((0, \pi)\). Since
  \(x\) is also positive it thus satisfies to check that
  \begin{equation*}
    S_{-\alpha - 1}(x(1 - t)) - S_{-\alpha - 1}(x(1 + t)) > 0,
  \end{equation*}
  for which it is enough to assert that \(S_{-\alpha - 1}\) is
  decreasing on \((0, 2\pi)\), which is the result of
  Lemma~\ref{lemma:clausens-monotone}. This proves the existence of a
  unique root \(r_{\alpha, x}\) on the interval \((0, 1)\).

  To prove that \(r_{\alpha,x}\) is decreasing in \(x\) we first prove
  that it is upper bounded by \(r_{\alpha,0}\). Expanding the Clausen
  functions gives us
  \begin{multline*}
    \hat{I}_{\alpha}(x, t) =
    \Gamma(1 + \alpha)\sin\left(-\frac{\pi}{2}\alpha\right)x^{-\alpha - 1}((1 - t)^{-\alpha - 1} + (1 + t)^{-\alpha - 1} - 2t^{\alpha - 1})\\
    + \sum_{m = 1}^{\infty}(-1)^{m}\zeta(-\alpha - 2m)\frac{x^{2m}}{(2m)!}((1 - t)^{2m} + (1 + t)^{2m} - 2t^{2m}).
  \end{multline*}
  All terms in the sum are positive, due to the location of the roots
  of the zeta function on the negative real axis, and to have a root
  we must hence have
  \begin{equation*}
    \Gamma(1 + \alpha)\sin\left(-\frac{\pi}{2}\alpha\right)x^{-\alpha - 1}((1 - t)^{-\alpha - 1} + (1 + t)^{-\alpha - 1} - 2t^{\alpha - 1}) < 0.
  \end{equation*}
  Since
  \(\Gamma(1 + \alpha)\sin\left(-\frac{\pi}{2}\alpha\right)x^{-\alpha
    - 1} > 0\) this means we must have
  \begin{equation*}
    (1 - t)^{-\alpha - 1} + (1 + t)^{-\alpha - 1} - 2t^{\alpha - 1} < 0,
  \end{equation*}
  but from the previous lemma we know that this only holds for
  \(t < r_{\alpha,0}\), it follows that
  \(r_{\alpha,x} < r_{\alpha,0}\). The monotonicity in \(x\) now
  follows directly from that \(\hat{I}_{\alpha}(x, t)\) is increasing
  in \(x\) for \(t \in (0, r_{\alpha,0})\). To get the lower bound
  \(\frac{1}{2} < r_{\alpha,x}\) it is enough to notice that
  \begin{equation*}
    \hat{I}_{\alpha}(\pi, 1 / 2) = C_{-\alpha}(\pi / 2) + C_{-\alpha}(\pi / 2) - 2C_{-\alpha}(\pi / 2) = 0.
  \end{equation*}
\end{proof}
In practice the root \(r_{\alpha,x}\) is decreasing also in
\(\alpha\). However, instead of proving this in the general case the
following lemma can be used as an a posteriori check.
\begin{lemma}
  \label{lemma:I-0-x-r-alpha}
  Let
  \(\boldsymbol{\alpha} = [\lo{\alpha}, \hi{\alpha}] \subseteq (-1,
  0)\) and \(x \in (0, \pi)\). If
  \begin{equation*}
    \frac{d}{d\alpha} \hat{I}_{\alpha}(x, r_{\lo{\alpha},x}) > 0
    \quad\text{for all}\quad
    \alpha \in \boldsymbol{\alpha},
  \end{equation*}
  then \(r_{\alpha,x} \leq r_{\lo{\alpha},x}\) for all
  \(\alpha \in \boldsymbol{\alpha}\). Similarly, if
  \(\frac{d}{d\alpha} \hat{I}_{\alpha}(x, r_{\hi{\alpha},x}) > 0\) for
  all \(\alpha \in \boldsymbol{\alpha}\), then
  \(r_{\alpha,x} \geq r_{\hi{\alpha},x}\) for all
  \(\alpha \in \boldsymbol{\alpha}\).
\end{lemma}
\begin{proof}
  We only prove the first statement, the second one is similar. By
  definition of \(r_{x,\lo{\alpha}}\) we have
  \({\hat{I}_{\lo{\alpha}}(x, r_{x,\lo{\alpha}}) = 0}\) and hence if
  \(\frac{d}{d\alpha} \hat{I}_{\alpha}(x, r_{x,\lo{\alpha}}) > 0\) we
  have \(\hat{I}_{\alpha}(x, r_{x,\lo{\alpha}}) \geq 0\). Since
  \(\hat{I}_{\alpha}(x, t)\) is increasing in \(t\) this means that
  \(r_{x,\alpha} \leq r_{x,\lo{\alpha}}\).
\end{proof}

\begin{lemma}
  \label{lemma:I-x-pi}
  For all \(\alpha \in (-1, 0)\) and \(x \in (0, \pi)\) we have
  \(I_{\alpha}(x, y) > 0\) for \(y \in (x, \pi)\) and
  \(\hat{I}_{\alpha}(x, t) > 0\) for \(t \in (1, \pi / x)\).
\end{lemma}
\begin{proof}
  The function \(C_{-\alpha}(y)\) is strictly convex for
  \(y \in (0, 2\pi)\), this can be seen from the integral
  representation of
  \(\frac{d^{2}}{dy^{2}}C_{-\alpha}(y) = -C_{-\alpha - 2}(y)\) given
  in the proof of Lemma~\ref{lemma:clausens-monotone}. It immediately
  follows that
  \begin{equation*}
    I_{\alpha}(x, y) = C_{-\alpha}(y - x) + C_{-\alpha}(y + x) - 2C_{-\alpha}(y) < 0.
  \end{equation*}
  A simple change of variables gives the result for
  \(\hat{I}_{\alpha}(x, t)\).
\end{proof}

With the help of Lemma~\ref{lemma:I-0-x} and~\ref{lemma:I-x-pi} we can
slightly simplify \(U_{1}\) and \(U_{2}\), we get
\begin{equation*}
  U_{\alpha,1}(x) = -x\int_{0}^{r_{\alpha,x}} \hat{I}(x, t)w_{\alpha}(tx)\ dt
  + x\int_{r_{\alpha,x}}^{1} \hat{I}(x, t)w_{\alpha}(tx)\ dt
  = U_{\alpha,1,1}(x) + U_{\alpha,1,2}(x)
\end{equation*}
and
\begin{equation*}
  U_{\alpha,2}(x) = \int_{x}^{\pi}I(x, y)w_{\alpha}(y)\ dy.
\end{equation*}

In general the integrals cannot be explicitly computed, the exception
is when \(w_{\alpha}(x) = |x|\) in which case we have the following
lemma.
\begin{lemma}
  \label{lemma:U-primitive}
  For \(\alpha \in (-1, 0)\), \(x \in (0, \pi)\) and
  \(w_{\alpha}(x) = |x|\) we have
  \begin{multline*}
    U_{\alpha}(x) = 2\tilde{C}_{2 - \alpha}(x) + 2(C_{2 - \alpha}(x + \pi) - C_{2 - \alpha}(\pi))\\
    - 2\left(C_{2 - \alpha}(x(1 - r_{\alpha,x})) + C_{2 - \alpha}(x(1 + r_{\alpha,x})) - 2C_{2 - \alpha}(xr_{\alpha,x})\right)\\
    - 2xr_{\alpha,x}\left(-S_{1 - \alpha}(x(1 - r_{\alpha,x})) + S_{1 - \alpha}(x(1 + r_{\alpha,x})) - 2S_{1 - \alpha}(xr_{\alpha,x})\right).
  \end{multline*}
\end{lemma}
\begin{proof}
  For \(w_{\alpha}(x) = |x|\) we have
  \begin{equation*}
    U_{\alpha}(x) = -U_{\alpha,1,1}(x) + U_{\alpha,1,2}(x) + U_{\alpha,2}(x)
  \end{equation*}
  with
  \begin{align*}
    U_{\alpha,1,1}(x) &= x^{2} \int_{0}^{r_{\alpha,x}}(C_{-\alpha}(x(1 - t)) + C_{-\alpha}(x(1 + t)) - 2C_{-\alpha}(xt)) t\ dt,\\
    U_{\alpha,1,2}(x) &= x^{2} \int_{r_{\alpha,x}}^{1}(C_{-\alpha}(x(1 - t)) + C_{-\alpha}(x(1 + t)) - 2C_{-\alpha}(xt)) t\ dt,\\
    U_{\alpha,2}(x) = & x^{2} \int_{1}^{\pi / x} (C_{-\alpha}(x(1 - t)) + C_{-\alpha}(x(1 + t)) - 2C_{-\alpha}(xt)) t\ dt.
  \end{align*}
  If we let
  \begin{align*}
    J_{\alpha}(x, t) &= \int (C_{-\alpha}(x(1 - t)) + C_{-\alpha}(x(1 + t)) - 2C_{-\alpha}(xt)) t\ dt\\
                     &= \frac{1}{x^{2}}\left(C_{2 - \alpha}(x(1 - t)) + C_{2 - \alpha}(x(1 + t)) - 2C_{2 - \alpha}(xt)\right)
                       + \frac{t}{x}\left(-S_{1 - \alpha}(x(1 - t)) + S_{1 - \alpha}(x(1 + t)) - 2S_{1 - \alpha}(xt)\right)
  \end{align*}
  we get
  \begin{align*}
    U_{\alpha}(x) &= -x^{2}(J_{\alpha}(x, r_{\alpha,x}) - J_{\alpha}(x, 0))
                    + x^{2}(J_{\alpha}(x, 1) - J_{\alpha,x}(x, r_{\alpha,x}))
                    + x^{2}(J_{\alpha}(x, \pi / x) - J_{\alpha}(x, 1))\\
                  &= x^{2}J_{\alpha}(x, 0) - 2x^{2}J_{\alpha}(x, r_{x,\alpha}) + x^{2}J_{\alpha}(x, \pi / x),
  \end{align*}
  where we have used that \(J_{\alpha}(x, 1)\) is finite. Both
  \(J_{\alpha}(x, 0)\) and \(J_{\alpha}(x, \pi / x)\) can be simplified
  further. We have
  \begin{equation*}
    J_{\alpha}(x, 0) = \frac{2}{x^{2}}\left(C_{2 - \alpha}(x) - C_{2 - \alpha}(0)\right)
    = \frac{2}{x^{2}}\tilde{C}_{2 - \alpha}(x).
  \end{equation*}
  and
  \begin{align*}
    J_{\alpha}(x, \pi / x) &= \frac{1}{x^{2}}\left(C_{2 - \alpha}(x - \pi) + C_{2 - \alpha}(x + \pi) - 2C_{2 - \alpha}(\pi)\right)
                             + \frac{\pi}{x}\left(-S_{1 - \alpha}(x - \pi) + S_{1 - \alpha}(x + \pi) - 2S_{1 - \alpha}(\pi)\right)\\
                           &= \frac{2}{x^{2}}(C_{2 - \alpha}(x + \pi) - C_{2 - \alpha}(\pi)),
  \end{align*}
  where we have used that \(C_{s}(x - \pi) = C_{s}(x + \pi)\),
  \(S_{s}(x - \pi) = S_{s}(x + \pi)\) and \(S_{s}(\pi) = 0\). Putting
  all of this together gives us the result.
\end{proof}

\subsection{Evaluation of \(\mathcal{T}_{\alpha}\)}
\label{sec:evaluation-T}
Similarly to in the previous sections we divide the interval
\([0, \pi]\), on which we take the supremum, into two parts,
\([0, \epsilon]\) and \([\epsilon, \pi]\).

\subsubsection{Evaluation of \(\mathcal{T}_{\alpha}\) for \(I_{2}\)}
\label{sec:evaluation-T-I-2}
For the weight \(w_{\alpha}(x) = |x|\) evaluation is straight forward
for \(x \in [\epsilon, \pi]\), using Lemma~\ref{lemma:U-primitive}.
For \(x \in [0, \epsilon]\) we write
\begin{equation}
  \label{eq:evaluation-T-I-2-split}
  \mathcal{T}_{\alpha}(x) = \frac{1}{\pi}\frac{x^{-\alpha}}{u_{\alpha}(x)} \cdot
  \frac{U_{\alpha}(x)}{x^{1 - \alpha}}.
\end{equation}
The first factor is enclosed using
Lemma~\ref{lemma:asymptotic-inv-u0}. For the second factor we use
Lemma~\ref{lemma:U-primitive} and expand the Clausen functions at zero
to be able to cancel the removable singularity. For both
\(x \in [0, \epsilon]\) and \(x \in [\epsilon, \pi]\) we need to
compute enclosures of the root of the integrand \(r_{\alpha,x}\), this
is done using standard interval methods for root enclosures, the
monotonicity in \(x\) from Lemma~\ref{lemma:I-0-x} and
Lemma~\ref{lemma:I-0-x-r-alpha}.

When the weight is \(w_{\alpha}(x) = |x|^{p}\) with \(p < 1\) more
work is required when evaluating \(U_{\alpha}\). For
\(x \in [\epsilon, \pi]\) it uses a combination of asymptotic analysis
near the singularities of the integrand and a rigorous integrator, as
described in Appendix~\ref{sec:rigorous-integration}. For
\(x \in [0, \epsilon]\) we split it as in
\eqref{eq:evaluation-T-I-2-split} and use the following lemma to
handle the second factor.
\begin{lemma}
  \label{lemma:evaluation-U-I-2}
  Let \(0 < \epsilon < \frac{\pi}{2}\), for \(\alpha \in (-1, 0)\),
  \(x \in [0, \epsilon]\) and \(w_{\alpha}(x) = |x|^{p}\) with
  \(-\alpha < p < 1\) and \(1 + \alpha \not= p\) we have
  \begin{equation*}
    \frac{U_{\alpha,1}(x)}{x^{-\alpha + p}} \leq c_{\alpha,1} + d_{\alpha,1}x^{3 + \alpha}
  \end{equation*}
  and
  \begin{equation*}
    \frac{U_{\alpha,2}(x)}{x^{-\alpha + p}} \leq c_{\alpha,2} + d_{\alpha,2}x^{2 + \alpha - p},
  \end{equation*}
  where
  \begin{align*}
    c_{\alpha,1} =& \Gamma(1 + \alpha)\sin\left(-\frac{\pi}{2}\alpha\right)\Bigg(\frac{2}{\alpha - p} + \frac{\Gamma(-\alpha)\Gamma(1 + p)}{\Gamma(1 - \alpha + p)} + \frac{{}_{2}F_{1}(1 + \alpha, 1 + p; 2 + p; -1)}{1 + p}\\
                  &- 2r_{\alpha,0}^{p}\left(\frac{2r_{\alpha,0}^{-\alpha}}{\alpha - p} + \frac{r_{\alpha,0}\ {}_{2}F_{1}(1 + \alpha, 1 + p, 2
                    + p, -r_{\alpha,0})}{1 + p} + \frac{r_{\alpha,0}\ {}_{2}F_{1}(1 + \alpha, 1 + p, 2 + p, r_{\alpha,0})}{1 + p}\right)\Bigg)\\
    c_{\alpha,2} =& \Gamma(1 + \alpha)\sin\left(-\frac{\pi}{2}\alpha\right)\left(\frac{\Gamma(-\alpha)\Gamma(\alpha - p)}{\Gamma(-p)} + \frac{{}_{2}F_{1}(1 + \alpha, \alpha - p; 1 + \alpha - p; -1) - 2}{\alpha - p}\right);\\
    d_{\alpha,1} =& 2\sum_{m = 1}^{M - 1} (-1)^{m}\zeta(-\alpha - 2m)\frac{\epsilon^{2m - 2}}{(2m)!}\sum_{k = 0}^{m - 1}\binom{2m}{2k}\frac{1}{2k + 1 + p}
                    + \frac{1}{\epsilon^{2}}\sum_{m = M}^{\infty} (-1)^{m}\zeta(-\alpha - 2m)\frac{(2\epsilon)^{2m}}{(2m)!};\\
    d_{\alpha,2} =& -\Gamma(1 + \alpha)\sin\left(-\frac{\pi}{2}\alpha\right)\frac{(1 + \alpha)(2 + \alpha)}{(2 + \alpha - p)\pi^{2 + \alpha - p}}\\
                  &+ 2\pi^{p - 1}\sum_{m = 1}^{M-1} (-1)^{m}\zeta(-\alpha - 2m)\frac{\pi^{2m}}{(2m)!}
                    \sum_{k = 0}^{m - 1}\binom{2m}{2k}\frac{\left(\frac{\epsilon}{\pi}\right)^{2(m - 1 - k)}}{2k + 1 + p}
                    + 6\pi^{p - 1}\sum_{m = M}^{\infty} (-1)^{m}\zeta(-\alpha - 2m)\frac{(\frac{3\pi}{2})^{2m}}{(2m)!}
  \end{align*}
  The tails for \(d_{\alpha,1}\) and \(d_{\alpha,2}\) are of the same
  form as in those for the Clausen functions and can be bounded using
  Lemma~\ref{lemma:clausen-tails}.
\end{lemma}
\begin{proof}
  Recall that
  \begin{align*}
    U_{\alpha,1}(x) &= x^{1 + p} \int_{0}^{1}\left|C_{-\alpha}(x(1 - t)) + C_{-\alpha}(x(1 + t)) - 2C_{-\alpha}(xt)\right| t^{p}\ dt,\\
    U_{\alpha,2}(x) = & x^{1 + p} \int_{1}^{\pi / x} (C_{-\alpha}(x(1 - t)) + C_{-\alpha}(x(1 + t)) - 2C_{-\alpha}(xt)) t^{p}\ dt.
  \end{align*}
  The idea is to expand the integrals in \(x\) and integrate the
  expansions termwise.

  From the asymptotic expansion of \(C_{s}(x)\) we get, with
  \(x, t > 0\),
  \begin{multline}
    \label{eq:integrand-expansion}
    C_{-\alpha}(x(1 - t)) + C_{-\alpha}(x(1 + t)) - 2C_{-\alpha}(xt) =
    \Gamma(1 + \alpha)\sin\left(-\frac{\pi}{2}\alpha\right)(|1 - t|^{-\alpha - 1} + (1 + t)^{-\alpha - 1} - 2t^{-\alpha - 1})x^{-\alpha - 1}\\
    + \sum_{m = 1}^{\infty} (-1)^{m}\zeta(-\alpha - 2m)((1 - t)^{2m} + (1 + t)^{2m} - 2t^{2m})\frac{x^{2m}}{(2m)!}.
  \end{multline}
  Using that
  \begin{equation*}
    (1 - t)^{2m} + (1 + t)^{2m} - 2t^{2m} = \sum_{k = 0}^{2m}\binom{2m}{k}(1 + (-1)^{k})t^{k} - 2t^{2m}
    = 2\sum_{k = 0}^{m - 1}\binom{2m}{2k}t^{2k}
  \end{equation*}
  the sum can be rewritten as
  \begin{equation*}
    \sum_{m = 1}^{\infty} (-1)^{m}\zeta(-\alpha - 2m)((1 - t)^{2m} + (1 + t)^{2m} - 2t^{2m})\frac{x^{2m}}{(2m)!}
    = 2\sum_{m = 1}^{\infty} (-1)^{m}\zeta(-\alpha - 2m)\frac{x^{2m}}{(2m)!}\sum_{k = 0}^{m - 1}\binom{2m}{2k}t^{2k}.
  \end{equation*}
  We can further note that \((-1)^{m}\zeta(-\alpha - 2m) > 0\) for
  \(-1 < \alpha < 0\) and \(m = 1, 2, \dots\), so all terms in the sum
  are positive.

  For \(U_{\alpha,1}\) we get
  \begin{multline*}
    U_{\alpha,1}(x) \leq \Gamma(1 + \alpha)\sin\left(-\frac{\pi}{2}\alpha\right)x^{-\alpha + p}\int_{0}^{1}\left|(1 - t)^{-\alpha - 1} + (1 + t)^{-\alpha - 1} - 2t^{-\alpha - 1}\right|t^{p}\ dt\\
    + 2x^{1 + p}\sum_{m = 1}^{\infty} (-1)^{m}\zeta(-\alpha - 2m)\frac{x^{2m}}{(2m)!}\sum_{k = 0}^{m - 1}\binom{2m}{2k}\int_{0}^{1}t^{2k + p}\ dt.
  \end{multline*}
  Where we have use that \(1 - t > 0\) and the positivity of the terms
  in the sum to remove some absolute values. For the first term we get
  from Lemma~\ref{lemma:I-0-x} that the integrand has the unique root
  \(r_{\alpha,0}\) on the interval \([0, 1]\). This gives us
  \begin{multline*}
    \int_{0}^{1}\left|(1 - t)^{-\alpha - 1} + (1 + t)^{-\alpha - 1} - 2t^{-\alpha - 1}\right|t^{p}\ dt\\
    = -\int_{0}^{r_{\alpha,0}}((1 - t)^{-\alpha - 1} + (1 + t)^{-\alpha - 1} - 2t^{-\alpha - 1})t^{p}\ dt\\
    + \int_{r_{\alpha,0}}^{1}((1 - t)^{-\alpha - 1} + (1 + t)^{-\alpha - 1} - 2t^{-\alpha - 1})t^{p}\ dt.
  \end{multline*}
  Integrating this gives us
  \begin{multline*}
    \int_{0}^{1}\left|(1 - t)^{-\alpha - 1} + (1 + t)^{-\alpha - 1} - 2t^{-\alpha - 1}\right|t^{p}\ dt
    = \frac{2}{\alpha - p} + \frac{\Gamma(-\alpha)\Gamma(1 + p)}{\Gamma(1 - \alpha + p)} + \frac{{}_{2}F_{1}(1 + \alpha, 1 + p; 2 + p; -1)}{1 + p}\\
    - 2r_{\alpha,0}^{p}\left(\frac{2r_{\alpha,0}^{-\alpha}}{\alpha - p} + \frac{r_{\alpha,0}\ {}_{2}F_{1}(1 + \alpha, 1 + p, 2 + p, -r_{\alpha,0})}{1 + p} + \frac{r_{\alpha,0}\ {}_{2}F_{1}(1 + \alpha, 1 + p, 2 + p, r_{\alpha,0})}{1 + p}\right)
  \end{multline*}
  This, together with the factor
  \(\Gamma(1 + \alpha)\sin\left(-\frac{\pi}{2}\alpha\right)\) gives us
  \(c_{\alpha,1}\).

  For the sum we have
  \(\int_{0}^{1}t^{2k + p}\ dt = \frac{1}{2k + 1 + p}\), giving us
  \begin{equation*}
    2x^{1 + p}\sum_{m = 1}^{\infty} (-1)^{m}\zeta(-\alpha - 2m)\frac{x^{2m}}{(2m)!}
    \sum_{k = 0}^{m - 1}\binom{2m}{2k}\frac{1}{2k + 1 + p}.
  \end{equation*}
  Factoring out \(x^{2}\) and using that the sum is increasing in
  \(x\) we get the bound
  \begin{equation*}
    2x^{3 + p}\sum_{m = 1}^{\infty} (-1)^{m}\zeta(-\alpha - 2m)\frac{\epsilon^{2m - 2}}{(2m)!}
    \sum_{k = 0}^{m - 1}\binom{2m}{2k}\frac{1}{2k + 1 + p}.
  \end{equation*}
  To bound the tail we use that
  \begin{equation*}
    \sum_{k = 0}^{m - 1}\binom{2m}{2k}\frac{1}{2k + 1 + p}
    \leq \sum_{k = 0}^{m}\binom{2m}{2k}\frac{1}{2k + 1}
    = \frac{2^{2m}}{1 + 2m} \leq 2^{2m}
  \end{equation*}
  and hence
  \begin{equation*}
    2\sum_{m = M}^{\infty} (-1)^{m}\zeta(-\alpha - 2m)\frac{\epsilon^{2m - 2}}{(2m)!}
    \sum_{k = 0}^{m - 1}\binom{2m}{2k}\frac{1}{2k + 1 + p}
    \leq \frac{1}{\epsilon^{2}}\sum_{m = M}^{\infty} (-1)^{m}\zeta(-\alpha - 2m)\frac{(2\epsilon)^{2m}}{(2m)!}
  \end{equation*}
  This together with the first \(M - 1\) terms in the sum gives us
  \(d_{\alpha,1}\).

  For \(U_{\alpha,2}\) there is no absolute value and integrating
  termwise gives us
  \begin{multline*}
    U_{\alpha,2}(x) = \Gamma(1 + \alpha)\sin\left(-\frac{\pi}{2}\alpha\right)x^{-\alpha + p}\int_{1}^{\pi / x}((t - 1)^{-\alpha - 1} + (1 + t)^{-\alpha - 1} - 2t^{-\alpha - 1})t^{p}\ dt\\
    + 2x^{1 + p}\sum_{m = 1}^{\infty} (-1)^{m}\zeta(-\alpha - 2m)\frac{x^{2m}}{(2m)!}\sum_{k = 0}^{m - 1}\binom{2m}{2k}\int_{1}^{\pi / x}t^{2k + p}\ dt
  \end{multline*}

  For the first term we get after long calculations that
  \begin{multline*}
    \int_{1}^{\pi / x}((t - 1)^{-\alpha - 1} + (1 + t)^{-\alpha - 1} - 2t^{-\alpha - 1})t^{p}\ dt =
    \frac{\Gamma(-\alpha)\Gamma(\alpha - p)}{\Gamma(-p)}
    + \frac{{}_{2}F_{1}(1 + \alpha, \alpha - p; 1 + \alpha - p; -1) - 2}{\alpha - p}\\
    - \frac{1}{\alpha - p}\left(\frac{x}{\pi}\right)^{\alpha - p}\left(
      {}_{2}F_{1}\left(1 + \alpha, \alpha - p; 1 + \alpha - p; \frac{x}{\pi}\right)
      + {}_{2}F_{1}\left(1 + \alpha, \alpha - p; 1 + \alpha - p; -\frac{x}{\pi}\right)
      -2
    \right).
  \end{multline*}
  To handle the hypergeometric functions we use the series
  representation
  \begin{equation*}
    {}_{2}F_{1}\left(1 + \alpha, \alpha - p; 1 + \alpha - p; \frac{x}{\pi}\right)
    = \sum_{k = 0}^{\infty}\frac{(1 + \alpha)_{k}(\alpha - p)_{k}}{(1 + \alpha - p)_{k}}\frac{1}{k!}\left(\frac{x}{\pi}\right)^{k},
  \end{equation*}
  and similarly for \(-\frac{x}{\pi}\), which holds since
  \(\frac{x}{\pi} < 1\) and \(1 + \alpha - p\) is not equal to a
  non-positive integer. This gives us
  \begin{multline*}
    {}_{2}F_{1}\left(1 + \alpha, \alpha - p; 1 + \alpha - p; \frac{x}{\pi}\right)
    + {}_{2}F_{1}\left(1 + \alpha, \alpha - p; 1 + \alpha - p; -\frac{x}{\pi}\right)
    -2\\
    = \sum_{k = 0}^{\infty}\frac{(1 + \alpha)_{k}(\alpha - p)_{k}}{(1 + \alpha - p)_{k}}\frac{1}{k!}\left(\left(\frac{x}{\pi}\right)^{k} + \left(-\frac{x}{\pi}\right)^{k}\right) - 2
    = 2\sum_{k = 1}^{\infty}\frac{(1 + \alpha)_{2k}(\alpha - p)_{2k}}{(1 + \alpha - p)_{2k}}\frac{1}{(2k)!}\left(\frac{x}{\pi}\right)^{2k}
  \end{multline*}
  Putting this together we have
  \begin{multline*}
    \int_{1}^{\pi / x}((t - 1)^{-\alpha - 1} + (1 + t)^{-\alpha - 1} - 2t^{-\alpha - 1})t^{p}\ dt
    = \frac{\Gamma(-\alpha)\Gamma(\alpha - p)}{\Gamma(-p)}
    + \frac{{}_{2}F_{1}(1 + \alpha, \alpha - p; 1 + \alpha - p; -1) - 2}{\alpha - p}\\
    - \frac{2}{\alpha - p}\left(\frac{x}{\pi}\right)^{\alpha - p}
    \sum_{k = 1}^{\infty}\frac{(1 + \alpha)_{2k}(\alpha - p)_{2k}}{(1 + \alpha - p)_{2k}}\frac{1}{(2k)!}\left(\frac{x}{\pi}\right)^{2k}
  \end{multline*}
  Noticing that
  \begin{equation*}
    \frac{1}{\alpha - p}\frac{(1 + \alpha)_{2k}(\alpha - p)_{2k}}{(1 + \alpha - p)_{2k}} = \frac{(1 + \alpha)_{2k}}{2k + \alpha - p} > 0
  \end{equation*}
  for \(k \geq 1\) we see that all terms in the sum are positive.
  Since we are subtracting the sum we get an upper bound even if we
  truncate the sum to any finite number of terms. In particular,
  keeping only the first term we get
  \begin{multline*}
    \int_{1}^{\pi / x}((t - 1)^{-\alpha - 1} + (1 + t)^{-\alpha - 1} - 2t^{-\alpha - 1})t^{p}\ dt
    \leq \frac{\Gamma(-\alpha)\Gamma(\alpha - p)}{\Gamma(-p)}
    + \frac{{}_{2}F_{1}(1 + \alpha, \alpha - p; 1 + \alpha - p; -1) - 2}{\alpha - p}\\
    - \frac{(1 + \alpha)(2 + \alpha)}{2 + \alpha - p}\left(\frac{x}{\pi}\right)^{2 + \alpha - p}
  \end{multline*}
  Multiplying this with
  \(\Gamma(1 + \alpha)\sin\left(-\frac{\pi}{2}\alpha\right)\) (which
  is positive) we get \(c_{\alpha,2}\) from the constant term and the
  first part of \(d_{\alpha,2}\) from the non-constant term.

  For the sum we get
  \begin{equation*}
    2x^{1 + p}\sum_{m = 1}^{\infty} (-1)^{m}\zeta(-\alpha - 2m)\frac{x^{2m}}{(2m)!}
    \sum_{k = 0}^{m - 1}\binom{2m}{2k}\frac{\left(\frac{\pi}{x}\right)^{2k + 1 + p} - 1}{2k + 1 + p}
  \end{equation*}
  Factoring out \((\pi / x)^{2m - 1 + p}\) gives us
  \begin{multline*}
    2x^{1 + p}\sum_{m = 1}^{\infty} (-1)^{m}\zeta(-\alpha - 2m)\frac{x^{2m}}{(2m)!}\left(\frac{\pi}{x}\right)^{2m - 1 + p}
    \sum_{k = 0}^{m - 1}\binom{2m}{2k}\frac{\left(\frac{x}{\pi}\right)^{2(m - 1 - k)} - \left(\frac{x}{\pi}\right)^{2m - 1 + p}}{2k + 1 + p}\\
    =  2x^{2}\pi^{p - 1}\sum_{m = 1}^{\infty} (-1)^{m}\zeta(-\alpha - 2m)\frac{\pi^{2m}}{(2m)!}
    \sum_{k = 0}^{m - 1}\binom{2m}{2k}\frac{\left(\frac{x}{\pi}\right)^{2(m - 1 - k)} - \left(\frac{x}{\pi}\right)^{2m - 1 + p}}{2k + 1 + p}.
  \end{multline*}
  Looking at the inner sum an upper bound is given by
  \begin{equation*}
    \sum_{k = 0}^{m - 1}\binom{2m}{2k}\frac{\left(\frac{x}{\pi}\right)^{2(m - 1 - k)}}{2k + 1 + p}.
  \end{equation*}
  This is increasing in \(x\) and an upper bound is hence given by
  \begin{equation*}
    2x^{2}\pi^{p - 1}\sum_{m = 1}^{\infty} (-1)^{m}\zeta(-\alpha - 2m)\frac{\pi^{2m}}{(2m)!}
    \sum_{k = 0}^{m - 1}\binom{2m}{2k}\frac{\left(\frac{\epsilon}{\pi}\right)^{2(m - 1 - k)}}{2k + 1 + p}.
  \end{equation*}
  To bound the tail we use that \(\epsilon \leq \pi / 2\) and hence
  \begin{equation*}
    \sum_{k = 0}^{m - 1}\binom{2m}{2k}\frac{\left(\frac{\epsilon}{\pi}\right)^{2(m - 1 - k)}}{2k + 1 + p}
    \leq \sum_{k = 0}^{m - 1}\binom{2m}{2k}\frac{\left(\frac{1}{2}\right)^{2(m - 1 - k)}}{2k + 1}
    = \frac{2^{-2m}(1 + 3^{1 + 2m}) - 4}{1 + 2m}
    < 3\left(\frac{3}{2}\right)^{2m}
  \end{equation*}
  Inserting this we have
  \begin{equation*}
    \sum_{m = M}^{\infty} (-1)^{m}\zeta(-\alpha - 2m)\frac{\pi^{2m}}{(2m)!}
    \sum_{k = 0}^{m - 1}\binom{2m}{2k}\frac{\left(\frac{\epsilon}{\pi}\right)^{2(m - 1 - k)}}{2k + 1}
    \leq 3\sum_{m = M}^{\infty} (-1)^{m}\zeta(-\alpha - 2m)\frac{(\frac{3\pi}{2})^{2m}}{(2m)!}.
  \end{equation*}
  This together with the first \(M - 1\) terms in the sum gives us the
  second part of \(d_{\alpha,2}\).
\end{proof}

\subsubsection{Evaluation of \(\mathcal{T}_{\alpha}\) for \(I_{1}\)}
\label{sec:evaluation-T-I-1}
For \(x \in [\epsilon, \pi]\) we make one optimization compared to
\(I_{2}\), instead of computing an enclosure we only compute an upper
bound. For this we use that \(u_{\alpha}(x) \geq u_{-1}(x)\) for
\(\alpha \in I_{1}\) and hence
\begin{equation*}
  \mathcal{T}_{\alpha}(x) \leq \frac{U_{\alpha}(x)}{\pi |w_{\alpha}(x)| |u_{-1}(x)|}.
\end{equation*}
The different weight also means that the asymptotic analysis near the
singularities of the integrand has to be adjusted, this is discussed
in Appendix~\ref{sec:rigorous-integration}.

For \(x \in [0, \epsilon]\) we write
\begin{equation*}
  \mathcal{T}_{\alpha}(x) = \frac{\log(1 / x)}{\pi\log(2e + 1 / x)} \cdot
  \frac{\Gamma(1 + \alpha)x^{-\alpha}(1 - x^{1 + \alpha + (1 + \alpha)^{2} / 2})}{u_{\alpha}(x)}
  \cdot \frac{U_{\alpha}(x)}{x^{(1 - \alpha) / 2 - \alpha}\log(1 / x)(1 - x^{1 + \alpha + (1 + \alpha)^{2} / 2})\Gamma(1 + \alpha)}.
\end{equation*}
The first two factors are handled in the same way as in
Section~\ref{sec:evaluation-F-I-1}. For the third factor we start with
the following lemma
\begin{lemma}
  \label{lemma:evaluation-U-I-1}
  For \(\alpha \in (-1, 0)\) and
  \(w_{\alpha}(x) = |x|^{(1 - \alpha) / 2}\log(2e + 1 / |x|)\) we have
  \begin{align*}
    \frac{U_{\alpha}(x)}{x^{(1 - \alpha) / 2 - \alpha}\log(1 / x)(1 - x^{1 + \alpha + (1 + \alpha)^{2} / 2})\Gamma(1 + \alpha)}
    \leq& \sin\left(-\frac{\pi}{2}\alpha\right)\left(G_{\alpha,1}(x) + G_{\alpha,2}(x)\right)\\
    &+ \frac{x^{1 + \alpha}}{\Gamma(1 + \alpha)\log(1 / x)(1 - x^{1 + \alpha + (1 + \alpha)^{2} / 2})}R_{\alpha}(x)
  \end{align*}
  with
  \begin{align}
    G_{\alpha,1}(x)
    &= \frac{
      \int_{0}^{1}\left|(1 - t)^{-\alpha - 1} + (1 + t)^{-\alpha - 1} - 2t^{-\alpha - 1}\right|t^{(1 - \alpha) / 2}\log(2e + 1 / (xt))\ dt
      }{
      \log(1 / x)(1 - x^{1 + \alpha + (1 + \alpha)^{2} / 2})
      }; \label{eq:norm-I-1-G1}\\
    G_{\alpha,2}(x)
    &= \frac{
      \int_{1}^{\pi / x}\left((t - 1)^{-\alpha - 1} + (1 + t)^{-\alpha - 1} - 2t^{-\alpha - 1}\right)t^{(1 - \alpha) / 2}\log(2e + 1 / (xt))\ dt
      }{
      \log(1 / x)(1 - x^{1 + \alpha + (1 + \alpha)^{2} / 2})
      }; \label{eq:norm-I-1-G2}\\
    R_{\alpha}(x)
    &=  2\sum_{m = 1}^{\infty} (-1)^{m}\zeta(-\alpha - 2m)\frac{x^{2m}}{(2m)!}\sum_{k = 0}^{m - 1}\binom{2m}{2k}\int_{0}^{\pi / x}t^{2k + (1 - \alpha) / 2}\log(2e + 1 / (xt))\ dt. \label{eq:norm-I-1-R}
  \end{align}
\end{lemma}
\begin{proof}
  With the given weight we have
  \begin{equation*}
    U_{\alpha}(x) = x\int_{0}^{\pi / x}\left|C_{-\alpha}(x(1 - t)) + C_{-\alpha}(x(1 + t)) - 2C_{-\alpha}(xt)\right|
    (xt)^{(1 - \alpha) / 2}\log(2e + 1 / (xt))\ dt.
  \end{equation*}
  This means that we want to bound
  \begin{equation*}
    \frac{
      x^{1 + \alpha}\int_{0}^{\pi / x}\left|C_{-\alpha}(x(1 - t)) + C_{-\alpha}(x(1 + t)) - 2C_{-\alpha}(xt)\right|
      t^{(1 - \alpha) / 2}\log(2e + 1 / (xt))\ dt
    }{
      \log(1 / x)(1 - x^{1 + \alpha + (1 + \alpha)^{2} / 2})\Gamma(1 + \alpha)
    }.
  \end{equation*}
  Using the asymptotic expansion of the Clausen terms in the integrand
  from \eqref{eq:integrand-expansion} we can split the integral as
  \begin{multline*}
    \int_{0}^{\pi / x}\left|C_{-\alpha}(x(1 - t)) + C_{-\alpha}(x(1 + t)) - 2C_{-\alpha}(xt)\right|
    t^{(1 - \alpha) / 2}\log(2e + 1 / (xt)) dt\\
    \leq \Gamma(1 + \alpha)\sin(-\pi\alpha / 2)x^{-\alpha - 1}\int_{0}^{\pi / x}\left||1 - t|^{-\alpha - 1} + (1 + t)^{-\alpha - 1} - 2t^{-\alpha - 1}\right|t^{(1 - \alpha) / 2}\log(2e + 1 / (xt))\ dt\\
    + 2\sum_{m = 1}^{\infty} (-1)^{m}\zeta(-\alpha - 2m)\frac{x^{2m}}{(2m)!}\sum_{k = 0}^{m - 1}\binom{2m}{2k}\int_{0}^{\pi / x}t^{2k + (1 - \alpha) / 2}\log(2e + 1 / (xt))\ dt
  \end{multline*}
  From this we get
  \begin{multline*}
    \frac{
      \sin(-\pi\alpha / 2)\int_{0}^{\pi / x}\left||1 - t|^{-\alpha - 1} + (1 + t)^{-\alpha - 1} - 2t^{-\alpha - 1}\right|t^{(1 - \alpha) / 2}\log(2e + 1 / (xt))\ dt
    }{
      \log(1 / x)(1 - x^{1 + \alpha + (1 + \alpha)^{2} / 2})
    }\\
    + \frac{
      x^{1 + \alpha}2\sum_{m = 1}^{\infty} (-1)^{m}\zeta(-\alpha - 2m)\frac{x^{2m}}{(2m)!}\sum_{k = 0}^{m - 1}\binom{2m}{2k}\int_{0}^{\pi / x}t^{2k + (1 - \alpha) / 2}\log(2e + 1 / (xt))\ dt
    }{
      \log(1 / x)(1 - x^{1 + \alpha + (1 + \alpha)^{2} / 2})\Gamma(1 + \alpha)
    }
  \end{multline*}
  as an upper bound. For the integral in the first term we can split
  the interval of integration at \(t = 1\) to get
  \begin{multline*}
    \int_{0}^{\pi / x}\left||1 - t|^{-\alpha - 1} + (1 + t)^{-\alpha - 1} - 2t^{-\alpha - 1}\right|t^{(1 - \alpha) / 2}\log(2e + 1 / (xt))\ dt\\
    = \int_{0}^{1}\left|(1 - t)^{-\alpha - 1} + (1 + t)^{-\alpha - 1} - 2t^{-\alpha - 1}\right|t^{(1 - \alpha) / 2}\log(2e + 1 / (xt))\ dt\\
    + \int_{1}^{\pi / x}\left((t - 1)^{-\alpha - 1} + (1 + t)^{-\alpha - 1} - 2t^{-\alpha - 1}\right)t^{(1 - \alpha) / 2}\log(2e + 1 / (xt))\ dt.
  \end{multline*}
  Here we have removed the absolute values around \(1 - t\) according
  to its sign and also used that
  \begin{equation*}
    (t - 1)^{-\alpha - 1} + (1 + t)^{-\alpha - 1} - 2t^{-\alpha - 1}
  \end{equation*}
  is positive for \(t > 1\), which can be shown in the same way as in
  Lemma~\ref{lemma:I-x-pi}.
\end{proof}

Bounds of \(G_{\alpha,1}\), \(G_{\alpha,2}\) and \(R_{\alpha}\) are
given in Appendix~\ref{sec:evaluation-T-I-1-asymptotic}. The factor
\begin{equation*}
  \frac{x^{1 + \alpha}}{\Gamma(1 + \alpha)\log(1 / x)(1 - x^{1 + \alpha + (1 + \alpha)^{2} / 2})}
\end{equation*}
has a removable singularity at \(\alpha = -1\) that needs to be
treated but is otherwise straightforward to enclose.

\subsubsection{Evaluation of \(\mathcal{T}_{\alpha}\) for \(I_{3}\)}
\label{sec:evaluation-T-I-3}
In this case we need to not only compute an enclosure of
\(\mathcal{T}_{\alpha}(x)\), but understand it behavior in \(\alpha\).
We therefore compute Taylor models of degree \(0\) centered at
\(\alpha = 0\). The weight is given by \(w_{\alpha}(x) = |x|\), and we
can therefore use the explicit expression for the integral given in
Lemma~\ref{lemma:U-primitive}.

We start with the following lemma that gives us information about the
first term in the Taylor model.
\begin{lemma}
  \label{lemma:evaluation-T-expansion-alpha-I-3}
  For \(x \in [0, \pi]\) the constant term in the expansion at
  \(\alpha = 0\) of \(\mathcal{T}_{\alpha}(x)\) is \(1\).
\end{lemma}
\begin{proof}
  Recall that
  \begin{equation*}
    \mathcal{T}_{\alpha}(x) = \frac{U_{\alpha}(x)}{\pi |w_{\alpha}(x)||u_{\alpha}(x)|}.
  \end{equation*}
  In this case \(w_{\alpha}(x) = |x|\), giving us
  \begin{equation*}
    \mathcal{T}_{\alpha}(x) = \frac{U_{\alpha}(x)}{\pi x|u_{\alpha}(x)|}.
  \end{equation*}
  By Lemma~\ref{lemma:u0-expansion-alpha-I-3} the constant term in the
  expansion at \(\alpha = 0\) of \(u_{\alpha}(x)\) is \(1\), it is
  therefore enough to show that the constant term for the numerator is
  \(\pi x\).

  From Lemma~\ref{lemma:U-primitive} we have that
  \begin{multline*}
    U_{\alpha}(x) = 2\tilde{C}_{2 - \alpha}(x) + 2(C_{2 - \alpha}(x + \pi) - C_{2 - \alpha}(\pi))\\
    - 2\left(C_{2 - \alpha}(x(1 - r_{\alpha,x})) + C_{2 - \alpha}(x(1 + r_{\alpha,x})) - 2C_{2 - \alpha}(xr_{\alpha,x})\right)\\
    - 2xr_{\alpha,x}\left(-S_{1 - \alpha}(x(1 - r_{\alpha,x})) + S_{1 - \alpha}(x(1 + r_{\alpha,x})) - 2S_{1 - \alpha}(xr_{\alpha,x})\right).
  \end{multline*}
  To get the constant term in the expansion we want to compute the
  limit of this as \(\alpha \to 0\), which we denote by \(U_{0}(x)\).

  To do that we first need to compute
  \(\lim_{\alpha \to 0^{-}} r_{\alpha,x}\). The function
  \(\hat{I}_{\alpha}(x, t)\) for which \(r_{\alpha,x}\) is a zero
  converges to \(0\) everywhere as \(\alpha \to 0\). If we divide by
  \(\alpha\) and compute the zero of
  \(\frac{\hat{I}_{\alpha}(x, t)}{\alpha}\) instead, we get a
  well-defined limit, and we get that \(r_{0,x}\) is the zero of
  \begin{equation*}
    \frac{d}{d\alpha} \hat{I}_{\alpha}(x, t) =
    C_{-\alpha}^{(1)}(x(1 - t)) + C_{-\alpha}^{(1)}(x(1 + t)) - 2C_{-\alpha}^{(1)}(xt).
  \end{equation*}
  This function satisfies the same properties as in
  Lemma~\ref{lemma:I-0-x} with respect to the root \(r_{0,x}\).

  Taking the limit in \(\alpha\) now gives us
  \begin{multline*}
    U_{0}(x) = 2\tilde{C}_{2}(x) + 2(C_{2}(x + \pi) - C_{2}(\pi))\\
    - 2\left(C_{2}(x(1 - r_{0,x})) + C_{2}(x(1 + r_{0,x})) - 2C_{2}(xr_{0,x})\right)\\
    - 2xr_{0,x}\left(-S_{1}(x(1 - r_{0,x})) + S_{1}(x(1 + r_{0,x})) - 2S_{1}(xr_{0,x})\right).
  \end{multline*}
  For these parameters to the Clausen functions we have the explicit
  expressions
  \begin{equation*}
    C_{2}(x) = \frac{\pi^{2}}{6} - \frac{\pi}{2}x + \frac{1}{4}x^{2},\quad
    S_{1}(x) = \frac{\pi}{2} - \frac{1}{2}x,
  \end{equation*}
  valid when \(x \in [0, 2\pi]\). For the different parts of
  \(U_{\alpha}(x)\) we get
  \begin{multline*}
    2\tilde{C}_{2}(x) + 2(C_{2}(x + \pi) - C_{2}(\pi))\\
    = 2\left(
      \left(-\frac{\pi}{2}x + \frac{1}{4}x^{2}\right)
      + \left(\frac{\pi^{2}}{6} - \frac{\pi}{2}(x + \pi) + \frac{1}{4}(x + \pi)^{2}\right)
      - \left(\frac{\pi^{2}}{6} - \frac{\pi}{2}\pi + \frac{1}{4}\pi^{2}\right)
    \right)\\
    = 2\left(-\frac{\pi}{2} x + \frac{1}{2}x^{2}\right) = -\pi x + x^{2},
  \end{multline*}
  \begin{multline*}
    2\left(C_{2}(x(1 - r_{0,x})) + C_{2}(x(1 + r_{0,x})) - 2C_{2}(xr_{0,x})\right)\\
    = 2\bigg(
      \left(\frac{\pi^{2}}{6} - \frac{\pi}{2}x(1 - r_{0,x}) + \frac{1}{4}x^{2}(1 - r_{0,x})^{2}\right)
      + \left(\frac{\pi^{2}}{6} - \frac{\pi}{2}x(1 + r_{0,x}) + \frac{1}{4}x^{2}(1 + r_{0,x})^{2}\right)\\
      - 2\left(\frac{\pi^{2}}{6} - \frac{\pi}{2}xr_{0,x} + \frac{1}{4}x^{2}r_{0,x}^{2}\right)
    \bigg)
    = 2\left(-\pi x (1 - r_{0,x}) + \frac{1}{2}x^{2}\right) =-2\pi x(1 - r_{0,x}) + x^{2},
  \end{multline*}
  \begin{multline*}
    2xr_{0,x}\left(-S_{1}(x(1 - r_{0,x})) + S_{1}(x(1 + r_{0,x})) - 2S_{1}(xr_{0,x})\right)\\
    = 2xr_{0,x}\left(
      - \left(\frac{\pi}{2} - \frac{1}{2}x(1 - r_{0,x})\right)
      + \left(\frac{\pi}{2} - \frac{1}{2}x(1 + r_{0,x})\right)
      - 2\left(\frac{\pi}{2} - \frac{1}{2}xr_{0,x}\right)
    \right) =
    -2\pi xr_{0,x}.
  \end{multline*}
  Putting this together we get
  \begin{equation*}
    U_{0}(x) = (-\pi x + x^{2}) - (-2\pi x(1 - r_{0,x}) + x^{2}) + 2\pi xr_{0,x} = \pi x,
  \end{equation*}
  which is exactly what we needed to show.
\end{proof}

For \(x \in [\epsilon, \pi]\) we compute the Taylor model of
\(u_{\alpha}(x)\) using the approach described in
Section~\ref{sec:evaluation-I-3}. For \(U_{\alpha}\) we use the
expression
\begin{multline*}
    U_{\alpha}(x) = 2\tilde{C}_{2 - \alpha}(x) + 2(C_{2 - \alpha}(x + \pi) - C_{2 - \alpha}(\pi))\\
    - 2\left(C_{2 - \alpha}(x(1 - r_{\alpha,x})) + C_{2 - \alpha}(x(1 + r_{\alpha,x})) - 2C_{2 - \alpha}(xr_{\alpha,x})\right)\\
    - 2xr_{\alpha,x}\left(-S_{1 - \alpha}(x(1 - r_{\alpha,x})) + S_{1 - \alpha}(x(1 + r_{\alpha,x})) - 2S_{1 - \alpha}(xr_{\alpha,x})\right)
\end{multline*}
from Lemma~\ref{lemma:U-primitive}. For the Clausen functions not
depending on \(r_{\alpha,x}\) we can compute Taylor models directly.
If we let
\begin{multline*}
  K_{\alpha}(x) = \left(C_{2 - \alpha}(x(1 - r_{\alpha,x})) + C_{2 - \alpha}(x(1 + r_{\alpha,x})) - 2C_{2 - \alpha}(xr_{\alpha,x})\right)\\
  + xr_{\alpha,x}\left(-S_{1 - \alpha}(x(1 - r_{\alpha,x})) + S_{1 - \alpha}(x(1 + r_{\alpha,x})) - 2S_{1 - \alpha}(xr_{\alpha,x})\right)
\end{multline*}
then the part of \(U_{\alpha}(x)\) depending on \(r_{\alpha,x}\) is
given by \(-2K_{\alpha}(x)\). For the Taylor model of
\(K_{\alpha}(x)\) we need to take into account that \(r_{\alpha,x}\)
also depends on \(\alpha\). The constant term in the Taylor model is
given by \(K_{0}(x)\), which we can compute. For the remainder term we
want to enclose \(\frac{d}{d\alpha} K_{\alpha}(x)\) for
\(\alpha \in I_{3}\). Differentiation gives us
\begin{multline*}
  \frac{d}{d\alpha}K_{\alpha}(x) =
  -\left(C_{2 - \alpha}^{(1)}(x(1 - r_{\alpha,x})) + C_{2 - \alpha}^{(1)}(x(1 + r_{\alpha,x})) - 2C_{2 - \alpha}^{(1)}(xr_{\alpha,x})\right)\\
  - x\left(\frac{d}{d\alpha}r_{\alpha,x}\right)\left(-S_{1 - \alpha}(x(1 - r_{\alpha,x})) + S_{1 - \alpha}^{(1)}(x(1 + r_{\alpha,x})) - 2S_{1 - \alpha}^{(1)}(xr_{\alpha,x})\right)\\
  + x\left(\frac{d}{d\alpha}r_{\alpha,x}\right)\left(-S_{1 - \alpha}(x(1 - r_{\alpha,x})) + S_{1 - \alpha}(x(1 + r_{\alpha,x})) - 2S_{1 - \alpha}(xr_{\alpha,x})\right)\\
  - xr_{\alpha,x}\left(-S_{1 - \alpha}^{(1)}(x(1 - r_{\alpha,x})) + S_{1 - \alpha}^{(1)}(x(1 + r_{\alpha,x})) - 2S_{1 - \alpha}^{(1)}(xr_{\alpha,x})\right)\\
  + x^{2}r_{\alpha,x}\left(\frac{d}{d\alpha}r_{\alpha,x}\right)\left(-C_{-\alpha}(x(1 - r_{\alpha,x})) + C_{-\alpha}(x(1 + r_{\alpha,x})) - 2C_{-\alpha}(xr_{\alpha,x})\right).
\end{multline*}
We see that two of the terms cancel out, and the last term is zero
since \(r_{\alpha,x}\) is a root of exactly this function, this gives
us
\begin{multline*}
  \frac{d}{d\alpha}K_{\alpha}(x) =
  -\left(C_{2 - \alpha}^{(1)}(x(1 - r_{\alpha,x})) + C_{2 - \alpha}^{(1)}(x(1 + r_{\alpha,x})) - 2C_{2 - \alpha}^{(1)}(xr_{\alpha,x})\right)\\
  - xr_{\alpha,x}\left(-S_{1 - \alpha}^{(1)}(x(1 - r_{\alpha,x})) + S_{1 - \alpha}^{(1)}(x(1 + r_{\alpha,x})) - 2S_{1 - \alpha}^{(1)}(xr_{\alpha,x})\right).
\end{multline*}
In particular we see that the derivative only depends on
\(r_{\alpha,x}\), for which we can easily compute an enclosure, and
not on \(\frac{d}{d\alpha}r_{\alpha,x}\).

For \(x \in [0, \epsilon]\) we write the function as
in~\ref{eq:evaluation-T-I-2-split} and the only problematic part is
the evaluation of \(U_{\alpha}(x)x^{\alpha - 1}\). The terms not
depending on \(r_{\alpha,x}\) are handled by expanding the Clausen
functions, following Appendix~\ref{sec:taylor-models-expansions-x} to
get a Taylor model of the expansion. What remains is to compute a
Taylor model of \(K_{\alpha}(x)x^{\alpha - 1}\). For that we compute
the constant term and an enclosure of the derivative separately. The
constant term is done by just expanding the Clausen functions. For the
derivative we get
\begin{equation*}
  \frac{d}{d\alpha} K_{\alpha}(x)x^{\alpha - 1} = \left(\frac{d}{d\alpha}K_{\alpha}(x)\right)x^{\alpha - 1}
  + K_{\alpha}(x)\log(x)x^{\alpha - 1}.
\end{equation*}
When expanding the Clausen functions for this we have to handle the
cancellations between the two terms since they individually blow up at
\(x = 0\).

\begin{lemma}
  We have the following expansion for \(0 < x < \pi\)
  \begin{align*}
    \frac{d}{d\alpha} K_{\alpha}(x)x^{\alpha - 1}
    =& -\frac{d}{ds}\left(\Gamma(1 - s)\sin\left(\frac{\pi}{2}s\right)\left(
        (1 - r_{\alpha,x})^{s - 1} + (1 + r_{\alpha,x})^{s - 1} - 2r_{\alpha,x}^{s - 1}
      \right)\right)\Bigg|_{s = 2 - \alpha}\\
    &- r_{\alpha,x}\frac{d}{ds}\left(\Gamma(1 - s)\cos\left(\frac{\pi}{2}s\right)\left(
        -(1 - r_{\alpha,x})^{s - 1} + (1 + r_{\alpha,x})^{s - 1} - 2r_{\alpha,x}^{s - 1}
      \right)\right)\Bigg|_{s = 1 - \alpha}\\
    &+ \log(x)\sum_{m = 1}^{\infty} \frac{(-1)^{m}}{(2m)!}\zeta(2 - \alpha - 2m)\left(
      (1 - r_{\alpha,x})^{2m} + (1 + r_{\alpha,x})^{2m} - 2r_{\alpha,x}^{2m}
    \right)x^{2m - 1 + \alpha}\\
    &+ r_{\alpha,x}\log(x)\sum_{m = 0}^{\infty} \frac{(-1)^{m}}{(2m + 1)!}\zeta(-\alpha - 2m)\left(
      -(1 - r_{\alpha,x})^{2m + 1} + (1 + r_{\alpha,x})^{2m + 1} - 2r_{\alpha,x}^{2m + 1}
    \right)x^{2m + 1 + \alpha}\\
    &- \sum_{m = 1}^{\infty} \frac{(-1)^{m}}{(2m)!}\zeta'(2 - \alpha - 2m)\left(
      (1 - r_{\alpha,x})^{2m} + (1 + r_{\alpha,x})^{2m} - 2r_{\alpha,x}^{2m}
    \right)x^{2m - 1 + \alpha}\\
    &- r_{\alpha,x}\sum_{m = 0}^{\infty} \frac{(-1)^{m}}{(2m + 1)!}\zeta'(-\alpha - 2m)\left(
      -(1 - r_{\alpha,x})^{2m + 1} + (1 + r_{\alpha,x})^{2m + 1} - 2r_{\alpha,x}^{2m + 1}
    \right)x^{2m + 1 + \alpha}.
  \end{align*}
\end{lemma}
\begin{proof}
  We get directly from expanding the Clausen functions that
  \begin{multline}
    \label{eq:evaluation-T-I-3-K}
    K_{\alpha}(x)\log(x)x^{\alpha - 1} = \Gamma(\alpha - 1)\sin\left(\frac{\pi}{2}(2 - \alpha)\right)\left(
      (1 - r_{\alpha,x})^{1 - \alpha} + (1 + r_{\alpha,x})^{1 - \alpha} - 2r_{\alpha,x}^{1 - \alpha}
    \right)\log(x)\\
    + r_{\alpha,x}\Gamma(\alpha)\cos\left(\frac{\pi}{2}(1 - \alpha)\right)\left(
      (1 - r_{\alpha,x})^{-\alpha} + (1 + r_{\alpha,x})^{-\alpha} - 2r_{\alpha,x}^{-\alpha}
    \right)\log(x)\\
    + \log(x)\sum_{m = 1}^{\infty} \frac{(-1)^{m}}{(2m)!}\zeta(2 - \alpha - 2m)\left(
      (1 - r_{\alpha,x})^{2m} + (1 + r_{\alpha,x})^{2m} - 2r_{\alpha,x}^{2m}
    \right)x^{2m - 1 + \alpha}\\
    + r_{\alpha,x}\log(x)\sum_{m = 0}^{\infty} \frac{(-1)^{m}}{(2m + 1)!}\zeta(-\alpha - 2m)\left(
      -(1 - r_{\alpha,x})^{2m + 1} + (1 + r_{\alpha,x})^{2m + 1} - 2r_{\alpha,x}^{2m + 1}
    \right)x^{2m + 1 + \alpha}
  \end{multline}
  and
  \begin{multline}
    \label{eq:evaluation-T-I-3-dK}
    \left(\frac{d}{d\alpha}K_{\alpha}(x)\right)x^{\alpha - 1} =
    -\frac{d}{ds}\left(\Gamma(1 - s)\sin\left(\frac{\pi}{2}s\right)\left(
        (1 - r_{\alpha,x})^{s - 1} + (1 + r_{\alpha,x})^{s - 1} - 2r_{\alpha,x}^{s - 1}
      \right)x^{s - 1}\right)\Bigg|_{s = 2 - \alpha}x^{\alpha - 1}\\
    - r_{\alpha,x}\frac{d}{ds}\left(\Gamma(1 - s)\cos\left(\frac{\pi}{2}s\right)\left(
        (1 - r_{\alpha,x})^{s - 1} + (1 + r_{\alpha,x})^{s - 1} - 2r_{\alpha,x}^{s - 1}
      \right)x^{s - 1}\right)\Bigg|_{s = 1 - \alpha}x^{\alpha}\\
    - \sum_{m = 1}^{\infty} \frac{(-1)^{m}}{(2m)!}\zeta'(2 - \alpha - 2m)\left(
      (1 - r_{\alpha,x})^{2m} + (1 + r_{\alpha,x})^{2m} - 2r_{\alpha,x}^{2m}
    \right)x^{2m - 1 + \alpha}\\
    - r_{\alpha,x}\sum_{m = 0}^{\infty} \frac{(-1)^{m}}{(2m + 1)!}\zeta'(-\alpha - 2m)\left(
      -(1 - r_{\alpha,x})^{2m + 1} + (1 + r_{\alpha,x})^{2m + 1} - 2r_{\alpha,x}^{2m + 1}
    \right)x^{2m + 1 + \alpha}.
  \end{multline}

  For the two first terms in the expansion for
  \(\left(\frac{d}{d\alpha}K_{\alpha}(x)\right)x^{\alpha - 1}\) we
  have
  \begin{multline*}
    \frac{d}{ds}\left(\Gamma(1 - s)\sin\left(\frac{\pi}{2}s\right)\left(
        (1 - r_{\alpha,x})^{s - 1} + (1 + r_{\alpha,x})^{s - 1} - 2r_{\alpha,x}^{s - 1}
      \right)x^{s - 1}\right)\Bigg|_{s = 2 - \alpha}x^{\alpha - 1}\\
    = \frac{d}{ds}\left(\Gamma(1 - s)\sin\left(\frac{\pi}{2}s\right)\left(
        (1 - r_{\alpha,x})^{s - 1} + (1 + r_{\alpha,x})^{s - 1} - 2r_{\alpha,x}^{s - 1}
      \right)\right)\Bigg|_{s = 2 - \alpha}\\
    + \Gamma(1 - s)\sin\left(\frac{\pi}{2}s\right)\left(
      (1 - r_{\alpha,x})^{s - 1} + (1 + r_{\alpha,x})^{s - 1} - 2r_{\alpha,x}^{s - 1}
    \right)\log(x)
  \end{multline*}
  and
  \begin{multline*}
    r_{\alpha,x}\frac{d}{ds}\left(\Gamma(1 - s)\cos\left(\frac{\pi}{2}s\right)\left(
        (1 - r_{\alpha,x})^{s - 1} + (1 + r_{\alpha,x})^{s - 1} - 2r_{\alpha,x}^{s - 1}
      \right)x^{s - 1}\right)\Bigg|_{s = 1 - \alpha}x^{\alpha}\\
    = r_{\alpha,x}\frac{d}{ds}\left(\Gamma(1 - s)\cos\left(\frac{\pi}{2}s\right)\left(
        (1 - r_{\alpha,x})^{s - 1} + (1 + r_{\alpha,x})^{s - 1} - 2r_{\alpha,x}^{s - 1}
      \right)\right)\Bigg|_{s = 1 - \alpha}\\
    + r_{\alpha,x}\Gamma(1 - s)\cos\left(\frac{\pi}{2}s\right)\left(
      (1 - r_{\alpha,x})^{s - 1} + (1 + r_{\alpha,x})^{s - 1} - 2r_{\alpha,x}^{s - 1}
    \right)\log(x).
  \end{multline*}
  For both of these the second term exactly cancels the corresponding
  one in~\eqref{eq:evaluation-T-I-3-K}, this gives us the result.
\end{proof}

\subsubsection{Evaluation of \(\mathcal{T}_{\alpha}\) in hybrid cases}
\label{sec:evaluation-T-hybrid}
For \(\alpha\) near \(0\) this is straight forward, the weight is in
this case given by \(w_{\alpha}(x) = |x|\), and we can use the same
approach as described in Section~\ref{sec:evaluation-T-I-2}.

For \(\alpha\) near \(-1\) we can use the same approach for
\(x \in [\epsilon, \pi]\). For \(x \in [0, \epsilon]\) this doesn't
work since the weight is not on the form \(w_{\alpha}(x) = |x|^{p}\)
so Lemma~\ref{lemma:evaluation-U-I-2} doesn't apply. Instead, we have
to use an approach more similar to that used for \(\alpha \in I_{1}\).
We write \(\mathcal{T}_{\alpha}\) as
\begin{equation*}
  \mathcal{T}_{\alpha}(x) = \frac{\log(1 / x)}{\pi\log(2e + 1 / x)} \cdot
  \frac{x^{-\alpha}}{u_{\alpha}(x)}
  \cdot \frac{U_{\alpha}(x)}{\log(1 / x)x^{-\alpha + p}}.
\end{equation*}
The first two factors are bounded in the same way as in the above
sections. For the third one we give a bound in
Appendix~\ref{sec:evaluation-T-hybrid-asymptotic}.

\section{Bounds for \(D_{\alpha}\), \(\delta_{\alpha}\) and \(n_{\alpha}\)}
\label{sec:bounds-for-values}
We are now ready to give bounds for \(D_{\alpha}\),
\(\delta_{\alpha}\) and \(n_{\alpha}\). Recall that they are given by
\begin{equation*}
  n_{\alpha} = \sup_{x \in [0, \pi]} |N_{\alpha}(x)|,\quad
  \delta_{\alpha} = \sup_{x \in [0, \pi]} |F_{\alpha}(x)|,\quad
  D_{\alpha} = \sup_{x \in [0, \pi]} |\mathcal{T}_{\alpha}(x)|.
\end{equation*}
In each case we split the interval \([0, \pi]\) into two parts,
\([0, \epsilon]\) and \([\epsilon, \pi]\), with \(\epsilon\) varying
for the different cases, and threat them separately. For the interval
\([0, \epsilon]\) we use the asymptotic bounds for the different
functions that were introduced in the previous sections. For the
interval \([\epsilon, \pi]\) we evaluate the functions directly using
interval arithmetic. The method we use for bounding the supremum is
the one described in Section~\ref{sec:enclosing-supremum}.

In most cases the subintervals in the computations are bisected at the
midpoint, meaning that the interval \([\lo{x}, \hi{x}]\) would be
bisected into the two intervals \([\lo{x}, (\lo{x} + \hi{x}) / 2]\)
and \([(\lo{x} + \hi{x}) / 2, \hi{x}]\). However, when the magnitude
of \(\lo{x}\) and \(\hi{x}\) are very different it can be beneficial
to bisect at the geometric midpoint (see e.g.~\cite{GmezSerrano2014}),
in that case we split the interval into
\(\left[\lo{x}, \sqrt{\lo{x}\hi{x}}\right]\) and
\(\left[\sqrt{\lo{x}\hi{x}}, \hi{x}\right]\), where we assume that
\(\lo{x} > 0\).

The code~\cite{HighestCuspedWave} for the computer-assisted parts is
implemented in Julia~\cite{Julia-2017}. The main tool for the rigorous
numerics is Arb~\cite{Johansson2013arb} which we use through the Julia
wrapper \textit{Arblib.jl}
\footnote{\url{https://github.com/kalmarek/Arblib.jl}}. Many of the
basic interval arithmetic algorithms, such as isolating roots or
enclosing maximum values, are implemented in a separate package,
\textit{ArbExtras.jl}
\footnote{\url{https://github.com/Joel-Dahne/ArbExtras.jl}}. For
finding the coefficients \(\{a_{j}\}\) with \(j \geq 1\) and
\(\{b_{n}\}\) of \(u_{\alpha}\) we make use of non-linear solvers from
\textit{NLsolve.jl}~\cite{patrick_kofod_mogensen_2020_4404703}.

For \(\alpha \in I_{1}\) and \(\alpha \in I_{3}\) the computations
were done on an AMD Ryzen 9 5900X processor with 32 GB of RAM using 12
threads. For \(\alpha \in I_{2}\) most of the computations where done
on the Dardel HPC system at the PDC Center for High Performance
Computing, KTH Royal Institute of Technology. The nodes are equipped
with two AMD EPYC\texttrademark{} Zen2 2.25 GHz 64 core processors and
256 GB of RAM.

We handle the intervals \(I_{1}\), \(I_{2}\) and \(I_{3}\) separately,
the all require slightly different approaches for computing the
bounds.

\subsection{Bounds for \(I_{1}\)}
\label{sec:bounds-I-1}
We here give bounds of \(n_{\alpha}\), \(\delta_{\alpha}\) and
\(D_{\alpha}\) for \(\alpha \in I_{1} = (-1, -1 + \delta_{1})\), with
\(\delta_{1} = 10^{-4}\). The bounds are split into three lemmas.
Recall that in this case the weight is given by
\(w_{\alpha}(x) = |x|^{(1 - \alpha) / 2}\log(2e + 1 / |x|)\). We take
\(u_{\alpha}\) is in~\eqref{eq:u0-I-1} with
\(\hat{\alpha} = -0.9997\), \(N_{\hat{\alpha},0} = 1929\) and
\(N_{-1,1} = 16\).

\begin{lemma}
  \label{lemma:bounds-I-1-n}
  The constant \(n_{\alpha}\) satisfies the inequality
  \(n_{\alpha} \leq \bar{n}_{1} = 1.79\) for all \(\alpha \in I_{1}\).
\end{lemma}
\begin{proof}
  A plot of \(N_{\alpha}(x)\) on the interval \([0, \pi]\) is given in
  Figure~\ref{fig:bounds-I-1-N}. It hints at a local maximum at
  \(x = \pi\) but does not fully show what happens near \(x = 0\), a
  plot closer to the origin is given in
  Figure~\ref{fig:bounds-I-1-N-asymptotic}, and we see that the
  function flattens out around \(\pi / 2\) when \(x\) approaches zero,
  which indicates that the maximum indeed is attained at \(x = \pi\).

  For the interval \([\epsilon, \pi]\) we don't compute with
  \(N_{\alpha}(x)\) directly but instead with
  \(\bar{N}_{\alpha}(x) = \frac{w_{\alpha}(x)}{2u_{-1}(x)}\) which
  satisfies \(N_{\alpha}(x) \leq \bar{N}_{\alpha}(x)\), as mentioned
  in Section~\ref{sec:evaluation-N-I-1}.

  The value of \(\epsilon\) is chosen dynamically. We don't compute an
  enclosure of the maximum for the interval \([0, \epsilon]\) but only
  prove that it is bounded by \(\bar{N}_{\alpha}(\pi)\). The value for
  \(\epsilon\) is determined in the following way. Starting with
  \(\epsilon = 0.45\) we compute an enclosure of
  \(N_{\alpha}([0, \epsilon])\). If this enclosure is bounded by
  \(\bar{N}_{\alpha}(\pi)\) we stop, otherwise we multiply
  \(\epsilon\) by \(0.8\) and try again. Once the bound holds we stop,
  in practice this happens for \(\epsilon \approx 0.075\).

  For the interval \([\epsilon, \pi]\) we compute an enclosure of the
  maximum of \(\bar{N}_{\alpha}(x)\). This gives us
  \begin{equation*}
    n_{\alpha} \leq [1.7870 \pm 8.20 \cdot 10^{-5}],
  \end{equation*}
  which is upper bounded by \(\bar{n}_{1}\). On this interval we are
  able to compute Taylor expansions of \(\bar{N}_{\alpha}(x)\) on the
  interval \([\epsilon, \pi]\), allowing us to use the better version
  of the algorithm, based on the Taylor polynomial, for enclosing the
  maximum. We use a Taylor expansion of degree \(0\), which is enough
  to pick up the monotonicity after only a few bisections in most
  cases.

  The runtime is about 6 seconds, most of it for handling the interval
  \([\epsilon, \pi]\).
\end{proof}

\begin{lemma}
  \label{lemma:bounds-I-1-delta}
  The constant \(\delta_{\alpha}\) satisfies the inequality
  \(\delta_{\alpha} \leq \bar{\delta}_{1} = 0.0005\) for all
  \(\alpha \in I_{1}\).
\end{lemma}
\begin{proof}
  A plot of \(F_{\alpha}(x)\) on the interval \([0, \pi]\) is given in
  Figure~\ref{fig:bounds-I-1-F}, a logarithmic plot on the interval
  \([10^{-5}, 10^{-1}]\) is given in
  Figure~\ref{fig:bounds-I-1-F-asymptotic}.

  In this case we don't try to compute a very accurate bound but
  instead only prove that \(\delta_{\alpha}\) is bounded by
  \(\bar{\delta}_{1}\).

  The value for \(\epsilon\) is chosen such that
  \(F_{\alpha}(\epsilon)\) evaluated using the asymptotic version
  gives an enclosure that satisfies the bound. We find that
  \(\epsilon \approx 0.48\) works.

  For the interval \([0, \epsilon]\) we first find \(\epsilon_{1}\)
  such that \(F_{\alpha}([0, \epsilon_{1}])\) directly gives an
  enclosure that is bounded by \(\bar{\delta}_{1}\). We get that
  \(\epsilon_{1} \approx 10^{-7000}\) works. For the interval
  \([\epsilon_{1}, \epsilon]\) we compute an enclosure of the maximum
  by iteratively bisecting the interval. Since the endpoints are very
  different in size we bisect at the geometric midpoint. We stop once
  the enclosure is bounded by \(\bar{\delta}_{1}\). The asymptotic
  version of \(F_{\alpha}\) doesn't allow for evaluation with Taylor
  expansions, it therefore requires many bisections to get a good
  enough enclosure, the maximum depth for the bisection is \(30\).

  For the interval \([\epsilon, \pi]\) we also compute a bound by
  bisection, stopping once the bound is less than
  \(\bar{\delta}_{1}\). Similar to in the previous lemma, and as
  mentioned in Section~\ref{sec:evaluation-F-I-1}, we use
  \(u_{-1}(x)\) in the numerator of \(F_{\alpha}(x)\), which gives an
  upper bound. In this case we use Taylor polynomials of degree \(4\)
  and maximum depth for the bisection is only \(4\).

  The runtime is about \(380\) seconds.
\end{proof}

\begin{lemma}
  \label{lemma:bounds-I-1-D}
  The constant \(D_{\alpha}\) satisfies the inequality
  \(D_{\alpha} \leq \bar{D}_{1} = 0.938\) for all
  \(\alpha \in I_{1}\).
\end{lemma}
\begin{proof}
  A plot of \(\mathcal{T}_{\alpha}(x)\) on the interval \([0, \pi]\)
  is given in Figure~\ref{fig:bounds-I-1-T}. It hints at the maximum being
  attained around \(x \approx 1.4\).

  We take \(\epsilon = 0.1\). We first bound the maximum on
  \([\epsilon, \pi]\), for which we get that the maximum is bounded by
  \([0.93 \pm 7.86 \cdot 10^{-3}]\), which in turn is bounded by
  \(\bar{D}_{1}\).

  For the interval \([0, \epsilon]\) we only prove that the maximum is
  bounded by \([0.93 \pm 7.86 \cdot 10^{-3}]\). This is done by first
  computing \(\mathcal{T}_{\alpha}([0, 10^{-10}])\) and verifying that
  this satisfies the bound. The interval \([10^{-10}, \epsilon]\) is
  then bisected until the bound can be verified, in this case the
  bisection is done at the geometric midpoint.

  For the interval \([0, \epsilon]\) we do not make use of Taylor
  expansions but only use direct evaluation. For the interval
  \([\epsilon, \pi]\) we are not able to compute Taylor expansion of
  \(U_{\alpha}\), and hence not of \(\mathcal{T}_{\alpha}\). We do
  however make one optimization to get better enclosures. Recall that
  \begin{equation*}
    \mathcal{T}_{\alpha}(x) \leq \frac{U_{\alpha}(x)}{\pi|w_{\alpha}(x)||u_{-1}(x)|}.
  \end{equation*}
  For a given interval \(\inter{x}\) both \(U_{\alpha}(\inter{x})\)
  and \(w_{\alpha}(\inter{x})\) are evaluated directly. To enclose
  \(u_{-1}(\inter{x})\) we first check if it is monotone by computing
  the derivative and checking that it is non-zero. If it is monotone
  we compute an enclosure by evaluating it at the endpoints of
  \(\inter{x}\).

  The runtime for the computation is around 410 seconds, the majority
  for handling the interval \([\epsilon, \pi]\).
\end{proof}

Combining the above three lemmas we get the following result.
\begin{lemma}
  \label{lemma:bounds-I-1-inequality}
  For all \(\alpha \in I_{1}\) the following inequality is satisfied
  \begin{equation*}
    \delta_{\alpha} < \frac{(1 - D_{\alpha})^{2}}{4n_{\alpha}}.
  \end{equation*}
\end{lemma}
\begin{proof}
  Follows immediately from the above lemmas, noticing that
  \begin{equation*}
    \delta_{\alpha} \leq \bar{\delta}_{-1}
    < \frac{(1 - \bar{D}_{\alpha})^{2}}{4\bar{n}_{\alpha}} \leq \frac{(1 - D_{\alpha})^{2}}{4n_{\alpha}}.
  \end{equation*}
\end{proof}

\begin{figure}
  \centering
  \begin{subfigure}[t]{0.3\textwidth}
    \includegraphics[width=\textwidth]{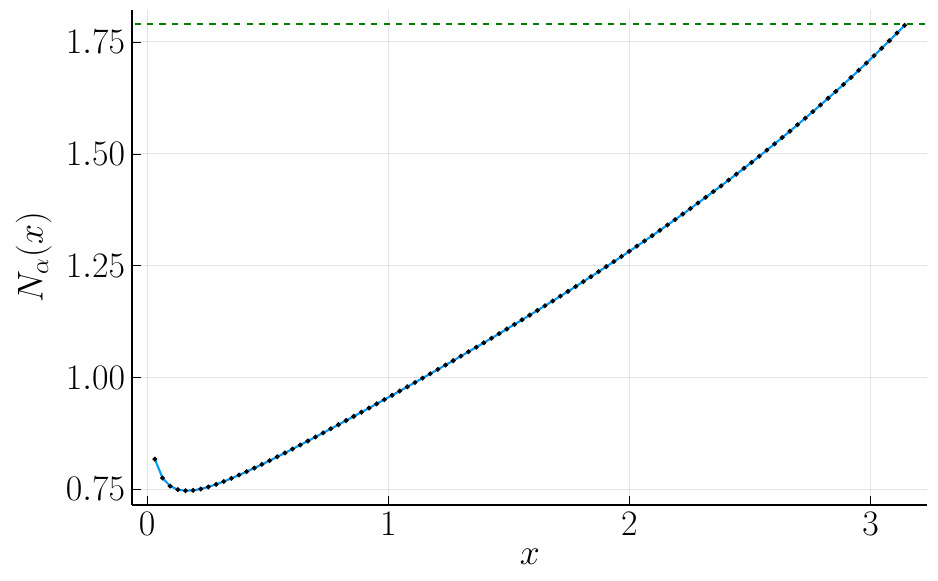}
    \caption{}
    \label{fig:bounds-I-1-N}
  \end{subfigure}
  \begin{subfigure}[t]{0.3\textwidth}
    \includegraphics[width=\textwidth]{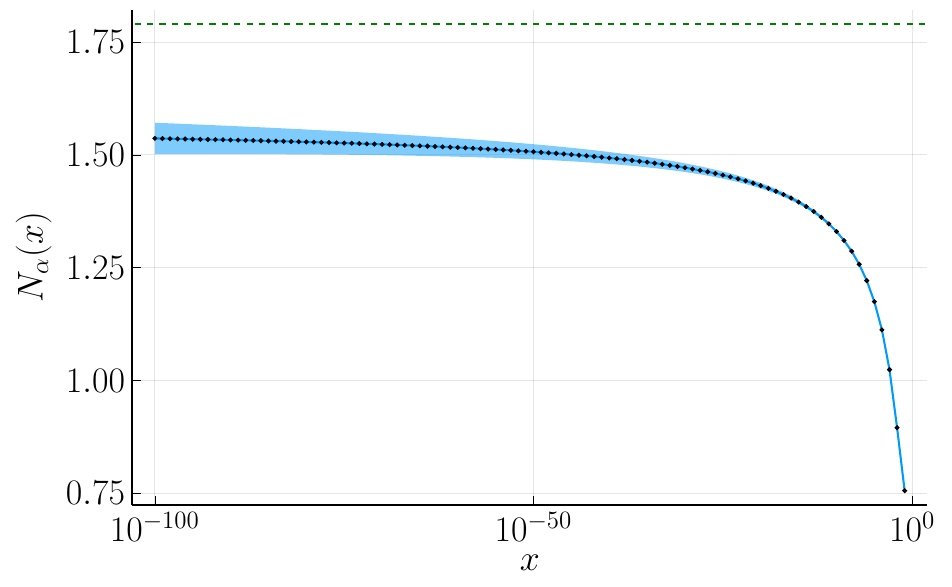}
    \caption{}
    \label{fig:bounds-I-1-N-asymptotic}
  \end{subfigure}
  \begin{subfigure}[t]{0.3\textwidth}
    \includegraphics[width=\textwidth]{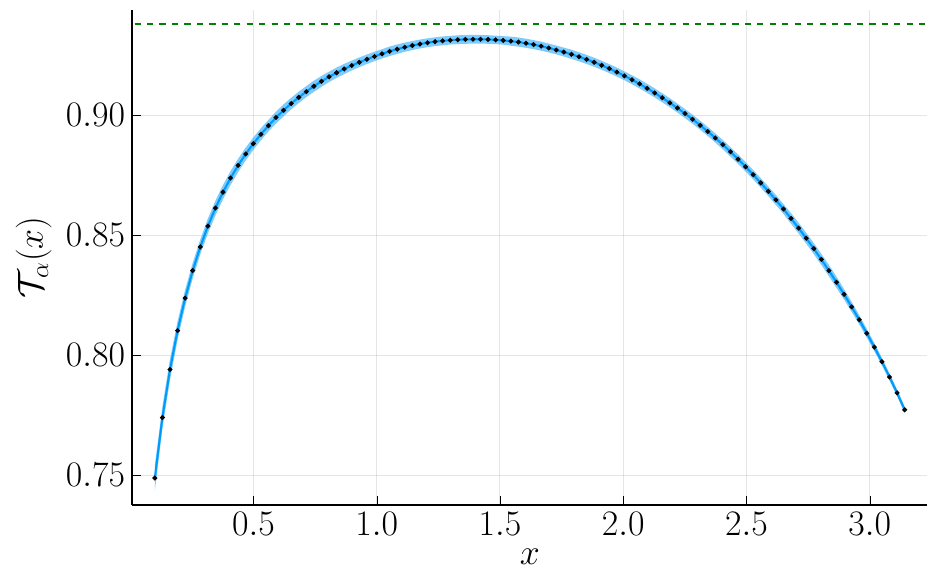}
    \caption{}
    \label{fig:bounds-I-1-T}
  \end{subfigure}
  \caption{Plot of the functions \(N_{\alpha}\) and
    \(\mathcal{T}_{\alpha}\) for \(\alpha \in I_{1}\). The dashed
    green lines show the upper bounds as given in
    Lemmas~\ref{lemma:bounds-I-1-n} and~\ref{lemma:bounds-I-1-D}}
\end{figure}

\begin{figure}
  \centering
  \begin{subfigure}[t]{0.45\textwidth}
    \includegraphics[width=\textwidth]{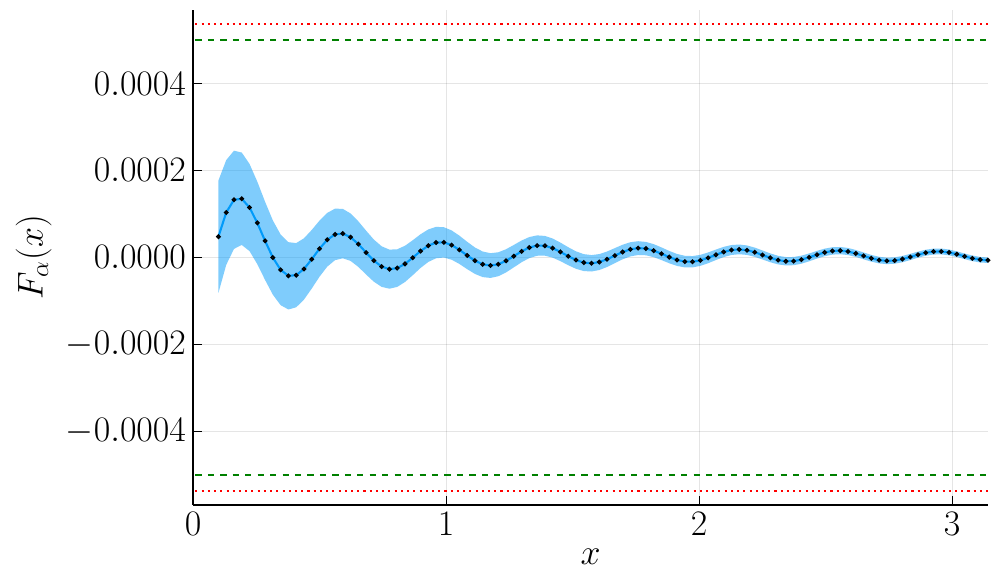}
    \caption{}
    \label{fig:bounds-I-1-F}
  \end{subfigure}
  \begin{subfigure}[t]{0.45\textwidth}
    \includegraphics[width=\textwidth]{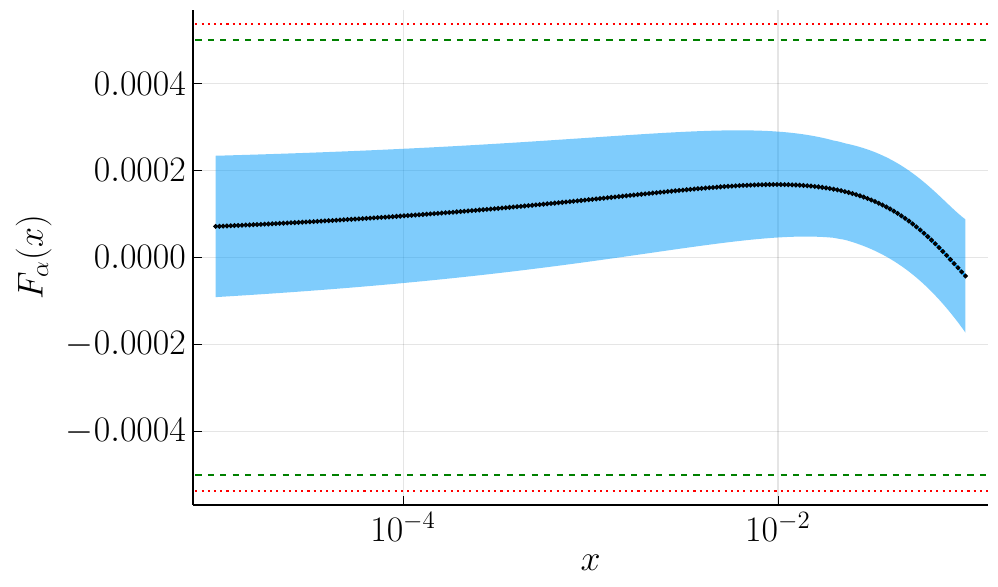}
    \caption{}
    \label{fig:bounds-I-1-F-asymptotic}
  \end{subfigure}
  \caption{Plot of the function \(F_{\alpha}\) for
    \(\alpha \in I_{1}\). The function is plotted on the interval
    \([0, \pi]\) and \([10^{-5}, 10^{-1}]\). The dashed green line
    shows the upper bound as given in
    Lemma~\ref{lemma:bounds-I-1-delta}. The dotted red line shows
    \(\frac{(1 - \bar{D}_{1})^{2}}{4\bar{n}_{1}}\), which is the value
    we want \(\delta_{\alpha}\) to be smaller than.}
\end{figure}

\subsection{Bounds for \(I_{2}\)}
\label{sec:bounds-I-2}
We here give bounds of \(n_{\alpha}\), \(\delta_{\alpha}\) and
\(D_{\alpha}\) for \(\alpha \in I_{2} = (-0.9999, -0.0012)\).

In this case it is not possible to give uniform bounds valid on the
whole interval \(I_{2}\). Instead we split \(I_{2}\) into many
subintervals and for each individual subinterval bounds are computed
and they are checked to satisfy
\begin{equation*}
  \delta_{\alpha} < \frac{(1 - D_{\alpha})^{2}}{4n_{\alpha}},
\end{equation*}
as required for Proposition~\ref{prop:contraction}.

When the midpoint of the subinterval is less than \(-0.95\) the hybrid
approach corresponding to \(\alpha\) near \(-1\) is used. In this case
the weight is given by \(w_{\alpha}(x) = |x|^{p}\log(2e + 1 / |x|)\).
If the midpoint of the subinterval is larger than \(-0.16\) then the
hybrid approach corresponding to \(\alpha\) near \(0\) is used, in
which case the weight is \(w_{\alpha}(x) = |x|\). For the rest of the
interval the default approach is used, with the weight
\(w_{\alpha}(x) = |x|^{p}\).

The value of \(p\) is chosen based on the midpoint of the subinterval.
If the midpoint is above \(-0.33\) we use \(p = 1\) and if it is
between \(-0.5\) and \(-0.33\) we use \(p = \frac{3}{4}\). If it is
below \(-0.5\) we take \(p = \frac{1 - \alpha}{2}\), evaluated at the
midpoint of the interval. Note that in all cases we have
\(-\alpha < p \leq 1\).

The interval is split into \(72 000\) subintervals. The sizes of the
subintervals are smaller near \(\alpha = -1\) and \(\alpha = 0\). The
interval \((-0.9999, -0.0012)\) is first split non-uniformly into
\(32\) subintervals, each of these subintervals are then further split
uniformly into a given number of subintervals, the number varying for
the different parts, see Table~\ref{table:I-2-subintervals}. Of these
\(72 000\) subintervals there are a few for which the computed bounds
are too bad to satisfy the required inequality. This happens in \(28\)
cases, in all of these cases it is enough to bisect the subinterval
once for the bounds to be good enough. The final result is therefore
based on a total of \(72 028\) subintervals.

\begin{table}
  \centering
  \begin{tabular}{c|c|c|c}
    Interval & Subintervals & Interval & Subintervals\\
    \hline
    \([-0.9999, -0.999875]\)  & \(500\)  & \([-0.7, -0.6]\)            & \(1000\) \\
    \([-0.999875, -0.99985]\) & \(500\)  & \([-0.6, -0.5]\)            & \(1500\) \\
    \([-0.99985, -0.9998]\)   & \(1000\) & \([-0.5, -0.45]\)           & \(1000\) \\
    \([-0.9998, -0.9996]\)    & \(1000\) & \([-0.45, -0.41]\)          & \(1000\) \\
    \([-0.9996, -0.9993]\)    & \(1000\) & \([-0.41, -0.37]\)          & \(1000\) \\
    \([-0.9993, -0.999]\)     & \(1000\) & \([-0.37, -0.33]\)          & \(1000\) \\
    \([-0.999, -0.998]\)      & \(1000\) & \([-0.33, -0.2]\)           & \(10000\) \\
    \([-0.998, -0.996]\)      & \(1000\) & \([-0.2, -0.16]\)           & \(3000\) \\
    \([-0.996, -0.993]\)      & \(1000\) & \([-0.16, -0.1]\)           & \(500\) \\
    \([-0.993, -0.99]\)       & \(1000\) & \([-0.1, -0.05]\)           & \(500\) \\
    \([-0.99, -0.95]\)        & \(2000\) & \([-0.05, -0.025]\)         & \(500\) \\
    \([-0.95, -0.935]\)       & \(5000\) & \([-0.025, -0.0125]\)       & \(1000\) \\
    \([-0.935, -0.9]\)        & \(8000\) & \([-0.0125, -0.00625]\)     & \(2000\) \\
    \([-0.9, -0.85]\)         & \(4000\) & \([-0.00625, -0.003125]\)   & \(4000\) \\
    \([-0.85, -0.8]\)         & \(2000\) & \([-0.003125, -0.0015625]\) & \(8000\) \\
    \([-0.8, -0.7]\)          & \(2000\) & \([-0.0015625, -0.0012]\)   & \(4000\) \\
  \end{tabular}
  \caption{The interval \([-0.9999, -0.0012]\) is split into these
    \(32\) subintervals which are then further split uniformly into
    the given number of subintervals. The first \(11\) use the hybrid
    approach corresponding to \(\alpha\) near \(-1\) and the last
    \(8\) use the hybrid approach corresponding to \(\alpha\) near
    \(0\), the ones in between use the default approach.}
  \label{table:I-2-subintervals}
\end{table}

Due to the large number of subintervals, explicit bounds for each case
is not given here, we refer to the repository~\cite{HighestCuspedWave}
where the data can be found. The bounds are visualized in
Figures~\ref{fig:bounds-I-2-n},~\ref{fig:bounds-I-2-delta}
and~\ref{fig:bounds-I-2-D}. However, these figures only give a rough
picture since the number of subintervals is larger than the number of
pixels in the picture and the subintervals are concentrated near
\(\alpha = -1\) and \(\alpha = 0\).

For the precise values of \(N_{\alpha,0}\) and \(N_{\alpha,1}\) we
refer to the repository~\cite{HighestCuspedWave}. In general
\(N_{\alpha,0}\) takes it maximum value near \(\alpha = -1\) and
minimum value near \(\alpha = 0\), varying from \(5799\) to \(2\). For
\(N_{\alpha,1}\) it varies between \(0\) and \(16\).

The approach for bounding \(n_{\alpha}\), \(\delta_{\alpha}\) and
\(D_{\alpha}\) is given in the following three lemmas. The total
runtime for the calculations on Dardel were around 12000 core hours
(corresponding to a wall time of around 48 hours using 256 cores).
\begin{lemma}
  \label{lemma:bounds-I-2-n}
  For all \(\alpha \in I_{2}\) the value \(N_{\alpha}\) is bounded. A
  sketch of the bound is given in Figure~\ref{fig:bounds-I-2-n}, for
  the precise bound for each value of \(\alpha\) we refer to the
  repository~\cite{HighestCuspedWave}.
\end{lemma}
\begin{proof}
  Similar to in Lemma~\ref{lemma:bounds-I-1-n} the maximum of
  \(N_{\alpha}\) is in practice attained at \(x = \pi\), which makes
  it easy to compute an accurate bound. As seen in the plot of
  \(n_{\alpha}\) in Figure~\ref{fig:bounds-I-2-n} it in general varies
  smoothly with \(\alpha\), the exception being the points where the
  choice of weight is changed. We can also note that the bounds are
  well-behaved near \(\alpha = -1\) and \(\alpha = 0\).

  When computing the bounds the value of \(\epsilon\) is not fixed,
  instead it is determined dynamically. Starting with \(\epsilon = 3\)
  we check if the asymptotic version of \(N_{\alpha}\) evaluated at
  this point is less than \(N_{\alpha}(\pi)\). if this is the case we
  take this value of \(\epsilon\), otherwise we try with a slightly
  smaller \(\epsilon\), stopping once we find one that works. On the
  interval \([0, \epsilon]\) we then check that \(N_{\alpha}\) is
  bounded by \(N_{\alpha}(\pi)\).

  For the interval \([\epsilon, \pi]\) we compute an enclosure of the
  maximum.

  In both cases we use Taylor expansions of degree \(0\) when
  computing the bounds.

  For the hybrid case when \(\alpha\) is close to \(-1\) the approach
  is exactly the same as above. For the hybrid case when \(\alpha\) is
  close to \(0\) we use the same approach as in
  Lemma~\ref{lemma:bounds-I-3-n} for computing a bound.

  The computational time varies with \(\alpha\), on average the
  computations took 4.05 seconds for each subinterval with a minimum
  and maximum of 0.018 and 144 seconds respectively.
\end{proof}

\begin{lemma}
  \label{lemma:bounds-I-2-delta}
  For all \(\alpha \in I_{2}\) the value \(\delta_{\alpha}\) is
  bounded. A sketch of the bound is given in
  Figure~\ref{fig:bounds-I-2-delta}, for the precise bound for each
  value of \(\alpha\) we refer to the
  repository~\cite{HighestCuspedWave}.
\end{lemma}
\begin{proof}
  The value of \(\delta_{\alpha}\) is heavily dependent on the precise
  approximation used. Very small changes in the approximation can give
  relatively large changes for \(\delta_{\alpha}\). This is clearly
  seen in Figure~\ref{fig:bounds-I-2-delta} where the bound is seen to
  vary a lot from subinterval to subinterval. The patches where the
  bound is more or less continuous correspond to certain values of
  \(N_{\alpha,0}\) in the approximation, when the value changes from
  one to the other it gives a large change in defect. There are also
  many isolated points where the bound is very different from the
  nearby points, this is mostly due to numerical instabilities when
  computing the coefficients for the approximation.

  Similar to in the previous lemma the value of \(\epsilon\) is not
  fixed but chosen dynamically. The precise choice is based on
  comparing the asymptotic and the non-asymptotic versions at
  different values and see where the asymptotic one starts to give
  tighter enclosures. The bounds on \([0, \epsilon]\) and
  \([\epsilon, \pi]\) are then computed separately, and their maximum
  is returned.

  In both cases we use Taylor expansions, the degree is tuned
  depending on \(\alpha\) and is between \(2\) and \(6\).

  For the hybrid case when \(\alpha\) is close to \(-1\) the approach
  is exactly the same as above. For the hybrid case when \(\alpha\) is
  close to \(0\) some modifications needs to be made. For the interval
  \([0, \epsilon]\) we use Taylor models in \(\alpha\) to compute
  enclosures for the coefficients in the expansion at \(x = 0\). Once
  we have the expansion we are able to compute Taylor expansions in
  \(x\) of it, where we use degree \(8\). For the interval
  \([\epsilon, \pi]\) we are not able to compute Taylor expansions in
  \(x\) and have to fall back to using direct enclosures.

  The computational time varies with \(\alpha\), on average the
  computations took 37 seconds for each subinterval with a minimum and
  maximum of 0.11 and 1084 seconds respectively.
\end{proof}

\begin{lemma}
  \label{lemma:bounds-I-2-D}
  For all \(\alpha \in I_{2}\) the value \(D_{\alpha}\) is bounded by
  a value smaller than \(1\). A sketch of the bound is given in
  Figure~\ref{fig:bounds-I-2-D}, for the precise bound for each value
  of \(\alpha\) we refer to the repository~\cite{HighestCuspedWave}.
\end{lemma}
\begin{proof}
  The computation of the bound of \(D_{\alpha}\) is the most
  time-consuming part. For that reason we do not attempt to compute a
  very accurate upper bound but only as much as we need for
  Lemma~\ref{lemma:bounds-I-2-inequality} to go through. We need
  \(D_{\alpha}\) to satisfy the inequality
  \begin{equation*}
    D_{\alpha} < 1 - 2\sqrt{n_{\alpha}\delta_{\alpha}}.
  \end{equation*}
  Taking the upper bounds of \(n_{\alpha}\) and \(\delta_{\alpha}\)
  from the previous two lemmas we compute the value that
  \(D_{\alpha}\) needs to be bounded by. To give a little headroom we
  subtract \(2^{-26}\) to get the bound that we use. This bound is
  seen in Figure~\ref{fig:bounds-I-2-D}, and we prove that
  \(D_{\alpha}\) is bounded by this value for all the subintervals.

  To get an understanding of how this bound compares to the actual
  value of \(D_{\alpha}\) we also compute a non-rigorous approximation
  of it. This approximation is computed by evaluating
  \(\mathcal{T}_{\alpha}\) on a few points in the interval
  \([0, \pi]\) and taking the maximum, but with no control of what
  happens in between these points. This estimate can also be seen in
  Figure~\ref{fig:bounds-I-2-D}. From this estimate we can also
  compute an estimate of \(\frac{(1 - D_{\alpha})^{2}}{4n_{\alpha}}\),
  which is the value we want \(\delta_{\alpha}\) to be less than, this
  estimate is seen in Figure~\ref{fig:bounds-I-2-delta}.

  Similar to in the previous two lemmas the value of \(\epsilon\) is
  not fixed but chosen dynamically. It is determined by starting at
  \(\epsilon = 1\) and then taking smaller and smaller values until
  the asymptotic version of \(\mathcal{T}_{\alpha}\) evaluated at
  \(\epsilon\) satisfies the prescribed bound. We then prove that the
  bound holds on both \([0, \epsilon]\) and \([\epsilon, \pi]\).

  For the interval \([0, \epsilon]\) we do not make use of Taylor
  expansions but only use direct evaluation. For the interval
  \([\epsilon, \pi]\) we are not able to compute Taylor expansion of
  \(U_{\alpha}\), and hence not of \(\mathcal{T}_{\alpha}\). We do
  however make one optimization to get better enclosures. Recall that
  \begin{equation*}
    \mathcal{T}_{\alpha}(x) = \frac{U_{\alpha}(x)}{\pi|w_{\alpha}(x)||u_{\alpha}(x)|}.
  \end{equation*}
  For a given interval \(\inter{x}\) both \(U_{\alpha}(\inter{x})\)
  and \(w_{\alpha}(\inter{x})\) are evaluated directly. To enclose
  \(u_{\alpha}(\inter{x})\) we first check if it is monotone by
  computing the derivative and checking that it is non-zero. If it is
  monotone we compute an enclosure by evaluating it at the endpoints
  of \(\inter{x}\).

  Both the hybrid cases use exactly the same approach.

  The computational time varies with \(\alpha\), on average the
  computations took 72 seconds for each subinterval with a minimum and
  maximum of 0.740.68 and 1075 seconds respectively.
\end{proof}

\begin{lemma}
  \label{lemma:bounds-I-2-inequality}
  For all \(\alpha \in I_{2}\) the following inequality is satisfied
  \begin{equation*}
    \delta_{\alpha} < \frac{(1 - D_{\alpha})^{2}}{4n_{\alpha}}.
  \end{equation*}
\end{lemma}
\begin{proof}
  Using the computed upper bounds from Lemma~\ref{lemma:bounds-I-2-n},
  \ref{lemma:bounds-I-2-delta} and~\ref{lemma:bounds-I-2-D} it is
  straight forward to check that the inequality
  \begin{equation*}
    \delta_{\alpha} < \frac{(1 - D_{\alpha})^{2}}{4n_{\alpha}}
  \end{equation*}
  holds for each of the \(72 028\) subintervals. Since the union of
  these subintervals cover the entire of \(I_{2}\) it follows that the
  inequality holds for all \(\alpha \in I_{2}\).
\end{proof}

\begin{figure}
  \centering
  \begin{subfigure}[t]{0.3\textwidth}
    \includegraphics[width=\textwidth]{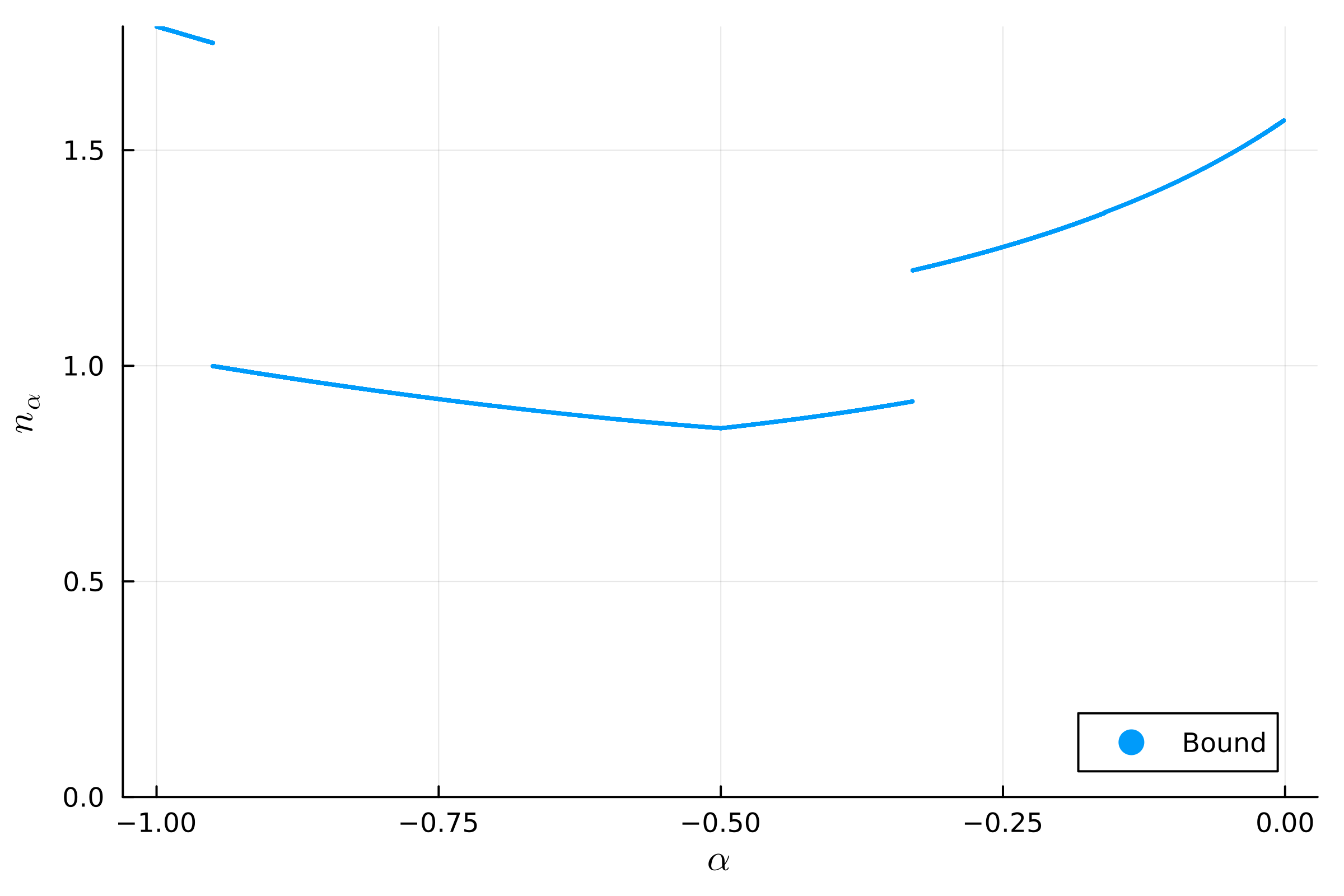}
    \caption{}
    \label{fig:bounds-I-2-n}
  \end{subfigure}
  \begin{subfigure}[t]{0.3\textwidth}
    \includegraphics[width=\textwidth]{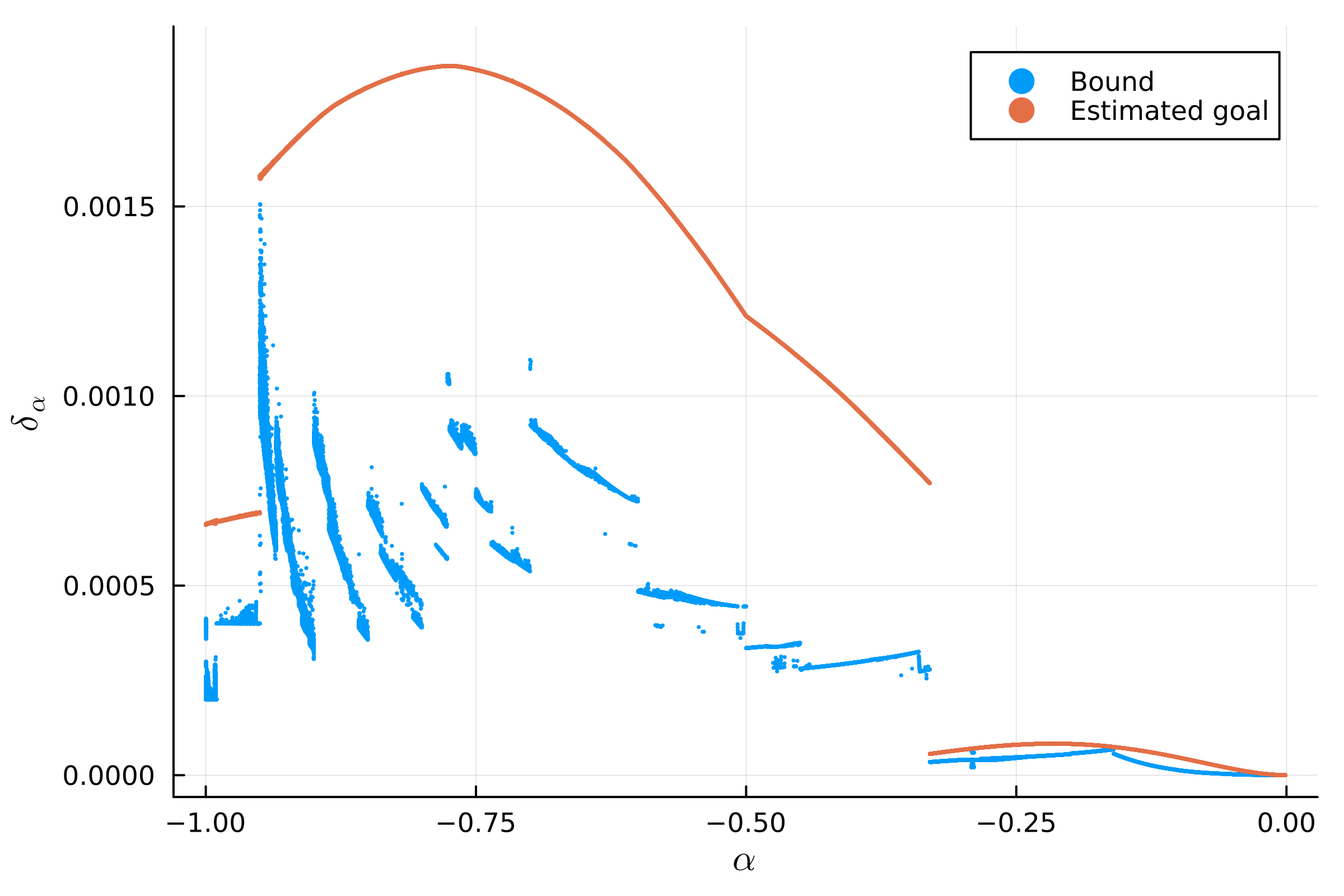}
    \caption{}
    \label{fig:bounds-I-2-delta}
  \end{subfigure}
  \begin{subfigure}[t]{0.3\textwidth}
    \includegraphics[width=\textwidth]{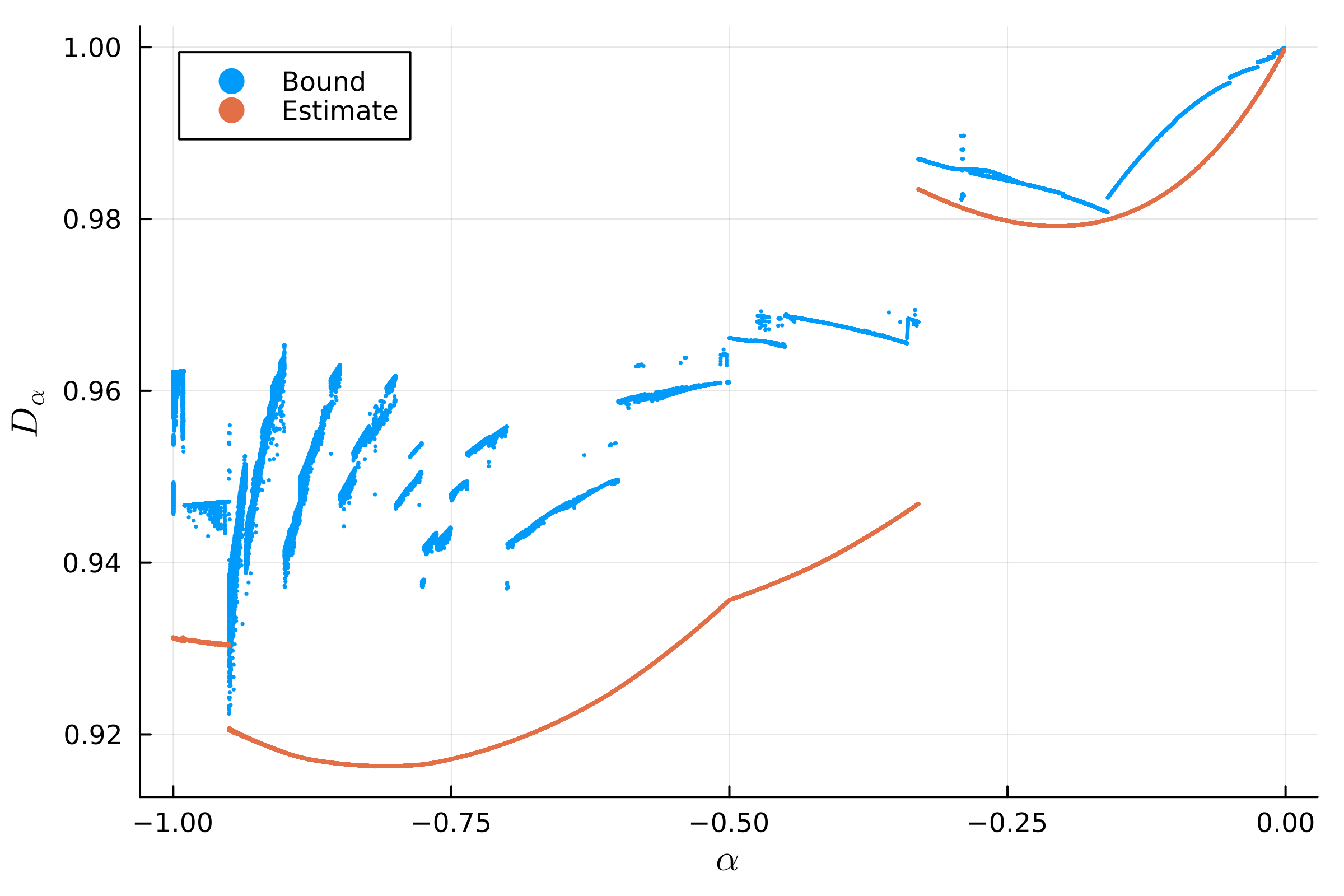}
    \caption{}
    \label{fig:bounds-I-2-D}
  \end{subfigure}
  \caption{Plot of bounds of \(n_{\alpha}\), \(\delta_{\alpha}\) and
    \(D_{\alpha}\) for \(\alpha \in I_{2}\). For \(\delta_{\alpha}\)
    the estimated goal is given by
    \(\frac{(1 - D_{\alpha})^{2}}{4n_{\alpha}}\) computed with the
    upper bound of \(n_{\alpha}\) and the estimate of \(D_{\alpha}\).
    For \(D_{\alpha}\) the non-rigorous estimate is also included.}
\end{figure}

\subsection{Bounds for \(I_{3}\)}
\label{sec:bounds-I-3}
For this interval the bounds are given in a slightly different form.
As \(\alpha\) gets close to \(0\) the value of \(\delta_{\alpha}\)
tends to zero and \(D_{\alpha}\) tends to one. To be able to say
something about the inequality
\begin{equation*}
  \delta_{\alpha} < \frac{(1 - D_{\alpha})^{2}}{4n_{\alpha}}
\end{equation*}
in Proposition~\ref{prop:contraction} in the limit we therefore need
information about the rate that they go to zero and one respectively.
This is however not the case for \(n_{\alpha}\), for which we have the
following lemma:

\begin{lemma}
  \label{lemma:bounds-I-3-n}
  The constant \(n_{\alpha}\) satisfies the inequality
  \(n_{\alpha} \leq \bar{n}_{3} = 1.58\) for all \(\alpha \in I_{3}\).
\end{lemma}
\begin{proof}
  A plot of \(N_{\alpha}(x)\) on the interval \([0, \pi]\) is given in
  Figure~\ref{fig:bounds-I-3-N}, clearly hinting at the maximum being
  attained at \(x = \pi\). We use the asymptotic expansion on the
  entire interval, i.e.\ \(\epsilon = \pi\). The computation of the
  enclosure takes less than a second and gives us
  \begin{equation*}
    n_{\alpha} \in [1.57 \pm 2.45 \cdot 10^{-3}],
  \end{equation*}
  which is upper bounded by \(\bar{n}_{3}\).
\end{proof}

For \(D_{\alpha}\) and \(\delta_{\alpha}\) we give bounds which are
functions of \(\alpha\).

\begin{lemma}
  \label{lemma:bounds-I-3-delta}
  The constant \(\delta_{\alpha}\) satisfies the inequality
  \(\delta_{\alpha} \leq \bar{\Delta}_{\delta} \cdot \alpha^{2}\),
  where \(\bar{\Delta}_{\delta} = 0.0077\), for all
  \(\alpha \in I_{3}\).
\end{lemma}
\begin{proof}
  For each \(x \in [0, \pi]\) we can compute a Taylor model
  \(M_{F}(x)\) of degree \(1\) in \(\alpha\) of \(F_{\alpha}(x)\)
  centered at \(\alpha = 0\) and valid for \(\alpha \in I_{3}\). This
  means that \(M_{F}(x)\) satisfies
  \begin{equation*}
    F_{\alpha}(x) \in p_{M_{F}(x)}(\alpha) + \Delta_{M_{F}(x)} \alpha^{2}
  \end{equation*}
  for all \(\alpha \in I_{3}\), where \(p_{M_{F}(x)}\) is the degree
  \(1\) polynomial associated with the Taylor model and
  \(\Delta_{M_{F}(x)}\) is the remainder. From
  Lemma~\ref{lemma:evaluation-F-expansion-alpha-I-3} we get that
  \(p_{M_{F}(x)} = 0\) for all \(x \in [0, \pi]\), hence
  \begin{equation*}
    |F_{\alpha}(x)| \in |\Delta_{M_{F}(x)}| \alpha^{2}.
  \end{equation*}
  This gives us
  \begin{equation*}
    \delta_{\alpha} = \sup_{x \in [0, \pi]} |F_{\alpha}(x)|
    \in \left(\sup_{x \in [0, \pi]} |\Delta_{M_{F}(x)}|\right) \alpha^{2}
    =: \Delta_{\delta} \alpha^{2},
  \end{equation*}
  where we with the supremum of intervals we mean the interval whose
  upper endpoint is the supremum of all upper endpoints and the lower
  endpoint is the supremum of all lower endpoints. A plot of
  \(\Delta_{M_{F}(x)}\) on the interval \([0, \pi]\) is given in
  Figure~\ref{fig:bounds-I-3-Delta-delta}, we are interested in
  enclosing \(\Delta_{\delta}\).

  We use the asymptotic expansion on the entire interval, i.e.
  \(\epsilon = \pi\). Enclosing the supremum we get the interval
  \begin{equation*}
    \Delta_{\delta} \subseteq [0, 7.65 \cdot 10^{-3}],
  \end{equation*}
  for which an upper bound is given by \(\bar{\Delta}_{\delta}\). The
  runtime is around 3 seconds.
\end{proof}

\begin{lemma}
  \label{lemma:bounds-I-3-D}
  The constant \(D_{\alpha}\) satisfies the inequality
  \(D_{\alpha} \leq 1 + \bar{\Delta}_{D} \cdot \alpha\), where
  \(\bar{\Delta}_{D} = 0.226\), for all \(\alpha \in I_{3}\).
\end{lemma}
\begin{proof}
  The proof is similar to the previous lemma. For each
  \(x \in [0, \pi]\) we can compute a Taylor model
  \(M_{\mathcal{T}}(x)\) of degree \(0\) in \(\alpha\) of
  \(\mathcal{T}_{\alpha}(x)\) centered at \(\alpha = 0\) and valid for
  \(\alpha \in I_{3}\). This means that \(M_{\mathcal{T}}(x)\)
  satisfies
  \begin{equation*}
    \mathcal{T}_{\alpha}(x) \in p_{M_{\mathcal{T}}(x)}(\alpha) + \Delta_{M_{\mathcal{T}}(x)}\alpha
  \end{equation*}
  for all \(\alpha \in I_{3}\), where \(p_{M_{\mathcal{T}}(x)}\) is
  the degree \(0\) polynomial associated with the Taylor model and
  \(\Delta_{M_{\mathcal{T}}(x)}\) is the remainder. From
  Lemma~\ref{lemma:evaluation-T-expansion-alpha-I-3} we get that
  \(p_{M_{\mathcal{T}}(x)} = 1\) for all \(x \in [0, \pi]\), hence
  \begin{equation*}
    \mathcal{T}_{\alpha}(x) \in 1 + \Delta_{M_{\mathcal{T}}(x)} \alpha.
  \end{equation*}
  This gives us
  \begin{equation*}
    D_{\alpha} = \sup_{x \in [0, \pi]} |\mathcal{T}_{\alpha}(x)|
    \in 1 + \left(\inf_{x \in [0, \pi]} \Delta_{M_{\mathcal{T}}(x)}\right) \alpha
    =: 1 + \Delta_{D} \alpha,
  \end{equation*}
  where we take the infimum since \(\alpha\) is negative. A plot of
  \(\Delta_{M_{\mathcal{T}}(x)}\) is given in
  Figure~\ref{fig:bounds-I-3-Delta-D}, we are interested in enclosing
  \(\Delta_{D}\).

  We take \(\epsilon = 1\). From Figure~\ref{fig:bounds-I-3-Delta-D}
  it seems like the infimum is attained at \(x = 0\), we therefore
  compute the infimum on the interval \([0, \epsilon]\) first and then
  only prove that the value on \([\epsilon, \pi]\) is lower bounded by
  this. Enclosing the infimum on \([0, \epsilon]\) we get the
  enclosure
  \begin{equation*}
    \Delta_{D} \subset [0.2413 \pm 0.0151],
  \end{equation*}
  which is then checked to be a lower bound for \([\epsilon, \pi]\).
  This computed enclosure is then lower bounded by
  \(\bar{\Delta}_{D}\). The runtime is around 7 seconds.
\end{proof}

Combining the above we get the following result.
\begin{lemma}
  \label{lemma:bounds-I-3-inequality}
  For all \(\alpha \in I_{3}\) the following inequality is satisfied
  \begin{equation*}
    \delta_{\alpha} < \frac{(1 - D_{\alpha})^{2}}{4n_{\alpha}}.
  \end{equation*}
\end{lemma}
\begin{proof}
  From lemma~\ref{lemma:bounds-I-3-n} and~\ref{lemma:bounds-I-3-D} we
  get
  \begin{equation*}
    \frac{(1 - D_{\alpha})^{2}}{4n_{\alpha}} \geq \frac{\bar{\Delta}_{D}^{2}}{4\bar{n}_{3}} \cdot \alpha^{2}
  \end{equation*}
  for all \(\alpha \in I_{3}\). Combining this with
  lemma~\ref{lemma:bounds-I-3-delta} means we only have to check the
  inequality
  \begin{equation*}
    \bar{\Delta}_{\delta} < \frac{\bar{\Delta}_{D}^{2}}{4\bar{n}_{3}},
  \end{equation*}
  which indeed holds.
\end{proof}

\begin{figure}
  \centering
  \begin{subfigure}[t]{0.3\textwidth}
    \includegraphics[width=\textwidth]{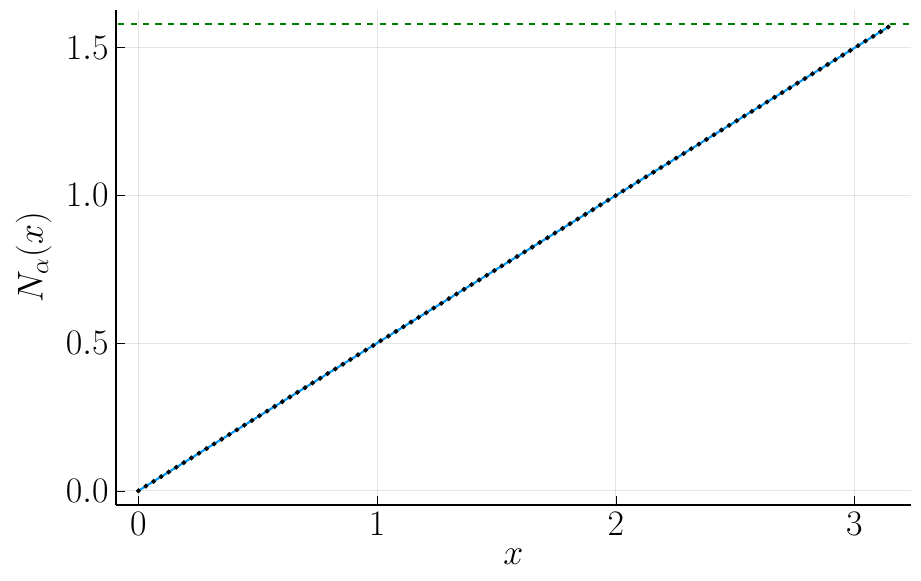}
    \caption{}
    \label{fig:bounds-I-3-N}
  \end{subfigure}
    \begin{subfigure}[t]{0.3\textwidth}
    \includegraphics[width=\textwidth]{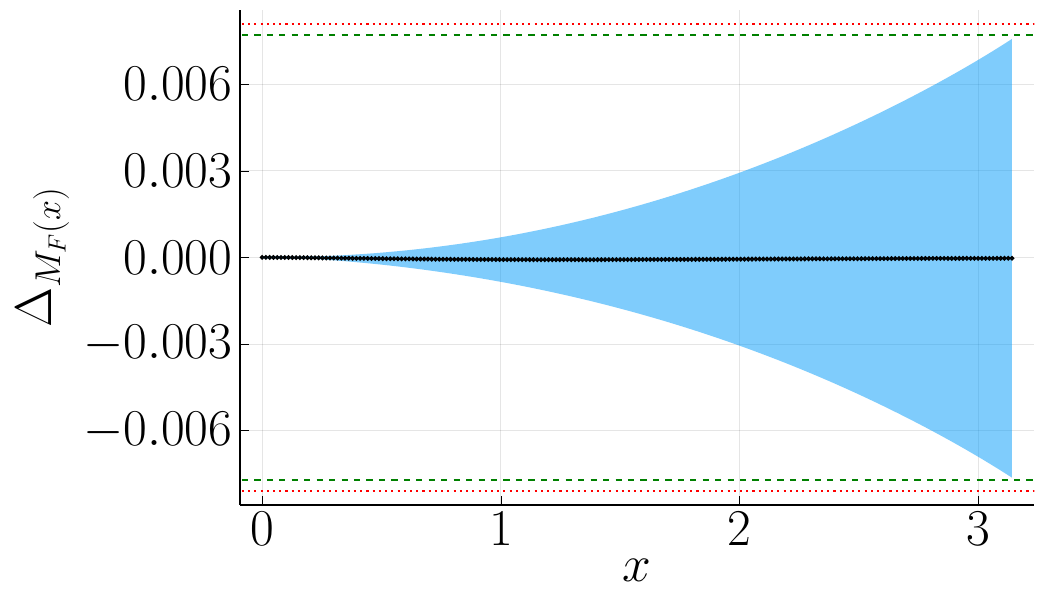}
    \caption{}
    \label{fig:bounds-I-3-Delta-delta}
  \end{subfigure}
  \begin{subfigure}[t]{0.3\textwidth}
    \includegraphics[width=\textwidth]{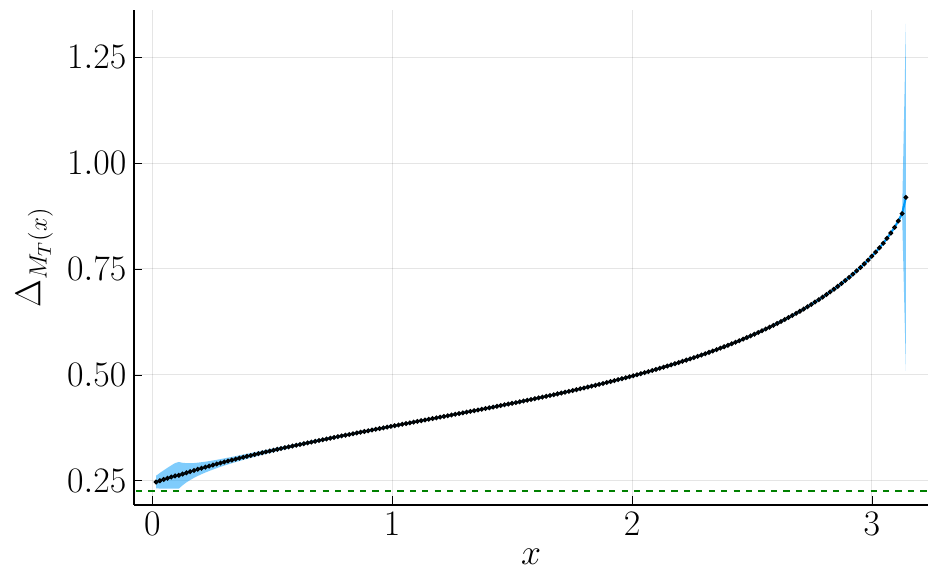}
    \caption{}
    \label{fig:bounds-I-3-Delta-D}
  \end{subfigure}
  \caption{Plots for \(\alpha \in I_{3}\). The leftmost figure shows a
    plot of \(N_{\alpha}(x)\). The two other figures show enclosures
    of \(\Delta_{M_{F}(x)}\) and \(\Delta_{M_{\mathcal{T}}(x)}\), the
    remainder terms of the Taylor models for \(F_{\alpha}(x)\) and
    \(\mathcal{T}_{\alpha}(x)\) respectively. The dashed green lines
    shows the upper/lower bounds given in
    Lemmas~\ref{lemma:bounds-I-3-n}, \ref{lemma:bounds-I-3-delta}
    and~\ref{lemma:bounds-I-3-D}. The dotted red line in the middle
    figure shows \(\frac{\bar{\Delta}_{D}^{2}}{4\bar{n}_{3}}\), the
    value that \(\Delta_{\delta}\) needs to be smaller than.}
\end{figure}

\section{Proof of Theorem~\ref{thm:main}}
\label{sec:proof-main-theorem}
We are now ready to give the proof of Theorem~\ref{thm:main}.

\begin{proof}[Proof of Theorem~\ref{thm:main}]
  Consider the operator \(G_{\alpha}\) from \eqref{eq:G} given by
  \begin{equation*}
    G_{\alpha}[v] = (I - T_{\alpha})^{-1}(-F_{\alpha} - N_{\alpha}v^{2}).
  \end{equation*}
  Combining Lemmas~\ref{lemma:bounds-I-1-D}, \ref{lemma:bounds-I-2-D}
  and~\ref{lemma:bounds-I-3-D} we have \(\|T_{\alpha}\| < 1\) for all
  \(\alpha \in (-1, 0)\), so the inverse of the operator
  \(I - T_{\alpha}\) is well-defined. Combining Lemmas
 ~\ref{lemma:bounds-I-1-inequality}, \ref{lemma:bounds-I-2-inequality}
  and~\ref{lemma:bounds-I-3-inequality} gives us that the inequality
  \begin{equation*}
    \delta_{\alpha} < \frac{(1 - D_{\alpha})^{2}}{4n_{\alpha}}
  \end{equation*}
  holds for all \(\alpha \in (-1, 0)\). This, together with
  Proposition~\ref{prop:contraction} and Banach fixed-point theorem,
  proves that for
  \begin{equation*}
    \epsilon_{\alpha} = \frac{1 - D_{\alpha} - \sqrt{(1 - D_{\alpha})^{2} - 4\delta_{\alpha}n_{\alpha}}}{2n_{\alpha}}
  \end{equation*}
  the operator \(G_{\alpha}\) has a unique fixed-point \(v_{\alpha}\)
  in \(X_{\epsilon_{\alpha}} \subseteq L^{\infty}(\mathbb{T})\).

  By the construction of the operator \(G_{\alpha}\) this means that
  the function
  \begin{equation*}
    u(x) = u_{\alpha}(x) + w_{\alpha}(x)v_{\alpha}(x)
  \end{equation*}
  solves \eqref{eq:main}, given by
  \begin{equation*}
    \frac{1}{2}u^{2} = -\Hop{\alpha}[u].
  \end{equation*}
  For any wave speed \(c \in \mathbb{R}\) we then have that the
  function
  \begin{equation*}
    \varphi(x) = c - u(x)
  \end{equation*}
  is a traveling wave solution to \eqref{eq:fkdv}. This proves the
  existence of a \(2\pi\)-periodic highest cusped traveling wave.

  The asymptotic behavior of \(u_{\alpha}\) close to \(x = 0\)
  is given by
  \begin{equation*}
    u_{\alpha}(x) = \nu_{\alpha}|x|^{-\alpha} + \mathcal{O}(|x|^{-\alpha + p_{\alpha}}).
  \end{equation*}
  with \(\nu_{\alpha}\) as in
  Lemma~\ref{lemma:asymptotic-coefficient}. The precise value for
  \(p_{\alpha}\) varies depending on the type of construction used,
  but in all cases it satisfies \(-\alpha + p_{\alpha} > 1\).
  Furthermore the weight satisfies
  \begin{equation*}
    w_{\alpha}(x) = \mathcal{O}(|x|^{p})
  \end{equation*}
  for some \(p\) with \(-\alpha < p \leq 1\). Hence
  \begin{equation*}
    u(x) = u_{\alpha}(x) + w_{\alpha}(x)v_{\alpha}(x) = \nu_{\alpha}|x|^{-\alpha} + \mathcal{O}(|x|^{p})
  \end{equation*}
  and
  \begin{equation*}
    \varphi(x) = c - \nu_{\alpha}|x|^{-\alpha} + \mathcal{O}(|x|^{p}),
  \end{equation*}
  as we wanted to show.
\end{proof}

\appendix

\section{Removable singularities}
\label{sec:removable-singularities}

In several cases we have to compute enclosures of functions with
removable singularities. For example the function
\begin{equation*}
  \Gamma(1 - s)\cos(\pi(1 - s) / 2)
\end{equation*}
comes up when computing \(C_{s}\) through equation
\eqref{eq:clausenc-periodic-zeta} and has a removable singularity
whenever \(s\) is a positive even integer. For this we follow the same
approach as in~\cite[Appendix A]{dahne2022burgershilbert}, where more
details are given. The main tool is the following lemma:

\begin{lemma}
  Let \(m \in \mathbb{Z}_{\geq 0}\) and let \(I\) be an interval
  containing zero. Consider a function \(f(x)\) with a zero of order
  \(n\) at \(x = 0\) and such that \(f^{(m + n)}(x)\) is absolutely
  continuous on \(I\). Then for all \(x \in I\) we have
  \begin{equation*}
    \frac{f(x)}{x^{n}} = \sum_{k = 0}^{m}f_{k + n}(0)x^{k} + f_{m + n + 1}(\xi)x^{m + 1}
  \end{equation*}
  for some \(\xi\) between \(0\) and \(x\). Furthermore, if
  \(f^{m + n + p}(x)\) is absolutely continuous for
  \(p \in \mathbb{Z}_{\geq 0}\) we have
  \begin{equation*}
    \frac{d^{p}}{dx^{p}}\frac{f(x)}{x^{n}}
    = \sum_{k = 0}^{m}\frac{(k + p)!}{k!}f_{k + n + p}(0)x^{k}
    + \frac{(m + p + 1)!}{(m + 1)!}f_{m + n + p + 1}(\xi)x^{m + 1}
  \end{equation*}
  for some \(\xi\) between \(0\) and \(x\).
\end{lemma}

We also make use of the first statement of the lemma for
\(x \in \mathbb{C}\), the proof is the same as for
\(x \in \mathbb{R}\).

\section{Taylor models}
\label{sec:taylor-models}
An important tool for handling the limit \(\alpha \to 0\) is the use
of Taylor models for enclosing \(\delta_{\alpha}\) and \(D_{\alpha}\).
We here give a brief introduction to Taylor models, for a more
thorough introduction we refer to~\cite{Joldes2011}.

With a Taylor model we mean what in~\cite{Joldes2011} is refereed to
as a \emph{Taylor model with relative error}. We use the following
definition, compare with~\cite[Definition 2.3.2]{Joldes2011}.
\begin{definition}
  A Taylor model \(M = (p, \Delta)\) of degree \(n\) for a function
  \(f\) on an interval \(I\) centered at a point \(x_{0} \in I\) is a
  polynomial \(p\) of degree \(n\) together with an interval
  \(\Delta\), satisfying that for all \(x \in I\) there is
  \(\delta \in \Delta\) such that
  \begin{equation*}
    f(x) - p(x - x_{0}) = \delta (x - x_{0})^{n + 1}.
  \end{equation*}
\end{definition}
For an \(n + 1\) times differentiable function \(f\) the polynomial
\(p\) is the Taylor polynomial of degree \(n\) of \(f\) centered at
\(x_{0}\)~\cite[Lemma 2.3.3]{Joldes2011}. A Taylor model is thus given
by a truncated Taylor expansion plus a bound on the remainder term
valid on some interval \(I\). From Taylor's theorem we see that we can
take \(\Delta\) to be an enclosure of
\(\frac{f^{n + 1}(x)}{(n + 1)!}\) on the interval \(I\).

We can perform arithmetic on Taylor models. Given two functions \(f\)
and \(g\) with corresponding Taylor models
\(M_{f} = (p_{f}, \Delta_{f})\) and \(M_{g} = (p_{g} + \Delta_{g})\)
we can compute a Taylor model of \(f + g\) as
\(M_{f + g} = (p_{f} + p_{g}, \Delta_{f} + \Delta_{g})\), similarly
for \(f - g\). With slightly more work we can compute a Taylor model
of \(f \cdot g\), see~\cite[Algorithm 2.3.6]{Joldes2011}, as well as
of \(f / g\), see~\cite[Algorithm 2.3.12]{Joldes2011}. The division
can directly handle removable singularities for \(f / g\), if there is
a removable singularity of order \(k\) then we need Taylor models of
degree \(n + k\) for \(f\) and \(g\) to get a Taylor model of degree
\(n\) for \(f / g\). We can also compose Taylor models with arbitrary
functions, given a Taylor model \(M_{f}\) of \(f\) and a function
\(g\) we can compute a Taylor model \(M_{g \circ f}\) of
\(g \circ f\), see~\cite[Algorithm 2.3.8]{Joldes2011}.

\subsection{Taylor models in \(\alpha\) of expansions in \(x\)}
\label{sec:taylor-models-expansions-x}
For both \(F_{\alpha}(x)\) and \(\mathcal{T}_{\alpha}(x)\) we need to
compute expansions in \(x\) to evaluate them near \(x = 0\). For
\(\alpha \in I_{3}\) we then need to compute Taylor models in
\(\alpha\) from these expansions, see
Section~\ref{sec:evaluation-F-I-3} and~\ref{sec:evaluation-T-I-3}.
This is done by computing Taylor models of the coefficients in the
expansion in \(x\).

For example, consider the problem of computing an expansion in \(x\)
of \(u_{\alpha}(x)\), used e.g.\ when computing
\(\frac{x^{-\alpha}}{u_{\alpha}(x)}\). For \(\alpha \in I_{3}\) we
have from Lemma~\ref{lemma:u0-asymptotic}, and using that
\(N_{\alpha,1} = 0\) in this case,
\begin{equation*}
  u_{\alpha}(x) = \sum_{j = 0}^{N_{\alpha,0}}a_{\alpha,j}^{0}|x|^{-\alpha + jp_{\alpha}}
  + \sum_{m = 1}^{\infty}\frac{(-1)^{m}}{(2m)!}\left(
    \sum_{j = 1}^{N_{\alpha,0}}a_{\alpha,j}\zeta(1 - \alpha + jp_{\alpha} - 2m)
  \right)x^{2m}.
\end{equation*}
It is straight forward to compute Taylor models of the coefficients
\(a_{\alpha,j}^{0}\) and
\begin{equation*}
  \frac{(-1)^{m}}{(2m)!}\left(
    \sum_{j = 1}^{N_{\alpha,0}}a_{\alpha,j}\zeta(1 - \alpha + jp_{\alpha} - 2m)
  \right).
\end{equation*}
for \(m < M\). What remains is computing a Taylor model bounding the
tail. To compute a Taylor model enclosing a sum of the form
\begin{equation*}
  \sum_{m = M}^{\infty}c_{\alpha,m}x^{2m}.
\end{equation*}
it is enough to have a method for enclosing
\begin{equation*}
  \frac{d^{k}}{d\alpha^{k}}\sum_{m = M}^{\infty}c_{\alpha,m}x^{2m}.
\end{equation*}
for \(k \geq 0\). For the Clausen functions this is given in
Lemma~\ref{lemma:clausen-derivative-tails}.

\subsection{Taylor model of \(p_{\alpha}\)}
\label{sec:taylor-models-p-alpha}
The above is enough to compute Taylor models of functions in the paper
for which we have explicit expressions. For \(p_{\alpha}\) from
\eqref{eq:p-alpha} we only have an implicit equation defining it, and
we have to take a slightly different approach.

Consider the equation
\begin{equation*}
  f(x, y) = 0.
\end{equation*}
We want to compute a Taylor model of \(y = y(x)\). Consider a Taylor
model of degree \(n\), centered at \(x_{0}\) and which should be valid
on the interval \(I\). As mentioned above it is enough to compute the
Taylor polynomial of \(y(x)\) at \(x = x_{0}\) of degree \(n\) and
enclose \(\frac{y^{n + 1}(x)}{(n + 1)!}\) on \(I\).

Let \(y_{0}\) be such that \(f(x_{0}, y_{0}) = 0\). We first consider
the case when \(f_{y}(x_{0}, y_{0}) \not= 0\). If \(f\) is
sufficiently smooth we can use the implicit function theorem to
compute a Taylor polynomial of \(y\) at \(x_{0}\) to degree \(n\).
Expanding \(f(x, y(x))\) at \(x_{0}\) we get
\begin{multline*}
  f(x_{0}, y_{0})
  + (f_{x}(x_{0}, y_{0}) + y'(x_{0})f_{y}(x_{0}, y_{0}))(x - x_{0})\\
  + \frac{1}{2}(f_{xx}(x_{0}, y_{0}) + 2y'(x_{0})f_{xy}(x_{0}, y_{0}) + y'(x_{0})f_{yy}(x_{0}, y_{0}) + y''(x_{0})f_{y}(x_{0}, y_{0}))(x - x_{0})^{2} \cdots = 0.
\end{multline*}
Solving for \(y'(x_{0})\), \(y''(x_{0})\), etc., we get
\begin{align*}
  y'(x_{0}) &= \frac{f_{x}(x_{0}, y(x_{0}))}{f_{y}(x_{0}, y(x_{0}))},\\
  y''(x_{0}) &= \frac{f_{xx}(x_{0}, y_{0}) + 2y'(x_{0})f_{xy}(x_{0}, y_{0}) + y'(x_{0})f_{yy}(x_{0}, y_{0})}{2f_{y}(x_{0}, y(x_{0}))}
\end{align*}
and similarly for higher orders.

While it is possible to use the above formulas for the derivatives to
enclose \(\frac{y^{n + 1}(x)}{(n + 1)!}\) on \(I\) this gets very
complicated for \(p_{\alpha}\) at \(\alpha = 0\) due to the large
number of removable singularities to handle. Instead, we take a
slightly different, more direct, approach for computing \(\Delta\). We
want to find \(\Delta\) such that for all \(x \in I\) there is
\(\delta \in \Delta\) such that
\begin{equation*}
  f(x, p(x - x_{0}) + \delta(x - x_{0})^{n + 1}) = 0.
\end{equation*}
If we let
\begin{equation*}
  g(x, \delta) = f(x, p(x - x_{0}) + \delta(x - x_{0})^{n + 1})
\end{equation*}
we want \(\Delta\) to be an enclosure of a root in \(\delta\) of \(g\)
for all \(x \in I\). To make the derivative in \(\delta\) non-zero at
\(x = x_{0}\) we normalize \(g\) as
\begin{equation*}
  h(x, \delta) = \frac{g(x, \delta)}{(x - x_{0})^{n + 1}}.
\end{equation*}
Given a guess \(\Delta = [\lo{\Delta}, \hi{\Delta}]\), to prove that
this is an enclosure of a root we only need to verify that
\(h(x, \lo{\Delta})\) and \(h(x, \hi{\Delta})\) both have constant
signs for all \(x \in I\) and that the two signs are different. The
guess for \(\Delta\) can be given using non-rigorous means, for
example by looking at the error at the endpoints of \(I\). Only the
verification of the signs at the endpoints of \(\Delta\) needs to be
done rigorously.

For \(p_{\alpha}\) we get from \eqref{eq:p-alpha} the function
\begin{equation*}
  f(\alpha, p) = \frac{
    \Gamma(2\alpha - p)\cos\left(\frac{\pi}{2}(2\alpha - p)\right)
  }{
    \Gamma(\alpha - p)\cos\left(\frac{\pi}{2}(\alpha - p)\right)
  }
  - \frac{
    2\Gamma(2\alpha)\cos(\pi\alpha)
  }{
    \Gamma(\alpha)\cos\left(\frac{\pi}{2}\alpha\right)
  }.
\end{equation*}
In this case \(\alpha = 0\) we have \(f_{p}(0, p_{0}) = 0\), and we
cannot apply the above approach directly. For the coefficients of
polynomial of the Taylor model we expand
\begin{equation*}
  f(\alpha, p_{0} + p_{1}\alpha + p_{2}\alpha^{2})
\end{equation*}
at \(\alpha = 0\). The constant coefficient is exactly zero, and we
then solve for \(p_{0}\), \(p_{1}\) and \(p_{2}\) making the three
coming coefficients zero. For the remainder term of the Taylor model
we instead we consider the function \(\frac{f(\alpha, p)}{\alpha}\),
for which the above approach does work.

\section{Computing Clausen functions}
\label{sec:computing-clausen}
To be able to compute bounds of \(D_{\alpha}\), \(\delta_{\alpha}\)
and \(n_{\alpha}\) it is critical that we can compute accurate
enclosures of \(C_{s}(x)\) and \(S_{s}(x)\), including expansions in
the argument and derivatives in the parameter. For this we follow the
same approach as in~\cite[Appendix B]{dahne2022burgershilbert} with a
few additions and improvements. We here give a shortened version of
the approach while highlighting the additions and improvements.

We start by going through how to compute \(C_{s}(x)\) and \(S_{s}(x)\)
for \(s, x \in \mathbb{R}\). Since both \(C_{s}(x)\) and \(S_{s}(x)\)
are \(2\pi\)-periodic we can reduce it to \(x = 0\) or
\(0 < x < 2\pi\).

For \(x = 0\) and \(s > 1\) have \(C_{s}(0) = \zeta(s)\) and
\(S_{s}(0) = 0\). For \(s \leq 1\) both functions typically diverge at
\(x = 0\).

For \(0 < x < 2\pi\) we can compute the Clausen functions by going
through the polylog function,
\begin{equation*}
  C_{s}(x) = \real\left(\polylog_{s}(e^{ix})\right),\quad S_{s}(x) = \imag\left(\polylog_{s}(e^{ix})\right).
\end{equation*}
However it is computationally beneficial to instead go through the
periodic zeta function~\cite[Sec.~25.13]{NIST:DLMF},
\begin{equation*}
  F(x, s) := \polylog_{s}(e^{2\pi i x}) = \sum_{n = 1}^{\infty}\frac{e^{2\pi i nx}}{n^{s}},
\end{equation*}
for which we have
\begin{equation*}
  C_{s}(x) = \real F\left(\frac{x}{2\pi}, s\right),\quad S_{s}(x) = \imag F\left(\frac{x}{2\pi}, s\right).
\end{equation*}
Using~\cite[Eq.~25.13.2]{NIST:DLMF} we get, for \(0 < x < 2\pi\),
\begin{align}
  \label{eq:clausenc-periodic-zeta}
  C_{s}(x) &= \frac{\Gamma(1 - s)}{(2\pi)^{1 - s}}\cos(\pi(1 - s) / 2)
             \left(\zeta\left(1 - s, \frac{x}{2\pi}\right) + \zeta\left(1 - s, 1 - \frac{x}{2\pi}\right)\right),\\
  \label{eq:clausens-periodic-zeta}
  S_{s}(x) &= \frac{\Gamma(1 - s)}{(2\pi)^{1 - s}}\sin(\pi(1 - s) / 2)
             \left(\zeta\left(1 - s, \frac{x}{2\pi}\right) - \zeta\left(1 - s, 1 - \frac{x}{2\pi}\right)\right).
\end{align}
This formulation works well as long as \(s\) is not a non-negative
integer. For non-negative integers we have to handle some removable
singularities, for details see~\cite[Appendix
B]{dahne2022burgershilbert}.

Compared to~\cite{dahne2022burgershilbert} we also need to compute
\(C_{s}(x)\) for real \(s\) but complex \(x\) during the rigorous
integration discussed in Appendix~\ref{sec:rigorous-integration} In
those cases \(s\) is never an integer, and we use
Equation~\eqref{eq:clausenc-periodic-zeta} which holds for
\(0 < \real x < 2\pi\).

\subsection{Interval arguments}
\label{sec:interval-arguments}
Let \(\inter{x} = [\lo{x}, \hi{x}]\) and
\(\inter{s} = [\lo{s}, \hi{s}]\) be two finite intervals, we are
interested in computing an enclosure of \(C_{\inter{s}}(\inter{x})\)
and \(S_{\inter{s}}(\inter{x})\). Due to the periodicity we can reduce
it to three different cases for \(\inter{x}\)
\begin{enumerate}
\item \(\inter{x}\) doesn't contain a multiple of \(2\pi\), by adding
  or subtracting a suitable multiple of \(2\pi\) we can assume that
  \(0 < \inter{x} < 2\pi\);
\item \(\inter{x}\) has a diameter of at least \(2\pi\), it then
  covers a full period and can without loss of generality be taken as
  \(\inter{x} = [0, 2\pi]\);
\item \(\inter{x}\) contains a multiple of \(2\pi\) but has a diameter
  less than \(2\pi\), by adding or subtracting a suitable multiple of
  \(2\pi\) we can take \(\inter{x}\) such that
  \(-2\pi < \lo{x} \leq 0 \leq \hi{x} < 2\pi\).
\end{enumerate}
The second and third case are handled exactly as
in~\cite{dahne2022burgershilbert}, for the first case we make some
minor improvements.

For \(C_{s}\) we make use of the following lemma which is a slightly
upgraded version of~\cite[Lemma B.1]{dahne2022burgershilbert}, see
also~\cite[Lemma B.1]{enciso2018convexity}.
\begin{lemma}
  \label{lemma:clausenc-monotone}
  For all \(s \in \mathbb{R}\) the Clausen function \(C_{s}(x)\) is
  monotone in \(x\) on the interval \((0, \pi)\). For \(s > 0\) it is
  decreasing. For \(s \leq 0\) the sign of the derivative is the same
  as that of \(-\sin\left(\frac{\pi}{2}s\right)\).
\end{lemma}
\begin{proof}
  The derivative in \(x\) is given by \(-S_{s - 1}(x)\). For \(s > 1\)
  the proof is the same as in~\cite[Lemma
  B.1]{dahne2022burgershilbert}.

  For \(s < 1\) we use equation \eqref{eq:clausens-periodic-zeta}
  together with~\cite[Eq.~25.11.25]{NIST:DLMF} to get
  \begin{equation*}
    S_{s - 1}(x) = \frac{\sin\left(\frac{\pi}{2}(2 - s)\right)}{(2\pi)^{2 - s}}\int_{0}^{\infty}t^{1 - s}e^{-\frac{xt}{2\pi}}\frac{1 - e^{(x / \pi - 1)t}}{1 - e^{-t}}\ dt,
  \end{equation*}
  which sign depends only on the value \(s\) and is the same as that
  of \(\sin\left(\frac{\pi}{2}s\right)\). In particular, for
  \(0 < s < 1\) we have \(\sin\left(\frac{\pi}{2}s\right) > 0\) so
  \(C_{s}(x)\) is decreasing for \(0 < s < 1\).

  For \(s = 1\) the result follows directly from that
  \(C_{1}(x) = -\log(2\sin(x / 2))\) on the interval.
\end{proof}

For \(S_{s}\) we have the following, slightly weaker, result
\begin{lemma}
  \label{lemma:clausens-monotone}
  For \(s \leq 1\) the Clausen function \(S_{s}(x)\) is monotone in
  \(x\) on the interval \((0, 2\pi)\). The sign of the derivative is
  the same as that of \(-\cos\left(\frac{\pi}{2}s\right)\).
\end{lemma}
\begin{proof}
  The derivative in \(x\) is given by \(C_{s - 1}(x)\). For \(s < 1\)
  we use equation \eqref{eq:clausenc-periodic-zeta} together
  with~\cite[Eq.~25.11.25]{NIST:DLMF} to get
  \begin{equation*}
    C_{s - 1}(x) = \frac{\cos\left(\frac{\pi}{2}(2 - s)\right)}{(2\pi)^{2 - s}}\int_{0}^{\infty}t^{1 - s}e^{-\frac{xt}{2\pi}}\frac{1 + e^{(x / \pi - 1)t}}{1 - e^{-t}}\ dt,
  \end{equation*}
  which sign depends only on the value \(s\) and is the same as that
  of \(-\cos\left(\frac{\pi}{2}s\right)\).

  For \(s = 1\) the result follows directly from that
  \(S_{1}(x) = \frac{\pi}{2} - \frac{x}{2}\) on the interval.
\end{proof}
For \(s \leq 1\) the extrema of \(S_{s}(\inter{x})\) are thus always
attained at \(x = \lo{x}\) and \(\hi{x}\). We handle \(s > 1\) by
computing and enclosure of the derivative
\(S_{\inter{s}}'(\inter{x}) = C_{\inter{s} - 1}(\inter{x})\), if the
enclosure of the derivative doesn't contain zero then the function is
monotone, and we evaluate it at the endpoints, if the derivative
contains zero we instead use the midpoint approximation
\(S_{\inter{s}}(\inter{x}) = S_{\inter{s}}(x_{0}) + (\inter{x} -
x_{0})C_{\inter{s} - 1}(\inter{x})\) where \(x_{0}\) is the midpoint
of \(\inter{x}\).

For \(\inter{s}\) we follow the same approach as
in~\cite{dahne2022burgershilbert}. The only difference is that we also
have to handle \(\inter{s}\) overlapping \(1\), in which case the
implementation of the deflated zeta function,
\begin{equation}
  \label{eq:deflated-zeta}
  \underline{\zeta}(s, x) = \zeta(s, x) + \frac{1}{1 - s} = \sum_{n = 0}^{\infty} \frac{(-1)^{n}}{n!}\gamma_{n}(x)(s - 1)^{n},
\end{equation}
in Arb does not work directly. It only supports \(s\) exactly equal to
\(1\), and not intervals containing \(1\). Looking at the
implementation it is however fairly easy to adapt it to also work for
intervals containing \(1\), which we have done.

\subsection{Expansion in \(x\)}
\label{sec:expansion-x}
We now go through how to compute expansions of the Clausen functions
in the argument \(x\). In general the procedure is the same as
in~\cite{dahne2022burgershilbert}, except that we also have to handle
the case when \(s\) overlaps positive integers for the expansion at
\(x = 0\).

At \(x = 0\) we have the following asymptotic
expansions~\cite{enciso2018convexity}
\begin{align}
  \label{eq:clausenc-asymptotic-x}
  C_{s}(x) &= \Gamma(1 - s)\sin\left(\frac{\pi}{2}s\right)|x|^{s - 1}
             + \sum_{m = 0}^{\infty} (-1)^{m}\zeta(s - 2m)\frac{x^{2m}}{(2m)!};\\
  \label{eq:clausens-asymptotic-x}
  S_{s}(x) &= \Gamma(1 - s)\cos\left(\frac{\pi}{2}s\right)\sign(x)|x|^{s - 1}
             + \sum_{m = 0}^{\infty} (-1)^{m}\zeta(s - 2m - 1)\frac{x^{2m + 1}}{(2m + 1)!}.
\end{align}
For positive integers we have to handle the poles of \(\Gamma(s)\) at
non-positive integers and the pole of \(\zeta(s)\) at \(s = 1\).

For \(C_{s}(x)\) with positive even integers \(s\) the only
problematic term is
\begin{equation*}
  \Gamma(1 - s)\sin\left(\frac{\pi}{2}s\right)|x|^{s - 1}
\end{equation*}
which has a removable singularity. Similarly, for \(S_{s}(x)\) with
positive odd integers \(s\). For \(C_{s}(x)\) with positive odd
integers \(s\) and \(S_{s}\) with positive even integers \(s\) the
singularities are not removable. The only case we encounter in the
paper is for \(C_{s}\), for which we make use of the following lemma
to bound the sum of the two singular terms.
\begin{lemma}
  \label{lemma:clausenc-expansion-singular-term}
  Let \(m \geq 1\) and
  \(s \in \left[2m + \frac{1}{2}, 2m + \frac{3}{2}\right]\), then
  \begin{equation*}
    \Gamma(1 - s)\sin\left(\frac{\pi}{2}s\right)|x|^{s - 1} + (-1)^{m}\zeta(s - 2m)\frac{x^{2m}}{(2m)!}
    = K_{1}|x|^{s - 1} + K_{2}x^{2m} + K_{3}x^{2m}\frac{|x|^{s - (2m + 1)} - 1}{s - (2m + 1)}
  \end{equation*}
  with
  \begin{align*}
    K_{1} &= \frac{\frac{\Gamma(2m + 2 - s)}{(1 - s)_{2m}}\sin\left(\frac{\pi}{2}s\right) - \frac{(-1)^{m}}{(2m!)}}{2m + 1 - s},\\
    K_{2} &= -\frac{(-1)^{m}\underline{\zeta}(s - 2m)}{(2m)!},\\
    K_{3} &= -\frac{(-1)^{m}}{(2m)!}.
  \end{align*}
\end{lemma}
\begin{proof}
  The result follows directly from adding and subtracting
  \begin{equation*}
    (-1)^{m}\frac{|x|^{s - 1}}{(2m + 1 - s)(2m)!}
  \end{equation*}
  and collecting the terms.
\end{proof}
Computing \(K_{2}\) and \(K_{3}\) is straight forward, for \(K_{1}\)
there is a removable singularity to handle. The function
\begin{equation*}
  x^{2m}\frac{|x|^{s - (2m + 1)} - 1}{s - (2m + 1)}
\end{equation*}
also has a removable singularity, to compute an accurate enclosure we
use the following lemma.
\begin{lemma}
  \label{lemma:x_pow_t_m1_div_t}
  For \(x \not= 0\) the function
  \begin{equation*}
    \frac{|x|^{t} - 1}{t}
  \end{equation*}
  is non-decreasing in \(t\). In the limit \(t \to 0\) it is equal to
  \(\log|x|\).
\end{lemma}
\begin{proof}
  The derivative in \(t\) is given by
  \begin{equation}
    \label{eq:x_pow_t_m1_div_t_dt}
    \frac{1 + (t \log|x| - 1)|x|^{t}}{t^{2}}.
  \end{equation}
  The sign depends on the numerator, which we can write as
  \begin{equation*}
    1 + (t \log|x| - 1)e^{t \log|x|}.
  \end{equation*}
  Letting \(v = t \log |x|\) we can write this as
  \(1 + (v - 1)e^{v}\), which has the unique root \(v = 0\) and is
  positive for other values of \(v\). It follows that
  \eqref{eq:x_pow_t_m1_div_t_dt} is non-negative and hence the
  function non-decreasing.

  For the limit \(t \to 0\) we directly get
  \begin{equation*}
    \lim_{t \to 0} \frac{|x|^{t} - 1}{t} = \lim_{t \to 0} \frac{\log|x|\ |x|^{t}}{1} = \log|x|.
  \end{equation*}
\end{proof}
These two lemmas are only strictly required for interval \(\inter{s}\)
overlapping positive odd integers. However, it gives better enclosures
also when \(\inter{s}\) is close to such integers, even if it doesn't
overlap them. This is used to compute better enclosures in some cases.

The tails for the asymptotic expansions are bounded using the
following lemma, see also~\cite[Lemma 2.1]{enciso2018convexity}. We
omit the proof since it is very similar to that
in~\cite{enciso2018convexity}.
\begin{lemma}
  \label{lemma:clausen-tails}
  Let \(s \geq 0\), \(2M \geq s + 1\) and \(|x| < 2\pi\), we then have
  the following bounds for the tails in equations
  \eqref{eq:clausenc-asymptotic-x} and
  \eqref{eq:clausens-asymptotic-x}
  \begin{align*}
    \left|\sum_{m = M}^{\infty} (-1)^{m}\zeta(s - 2m)\frac{x^{2m}}{(2m)!}\right|
    &\leq 2(2\pi)^{1 + s - 2M}\left|\sin\left(\frac{\pi}{2}s\right)\right|\zeta(2M + 1 - s) \frac{x^{2M}}{4\pi^{2} - x^{2}},\\
    \left|\sum_{m = M}^{\infty} (-1)^{m}\zeta(s - 2m - 1)\frac{x^{2m + 1}}{(2m + 1)!}\right|
    &\leq 2(2\pi)^{s - 2M}\left|\cos\left(\frac{\pi}{2}s\right)\right|\zeta(2M + 2 - s) \frac{x^{2M + 1}}{4\pi^{2} - x^{2}}.
  \end{align*}
\end{lemma}

\subsection{Derivatives in \(s\)}
\label{sec:clausen-derivatives-s}
For \(C_{s}^{(\beta)}(x)\) and \(S_{s}^{(\beta)}(x)\) we use
\eqref{eq:clausenc-periodic-zeta} and
\eqref{eq:clausens-periodic-zeta} and differentiate directly in \(s\).
When \(s\) is not an integer this is handled directly using Taylor
arithmetic, for integers we use the approach in
Appendix~\ref{sec:removable-singularities} to handle the removable
singularities.

To get asymptotic expansions at \(x = 0\) we take the expansions
\eqref{eq:clausenc-asymptotic-x} and \eqref{eq:clausens-asymptotic-x}
and differentiate them with respect to \(s\). Giving us
\begin{align}
  \label{eq:clausenc-derivative-asymptotic-x}
  C_{s}^{(\beta)}(x) &= \frac{d}{ds^{\beta}}\left(\Gamma(1 - s)\sin\left(\frac{\pi}{2}s\right)|x|^{s - 1}\right)
             + \sum_{m = 0}^{\infty} (-1)^{m}\zeta^{(\beta)}(s - 2m)\frac{x^{2m}}{(2m)!};\\
  \label{eq:clausens-derivative-asymptotic-x}
  S_{s}^{(\beta)}(x) &= \frac{d}{ds^{\beta}}\left(\Gamma(1 - s)\cos\left(\frac{\pi}{2}s\right)\sign(x)|x|^{s - 1}\right)
             + \sum_{m = 0}^{\infty} (-1)^{m}\zeta^{(\beta)}(s - 2m - 1)\frac{x^{2m + 1}}{(2m + 1)!}.
\end{align}
These formulas work well when \(s\) is not a positive odd integer for
\(C_{s}^{(\beta)}(x)\) or a positive even integer for
\(S_{s}^{(\beta)}(x)\), and the derivatives can be computed using
Taylor expansions. We mostly make use of the functions
\(C_{2}^{(1)}(x)\) and \(C_{3}^{(1)}(x)\), in which case the
expansions can be computed explicitly using~\cite[Eq.~16]{Bailey2015},
for \(|x| < 2\pi\) we have
\begin{align*}
  C_{2}^{(1)}(x) =& \zeta^{(1)}(2) - \frac{\pi}{2}|x|\log|x| - (\gamma - 1)\frac{\pi}{2}|x|
                   + \sum_{m = 1}^{\infty}(-1)^{m}\zeta^{(1)}(2 - 2m)\frac{x^{2m}}{(2m)!}\\
  C_{3}^{(1)}(x) =& \zeta^{(1)}(3) - \frac{1}{4}x^{2}\log^2|x|
                    + \frac{3 - 2\gamma}{4}x^2\log|x|
                    - \frac{36\gamma - 12\gamma^2 - 24\gamma_{1} - 42 + \pi^{2}}{48}x^2\\
                 &+ \sum_{m = 2}^{\infty}(-1)^{m}\zeta^{(1)}(3 - 2m)\frac{x^{2m}}{(2m)!}.
\end{align*}
Where \(\gamma_{n}\) is the Stieltjes constant and
\(\gamma = \gamma_{0}\). To bound the tails we have the following
lemma from~\cite{dahne2022burgershilbert}.
\begin{lemma}
  \label{lemma:clausen-derivative-tails}
  Let \(\beta \geq 1\), \(s \geq 0\), \(2M \geq s + 1\) and
  \(|x| < 2\pi\), we then have the following bounds:
  \begin{multline*}
    \left|\sum_{m = M}^{\infty} (-1)^{m}\zeta^{(\beta)}(s - 2m)\frac{x^{2m}}{(2m)!}\right|\\
    \leq \sum_{j_{1} + j_{2} + j_{3} = \beta} \binom{\beta}{j_{1},j_{2},j_{3}}
    2\left(\log(2\pi) + \frac{\pi}{2}\right)^{j_{1}}(2\pi)^{s - 1}|\zeta^{(j_{3})}(1 + 2M - s)|
    \sum_{m = M}^{\infty} \left|p_{j_{2}}(1 + 2m - s)\left(\frac{x}{2\pi}\right)^{2m}\right|,
  \end{multline*}
  \begin{multline*}
    \left|\sum_{m = M}^{\infty} (-1)^{m}\zeta^{(\beta)}(s - 2m - 1)\frac{x^{2m + 1}}{(2m + 1)!}\right|\\
    \leq \sum_{j_{1} + j_{2} + j_{3} = \beta} \binom{\beta}{j_{1},j_{2},j_{3}}
    2\left(\log(2\pi) + \frac{\pi}{2}\right)^{j_{1}}(2\pi)^{s - 2}|\zeta^{(j_{3})}(2 + 2M - s)|
    \sum_{m = M}^{\infty} \left|p_{j_{2}}(2 + 2m - s)\left(\frac{x}{2\pi}\right)^{2m + 1}\right|.
  \end{multline*}
  Here \(p_{j_{2}}\) is given recursively by
  \begin{equation*}
    p_{k + 1}(s) = \psi^{(0)}(s)p_{k}(s) + p'_{k}(s),\quad p_{0} = 1,
  \end{equation*}
  where \(\psi^{(0)}\) is the polygamma function. It is given by a
  linear combination of terms of the form
  \begin{equation*}
    (\psi^{(0)}(s))^{q_{0}}(\psi^{(1)}(s))^{q_{1}}\cdots (\psi^{(j_{2} - 1)}(s))^{q_{j_{2} - 1}}.
  \end{equation*}
  We have the following bounds
  \begin{multline*}
    \sum_{m = M}^{\infty}\left|
      (\psi^{(0)}(1 + 2m - s))^{q_{0}}\cdots (\psi^{(j_{2} - 1)}(1 + 2m - s))^{q_{j_{2} - 1}}
      \left(\frac{x}{2\pi}\right)^{2m}
    \right|\\
    \leq |(\psi^{(1)}(1 + 2M - s))^{q_{1}}\cdots (\psi^{(j_{2} - 1)}(1 + 2M - s))^{q_{j_{2} - 1}}|
    \frac{1}{2^{q_{0} / 2}}(2\pi)^{-2M}
    \Phi\left(\frac{x^{2}}{4\pi^{2}}, -\frac{q_{0}}{2}, M + \frac{1}{2}\right)x^{2M}
  \end{multline*}
  and
  \begin{multline*}
    \sum_{m = M}^{\infty}\left|
      (\psi^{(0)}(2 + 2m - s))^{q_{0}}\cdots (\psi^{(j_{2} - 1)}(2 + 2m - s))^{q_{j_{2} - 1}}
      \left(\frac{x}{2\pi}\right)^{2m + 1}
    \right|\\
    \leq |(\psi^{(1)}(2 + 2M - s))^{q_{1}}\cdots (\psi^{(j_{2} - 1)}(2 + 2M - s))^{q_{j_{2} - 1}}|
    \frac{1}{2^{q_{0} / 2}}(2\pi)^{-2M - 1}
    \Phi\left(\frac{x^{2}}{4\pi^{2}}, -\frac{q_{0}}{2}, M + 1\right)x^{2M + 1}
  \end{multline*}
  where
  \begin{equation*}
    \Phi(z, s, a) = \sum_{m = 0}^{\infty}\frac{z^{m}}{(a + m)^{s}},
  \end{equation*}
  is the Lerch transcendent.
\end{lemma}

For \(\alpha \in I_{1}\) the expansions in \(x\) of \(u_{\alpha}\) and
\(\Hop{\alpha}[u_{\alpha}]\) are given in
Lemma~\ref{lemma:u0-asymptotic-I-1}. Bounding the remainder terms
for the sums
\begin{equation*}
  \sum_{m = 1}^{\infty}\frac{(-1)^{m}}{(2m)!}a_{\alpha,0}\left(
    \zeta(1 - \alpha - 2m) - \zeta(2 + (1 + \alpha)^{2} / 2 - 2m)
  \right)x^{2m}
\end{equation*}
and
\begin{equation*}
  \sum_{m = 2}^{\infty}\frac{(-1)^{m}}{(2m)!}a_{\alpha,0}\left(
    \zeta(1 - 2\alpha - 2m) - \zeta(2 - \alpha + (1 + \alpha)^{2} / 2 - 2m)
  \right)x^{2m}
\end{equation*}
requires a bound similar to that in
Lemma~\ref{lemma:clausen-derivative-tails} above. Factoring out
\(a_{\alpha,0}(1 + \alpha)\), which is bounded, we are left bounding
\begin{equation*}
  \sum_{m = 1}^{\infty}\frac{(-1)^{m}}{(2m)!}\frac{
    \zeta(1 - \alpha - 2m) - \zeta(2 + (1 + \alpha)^{2} / 2 - 2m)
  }{1 + \alpha}x^{2m}
\end{equation*}
and
\begin{equation*}
  \sum_{m = 2}^{\infty}\frac{(-1)^{m}}{(2m)!}\frac{
    \zeta(1 - 2\alpha - 2m) - \zeta(2 - \alpha + (1 + \alpha)^{2} / 2 - 2m)
  }{1 + \alpha}x^{2m}.
\end{equation*}
The following lemma allows us to write them in a form similar to that
in Lemma~\ref{lemma:clausen-derivative-tails}.
\begin{lemma}
  \label{lemma:bound-tail-I-1}
  For \(-1 < \alpha < 0\) and \(M_{1} \geq 1\) we have
  \begin{multline*}
    \left|\sum_{m = M_{1}}^{\infty}\frac{(-1)^{m}}{(2m)!}\frac{
      \zeta(1 - \alpha - 2m) - \zeta(2 + (1 + \alpha)^{2} / 2 - 2m)
    }{1 + \alpha}x^{2m}\right|\\
  \leq \left|\sum_{m = M_{1}}^{\infty}\frac{(-1)^{m}}{(2m)!}\zeta'(s_{1,m} - 2m)x^{2m}\right|
    + \frac{1 + \alpha}{2}\left|\sum_{m = M_{1}}^{\infty}\frac{(-1)^{m}}{(2m)!}\zeta'(s_{2,m} - 2m)x^{2m}\right|
  \end{multline*}
  with \(s_{1,m} \in [1 - \alpha, 2]\) and
  \(s_{2,m} \in [2, 2 + (1 + \alpha)^{2} / 2]\). Similarly, for
  \(M_{2} \geq 2\),
  \begin{multline*}
    \left|\sum_{m = M_{2}}^{\infty}\frac{(-1)^{m}}{(2m)!}\frac{
      \zeta(1 - 2\alpha - 2m) - \zeta(2 - \alpha + (1 + \alpha)^{2} / 2 - 2m)
    }{1 + \alpha}x^{2m}\right|\\
  \leq 2\left|\sum_{m = M_{2}}^{\infty}\frac{(-1)^{m}}{(2m)!}\zeta'(s_{3,m} - 2m)x^{2m}\right|
  + \left(1 - \frac{1 + \alpha}{2}\right)\left|\sum_{m = M_{2}}^{\infty}\frac{(-1)^{m}}{(2m)!}\zeta'(s_{4,m} - 2m)x^{2m}\right|
  \end{multline*}
  with \(s_{3,m} \in [1 - 2\alpha, 3]\) and
  \(s_{4,m} \in [2 - \alpha + (1 + \alpha)^{2} / 2, 3]\).
\end{lemma}
\begin{proof}
  The first inequality follows directly from that
  \begin{equation*}
    \frac{\zeta(1 - \alpha - 2m) - \zeta(2 + (1 + \alpha)^{2} / 2 - 2m)}{1 + \alpha}
    = \frac{\zeta(1 - \alpha - 2m) - \zeta(2 - 2m)}{1 + \alpha}
    + \frac{\zeta(2 - 2m) - \zeta(2 + (1 + \alpha)^{2} / 2 - 2m)}{1 + \alpha}
  \end{equation*}
  together with
  \begin{equation*}
    \frac{\zeta(1 - \alpha - 2m) - \zeta(2 - 2m)}{1 + \alpha}
    = \zeta'(s_{1,m} - 2m)
  \end{equation*}
  for some \(s_{1,m} \in [1 - \alpha, 2]\) and
  \begin{equation*}
    \frac{\zeta(2 - 2m) - \zeta(2 + (1 + \alpha)^{2} / 2 - 2m)}{1 + \alpha}
    = \frac{1 + \alpha}{2}\zeta'(s_{2,m} - 2m)
  \end{equation*}
  for some \(s_{2,m} \in [2, 2 + (1 + \alpha)^{2} / 2]\). The second
  one is similar.
\end{proof}

Lemma~\ref{lemma:clausen-derivative-tails} cannot be directly applied
to the result of the above lemma since the argument of the zeta
function depends on \(m\). This can however be fixed, as the following
lemma shows.

\begin{lemma}
  \label{lemma:clausen-derivative-tails-alt}
  Let \(\beta \geq 1\), \(\inter{s} = [\lo{s}, \hi{s}]\) with
  \(\lo{s} \geq 0\), \(2M \geq \hi{s} + 1\), \(s_{m} \in \inter{s}\)
  for \(m \geq M\) and \(|x| < 2\pi\), we then have the following
  bounds:
  \begin{multline*}
    \left|\sum_{m = M}^{\infty} (-1)^{m}\zeta^{(\beta)}(s_{m} - 2m)\frac{x^{2m}}{(2m)!}\right|\\
    \leq \sum_{j_{1} + j_{2} + j_{3} = \beta} \binom{\beta}{j_{1},j_{2},j_{3}}
    2\left(\log(2\pi) + \frac{\pi}{2}\right)^{j_{1}}(2\pi)^{\hi{s} - 1}|\zeta^{(j_{3})}(1 + 2M - \hi{s})|
    \sum_{m = M}^{\infty} \left|p_{j_{2}}(1 + 2m - s_{m})\left(\frac{x}{2\pi}\right)^{2m}\right|,
  \end{multline*}
  \begin{multline*}
    \left|\sum_{m = M}^{\infty} (-1)^{m}\zeta^{(\beta)}(s_{m} - 2m - 1)\frac{x^{2m + 1}}{(2m + 1)!}\right|\\
    \leq \sum_{j_{1} + j_{2} + j_{3} = \beta} \binom{\beta}{j_{1},j_{2},j_{3}}
    2\left(\log(2\pi) + \frac{\pi}{2}\right)^{j_{1}}(2\pi)^{s - 2}|\zeta^{(j_{3})}(2 + 2M - s)|
    \sum_{m = M}^{\infty} \left|p_{j_{2}}(2 + 2m - s_{m})\left(\frac{x}{2\pi}\right)^{2m + 1}\right|.
  \end{multline*}
  Here \(p_{j_{2}}\) is the same as in
  Lemma~\ref{lemma:clausen-derivative-tails} and given recursively by
  \begin{equation*}
    p_{k + 1}(s) = \psi^{(0)}(s)p_{k}(s) + p'_{k}(s),\quad p_{0} = 1,
  \end{equation*}
  where \(\psi^{(0)}\) is the polygamma function. It is given by a
  linear combination of terms of the form
  \begin{equation*}
    (\psi^{(0)}(s))^{q_{0}}(\psi^{(1)}(s))^{q_{1}}\cdots (\psi^{(j_{2} - 1)}(s))^{q_{j_{2} - 1}}.
  \end{equation*}
  We have the following bounds
  \begin{multline*}
    \sum_{m = M}^{\infty}\left|
      (\psi^{(0)}(1 + 2m - s_{m}))^{q_{0}}\cdots (\psi^{(j_{2} - 1)}(1 + 2m - s_{m}))^{q_{j_{2} - 1}}
      \left(\frac{x}{2\pi}\right)^{2m}
    \right|\\
    \leq |(\psi^{(1)}(1 + 2M - \hi{s}))^{q_{1}}\cdots (\psi^{(j_{2} - 1)}(1 + 2M - \hi{s}))^{q_{j_{2} - 1}}|
    \frac{1}{2^{q_{0} / 2}}(2\pi)^{-2M}
    \Phi\left(\frac{x^{2}}{4\pi^{2}}, -\frac{q_{0}}{2}, M + \frac{1}{2}\right)x^{2M}
  \end{multline*}
  and
  \begin{multline*}
    \sum_{m = M}^{\infty}\left|
      (\psi^{(0)}(2 + 2m - s_{m}))^{q_{0}}\cdots (\psi^{(j_{2} - 1)}(2 + 2m - s_{m}))^{q_{j_{2} - 1}}
      \left(\frac{x}{2\pi}\right)^{2m + 1}
    \right|\\
    \leq |(\psi^{(1)}(2 + 2M - \hi{s}))^{q_{1}}\cdots (\psi^{(j_{2} - 1)}(2 + 2M - \hi{s}))^{q_{j_{2} - 1}}|
    \frac{1}{2^{q_{0} / 2}}(2\pi)^{-2M - 1}
    \Phi\left(\frac{x^{2}}{4\pi^{2}}, -\frac{q_{0}}{2}, M + 1\right)x^{2M}
  \end{multline*}
\end{lemma}
\begin{proof}
  The proof is the same as for
  Lemma~\ref{lemma:clausen-derivative-tails}, and we therefore omit
  the details. The only difference is that care has to be taken so
  that the termwise bounds holds for all \(s_{m}\).
\end{proof}

\section{Rigorous integration with singularities}
\label{sec:rigorous-integration}
In this section we discuss how to compute enclosures of the integrals
\(U_{\alpha,1,1}(x)\), \(U_{\alpha,1,2}(x)\) and \(U_{\alpha,2}(x)\)
in the non-asymptotic case. We focus on the case when
\(w_{\alpha}(x)\) is not given by \(|x|\), since otherwise
Lemma~\ref{lemma:U-primitive} allows us to directly compute the
integral. In particular this means that we don't have to consider the
case when \(\alpha \to 0\), though we have to handle
\(\alpha \to -1\).

Recall that
\begin{align*}
  U_{\alpha,1,1}(x) &= -x\int_{0}^{r_{\alpha,x}} \hat{I}_{\alpha}(x, t)w_{\alpha}(tx)\ dt,\\
  U_{\alpha,1,2}(x) &= x\int_{r_{\alpha,x}}^{1} \hat{I}_{\alpha}(x, t)w_{\alpha}(tx)\ dt,\\
  U_{\alpha,2}(x) &= \int_{x}^{\pi}I_{\alpha}(x, y)w_{\alpha}(y)\ dy.
\end{align*}
The integrand for \(U_{1,2}\) has a (integrable) singularity at
\(t = 1\) and the integrand for \(U_{2}\) has one at \(y = x\). As a
first step we split these off to handle them separately.

Let
\begin{align*}
  U_{\alpha,1,2}(x) =& x\int_{r_{\alpha,x}}^{1 - \delta_{U,1}} \hat{I}_{\alpha}(x, t)w_{\alpha}(tx)\ dt
                       + x\int_{1 - \delta_{U,1}}^{1} \hat{I}_{\alpha}(x, t)w_{\alpha}(tx)\ dt
                       = U_{\alpha,1,2,1}(x) + U_{\alpha,1,2,2}(x),\\
     U_{\alpha,2}(x) =& \int_{x}^{x + \delta_{U,2}}I_{\alpha}(x, y)w_{\alpha}(y)\ dy
                        + \int_{x + \delta_{U,2}}^{\pi}I_{\alpha}(x, y)w_{\alpha}(y)\ dy
                        = U_{\alpha,2,1}(x) + U_{\alpha,2,2}(x).
\end{align*}
These integrals are handled by noticing that \(w_{\alpha}\) is bounded
on the interval of integration. This allows us to compute an enclosure
as
\begin{align*}
  U_{\alpha,1,2,2}(x) \in&\ xw_{\alpha}(x[1 - \delta_{U,1}, 1]) \int_{1 - \delta_{U,1}}^{1} \hat{I}_{\alpha}(x, t)\ dt,\\
  U_{\alpha,2,1}(x) \in&\ w_{\alpha}([x, x + \delta_{U,2}]) \int_{x}^{x + \delta_{U,2}}I_{\alpha}(x, y)\ dy.
\end{align*}
The integrals can be computed explicitly, using that
\begin{align*}
  \int \hat{I}_{\alpha}(x, t)\ dt &= \frac{1}{x}\left(-S_{1 - \alpha}(x(1 - t)) + S_{1 - \alpha}(x(1 + t)) - 2S_{1 - \alpha}(xt)\right),\\
  \int I_{\alpha}(x, y)\ dy &= -S_{1 - \alpha}(x - y) + S_{1 - \alpha}(x + y) - 2S_{1 - \alpha}(y).
\end{align*}
From which we get
\begin{equation*}
  \int_{1 - \delta_{U,1}}^{1} \hat{I}_{\alpha}(x, t)\ dt = \frac{1}{x}\Big(
  S_{1 - \alpha}(2x) - 2S_{1 - \alpha}(x)
  + S_{1 - \alpha}(x\delta_{U,1}) - S_{1 - \alpha}(x(2 - \delta_{U,1})) + 2S_{1 - \alpha}(x(1 - \delta_{U,1}))
  \Big)
\end{equation*}
and
\begin{equation*}
  \int_{x}^{x + \delta_{U,2}}I_{\alpha}(x, y)\ dy = -S_{1 - \alpha}(-\delta_{U,2}) + S_{1 - \alpha}(2x + \delta_{U,2}) - 2S_{1 - \alpha}(x + \delta_{U,2})
  - S_{1 - \alpha}(2x) + 2S_{1 - \alpha}(x).
\end{equation*}
For \(\alpha \in I_{1}\) we use a slightly modified approach which
gives better bounds. Using that
\(w_{\alpha}(x) = |x|^{(1 - \alpha) / 2}\log(2e + 1/|x|)\) in this
case we have
\begin{multline*}
  U_{\alpha,1,2,2}(x)
  = x^{1 + (1 - \alpha) / 2}\int_{1 - \delta_{U,1}}^{1} \hat{I}_{\alpha}(x, t)t^{(1 - \alpha) / 2}\log(2e + 1/(xt))\ dt\\
  \in x^{1 + (1 - \alpha) / 2}[1 - \delta_{U,1}, 1]^{-(1 + \alpha) / 2}\log(2e + 1/(x[1 - \delta_{U,1}, 1]))\int_{1 - \delta_{U,1}}^{1} \hat{I}_{\alpha}(x, t)t\ dt
\end{multline*}
and can use the same approach as in Lemma~\ref{lemma:U-primitive} for
computing the integral.

What remains is computing \(U_{\alpha,1,1}\), \(U_{\alpha,1,2,1}\) and
\(U_{\alpha,2,2}\). In this case the integrands are bounded everywhere
on the intervals of integration. To enclose the integrals we make use
of two different rigorous numerical integrators, depending on the
value of \(\alpha\). For \(\alpha \in I_{2}\) we use the rigorous
numerical integrator implemented by Arb~\cite{Johansson2018numerical}.
For functions that are analytic on the interval the integrator uses
Gaussian quadratures with error bounds computed through complex
magnitudes, we therefore need to evaluate the integrands on complex
intervals. When the function is not analytic it falls back to naive
enclosures using interval arithmetic. For \(\alpha \in I_{1}\) it is
harder to enclose the integrand for complex values, we therefore use a
quadrature rule given by
\begin{equation*}
  \int_{a}^{b}f(y)\ dy \in \frac{b - a}{2}\left(f(y_{1}) + f(y_{2})\right) + \frac{1}{4320}(b - a)^{5}f^{(4)}([a, b]),
\end{equation*}
where
\begin{equation*}
  y_{1} = \frac{b + a}{2} + \frac{b - a}{2\sqrt{3}},\quad
  y_{2} = \frac{b + a}{2} - \frac{b - a}{2\sqrt{3}},
\end{equation*}
which doesn't require us to evaluate on complex balls, see
e.g.~\cite{Kramer1996}. On intervals where the function is not two
times differentiable it falls back to naive enclosures using interval
arithmetic.

For \(U_{\alpha,1,1}\) the integrand is bounded but not analytic at
\(t = 0\). Slightly more work is required to compute an enclosure of
the integrand for intervals containing \(t = 0\). The only non-trivial
part of the integrand is the term \(C_{-\alpha}(xt) w_{\alpha}(xt)\)
where the first factor diverges at \(t = 0\) and the second factor
tends to zero. For \(\alpha \in I_{2}\) this is handled by expanding
the Clausen function at zero, which allows us to simplify the terms
and bound them individually. For \(\alpha \in I_{1}\) this doesn't
work directly because the individual terms are not bounded as
\(\alpha \to -1\), instead we use the following lemma.
\begin{lemma}
  For \(\alpha \in \left(-1, -\frac{1}{2}\right)\) and
  \(x \in (0, \pi)\) the function
  \begin{equation*}
    C_{-\alpha}(xt)|xt|^{(1 - \alpha) / 2}\log(2e + 1 / |xt|)
  \end{equation*}
  is increasing in \(t\) on the interval \((0, t_{0})\) for
  \begin{equation*}
    t_{0} = \left(
      -\frac{2\Gamma(1 + \alpha)\sin\left(-\frac{\pi}{2}\alpha\right)x^{-\alpha - 1}\left(-\alpha - \frac{1}{2}\right)}{\zeta(-\alpha)}
    \right)^{\frac{1}{\alpha + 1}}.
  \end{equation*}
\end{lemma}
\begin{proof}
  For \(t \geq 0\) can remove the absolute values and rewrite the
  function as
  \begin{equation*}
    C_{-\alpha}(xt)(xt)^{1 / 2} \cdot (xt)^{-\alpha / 2}\log(2e + 1 / (xt)).
  \end{equation*}
  It is enough to prove that the two factors both are increasing in
  \(t\).

  For the first factor, \(C_{-\alpha}(xt)(xt)^{1 / 2}\), we expand the
  Clausen function, giving us
  \begin{equation*}
    C_{-\alpha}(xt)(xt)^{1 / 2} =
    \Gamma(1 + \alpha)\sin\left(-\frac{\pi}{2}\alpha\right)(xt)^{-\alpha - 1 / 2}
    + \sum_{m = 0}^{\infty} (-1)^{m}\zeta(-\alpha - 2m)\frac{(xt)^{2m + 1 / 2}}{(2m)!}.
  \end{equation*}
  All terms in the sum with \(m \geq 1\) are positive, due to the
  location of the zeros of the zeta function on the negative real
  axis, and hence increasing in \(t\). We are left with proving that
  \begin{equation*}
    \Gamma(1 + \alpha)\sin\left(-\frac{\pi}{2}\alpha\right)(xt)^{-\alpha - 1 / 2}
    + \zeta(-\alpha)(xt)^{1 / 2}
  \end{equation*}
  is increasing in \(t\). Differentiating with respect to \(t\) gives
  us
  \begin{equation*}
    \Gamma(1 + \alpha)\sin\left(-\frac{\pi}{2}\alpha\right)x^{-\alpha - 1 / 2}\left(-\alpha - \frac{1}{2}\right)t^{-\alpha - 3 / 2}
    + \frac{1}{2}\zeta(-\alpha)x^{1 / 2}t^{-1 / 2}.
  \end{equation*}
  Since we are only interested in the sign we can multiply by
  \(t^{\alpha + 3 / 2}\), giving us
    \begin{equation*}
    \Gamma(1 + \alpha)\sin\left(-\frac{\pi}{2}\alpha\right)x^{-\alpha - 1 / 2}\left(-\alpha - \frac{1}{2}\right)
    + \frac{1}{2}\zeta(-\alpha)x^{1 / 2}t^{\alpha + 1}.
  \end{equation*}
  The positiveness of the first term together with \(1 + \alpha > 0\)
  means that it is positive at \(t = 0\). The unique root for
  \(t > 0\) is given by
  \begin{equation*}
    \left(
      -\frac{2\Gamma(1 + \alpha)\sin\left(-\frac{\pi}{2}\alpha\right)x^{-\alpha - 1}\left(-\alpha - \frac{1}{2}\right)}{\zeta(-\alpha)}
    \right)^{\frac{1}{\alpha + 1}},
  \end{equation*}
  which is exactly the value for \(t_{0}\).

  For the second factor, \((xt)^{-\alpha / 2}\log(2e + 1 / (xt))\),
  differentiation with respect to \(t\) gives us
  \begin{equation*}
    -xt^{-\frac{\alpha}{2} - 1}\frac{1 + \alpha(1 + 2ext)\log(2e + 1/(xt))}{2(1 + 2ext)}.
  \end{equation*}
  Since \(xt^{-\frac{\alpha}{2} - 1}\) and \(1 + 2ext\) are positive
  the sign is the same as that for
  \begin{equation*}
    -(1 + \alpha(1 + 2ext)\log(2e + 1/(xt))).
  \end{equation*}
  For \(x, t > 0\) a lower bound is given by \(-1 - 2\alpha\), which
  is positive for \(\alpha < -\frac{1}{2}\) and hence for
  \(\alpha \in I_{1}\). The derivative is hence positive for all
  \(t > 0\) and the factor is always increasing.
\end{proof}

For the endpoint \(r_{\alpha,x}\) of \(U_{\alpha,1,1}\) and
\(U_{\alpha,1,2,1}\) we don't have an exact value but only an
enclosure. The width of this enclosure depends on how large the
intervals for \(\alpha\) and \(x\) are. The integrators we use don't
handle wide endpoints very well, it is therefore beneficial to
slightly modify the endpoints. If we have the enclosure
\(r_{\alpha,x} \in [\lo{r}_{\alpha,x}, \hi{r}_{\alpha,x}]\) then for
\(U_{\alpha,1,1}\) we only integrate up to \(\lo{r}_{\alpha,x}\) and
for \(U_{\alpha,1,2,1}\) we integrate from \(\hi{r}_{\alpha,x}\). The
part we have missed is given by
\begin{equation*}
  \int_{\lo{r}_{\alpha,x}}^{\hi{r}_{\alpha,x}} |\hat{I}_{\alpha}(x, t)|w_{\alpha}(xt)\ dt.
\end{equation*}
Since the interval centered at the root of \(\hat{I}_{\alpha}(x, t)\)
the integrand will be very small. It is therefore enough to just
compute a naive enclosure given by the diameter of the interval times
an enclosure of the integrand on the interval.

\section{Details for evaluating \(F_{\alpha}(x)\) with \(\alpha \in I_{1}\) and \(x\) near zero}
\label{sec:evaluation-F-I-1-asymptotic}
We here explain how to bound \eqref{eq:evaluation-F-I-1-factor}, given
by
\begin{equation*}
  \frac{\Hop{\alpha}[u_{\alpha}](x) + \frac{1}{2}u_{\alpha}(x)^{2}}{\Gamma(1 + \alpha)\log(1 / x)(1 - x^{1 + \alpha + (1 + \alpha)^{2} / 2})x^{1 - \alpha}}
\end{equation*}
for \(\alpha \in I_{1}\) and \(x\) close to zero. This is used when
bounding \(F_{\alpha}(x)\) for \(x \in [0, \epsilon]\). To simplify
the notation we in this section let
\(p_{\alpha} = 1 + \alpha + (1 + \alpha)^{2} / 2\), note that
\(p_{\hat{\alpha}}\) still denotes a numerical value, we also assume
that \(0 < x < 1\).

From Lemma~\ref{lemma:u0-asymptotic-I-1} we can get the expansion of
\(\Hop{\alpha}[u_{\alpha}](x) + \frac{1}{2}u_{\alpha}(x)^{2}\) at
\(x = 0\). To handle the cancellations between some of the terms in
the expansion we extract them and handle them separately.

To begin with we take out the leading terms in the expansions of
\(u_{\alpha}\) and \(\Hop{\alpha}[u_{\alpha}]\), namely
\begin{align*}
  P &= a_{\alpha,0}\left(
      \Gamma(\alpha)\cos\left(\frac{\pi}{2}\alpha\right)
      - \Gamma(-1 - (1 + \alpha)^{2} / 2)\cos\left(\frac{\pi}{2}(1 + (1 + \alpha)^{2} / 2)\right)x^{p_{\alpha}}
      \right)x^{-\alpha},\\
  Q &= -a_{\alpha,0}\Big(
      \Gamma(2\alpha)\cos\left(\pi\alpha\right)\\
    &\quad\quad - \Gamma(-1 + \alpha - (1 + \alpha)^{2} / 2)\cos\left(\frac{\pi}{2}(1 - \alpha + (1 + \alpha)^{2} / 2)\right)x^{p_{\alpha}}\\
    &\quad\quad + \frac{1}{2}(\zeta(-1 - 2\alpha) - \zeta(-\alpha + (1 + \alpha)^{2} / 2))x^{2(1 + \alpha)}
      \Big)x^{-2\alpha},
\end{align*}
and treat them together. For small values of \(j\) the terms
\(-a_{j}\tilde{C}_{1 \alpha - \hat{\alpha} + jp_{\hat{\alpha}}}\) in
\(\Hop{\alpha}[u_{\alpha}]\) have large cancellations between the
\(x^{-\alpha - \hat{\alpha} + jp_{\hat{\alpha}}}\) and \(x^{2}\) terms
in their expansion, we therefore also treat these separately. More
precisely we consider the terms
\begin{equation*}
  \sum_{j = 1}^{M}-\hat{A}_{\alpha,j}^{0}x^{-\alpha - \hat{\alpha} + jp_{\hat{\alpha}}} + \frac{a_{\hat{\alpha},j}\zeta(-1 - \alpha - \hat{\alpha} + jp_{\hat{\alpha}})}{2}x^{2}
\end{equation*}
Where \(M \leq N_{0}\) is some fixed limit for which terms we treat
separately.

\subsection{First part}
For the first part we are interested in the term
\begin{equation*}
  \frac{Q + P^{2} / 2}{\Gamma(1 + \alpha)\log(1 / x)(1 - x^{p_{\alpha}})x^{1 - \alpha}}
\end{equation*}
with \(P\) and \(Q\) as above. If we let
\(c(\alpha) = \Gamma(\alpha)\cos\left(\frac{\pi}{2}\alpha\right)\) we
can write \(P\) and \(Q\) as
\begin{align*}
  P &= a_{\alpha,0}(c(\alpha) - c(\alpha - p_{\alpha})x^{p_{\alpha}})x^{-\alpha};\\
  Q &= -a_{\alpha,0}(c(2\alpha) - c(2\alpha - p_{\alpha})x^{p_{\alpha}})x^{-2\alpha} + \frac{a_{\alpha,0}}{2}(\zeta(-1 - 2\alpha) - \zeta(-1 - 2\alpha + p_{\alpha}))x^{2}.
\end{align*}
This gives us
\begin{multline*}
  Q + P^{2} / 2 = a_{\alpha,0}\Big(
  (a_{\alpha,0}c(\alpha)^{2} / 2 - c(2\alpha))x^{-2\alpha}
  + (c(2\alpha - p_{\alpha}) - a_{\alpha,0}c(\alpha)c(\alpha - p_{\alpha}))x^{-2\alpha + p_{\alpha}}\\
  + \frac{1}{2}(\zeta(-1 - 2\alpha) - \zeta(-1 - 2\alpha + p_{\alpha}))x^{2}
  + \frac{a_{\alpha,0}}{2}c(\alpha - p_{\alpha})^{2}x^{-2\alpha + 2p_{\alpha}}
  \Big),
\end{multline*}
where we can note that \(2 < -2\alpha + 2p_{\alpha}\). The value of
\(a_{\alpha,0}\) is taken such that the first term is identically equal to
zero, we are thus left with
\begin{multline*}
  Q + P^{2} / 2 = a_{\alpha,0}\Big(
  (c(2\alpha - p_{\alpha}) - a_{\alpha,0}c(\alpha)c(\alpha - p_{\alpha}))x^{-2\alpha + p_{\alpha}}\\
  + \frac{1}{2}(\zeta(-1 - 2\alpha) - \zeta(-1 - 2\alpha + p_{\alpha}))x^{2}
  + \frac{a_{\alpha,0}}{2}c(\alpha - p_{\alpha})^{2}x^{-2\alpha + 2p_{\alpha}}
  \Big).
\end{multline*}
Now, cancelling the \(x^{1 - \alpha}\) and reordering the terms a bit
gives us
\begin{multline*}
  \frac{Q + P^{2} / 2}{\Gamma(1 + \alpha)\log(1 / x)(1 - x^{p_{\alpha}})x^{1 - \alpha}}
  = \frac{a_{\alpha,0}}{\Gamma(1 + \alpha)}\Big(
  \frac{
    (c(2\alpha - p_{\alpha}) - a_{\alpha,0}c(\alpha)c(\alpha - p_{\alpha}))x^{-1 - \alpha + p_{\alpha}}
  }{
    \log(1 / x)(1 - x^{p_{\alpha}})
  }\\
  + \frac{
    \frac{1}{2}(\zeta(-1 - 2\alpha) - \zeta(-1 - 2\alpha + p_{\alpha}))x^{1 + \alpha} + \frac{a_{\alpha,0}}{2}c(\alpha - p_{\alpha})^{2}x^{-1 - \alpha + 2p_{\alpha}}
  }{
    \log(1 / x)(1 - x^{p_{\alpha}})
  }
  \Big).
\end{multline*}
The factor \(\frac{a_{\alpha,0}}{\Gamma(1 + \alpha)}\) can be handled the
same way as in Lemma~\ref{lemma:asymptotic-inv-u0-I-1}. We consider
the two terms separately.

The first term we split as
\begin{equation*}
  \frac{x^{-1 - \alpha + p_{\alpha}}}{\log(1 / x)}
  \frac{1 + \alpha}{1 - x^{p_{\alpha}}}
  \frac{c(2\alpha - p_{\alpha}) - a_{\alpha,0}c(\alpha)c(\alpha - p_{\alpha})}{1 + \alpha}
\end{equation*}
For the first factor we have
\(-1 - \alpha + p_{\alpha} = (1 + \alpha)^{2} / 2 > 0\), so the factor
it is zero for \(x = 0\) and increasing in \(x\), allowing us to
compute an enclosure. The second factor is also increasing in \(x\),
for \(x = 0\) it is \(1 + \alpha\) and for non-zero \(x\) we can
handle the removable singularity in \(\alpha\). For the third factor
we note that \(a_{\alpha,0} = 2c(2\alpha) / c(\alpha)^{2}\), giving us
\begin{equation*}
  \frac{c(2\alpha - p_{\alpha}) - a_{\alpha,0}c(\alpha)c(\alpha - p_{\alpha})}{1 + \alpha}
  = \frac{c(2\alpha - p_{\alpha}) - 2c(2\alpha)c(\alpha - p_{\alpha})/c(\alpha)}{1 + \alpha}.
\end{equation*}
This can also be computed by handling the removable singularities.

The second term we split in a similar way
\begin{multline*}
  \frac{
    \frac{1}{2}(\zeta(-1 - 2\alpha) - \zeta(-1 - 2\alpha + p_{\alpha}))x^{1 + \alpha} + \frac{a_{\alpha,0}}{2}c(\alpha - p_{\alpha})^{2}x^{-1 - \alpha + 2p_{\alpha}}
  }{
    \log(1 / x)(1 - x^{p_{\alpha}})
  }\\
  = \frac{x^{1 + \alpha}}{2}\frac{1 + \alpha}{1 - x^{p_{\alpha}}}
  \frac{\frac{1}{2}(\zeta(-1 - 2\alpha) - \zeta(-1 - 2\alpha + p_{\alpha})) + \frac{a_{\alpha,0}}{2}c(\alpha - p_{\alpha})^{2}x^{(1 + \alpha)^{2}}}{(1 + \alpha)\log(1/x)},
\end{multline*}
where in the last exponent we have used
\(-2 - \alpha + 2p_{\alpha} = (1 + \alpha)^{2}\). The first factor is
easily handled, and the second factor is the same as in the previous
term. For the third factor we let
\begin{align*}
  v_{\alpha} &= \frac{1}{2}(\zeta(-1 - 2\alpha) - \zeta(-1 - 2\alpha + p_{\alpha})),\\
  w_{\alpha} &= \frac{a_{\alpha,0}}{2}c(\alpha - p_{\alpha})^{2}.
\end{align*}
Allowing us to write it as
\begin{equation*}
  \frac{v_{\alpha} + w_{\alpha}x^{(1 + \alpha)^{2}}}{(1 + \alpha)\log(1/x)}
  = \frac{v_{\alpha} - w_{\alpha}}{1 + \alpha}\frac{1}{\log(1/x)}
  + (1 + \alpha)w_{\alpha}\frac{1 - x^{(1 + \alpha)^{2}}}{(1 + \alpha)^{2}\log(1/x)}
\end{equation*}
The first factor in the first term has a removable singularity we can
handle, and its second factor is easily enclosed. For the second term
we can enclose \((1 + \alpha)w_{\alpha}\) which has a removable
singularity at \(\alpha = -1\). For the remaining factor we let
\(t = (1 + \alpha)^{2}\log(x)\) and rewrite it as
\begin{equation}
  \label{eq:I1-defect-exp-t}
  \frac{1 - x^{(1 + \alpha)^{2}}}{(1 + \alpha)^{2}\log(1/x)} =
  \frac{x^{(1 + \alpha)^{2}} - 1}{(1 + \alpha)^{2}\log(x)}=
  \frac{e^{t} - 1}{t}.
\end{equation}
This function is increasing in \(t\), so it is enough to get an
enclosure of \(t\) and evaluate at the endpoints of \(t\).

\subsection{Second part}
For the second part we are interested in the term
\begin{equation*}
  \sum_{j = 1}^{M}\frac{-\hat{A}_{\alpha,j}^{0}x^{-\alpha - \hat{\alpha} + jp_{\hat{\alpha}}} + \frac{a_{\hat{\alpha},j}\zeta(-1 - \alpha - \hat{\alpha} + jp_{\hat{\alpha}})}{2}x^{2}}{\Gamma(1 + \alpha)\log(1 / x)(1 - x^{p_{\alpha}})x^{1 - \alpha}}
  = \frac{1}{\Gamma(1 + \alpha)(1 - x^{p_{\alpha}})}
  \sum_{j = 1}^{M}\frac{-\hat{A}_{\alpha,j}^{0}x^{-\alpha - \hat{\alpha} + jp_{\hat{\alpha}}} + \frac{a_{\hat{\alpha},j}\zeta(-1 - \alpha - \hat{\alpha} + jp_{\hat{\alpha}})}{2}x^{2}}{\log(1 / x)x^{1 - \alpha}}.
\end{equation*}
We begin by noting that
\begin{equation*}
  \frac{1}{\Gamma(1 + \alpha)(1 - x^{p_{\alpha}})}
\end{equation*}
is increasing in \(x\). For \(x = 0\) it is lower bounded by \(0\) and
increasing in \(\alpha\). For \(x > 0\) it has a removable singularity
at \(\alpha = -1\) which we can handle. What remains is to handle the
sum, this we do term by term.

Let \(1 \leq j \leq M\) and
\(r_{j} = -1 - \hat{\alpha} + jp_{\hat{\alpha}} = (j - 1)(1 +
\hat{\alpha}) + (1 + \hat{\alpha})^{2} / 2\), note that \(r_{j} > 0\),
this gives us
\begin{equation*}
  \frac{-\hat{A}_{\alpha,j}^{0}x^{1 - \alpha + r_{j}} + \frac{a_{j}\zeta(-\alpha + r_{j})}{2}x^{2}}{\log(1 / x)x^{1 - \alpha}}
  =   \frac{-\hat{A}_{\alpha,j}^{0}x^{r_{j}} + \frac{a_{\hat{\alpha},j}\zeta(-\alpha + r_{j})}{2}x^{1 + \alpha}}{\log(1 / x)}.
\end{equation*}
Note that which exponent is the largest depends on the precise choice
of both \(j\) and \(\alpha\) and is in general not the same for all
\(\alpha \in I_{1}\). Using that
\begin{equation*}
  \hat{A}_{\alpha,j}^{0} = \Gamma(\alpha + \hat{\alpha} - jp_{\hat{\alpha}})\cos\left((\alpha + \hat{\alpha} - jp_{\hat{\alpha}})\frac{\pi}{2}\right)a_{\hat{\alpha},j}
  = \Gamma(\alpha - 1 - r_{j})\cos\left((\alpha - 1 - r_{j})\frac{\pi}{2}\right)a_{\hat{\alpha},j}
\end{equation*}
we get
\begin{equation*}
  a_{\hat{\alpha},j}\frac{-\Gamma(\alpha - 1 - r_{j})\cos\left((\alpha - 1 - r_{j})\frac{\pi}{2}\right)x^{r_{j}} + \frac{\zeta(-\alpha + r_{j})}{2}x^{1 + \alpha}}{\log(1 / x)}.
\end{equation*}
The \(a_{\hat{\alpha},j}\) is a constant, we therefore focus on the
rest. Adding and subtracting
\(\frac{\zeta(-\alpha + r_{j})}{2}x^{r_{j}}\) we can write this as
\begin{equation*}
  \left(-\Gamma(\alpha - 1 - r_{j})\cos\left((\alpha - 1 - r_{j})\frac{\pi}{2}\right) + \frac{\zeta(-\alpha + r_{j})}{2}\right)\frac{x^{r_{j}}}{\log(1 / x)}
  + \frac{\zeta(-\alpha + r_{j})}{2}\frac{x^{1 + \alpha}-x^{r_{j}}}{\log(1 / x)}.
\end{equation*}
For the first term we can easily enclose the factor depending on \(x\)
using that it is monotone. The factor depending on \(\alpha\) we
either enclose directly, if \(-\alpha - 1 - r_{j} \not= 0\), or handle
the removable singularity otherwise.

For the second term we use the deflated zeta function
\eqref{eq:deflated-zeta} to split it further into two terms
\begin{equation*}
  \frac{\underline{\zeta}(-\alpha + r_{j})}{2}\frac{x^{1 + \alpha}-x^{r_{j}}}{\log(1 / x)}
  - \frac{1}{2(1 + \alpha - r_{j})}\frac{x^{1 + \alpha}-x^{r_{j}}}{\log(1 / x)}.
\end{equation*}
The first term can be enclosed directly. The second term we write as
\begin{equation*}
  \frac{x^{r_{j}} - x^{1 + \alpha}}{2(r_{j} - (1 + \alpha))\log(1 / x)}.
\end{equation*}

We first consider the case when \(r \geq 1 + \alpha\), we then factor
out \(x^{1 + \alpha}\), giving us
\begin{equation*}
  \frac{x^{1 + \alpha}}{2}\frac{x^{r_{j} - (1 + \alpha)} - 1}{(r_{j} - (1 + \alpha))\log(1 / x)}
  = \frac{x^{1 + \alpha}}{2}\frac{1 - x^{r_{j} - (1 + \alpha)}}{(r_{j} - (1 + \alpha))\log(x)}
\end{equation*}
If we let \(t = (r_{j} - (1 + \alpha))\log(x)\) the second factor can
be written as
\begin{equation*}
  \frac{1 - e^{t}}{t}
\end{equation*}
which we handle similarly to \eqref{eq:I1-defect-exp-t}, note that
since \(r_{j} \geq 1 + \alpha\) and \(x < 1\) we have \(t \leq 0\).

Next we consider the case when \(r_{j} < 1 + \alpha\), depending on
the value of \(r_{j}\) this case might or might not occur. We then
factor out \(x^{r_{j}}\), giving us
\begin{equation*}
  \frac{x^{r_{j}}}{2}\frac{1 - x^{1 + \alpha - r_{j}}}{(r_{j} - (1 + \alpha))\log(1 / x)}
  = \frac{x^{r_{j}}}{2}\frac{1 - x^{1 + \alpha - r_{j}}}{(1 + \alpha - r_{j})\log(x)}.
\end{equation*}
We here let \(t = 1 + \alpha - r_{j}\) and the procedure is then the
same as the previous case.

\subsection{Remaining part}
Similarly to in the above section we factor out
\(\frac{1}{\Gamma(1 + \alpha)(1 - x^{p_{\alpha}})}\) and bound it
separately. For the terms in the expansion we note that they all have
an exponent greater than \(1 - \alpha\), we can thus cancel the
\(x^{1 - \alpha}\) in the denominator directly. It is then possible to
enclose the terms in the numerator and the factor
\(\frac{1}{\log(1 / x)}\) separately. For some terms we can however
improve the computed enclosures by incorporating the division by
\(\log(1 / x)\). These are the terms coming from \(u_{\alpha}(x)^{2}\)
containing a factor \(P\). In that case we want to enclose the factor
\begin{equation*}
  \frac{a_{\alpha,0}(c(\alpha) - c(\alpha - p_{\alpha})x^{p_{\alpha}})}{\log(1 / x)}.
\end{equation*}
This can be done using the following lemma and combining it with
Lemma~\ref{lemma:x_pow_t_m1_div_t} to enclose
\(\frac{x^{p_{\alpha}} - 1}{p_{\alpha}\log(1 / x)}\).
\begin{lemma}
  The factor
  \begin{equation*}
    a_{\alpha,0}(c(\alpha) - c(\alpha - p_{\alpha})x^{p_{\alpha}}).
  \end{equation*}
  can be written as
  \begin{equation*}
    C_{1}x^{p_{\alpha}} - C_{2}\frac{x^{p_{\alpha}} - 1}{p_{\alpha}}
  \end{equation*}
  with
  \begin{equation*}
    C_{1} = a_{\alpha,0}(c(\alpha) - c(\alpha - p_{\alpha}))\quad\text{ and }\quad
    C_{2} = a_{\alpha,0}c(\alpha)p_{\alpha}.
  \end{equation*}
  Both \(C_{1}\) and \(C_{2}\) have removable singularities at
  \(\alpha = -1\) and are finite there.
\end{lemma}
\begin{proof}
  The expression follows directly by adding and subtracting
  \(a_{\alpha,0}c(\alpha)x^{p_{\alpha}}\) and collecting the terms.
  The removable singularities of \(C_{1}\) and \(C_{2}\) can be
  handled by writing them as
  \begin{align*}
    C_{1} &= a_{\alpha,0}(1 + \alpha) \cdot \frac{c(\alpha) - c(\alpha - p_{\alpha})}{1 + \alpha},\\
    C_{2} &= a_{\alpha,0}(1 + \alpha) \cdot c(\alpha)(1 + (1 + \alpha) / 2)
  \end{align*}
  and enclosing all the factors separately.
\end{proof}

\section{Details for evaluating \(\mathcal{T}_{\alpha}(x)\) with \(\alpha \in I_{1}\) and \(x\) near zero}
\label{sec:evaluation-T-I-1-asymptotic}
In this section we give bounds for \(G_{\alpha,1}\), \(G_{\alpha,2}\)
and \(R_{\alpha}\) occurring in Lemma~\ref{lemma:evaluation-U-I-1}.

\subsection{\(G_{\alpha,1}\)}
Recall that
\begin{equation*}
  G_{\alpha,1}(x) = \frac{
    \int_{0}^{1}\left|(1 - t)^{-\alpha - 1} + (1 + t)^{-\alpha - 1} - 2t^{-\alpha - 1}\right|t^{(1 - \alpha) / 2}\log(2e + 1 / (xt))\ dt
  }{
    \log(1 / x)(1 - x^{1 + \alpha + (1 + \alpha)^{2} / 2})
  }.
\end{equation*}
To bound it we use the following lemma.
\begin{lemma}
  For \(\alpha \in (-1, 0)\) and \(x < 1\), \(G_{\alpha,1}\) satisfies
  \begin{equation*}
    G_{\alpha,1}(x) \leq \frac{1 + \alpha}{1 - x^{1 + \alpha + (1 + \alpha)^{2} / 2}}\left(
      \frac{\log(2e + 1 / x)}{\log(1 / x)} J_{\alpha,1}
      - \frac{1}{\log(1 / x)} J_{\alpha,2}
    \right),
  \end{equation*}
  where \(J_{\alpha,1}\) and \(J_{\alpha,2}\) are independent of \(x\)
  and given by
  \begin{align*}
    J_{\alpha,1} &= \int_{0}^{1}\left|\frac{(1 - t)^{-(1 + \alpha)} + (1 + t)^{-(1 + \alpha)} - 2t^{-(1 + \alpha)}}{1 + \alpha}\right|t^{(1 - \alpha) / 2}\ dt,\\
    J_{\alpha,2} &= \int_{0}^{1}\left|\frac{(1 - t)^{-(1 + \alpha)} + (1 + t)^{-(1 + \alpha)} - 2t^{-(1 + \alpha)}}{1 + \alpha}\right|t^{(1 - \alpha) / 2}\log(t)\ dt.
  \end{align*}
\end{lemma}
\begin{proof}
  Factoring out \(1 + \alpha\) we have
  \begin{multline*}
    G_{\alpha,1}(x) = \frac{1 + \alpha}{1 - x^{1 + \alpha + (1 + \alpha)^{2} / 2}}
    \frac{1}{\log(1 / x)}\\
    \int_{0}^{1}\left|\frac{(1 - t)^{-(1 + \alpha)} + (1 + t)^{-(1 + \alpha)} - 2t^{-(1 + \alpha)}}{1 + \alpha}\right|t^{(1 - \alpha) / 2}\log(2e + 1 / (xt))\ dt.
  \end{multline*}
  Let \(J_{\alpha}(x)\) denote
  \begin{equation*}
    J_{\alpha}(x) = \int_{0}^{1}\left|\frac{(1 - t)^{-(1 + \alpha)} + (1 + t)^{-(1 + \alpha)} - 2t^{-(1 + \alpha)}}{1 + \alpha}\right|t^{(1 - \alpha) / 2}\log(2e + 1 / (xt))\ dt
  \end{equation*}
  Splitting the logarithm as
  \begin{equation*}
    \log(2e + 1 / (xt)) = \log(1 + 2ext) - \log(x) - \log(t)
  \end{equation*}
  we can split \(J_{\alpha}(x)\) as
  \begin{multline*}
    J_{\alpha}(x) = \int_{0}^{1}\left|\frac{(1 - t)^{-(1 + \alpha)} + (1 + t)^{-(1 + \alpha)} - 2t^{-(1 + \alpha)}}{1 + \alpha}\right|t^{(1 - \alpha) / 2}\log(1 + 2ext)\ dt\\
    - \log(x)\int_{0}^{1}\left|\frac{(1 - t)^{-(1 + \alpha)} + (1 + t)^{-(1 + \alpha)} - 2t^{-(1 + \alpha)}}{1 + \alpha}\right|t^{(1 - \alpha) / 2}\ dt\\
    - \int_{0}^{1}\left|\frac{(1 - t)^{-(1 + \alpha)} + (1 + t)^{-(1 + \alpha)} - 2t^{-(1 + \alpha)}}{1 + \alpha}\right|t^{(1 - \alpha) / 2}\log(t)\ dt.
  \end{multline*}
  Using that \(\log(1 + 2ext) \leq \log(1 + 2ex)\) on the interval of
  integration we have
  \begin{multline*}
    J_{\alpha}(x) \leq
    (\log(1 + 2ex) - \log(x)) \int_{0}^{1}\left|\frac{(1 - t)^{-(1 + \alpha)} + (1 + t)^{-(1 + \alpha)} - 2t^{-(1 + \alpha)}}{1 + \alpha}\right|t^{(1 - \alpha) / 2}\ dt\\
    - \int_{0}^{1}\left|\frac{(1 - t)^{-(1 + \alpha)} + (1 + t)^{-(1 + \alpha)} - 2t^{-(1 + \alpha)}}{1 + \alpha}\right|t^{(1 - \alpha) / 2}\log(t)\ dt
    = \log(2e + 1 / x)J_{\alpha,1} - J_{\alpha,2}
  \end{multline*}
  and the result follows.
\end{proof}
The computation of \(J_{\alpha,1}\) and \(J_{\alpha,2}\) follow the
same approach as in Appendix~\ref{sec:rigorous-integration}, rigorous
numerical integration on most of the interval and explicitly handling
the singularity at \(t = 1\) by factoring out a bound for
\(t^{(1 - \alpha) / 2}\) and \(t^{(1 - \alpha) / 2}\log(t)\)
respectively.

\subsection{\(G_{\alpha,2}\)}
Recall that
\begin{equation*}
  G_{\alpha,2}(x)
  = \frac{
    \int_{1}^{\pi / x}\left((t - 1)^{-\alpha - 1} + (1 + t)^{-\alpha - 1} - 2t^{-\alpha - 1}\right)t^{(1 - \alpha) / 2}\log(2e + 1 / (xt))\ dt
  }{
    \log(1 / x)(1 - x^{1 + \alpha + (1 + \alpha)^{2} / 2})
  }.
\end{equation*}
Unfortunately it is not enough in this case to just factor out
\(1 + \alpha\) and bound the two factors independently, as for
\(G_{\alpha,1}\), the second factor would not be bounded. Instead, we
have to more carefully track the dependence on \(x\). We split the
work from getting an upper bound into several lemmas. In
Lemma~\ref{lemma:G2-split} we split \(G_{\alpha,2}\) into one main
part, \(G_{\alpha,2,M}\), and one remainder part, \(G_{\alpha,2,R}\).
In Lemma~\ref{lemma:G2M-bound} we give a bound for \(G_{\alpha,2,M}\).
The bound for the \(G_{\alpha,2,R}\) is handled in
Lemma~\ref{lemma:G2R-n1}, \ref{lemma:G2R1-bound}
and~\ref{lemma:G2R2-bound}. These bounds work well for very small
values of \(x\), less than around \(10^{-10}\), but work poorly for
larger values of \(x\). The methods based on direct integration, as
discussed in Appendix~\ref{sec:rigorous-integration}, work well for
\(x\) larger than around \(0.1\). In between these two intervals,
\(10^{-10} < x < 0.\), we use a slightly different approach for
bounding \(G_{\alpha,2}\). It is described at the end of this section,
with most of the details in Lemma~\ref{lemma:G21}.

\begin{lemma}
  \label{lemma:G2-split}
  For \(\alpha \in (-1, 0)\) and \(x < 1\), \(G_{\alpha,2}(x)\)
  satisfies
  \begin{equation*}
    G_{\alpha,2}(x) = G_{\alpha,2,M}(x) + G_{\alpha,2,R}(x)
  \end{equation*}
  with
  \begin{align*}
    G_{\alpha,2,M}(x) =&  \frac{1}{1 - x^{1 + \alpha + (1 + \alpha)^{2} / 2}}
                         \frac{1}{\log(1 / x)}
                         \frac{4 + 2\alpha}{3}\\
                       &\left(
                         \left(D_{x} - \log x - \frac{2}{3}\frac{1}{1 + \alpha}\right)
                         \left(1 - \left(\frac{x}{\pi}\right)^{\frac{3}{2}(1 + \alpha)}\right)
                         + \log\left(\frac{\pi}{x}\right)\left(\frac{x}{\pi}\right)^{\frac{3}{2}(1 + \alpha)}
                         \right),\\
    G_{\alpha,2,R}(x) =& \frac{1}{1 - x^{1 + \alpha + (1 + \alpha)^{2} / 2}}
                         \frac{1}{\log(1 / x)}\\
                       &\sum_{n = 1}^{\infty} (-1)^{n}\frac{(1 + \alpha)^{n}}{n!}
                         \sum_{k = 0}^{n - 1}\binom{n}{k}\frac{1}{2^{k}}\int_{1}^{\pi / x}
                         \log(t)^{k}h_{n - k}(t)t\log(2e + 1 / (xt))\ dt,
  \end{align*}
  for some \(D_{x} \in \left[-\log(1 + 2e\pi), \log(1 + 2e\pi)\right]\)
  and
  \begin{equation*}
    h_{k}(t) = \log(t - 1)^{k} + \log(t + 1)^{k} - 2\log(t)^{k}
    - \frac{k(k - 1 - \log t)\log(t)^{k - 2}}{t^{2}}.
  \end{equation*}
\end{lemma}
\begin{proof}
  Let us start by studying the integral in of \(G_{\alpha,2}\), for
  which we use the notation
  \begin{equation*}
    K_{\alpha}(x) = \int_{1}^{\pi / x}
    \left((t - 1)^{-\alpha - 1} + (1 + t)^{-\alpha - 1} - 2t^{-\alpha - 1}\right)
    t^{(1 - \alpha) / 2}\log(2e + 1 / (xt))\ dt.
  \end{equation*}

  We have
  \begin{multline*}
    \left((t - 1)^{-\alpha - 1} + (1 + t)^{-\alpha - 1} - 2t^{-\alpha - 1}\right)t^{(1 - \alpha) / 2}\\
    = t\left(
      e^{-(1 + \alpha)(\log(t - 1) + \log(t) / 2)}
      + e^{-(1 + \alpha)(\log(t + 1) + \log(t) / 2)}
      - 2e^{-(1 + \alpha)3\log(t)/2}
    \right)\\
    = t\sum_{n = 1}^{\infty}(-1)^{n}\frac{(1 + \alpha)^{n}}{n!}g_{n}(t)
  \end{multline*}
  with
  \begin{equation*}
    g_{n}(t) = \left(\log(t - 1) + \frac{\log t}{2}\right)^{n}
      + \left(\log(t + 1) + \frac{\log t}{2}\right)^{n}
      - 2\left(\frac{3\log t}{2}\right)^{n}.
  \end{equation*}
  Inserting this into the integral we can write it as
  \begin{equation*}
    K_{\alpha}(x) = \int_{1}^{\pi / x}
    t\sum_{n = 1}^{\infty}(-1)^{n}\frac{(1 + \alpha)^{n}}{n!}g_{n}(t)\log(2e + 1 / (xt))\ dt.
  \end{equation*}
  Switching the integral and the sum we get
  \begin{equation*}
    K_{\alpha}(x) = \sum_{n = 1}^{\infty}(-1)^{n}\frac{(1 + \alpha)^{n}}{n!}
    \int_{1}^{\pi / x} g_{n}(t)t\log(2e + 1 / (xt))\ dt.
  \end{equation*}

  Next we want to look closer at
  \begin{equation*}
    I_{n}(x) = \int_{1}^{\pi / x} g_{n}(t)t\log(2e + 1 / (xt))\ dt.
  \end{equation*}
  Using the binomial theorem and splitting
  \(\frac{3\log t}{2} = \log t + \frac{\log t}{2}\) we can write
  \(g_{n}(t)\) as
  \begin{equation*}
    g_{n}(t) = \sum_{k = 0}^{n - 1}\binom{n}{k}\left(\frac{\log t}{2}\right)^{k}\left(
      \log(t - 1)^{n - k} + \log(t + 1)^{n - k} - 2\log(t)^{n - k}
    \right).
  \end{equation*}
  This gives us
  \begin{equation*}
    I_{n}(x) = \sum_{k = 0}^{n - 1}\binom{n}{k}\frac{1}{2^{k}}\int_{1}^{\pi / x}
    \log(t)^{k}\left(
      \log(t - 1)^{n - k} + \log(t + 1)^{n - k} - 2\log(t)^{n - k}
    \right)t\log(2e + 1 / (xt))\ dt.
  \end{equation*}
  Asymptotically as \(t\) goes to infinity we have that
  \begin{equation*}
    \log(t - 1)^{l} + \log(t + 1)^{l} - 2\log(t)^{l}
  \end{equation*}
  behaves like
  \begin{equation*}
    \frac{l(l - 1 - \log t)\log(t)^{l - 2}}{t^{2}}.
  \end{equation*}
  That this indeed is the case is seen in the proof of
  Lemma~\ref{lemma:G2R2-bound}, and this will be what gives the main
  part of the integral for small \(x\). If we let
  \begin{equation*}
    h_{k}(t) = \log(t - 1)^{k} + \log(t + 1)^{k} - 2\log(t)^{k}
    - \frac{k(k - 1 - \log t)\log(t)^{k - 2}}{t^{2}}
  \end{equation*}
  and
  \begin{align*}
    I_{n,M}(x) =& \sum_{k = 0}^{n - 1}\binom{n}{k}\frac{1}{2^{k}}\int_{1}^{\pi / x}
                  \log(t)^{k}
                  \frac{(n - k)(n - k - 1 - \log t)\log(t)^{n - k - 2}}{t^{2}}
                  t\log(2e + 1 / (xt))\ dt,\\
    I_{n,R}(x) =& \sum_{k = 0}^{n - 1}\binom{n}{k}\frac{1}{2^{k}}\int_{1}^{\pi / x}
                  \log(t)^{k}h_{n - k}(t)t\log(2e + 1 / (xt))\ dt,
  \end{align*}
  then we have
  \begin{equation*}
    I_{n}(x) = I_{n,M}(x) + I_{n,R}(x)
  \end{equation*}
  and
  \begin{equation*}
    K_{\alpha}(x) = \sum_{n = 1}^{\infty}(-1)^{n}\frac{(1 + \alpha)^{n}}{n!}I_{n,M}(x)
    + \sum_{n = 1}^{\infty}(-1)^{n}\frac{(1 + \alpha)^{n}}{n!}I_{n,R}(x).
  \end{equation*}
  Multiplying the second term with the factor
  \begin{equation*}
    \frac{1}{1 - x^{1 + \alpha + (1 + \alpha)^{2} / 2}}\frac{1}{\log(1 / x)}
  \end{equation*}
  gives us \(G_{\alpha,2,R}\).

  For \(I_{n,M}(x)\) we note that
  \begin{equation*}
    \log(t)^{k}(n - k)(n - k - 1 - \log t)\log(t)^{n - k - 2}
    = (n - k)\left((n - k - 1)\log(t)^{n - 2} - \log(t)^{n - 1}\right).
  \end{equation*}
  Giving us
  \begin{align*}
    I_{n,M}(x) =& \sum_{k = 0}^{n - 1}\binom{n}{k}\frac{1}{2^{k}}(n - k)\\
                & \left(
                  (n - k - 1)\int_{1}^{\pi / x}\frac{\log(t)^{n - 2}}{t}\log(2e + 1 / (xt))\ dt
                  - \int_{1}^{\pi / x}\frac{\log(t)^{n - 1}}{t}\log(2e + 1 / (xt))\ dt
                  \right)\\
    =& \int_{1}^{\pi / x}\frac{\log(t)^{n - 2}}{t}\log(2e + 1 / (xt))\ dt
       \sum_{k = 0}^{n - 1}\binom{n}{k}\frac{1}{2^{k}}(n - k)(n - k - 1)\\
                & - \int_{1}^{\pi / x}\frac{\log(t)^{n - 1}}{t}\log(2e + 1 / (xt))\ dt
       \sum_{k = 0}^{n - 1}\binom{n}{k}\frac{1}{2^{k}}(n - k).
  \end{align*}
  The sums can be computed explicitly,
  \begin{align*}
    \sum_{k = 0}^{n - 1}\binom{n}{k}\frac{1}{2^{k}}(n - k)(n - k - 1)
    &= \left(\frac{3}{2}\right)^{n - 2}n(n - 1),\\
    \sum_{k = 0}^{n - 1}\binom{n}{k}\frac{1}{2^{k}}(n - k)
    &= \left(\frac{3}{2}\right)^{n - 1}n.
  \end{align*}
  Inserting this back we have
  \begin{equation*}
    I_{n,M}(x) = \left(\frac{3}{2}\right)^{n - 2}n \left(
      (n - 1)\int_{1}^{\pi / x}\frac{\log(t)^{n - 2}}{t}\log(2e + 1 / (xt))\ dt
      - \frac{3}{2}\int_{1}^{\pi / x}\frac{\log(t)^{n - 1}}{t}\log(2e + 1 / (xt))\ dt
    \right).
  \end{equation*}
  Next we are interested in computing integrals of the form
  \begin{equation*}
    \int_{1}^{\pi / x}\frac{\log(t)^{l}}{t}\log(2e + 1 / (xt))\ dt.
  \end{equation*}
  Since \(n \geq 1\) we need to consider \(l \geq -1\). However, the
  only case for which we have \(l = -1\) is for \(n = 1\), in which
  case the factor \(n - 1\) in front is zero. Hence, we only need to
  consider \(l \geq 0\). Using that
  \({\log(2e + 1 / (xt)) = \log(1 + 2ext) - \log(x) - \log(t)}\) we
  get
  \begin{equation*}
    \int_{1}^{\pi / x}\frac{\log(t)^{l}}{t}\log(2e + 1 / (xt))\ dt
    = \int_{1}^{\pi / x}\frac{\log(t)^{l}}{t}\log(1 + 2ext)\ dt
    - \log(x)\int_{1}^{\pi / x}\frac{\log(t)^{l}}{t}\ dt
    - \int_{1}^{\pi / x}\frac{\log(t)^{l + 1}}{t}\ dt.
  \end{equation*}
  For \(1 \leq t \leq \pi / x\) we have
  \(0 \leq \log(1 + 2ext) \leq \log(1 + 2e\pi)\), hence
  \begin{equation*}
    \int_{1}^{\pi / x}\frac{\log(t)^{l}}{t}\log(1 + 2ext)\ dt
    = D_{l,x}\int_{1}^{\pi / x}\frac{\log(t)^{l}}{t}\ dt
  \end{equation*}
  for some \(D_{l,x} \in [0, \log(1 + 2e\pi)]\). This gives us
  \begin{align*}
    \int_{1}^{\pi / x}\frac{\log(t)^{l}}{t}\log(2e + 1 / (xt))\ dt
    &= (D_{l,x} - \log(x))\int_{1}^{\pi / x}\frac{\log(t)^{l}}{t}\ dt
      - \int_{1}^{\pi / x}\frac{\log(t)^{l + 1}}{t}\ dt\\
    &= (D_{l,x} - \log(x))\frac{\log(\pi / x)^{l + 1}}{l + 1}
      - \frac{\log(\pi / x)^{l + 2}}{l + 2},
  \end{align*}
  valid for \(l \geq 0\). Inserting this back into \(I_{n,M}\) we get
  for \(n \geq 2\)
  \begin{align*}
    I_{n,M}(x) =& \left(\frac{3}{2}\right)^{n - 2}n
                  \Bigg(
                  (n - 1)\left(
                  (D_{n - 2,x} - \log(x))\frac{\log(\pi / x)^{n - 1}}{n - 1}
                  - \frac{\log(\pi / x)^{n}}{n}
                  \right)\\
                &- \frac{3}{2}\left(
                  (D_{n-1,x} - \log(x))\frac{\log(\pi / x)^{n}}{n}
                  - \frac{\log(\pi / x)^{n + 1}}{n + 1}
                  \right)
                  \Bigg)\\
                  =& \left(\frac{3}{2}\right)^{n - 2}\log(\pi / x)^{n - 1}
                     \Bigg(
                     n(D_{n - 2,x} - \log(x)) - (n - 1)\log(\pi / x)\\
                &- \frac{3}{2}(D_{n-1,x} - \log(x))\log(\pi / x)
                  + \frac{3n}{2(n + 1)}\log(\pi / x)^{2}
                  \Bigg).
  \end{align*}
  For \(n = 1\) we instead get
  \begin{align*}
    I_{1,M}(x) =& \left(\frac{3}{2}\right)^{-1}\log(\pi / x)^{-1}
                  \left(
                  - \frac{3}{2}(D_{-1,x} - \log(x))\log(\pi / x)
                  + \frac{3}{4}\log(\pi / x)^{2}
                  \right).
  \end{align*}

  Inserting this into the sum for \(n\) and splitting it in parts we
  get
  \begin{align*}
    \sum_{n = 1}^{\infty}(-1)^{n}\frac{(1 + \alpha)^{n}}{n!}I_{n,M}(x)
    =& \sum_{n = 2}^{\infty}(-1)^{n}\frac{(1 + \alpha)^{n}}{n!}
       \left(\frac{3}{2}\right)^{n - 2}\log(\pi / x)^{n - 1}nD_{n - 2,x}\\
     & -\log(x)\sum_{n = 2}^{\infty}(-1)^{n}\frac{(1 + \alpha)^{n}}{n!}
       \left(\frac{3}{2}\right)^{n - 2}\log(\pi / x)^{n - 1}n\\
     & -\sum_{n = 2}^{\infty}(-1)^{n}\frac{(1 + \alpha)^{n}}{n!}
       \left(\frac{3}{2}\right)^{n - 2}\log(\pi / x)^{n}(n - 1)\\
     & -\sum_{n = 1}^{\infty}(-1)^{n}\frac{(1 + \alpha)^{n}}{n!}
       \left(\frac{3}{2}\right)^{n - 1}\log(\pi / x)^{n}D_{n - 1,x}\\
     & +\log(x)\sum_{n = 1}^{\infty}(-1)^{n}\frac{(1 + \alpha)^{n}}{n!}
       \left(\frac{3}{2}\right)^{n - 1}\log(\pi / x)^{n}\\
     & +\sum_{n = 1}^{\infty}(-1)^{n}\frac{(1 + \alpha)^{n}}{n!}
       \left(\frac{3}{2}\right)^{n - 1}\log(\pi / x)^{n + 1}\frac{n}{n + 1}.
  \end{align*}
  We begin by noting that
  \begin{align*}
    \sum_{n = 2}^{\infty}(-1)^{n}\frac{(1 + \alpha)^{n}}{n!}
    \left(\frac{3}{2}\right)^{n - 2}\log(\pi / x)^{n - 1}nD_{n - 2,x}
    &= -(1 + \alpha)\sum_{n = 1}^{\infty}(-1)^{n}\frac{(1 + \alpha)^{n}}{n!}
      \left(\frac{3}{2}\right)^{n - 1}\log(\pi / x)^{n}D_{n - 1,x},\\
    \sum_{n = 2}^{\infty}(-1)^{n}\frac{(1 + \alpha)^{n}}{n!}
    \left(\frac{3}{2}\right)^{n - 2}\log(\pi / x)^{n - 1}n
    &= -(1 + \alpha)\sum_{n = 1}^{\infty}(-1)^{n}\frac{(1 + \alpha)^{n}}{n!}
      \left(\frac{3}{2}\right)^{n - 1}\log(\pi / x)^{n},\\
    \sum_{n = 2}^{\infty}(-1)^{n}\frac{(1 + \alpha)^{n}}{n!}
    \left(\frac{3}{2}\right)^{n - 2}\log(\pi / x)^{n}(n - 1)
    &= -(1 + \alpha)\sum_{n = 1}^{\infty}(-1)^{n}\frac{(1 + \alpha)^{n}}{n!}
      \left(\frac{3}{2}\right)^{n - 1}\log(\pi / x)^{n + 1}\frac{n}{n + 1},
  \end{align*}
  and that since \(D_{l,x} \in [0, \log(1 + 2e\pi)]\) for all values of
  \(l\) and \(x\) we have that
  \begin{align*}
    \sum_{n = 1}^{\infty}(-1)^{n}\frac{(1 + \alpha)^{n}}{n!}
    \left(\frac{3}{2}\right)^{n - 1}\log(\pi / x)^{n}D_{n - 1,x}
    &= D_{x}\sum_{n = 1}^{\infty}(-1)^{n}\frac{(1 + \alpha)^{n}}{n!}
      \left(\frac{3}{2}\right)^{n - 1}\log(\pi / x)^{n}
  \end{align*}
  for some \(D_{x} \in [-\log(1 + 2e\pi), \log(1 + 2e\pi)]\). With this
  we get
  \begin{align*}
    \sum_{n = 1}^{\infty}(-1)^{n}\frac{(1 + \alpha)^{n}}{n!}I_{n,M}(x)
    =& -(D_{x} - \log(x))(2 + \alpha)\sum_{n = 1}^{\infty}(-1)^{n}\frac{(1 + \alpha)^{n}}{n!}
       \left(\frac{3}{2}\right)^{n - 1}\log(\pi / x)^{n}\\
     & +(2 + \alpha)\sum_{n = 1}^{\infty}(-1)^{n}\frac{(1 + \alpha)^{n}}{n!}
       \left(\frac{3}{2}\right)^{n - 1}\log(\pi / x)^{n + 1}\frac{n}{n + 1}.
  \end{align*}
  Explicitly computing the two sums we have
  \begin{align*}
    \sum_{n = 1}^{\infty}(-1)^{n}\frac{(1 + \alpha)^{n}}{n!}
    \left(\frac{3}{2}\right)^{n - 1}\log(\pi / x)^{n}
    &= \frac{2}{3}\sum_{n = 1}^{\infty}(-1)^{n}
      \frac{\left(\frac{3}{2}(1 + \alpha)\log(\pi / x)\right)^{n}}{n!}\\
    &= \frac{2}{3}\left(e^{-\frac{3}{2}(1 + \alpha)\log(\pi / x)} - 1\right)\\
    &= \frac{2}{3}\left(\left(\frac{x}{\pi}\right)^{\frac{3}{2}(1 + \alpha)} - 1\right),\\
    \sum_{n = 1}^{\infty}(-1)^{n}\frac{(1 + \alpha)^{n}}{n!}
    \left(\frac{3}{2}\right)^{n - 1}\log(\pi / x)^{n + 1}\frac{n}{n + 1}
    &= -\left(\frac{2}{3}\right)^{2}\frac{1}{1 + \alpha}\sum_{n = 1}^{\infty}(-1)^{n + 1}
      \frac{\left(\frac{3}{2}(1 + \alpha)\log(\pi / x)\right)^{n + 1}}{(n + 1)!}n\\
    &= -\left(\frac{2}{3}\right)^{2}\frac{1}{1 + \alpha}\left(1 - \left(1 + \frac{3}{2}(1 + \alpha)\log(\pi / x)\right)e^{-\frac{3}{2}(1 + \alpha)\log(\pi / x)}\right)\\
    &= -\left(\frac{2}{3}\right)^{2}\frac{1}{1 + \alpha}\left(1 - \left(1 + \frac{3}{2}(1 + \alpha)\log(\pi / x)\right)\left(\frac{x}{\pi}\right)^{\frac{3}{2}(1 + \alpha)}\right).
  \end{align*}
  This gives us
  \begin{multline*}
    \sum_{n = 1}^{\infty} (-1)^{n}\frac{(1 + \alpha)^{n}}{n!}I_{n,M}(x)
    = -(D_{x} - \log x)(2 + \alpha)\frac{2}{3}\left(\left(\frac{x}{\pi}\right)^{\frac{3}{2}(1 + \alpha)} - 1\right)\\
    -(2 + \alpha)\left(\frac{2}{3}\right)^{2}\frac{1}{1 + \alpha}\left(1 - \left(1 + \frac{3}{2}(1 + \alpha)\log(\pi / x)\right)\left(\frac{x}{\pi}\right)^{\frac{3}{2}(1 + \alpha)}\right).
  \end{multline*}
  Simplifying we get
  \begin{equation*}
    \sum_{n = 1}^{\infty} (-1)^{n}\frac{(1 + \alpha)^{n}}{n!}I_{n,M}(x)
    = \frac{4 + 2\alpha}{3}\left(
      \left(D_{x} - \log x - \frac{2}{3}\frac{1}{1 + \alpha}\right)\left(1 - \left(\frac{x}{\pi}\right)^{\frac{3}{2}(1 + \alpha)}\right)
    + \log\left(\frac{\pi}{x}\right)\left(\frac{x}{\pi}\right)^{\frac{3}{2}(1 + \alpha)}
    \right).
  \end{equation*}
  Multiplying this with the factor
  \begin{equation*}
    \frac{1}{1 - x^{1 + \alpha + (1 + \alpha)^{2} / 2}}\frac{1}{\log(1 / x)}
  \end{equation*}
  gives us \(G_{\alpha,2,M}\).
\end{proof}

A bound for \(G_{\alpha,2,M}(x)\) is given in the following lemma.
\begin{lemma}
  \label{lemma:G2M-bound}
  For \(\alpha \in (-1, 0)\) and \(x < \frac{1}{\pi^{3}}\),
  \(G_{\alpha,2,M}\) satisfies the following bound
  \begin{equation*}
    G_{\alpha,2,M}(x) \leq \frac{4 + 2\alpha}{3}\left(\frac{2\log(1 + 2e\pi)}{\log(1 / x)} + 1\right).
  \end{equation*}
\end{lemma}
\begin{proof}
  As a first step we split \(G_{\alpha,2,M}(x)\) into two terms
  \begin{multline*}
    G_{\alpha,2,M}(x)
    = \frac{4 + 2\alpha}{3}\frac{1}{\log(1 / x)}\frac{D_{x}}{1 - x^{1 + \alpha + (1 + \alpha)^{2} / 2}}
    \left(1 - \left(\frac{x}{\pi}\right)^{\frac{3}{2}(1 + \alpha)}\right)\\
    + \frac{4 + 2\alpha}{3}\frac{1}{1 - x^{1 + \alpha + (1 + \alpha)^{2} / 2}}\frac{1}{\log(1 / x)}
    \left(
      \log\left(\frac{\pi}{x}\right)\left(\frac{x}{\pi}\right)^{\frac{3}{2}(1 + \alpha)}
      -\left(\log x + \frac{2}{3}\frac{1}{1 + \alpha}\right)
      \left(1 - \left(\frac{x}{\pi}\right)^{\frac{3}{2}(1 + \alpha)}\right)
    \right)
  \end{multline*}

  For the first term we focus on the part
  \begin{equation*}
    \frac{D_{x}}{1 - x^{1 + \alpha + (1 + \alpha)^{2} / 2}}
    \left(1 - \left(\frac{x}{\pi}\right)^{\frac{3}{2}(1 + \alpha)}\right)
    = D_{x}\frac{1 - e^{\frac{3}{2}(1 + \alpha)(\log x - \log \pi)}}{1 - e^{(1 + \alpha)(1 + (1 + \alpha) / 2)\log x}}.
  \end{equation*}
  Adding and subtracting
  \(e^{(1 + \alpha)(1 + (1 + \alpha) / 2)\log x}\) to the numerator we
  can write this as
  \begin{multline*}
    D_{x}\left(
      1
      + e^{(1 + \alpha)(1 + (1 + \alpha) / 2)\log x}\frac{1 - e^{\frac{3}{2}(1 + \alpha)(\log x - \log \pi) - (1 + \alpha)(1 + (1 + \alpha) / 2)\log x}}{1 - e^{(1 + \alpha)(1 + (1 + \alpha) / 2)\log x}}
    \right)\\
    = D_{x}\left(
      1
      + e^{(1 + \alpha)(1 + (1 + \alpha) / 2)\log x}\frac{1 - e^{(1 + \alpha)\left(\frac{3}{2}(\log x - \log \pi) - (1 + (1 + \alpha) / 2)\log x\right)}}{1 - e^{(1 + \alpha)(1 + (1 + \alpha) / 2)\log x}}
    \right)
  \end{multline*}
  Since \((1 + \alpha)(1 + (1 + \alpha) / 2)\log x\) is negative this
  is bounded by
  \begin{equation*}
    D_{x}\left(
      1
      + \frac{1 - e^{(1 + \alpha)\left(\frac{3}{2}(\log x - \log \pi) - (1 + (1 + \alpha) / 2)\log x\right)}}{1 - e^{(1 + \alpha)(1 + (1 + \alpha) / 2)\log x}}
    \right)
  \end{equation*}
  Since \(x < \frac{1}{\pi^{3}}\) we have
  \begin{multline*}
    (1 + \alpha)\left(\frac{3}{2}(\log x - \log \pi) - (1 + (1 + \alpha) / 2)\log x\right)
    > (1 + \alpha)\left(\frac{3}{2}\left(\log x + \frac{1}{3} \log x\right) - (1 + (1 + \alpha) / 2)\log x\right)\\
    = (1 + \alpha)(1 - (1 + \alpha) / 2)\log x.
  \end{multline*}
  From which we get
  \begin{equation*}
    e^{(1 + \alpha)\left(\frac{3}{2}(\log x - \log \pi) - (1 + (1 + \alpha) / 2)\log x\right)}
    > e^{(1 + \alpha)(1 + (1 + \alpha) / 2)\log x}
  \end{equation*}
  and hence
  \begin{equation*}
    \frac{1 - e^{(1 + \alpha)\left(\frac{3}{2}(\log x - \log \pi) - (1 + (1 + \alpha) / 2)\log x\right)}}{1 - e^{(1 + \alpha)(1 + (1 + \alpha) / 2)\log x}}
    < 1.
  \end{equation*}
  From this we get the upper bound
  \begin{equation*}
    \frac{D_{x}}{1 - x^{1 + \alpha + (1 + \alpha)^{2} / 2}}
    \left(1 - \left(\frac{x}{\pi}\right)^{\frac{3}{2}(1 + \alpha)}\right)
    < 2D_{x}.
  \end{equation*}
  Together with the bound \(D_{x} < \log(1 + 2e\pi)\) this gives us the
  first part of the bound for \(G_{\alpha,2,M}(x)\).

  For the second term we want to prove that
  \begin{equation}
    \label{eq:G2M2}
    \frac{1}{1 - x^{1 + \alpha + (1 + \alpha)^{2} / 2}}\frac{1}{\log(1 / x)}
    \left(
      \log\left(\frac{\pi}{x}\right)\left(\frac{x}{\pi}\right)^{\frac{3}{2}(1 + \alpha)}
      -\left(\log x + \frac{2}{3}\frac{1}{1 + \alpha}\right)
      \left(1 - \left(\frac{x}{\pi}\right)^{\frac{3}{2}(1 + \alpha)}\right)
    \right)
  \end{equation}
  is bounded by \(1\). To begin with we note that
  \begin{multline*}
    \log\left(\frac{\pi}{x}\right)\left(\frac{x}{\pi}\right)^{\frac{3}{2}(1 + \alpha)}
    -\left(\log x + \frac{2}{3}\frac{1}{1 + \alpha}\right)
    \left(1 - \left(\frac{x}{\pi}\right)^{\frac{3}{2}(1 + \alpha)}\right)\\
    = \frac{1}{1 + \alpha}\left(
      - \frac{2}{3}
      - (1 + \alpha)\log x
      + \left(
        \frac{2}{3}
        + (1 + \alpha)\log(\pi)
      \right)\left(\frac{x}{\pi}\right)^{\frac{3}{2}(1 + \alpha)}
    \right).
  \end{multline*}
  Allowing us to write \eqref{eq:G2M2} as
  \begin{equation*}
    \frac{1}{\left(1 - x^{1 + \alpha + (1 + \alpha)^{2} / 2}\right)(1 + \alpha)\log x}
    \left(
      \frac{2}{3}
      + (1 + \alpha)\log x
      - \left(
        \frac{2}{3}
        + (1 + \alpha)\log(\pi)
      \right)\left(\frac{x}{\pi}\right)^{\frac{3}{2}(1 + \alpha)}
    \right).
  \end{equation*}
  If we let \(s = (1 + \alpha)\log x\) we can write it as
  \begin{equation*}
    \frac{1}{\left(1 - e^{(1 + (1 + \alpha) / 2) s}\right)s}
    \left(
      \frac{2}{3}
      + s
      - \left(
        \frac{2}{3}
        + (1 + \alpha)\log \pi
      \right)e^{\frac{3}{2}s - \frac{3}{2}(1 + \alpha)\log \pi}
    \right).
  \end{equation*}
  Since the denominator is negative it is enough to verify that
  \begin{equation*}
    \frac{2}{3}
    + s
    - \left(
      \frac{2}{3}
      + (1 + \alpha)\log \pi
    \right)e^{\frac{3}{2}s - \frac{3}{2}(1 + \alpha)\log \pi}
    > \left(1 - e^{(1 + (1 + \alpha) / 2) s}\right)s
  \end{equation*}
  to prove that the quotient is bounded by \(1\). If we let
  \(v = (1 + \alpha)\log \pi\) and simplify the inequality we get
  \begin{equation}
    \label{eq:G2M2-inequality}
    \frac{2}{3}
    - \left(\frac{2}{3} - v\right) e^{\frac{3}{2}(s - v)}
    + se^{\frac{3 + \alpha}{2}s}
    > 0.
  \end{equation}
  The left-hand side is increasing in \(v\), to see this we
  differentiate it with respect to \(v\) to get
  \begin{equation*}
    \frac{4 + 3v}{2}e^{\frac{3}{2}(s - v)}.
  \end{equation*}
  Since \(0 < v\) this is positive and the value for the left-hand
  side of \eqref{eq:G2M2-inequality} is hence lower bounded by the
  value for \(v = 0\), from which we get the inequality
  \begin{equation}
    \label{eq:G2M2-inequality-2}
    \frac{2}{3} - \frac{2}{3}e^{\frac{3}{2}s} + se^{\frac{3 + \alpha}{2}s} > 0.
  \end{equation}
  Note that \(s = (1 + \alpha)\log x < 0\). We want to prove that the
  left-hand side is decreasing in \(s\) so that it is lower bounded by
  the value for \(s = 0\), which is \(0\). To see that the left-hand
  side of \eqref{eq:G2M2-inequality-2} is decreasing in \(s\) we
  differentiate it with respect to \(s\), giving us
  \begin{equation*}
    -e^{\frac{3}{2}s} + e^{\frac{3 + \alpha}{2}s} + s \frac{3 + \alpha}{2}e^{\frac{3 + \alpha}{2}s}
    = e^{\frac{3}{2}s}\left(e^{\frac{\alpha}{2}s}\left(1 + s\frac{3 + \alpha}{2}\right) - 1\right).
  \end{equation*}
  We want to check that
  \begin{equation}
    \label{eq:G2M2-inequality-3}
    e^{\frac{\alpha}{2}s}\left(1 + s\frac{3 + \alpha}{2}\right) - 1 < 0.
  \end{equation}
  If we again differentiate the left-hand side with respect to \(s\)
  we get
  \begin{equation*}
    \frac{1}{2}e^{\frac{\alpha}{2}s}\left(3 + 2\alpha + \frac{(3 + \alpha)\alpha}{2}s\right).
  \end{equation*}
  We have \(3 + 2\alpha > 0\) and
  \(\frac{(3 + \alpha)\alpha}{2}s > 0\), so this is positive and hence
  \eqref{eq:G2M2-inequality-3} is upper bounded by the value at
  \(s = 0\), where it is \(0\). It follows that the left-hand side of
  \eqref{eq:G2M2-inequality-2} is decreasing in \(s\). With all of
  this we have shown that \eqref{eq:G2M2} is bounded by \(1\) and the
  result follows.
\end{proof}

To bound \(G_{\alpha,2,R}(x)\) we factor out \(1 + \alpha\) from the
sum and take the absolute value termwise and move it inside the
integral to get
\begin{multline*}
  G_{\alpha,2,R}(x) \leq \frac{1 + \alpha}{1 - x^{1 + \alpha + (1 + \alpha)^{2} / 2}}
  \frac{1}{\log(1 / x)}\\
  \sum_{n = 1}^{\infty} \frac{(1 + \alpha)^{n - 1}}{n!}
  \sum_{k = 0}^{n - 1}\binom{n}{k}\frac{1}{2^{k}}\int_{1}^{\pi / x}
  \log(t)^{k}|h_{n - k}(t)|t\log(2e + 1 / (xt))\ dt.
\end{multline*}
Furthermore we use the bound \(\log(2e + 1 / (xt)) < \log(2e + 1 / x)\)
to get
\begin{equation}
  \label{eq:G2R-sum}
  G_{\alpha,2,R}(x) \leq \frac{1 + \alpha}{1 - x^{1 + \alpha + (1 + \alpha)^{2} / 2}}
  \frac{\log(2e + 1 / x)}{\log(1 / x)}
  \sum_{n = 1}^{\infty} \frac{(1 + \alpha)^{n - 1}}{n!}
  \sum_{k = 0}^{n - 1}\binom{n}{k}\frac{1}{2^{k}}\int_{1}^{\pi / x}
  \log(t)^{k}|h_{n - k}(t)|t\ dt.
\end{equation}

We can bound
\begin{equation*}
  \frac{1 + \alpha}{1 - x^{1 + \alpha + (1 + \alpha)^{2} / 2}}
\end{equation*}
by handling the removable singularity at \(\alpha = -1\) and the
second factor by writing it as
\begin{equation*}
  \frac{\log(2e + 1 / x)}{\log(1 / x)} = 1 + \frac{\log(1 + 2ex)}{\log(1 / x)}.
\end{equation*}

For the remaining part the first term in the sum, with \(n = 1\), is
treated explicitly in Lemma~\ref{lemma:G2R-n1}. For the remaining
terms we split the interval of integration into two parts and bound
them separately. More precisely we look at the two terms
\begin{align*}
  G_{\alpha,2,R,1}(x) =& \sum_{n = 2}^{\infty} \frac{(1 + \alpha)^{n - 1}}{n!}
                         \sum_{k = 0}^{n - 1}\binom{n}{k}\frac{1}{2^{k}}\int_{1}^{2}
                         \log(t)^{k}|h_{n - k}(t)|t\ dt,\\
  G_{\alpha,2,R,2}(x) =& \sum_{n = 2}^{\infty} \frac{(1 + \alpha)^{n - 1}}{n!}
                         \sum_{k = 0}^{n - 1}\binom{n}{k}\frac{1}{2^{k}}\int_{2}^{\pi / x}
                         \log(t)^{k}|h_{n - k}(t)|t\ dt
\end{align*}
which are handled in Lemma~\ref{lemma:G2R1-bound}
and~\ref{lemma:G2R2-bound}.

\begin{lemma}
  \label{lemma:G2R-n1}
  For \(\alpha \in (-1, 0)\) and \(x < 1\) the first term for the sum
  in~\eqref{eq:G2R-sum} is bounded by \(\frac{1}{2}\).
\end{lemma}
\begin{proof}
  The term is given by
  \begin{equation*}
    \frac{(1 + \alpha)^{n - 1}}{n!}
    \sum_{k = 0}^{n - 1}\binom{n}{k}\frac{1}{2^{k}}\int_{1}^{\pi / x}
    \log(t)^{k}|h_{n - k}(t)|t\ dt.
  \end{equation*}
  Inserting \(n = 1\) gives us
  \begin{equation*}
    \int_{1}^{\pi / x}|h_{1}(t)|t\ dt.
  \end{equation*}
  We have
  \begin{equation*}
    h_{1}(t) = \log(t - 1) + \log(t + 1) - 2\log(t) + \frac{1}{t^{2}},
  \end{equation*}
  giving us
  \begin{equation*}
    \int_{1}^{\pi / x}|h_{1}(t)|t\ dt
    = \int_{1}^{\pi / x}\left|\log(t - 1) + \log(t + 1) - 2\log(t) + \frac{1}{t^{2}}\right|t\ dt.
  \end{equation*}
  Note that for \(t > 1\) we have
  \begin{equation*}
    h_{1}(t) = \log(1 - 1 / t) + \log(1 + 1 / t) + \frac{1}{t^{2}}
    = -\sum_{n = 2}^{\infty} \frac{1}{nt^{2n}} < 0,
  \end{equation*}
  so the absolute value can be replaced with a minus sign. If we also
  integrate to infinity we get the upper bound
  \begin{equation*}
    \int_{1}^{\pi / x}|h_{1}(t)|t\ dt
    \leq \int_{1}^{\infty}-\left(\log(t - 1) + \log(t + 1) - 2\log(t) + \frac{1}{t^{2}}\right)t\ dt.
  \end{equation*}
  The integral can be computed explicitly to be \(\frac{1}{2}\),
  giving us the required upper bound.
\end{proof}

\begin{lemma}
  \label{lemma:G2R1-bound}
  For \(\alpha \in (-1, 0)\) and \(x < 1\), \(G_{\alpha,2,R,1}\)
  satisfies the following bound
  \begin{equation*}
    G_{\alpha,2,R,1}(x) \leq 2\left(
      \sqrt{e}\frac{1 + \alpha}{-\alpha}
      + 9\frac{e^{3(1 + \alpha)} -3(1 + \alpha) - 1}{3(1 + \alpha)}
      + (2 + \alpha)e^{3(1 + \alpha) / 2} - 1
    \right).
  \end{equation*}
\end{lemma}
\begin{proof}
  As a first step our goal is to compute an upper bound of the
  integral
  \begin{equation*}
    H_{k} = \int_{1}^{2} \log(t)^{k}|h_{n - k}(t)|t\ dt.
  \end{equation*}
  On the interval of integration we have
  \begin{equation*}
    \log(t)^{k}t \leq 2\log(2)^{k}
  \end{equation*}
  and hence
  \begin{equation*}
    H_{k} \leq 2\log(2)^{k}\int_{1}^{2} \left|h_{n - k}(t)\right|\ dt.
  \end{equation*}
  For \(1 < t < 2\) we have
  \begin{align*}
    \left|h_{k}(t)\right|
    &= \left|
      \log(t - 1)^{k} + \log(t + 1)^{k} - 2\log(t)^{k}- \frac{k(k - 1 - \log t)\log(t)^{k - 2}}{t^{2}}
      \right|\\
    &\leq (-1)^{k}\log(t - 1)^{k}
      + \log(t + 1)^{k}
      + 2\log(t)^{k}
      + \frac{k(k - 1)}{t^{2}}\log(t)^{k - 2}
      + \frac{k}{t^{2}}\log(t)^{k - 1}\\
    &\leq (-1)^{k}\log(t - 1)^{k} + \log(3)^{k} + 2\log(2)^{k} + k(k - 1)\log(2)^{k - 2} + k\log(2)^{k - 1}.
  \end{align*}
  This, together with
  \(\int_{1}^{2}(-1)^{k}\log(t - 1)^{k}\ dt = k!\), gives us
  \begin{equation*}
    \int_{1}^{2} \left|h_{n - k}(t)\right|\ dt
    \leq (n - k)! + \log(3)^{n - k} + 2\log(2)^{n - k}
    + (n - k)(n - k - 1)\log(2)^{n - k - 2} + (n - k)\log(2)^{n - k - 1}.
  \end{equation*}
  Combining this with \(\log(2) < 1 < \log(3) < 2\) and \(n \geq 2\)
  we get
  \begin{equation*}
    H_{k} \leq 2\left((n - k)! + 3 \cdot 2^{n} + (n - k)^{2}\right).
  \end{equation*}
  This gives us the upper bound
  \begin{equation*}
    G_{\alpha,2,R,1}(x) \leq 2\sum_{n = 2}^{\infty}\frac{(1 + \alpha)^{n - 1}}{n!}\Bigg(
    \sum_{k = 0}^{n - 1}\binom{n}{k}\frac{(n - k)!}{2^{k}}
    + 3 \cdot 2^{n}\sum_{k = 0}^{n - 1}\binom{n}{k}\frac{1}{2^{k}}
    + \sum_{k = 0}^{n - 1}\binom{n}{k}\frac{(n - k)^{2}}{2^{k}}
    \Bigg).
  \end{equation*}
  Explicitly computing the inner sums we have
  \begin{align*}
    \sum_{k = 0}^{n - 1}\binom{n}{k}\frac{(n - k)!}{2^{k}}
    &= \sqrt{e}n\Gamma\left(n, \frac{1}{2}\right) < \sqrt{e}n!,\\
    \sum_{k = 0}^{n - 1}\binom{n}{k}\frac{1}{2^{k}}
    &= \left(\frac{3}{2}\right)^{n} - \frac{1}{2^{n}} < \left(\frac{3}{2}\right)^{n},\\
    \sum_{k = 0}^{n - 1}\binom{n}{k}\frac{(n - k)^{2}}{2^{k}}
    &= \frac{n(2n + 1)}{3}\left(\frac{3}{2}\right)^{n - 1},
  \end{align*}
  here \(\Gamma(n, z)\) is the upper incomplete gamma function, which
  satisfies \(\Gamma(n, z) < \Gamma(n) = (n - 1)!\) for \(z > 0\).
  Inserting this back gives us
  \begin{equation*}
    G_{\alpha,2,R,1}(x) \leq 2\sum_{n = 2}^{\infty}\frac{(1 + \alpha)^{n - 1}}{n!}
    \left(
      \sqrt{e}n!
      + 3 \cdot 2^{n}\left(\frac{3}{2}\right)^{n}
      + \frac{n(2n + 1)}{3}\left(\frac{3}{2}\right)^{n - 1}
    \right).
  \end{equation*}
  Splitting it into the three sums we have
  \begin{equation*}
    G_{\alpha,2,R,1}(x) \leq 2\left(\sqrt{e}\sum_{n = 2}^{\infty}(1 + \alpha)^{n - 1}
    + 9\sum_{n = 2}^{\infty}\frac{(1 + \alpha)^{n - 1}}{n!}3^{n - 1}
    + \frac{1}{3}\sum_{n = 2}^{\infty}\frac{(1 + \alpha)^{n - 1}}{n!}n(2n + 1)\left(\frac{3}{2}\right)^{n - 1}\right).
  \end{equation*}
  For the individual sums we get
  \begin{align*}
    \sum_{n = 2}^{\infty}(1 + \alpha)^{n - 1}
    &= \frac{1 + \alpha}{-\alpha},\\
    \sum_{n = 2}^{\infty}\frac{(1 + \alpha)^{n - 1}}{n!}3^{n - 1}
    &= \frac{e^{3(1 + \alpha)} -3(1 + \alpha) - 1}{3(1 + \alpha)},\\
    \sum_{n = 2}^{\infty}\frac{(1 + \alpha)^{n - 1}}{n!}n(2n + 1)\left(\frac{3}{2}\right)^{n - 1}
    &= 3(2 + \alpha)e^{3(1 + \alpha) / 2} - 3.
  \end{align*}
  Combining all of this we get the bound
  \begin{equation*}
    G_{\alpha,2,R,1}(x) \leq 2\left(
      \sqrt{e}\frac{1 + \alpha}{-\alpha}
      + 9\frac{e^{3(1 + \alpha)} -3(1 + \alpha) - 1}{3(1 + \alpha)}
      + (2 + \alpha)e^{3(1 + \alpha) / 2} - 1
    \right),
  \end{equation*}
  which is what we wanted to prove.
\end{proof}

\begin{lemma}
  \label{lemma:G2R2-bound}
  For \(\alpha \in (-1, 0)\) and \(x < 1\), \(G_{\alpha,2,R,2}\)
  satisfies the following bound
  \begin{equation*}
    G_{\alpha,2,R,2}(x) < 192\log(2)\frac{1 + \alpha}{-\alpha}.
  \end{equation*}
\end{lemma}
\begin{proof}
  To bound the integral we need to better understand the behavior of
  \(h_{k}(t)\). Recall that we have
  \begin{equation*}
    h_{k}(t) = \log(t - 1)^{k} + \log(t + 1)^{k} - 2\log(t)^{k}
    - \frac{k(k - 1 - \log t)\log(t)^{k - 2}}{t^{2}}.
  \end{equation*}
  As a first step we use that
  \begin{align*}
    \log(t - 1)^{k} + \log(t + 1)^{k} - 2\log(t)^{k}
    &= \left(\log(t) + \log(1 - 1 / t)\right)^{k}
    + \left(\log(t) + \log(1 + 1 / t)\right)^{k}
    - 2\log(t)^{k}\\
    &= \sum_{j = 0}^{k - 1}\binom{k}{j}\log(t)^{j}
      \left(\log(1 - 1 / t)^{k - j} + \log(1 + 1 / t)^{k - j}\right)
  \end{align*}
  to get
  \begin{equation*}
    h_{k}(t) = \sum_{j = 0}^{k - 1}\binom{k}{j}\log(t)^{j}
    \left(\log(1 - 1 / t)^{k - j} + \log(1 + 1 / t)^{k - j}\right)
    - \frac{k(k - 1 - \log t)\log(t)^{k - 2}}{t^{2}}.
  \end{equation*}
  Next, using that
  \begin{equation*}
    \log(1 - 1 / t) = -t\sum_{i = 0}^{\infty} \frac{1}{i + 1}\frac{1}{t^{i}},\quad
    \log(1 + 1 / t) = t\sum_{i = 0}^{\infty} (-1)^{i}\frac{1}{i + 1}\frac{1}{t^{i}},
  \end{equation*}
  we can write it as
  \begin{multline*}
    h_{k}(t) = \sum_{j = 0}^{k - 1}\binom{k}{j}\log(t)^{j}\frac{1}{t^{k - j}}
    \left(
      (-1)^{k - j}\left(\sum_{i = 0}^{\infty} \frac{1}{i + 1}\frac{1}{t^{i}}\right)^{k - j}
      + \left(\sum_{i = 0}^{\infty} (-1)^{i}\frac{1}{i + 1}\frac{1}{t^{i}}\right)^{k - j}
    \right)\\
    - \frac{k(k - 1 - \log t)\log(t)^{k - 2}}{t^{2}}.
  \end{multline*}

  Using formulas for powers of power series (see
  e.g.~\cite{Gould1974}) we have
  \begin{align*}
    \left(\sum_{i = 0}^{\infty} \frac{1}{i + 1}\frac{1}{t^{i}}\right)^{j}
    &= \sum_{i = 0}^{\infty}c_{i,j}\frac{1}{t^{i}},\\
    \left(\sum_{i = 0}^{\infty} (-1)^{i}\frac{1}{i + 1}\frac{1}{t^{i}}\right)^{j}
    &= \sum_{i = 0}^{\infty}d_{i,j}\frac{1}{t^{i}}.
  \end{align*}
  with
  \begin{equation*}
    c_{0,j} = 1,\quad
    c_{i,j} = \frac{1}{i}\sum_{l = 1}^{i} \frac{lj - i + l}{l + 1}c_{i-l,j},\ i \geq 1
  \end{equation*}
  and
  \begin{equation*}
    d_{0,j} = 1,\quad
    d_{i,j} = \frac{1}{i}\sum_{l = 1}^{i} (-1)^{l}\frac{lj - i + l}{l + 1}d_{i-l,j},\ i \geq 1.
  \end{equation*}
  By induction it can be checked that \(d_{i,j} = (-1)^{i}c_{i,j}\),
  hence
  \begin{equation*}
    \left(\sum_{i = 0}^{\infty} (-1)^{i}\frac{1}{i + 1}\frac{1}{t^{i}}\right)^{j}
    = \sum_{i = 0}^{\infty}(-1)^{i}c_{i,j}\frac{1}{t^{i}}.
  \end{equation*}
  Inserting this back into \(h_{k}(t)\) we get
  \begin{equation*}
    h_{k}(t) = \sum_{j = 0}^{k - 1}\binom{k}{j}\log(t)^{j}\frac{1}{t^{k - j}}
    \sum_{i = 0}^{\infty} \left((-1)^{k - j} + (-1)^{i}\right)c_{i,k-j}\frac{1}{t^{i}}
    - \frac{k(k - 1 - \log t)\log(t)^{k - 2}}{t^{2}},
  \end{equation*}
  which we can rewrite as
  \begin{equation*}
    h_{k}(t) = \sum_{j = 0}^{k - 1}\sum_{i = 0}^{\infty}
    \binom{k}{j}\left((-1)^{k - j} + (-1)^{i}\right)c_{i,k-j}\log(t)^{j}\frac{1}{t^{k - j + i}}
    - \frac{k(k - 1 - \log t)\log(t)^{k - 2}}{t^{2}}.
  \end{equation*}
  Reordering the sums we can further write it as
  \begin{align*}
    h_{k}(t) &= \sum_{m = 0}^{\infty}\sum_{j = \max(k - m, 0)}^{k - 1}
               \binom{k}{j}\left((-1)^{k - j} + (-1)^{m + j - k}\right)c_{m+j-k,k-j}\log(t)^{j}\frac{1}{t^{m}}
               - \frac{k(k - 1 - \log t)\log(t)^{k - 2}}{t^{2}}\\
             &= \sum_{m = 1}^{\infty}\sum_{j = \max(k - m, 0)}^{k - 1}
               \binom{k}{j}\left((-1)^{k - j} + (-1)^{m + j - k}\right)c_{m+j-k,k-j}\log(t)^{j}\frac{1}{t^{m}}
               - \frac{k(k - 1 - \log t)\log(t)^{k - 2}}{t^{2}}\\
             &= \sum_{m = 1}^{\infty}\sum_{j = \max(k - m, 0)}^{k - 1}
               \binom{k}{j}(-1)^{k - j}\left(1 + (-1)^{m}\right)c_{m+j-k,k-j}\log(t)^{j}\frac{1}{t^{m}}
               - \frac{k(k - 1 - \log t)\log(t)^{k - 2}}{t^{2}}\\
             &= 2\sum_{m = 1}^{\infty}\sum_{j = \max(k - 2m, 0)}^{k - 1}
               \binom{k}{j}(-1)^{k - j}c_{2m+j-k,k-j}\log(t)^{j}\frac{1}{t^{2m}}
               - \frac{k(k - 1 - \log t)\log(t)^{k - 2}}{t^{2}}.
  \end{align*}
  The \(m = 1\) term exactly cancels the subtracted term, giving us
  \begin{equation*}
    h_{k}(t) = 2\sum_{m = 2}^{\infty}\sum_{j = \max(k - 2m, 0)}^{k - 1}
    \binom{k}{j}(-1)^{k - j}c_{2m+j-k,k-j}\log(t)^{j}\frac{1}{t^{2m}}.
  \end{equation*}
  Taking the absolute value we have
  \begin{equation*}
    \left|h_{k}(t)\right| \leq 2\sum_{m = 2}^{\infty}\sum_{j = \max(k - 2m, 0)}^{k - 1}
    \binom{k}{j}c_{2m+j-k,k-j}\log(t)^{j}\frac{1}{t^{2m}}.
  \end{equation*}
  Using that \(\log(t)^{j} \leq 1 + \log(t)^{k - 1}\) and
  \(t^{-2m} \leq 16t^{-4}2^{-2m}\)for \(t \geq 2\) we get the upper
  bound
  \begin{equation*}
    \left|h_{k}(t)\right| \leq 32\frac{1 + \log(t)^{k - 1}}{t^{4}}
    \sum_{m = 2}^{\infty}\sum_{j = \max(k - 2m, 0)}^{k - 1}
    \binom{k}{j}\frac{c_{2m+j-k,k-j}}{2^{2m}}.
  \end{equation*}
  Note that the sum no longer depends on \(t\), we denote it by
  \(S_{k}\).

  Inserting this back into \(G_{\alpha,2,R,2}\) we get the upper bound
  \begin{equation*}
    G_{\alpha,2,R,2}(x) \leq 32\sum_{n = 2}^{\infty} \frac{(1 + \alpha)^{n - 1}}{n!}
    \sum_{k = 0}^{n - 1}\binom{n}{k}\frac{S_{n - k}}{2^{k}}\int_{2}^{\pi / x}
    \frac{\log(t)^{k} + \log(t)^{n - 1}}{t^{3}}\ dt.
  \end{equation*}
  The integral can be explicitly computed to be
  \begin{equation*}
    \int_{2}^{\pi / x}\frac{\log(t)^{k} + \log(t)^{n - 1}}{t^{3}}\ dt
    = \frac{\Gamma(k + 1, 2\log(2)) - \Gamma\left(k + 1, 2\log\left(\frac{\pi}{x}\right)\right)}{2^{k + 1}}
      + \frac{\Gamma(n, 2\log(2)) - \left(n, 2\log\left(\frac{\pi}{x}\right)\right)}{2^{n}}.
  \end{equation*}
  The differences between the incomplete gamma functions are bounded
  by \(\Gamma(k + 1) = k!\) and \(\Gamma(n) = (n - 1)!\) respectively,
  giving us
  \begin{equation*}
    \int_{2}^{\pi / x}\frac{\log(t)^{k} + \log(t)^{n - 1}}{t^{3}}\ dt
    \leq \frac{k!}{2^{k + 1}} + \frac{(n-1)!}{2^{n}}.
  \end{equation*}
  For \(0 \leq k \leq n - 1\) we have
  \(\frac{k!}{2^{k + 1}} < 2\frac{(n-1)!}{2^{n}}\) and hence
  \begin{equation*}
    \int_{2}^{\pi / x}\frac{\log(t)^{k} + \log(t)^{n - 1}}{t^{3}}\ dt \leq 3\frac{(n-1)!}{2^{n}}.
  \end{equation*}
  Inserting this back into \(G_{\alpha,2,R,2}\) gives
  \begin{align*}
    G_{\alpha,2,R,2}(x)
    &\leq 96\sum_{n = 2}^{\infty} \frac{(1 + \alpha)^{n - 1}}{n!}
      \sum_{k = 0}^{n - 1}\binom{n}{k}\frac{S_{n - k}}{2^{k}}\frac{(n-1)!}{2^{n}}\\
    &= 48\sum_{n = 2}^{\infty} \left(\frac{1 + \alpha}{2}\right)^{n - 1}\frac{1}{n}
      \sum_{k = 0}^{n - 1}\binom{n}{k}\frac{S_{n - k}}{2^{k}}.
  \end{align*}

  The sum no longer depends on \(x\) and to bound it we need to
  analyze
  \begin{equation*}
    S_{k} = \sum_{m = 2}^{\infty}\sum_{j = \max(k - 2m, 0)}^{k - 1}\binom{k}{j}\frac{c_{2m+j-k,k-j}}{2^{2m}}.
  \end{equation*}
  The bound \(\binom{k}{j} \leq 2^{k}\) gives us
  \begin{equation*}
    S_{k} \leq 2^{k}\sum_{m = 2}^{\infty}\sum_{j = \max(k - 2m, 0)}^{k - 1}\frac{c_{2m+j-k,k-j}}{2^{2m}}.
  \end{equation*}
  Starting from \(m = 0\) and adding back the odd powers of \(2\)
  gives us
  \begin{equation*}
    S_{k} \leq 2^{k}\sum_{m = 0}^{\infty}\sum_{j = \max(k - m, 0)}^{k - 1}\frac{c_{m+j-k,k-j}}{2^{m}}.
  \end{equation*}
  Reordering the sums we have
  \begin{equation*}
    \sum_{m = 0}^{\infty}\sum_{j = \max(k - m, 0)}^{k - 1}\frac{c_{m+j-k,k-j}}{2^{m}}
    = \sum_{j = 0}^{k - 1}\frac{1}{2^{k - j}}\sum_{i = 0}^{\infty}c_{i,k-j}\frac{1}{2^{i}}.
  \end{equation*}
  From the definition of \(c_{i,j}\) we get
  \begin{equation*}
    \sum_{i = 0}^{\infty}c_{i,k-j}\frac{1}{2^{i}}
    = \left(\sum_{i = 0}^{\infty}\frac{1}{i + 1}\frac{1}{2^{i}}\right)^{k - j}
    = \left(-2\log(1 - 1 / 2)\right)^{k - j} = \log(4)^{k - j}.
  \end{equation*}
  Inserting this back into \(S_{k}\) we have
  \begin{equation*}
    S_{k} \leq 2^{k}\sum_{j = 0}^{k - 1}\left(\frac{\log(4)}{2}\right)^{k - j}
    = 2^{k}\sum_{j = 0}^{k - 1}\log(2)^{k - j}
    = 2^{k}\log(2)\frac{1 - \log(2)^{k}}{1 - \log(2)}.
  \end{equation*}
  We have \(\frac{1 - \log(2)^{k}}{1 - \log(2)} \leq 4\) and hence
  \(S_{k} \leq 2^{k + 2}\log(2)\). Inserting this back into
  \(G_{\alpha,2,R,2}\) gives
  \begin{align*}
    G_{\alpha,2,R,2}(x)
    &= 48\sum_{n = 2}^{\infty} \left(\frac{1 + \alpha}{2}\right)^{n - 1}\frac{1}{n}
      \sum_{k = 0}^{n - 1}\binom{n}{k}\frac{2^{k + 2}\log(2)}{2^{k}}\\
    &= 192\log(2)\sum_{n = 2}^{\infty} \left(\frac{1 + \alpha}{2}\right)^{n - 1}\frac{1}{n}\sum_{k = 0}^{n - 1}\binom{n}{k}\\
    &= 192\log(2)\sum_{n = 2}^{\infty} \left(\frac{1 + \alpha}{2}\right)^{n - 1}\frac{1}{n}(2^{n} - 1)\\
    &\leq 192\log(2)\sum_{n = 2}^{\infty} (1 + \alpha)^{n - 1}\\
    &= 192\log(2)\frac{1 + \alpha}{-\alpha},
  \end{align*}
  which is what we wanted to prove.
\end{proof}

As mentioned in the beginning of the section the above methods work
well for very small values of \(x\). For small, but not that small,
values we have to take another approach. We let
\begin{equation*}
  G_{\alpha,2,1}(x)
  = \frac{1}{\log(1 / x)}\int_{1}^{2}\frac{(t - 1)^{-\alpha - 1} + (1 + t)^{-\alpha - 1} - 2t^{-\alpha - 1}}{1 + \alpha}t^{(1 - \alpha) / 2}\log(2e + 1 / (xt))\ dt
\end{equation*}
and
\begin{equation*}
  G_{\alpha,2,2}(x)
  = \frac{1}{\log(1 / x)}\int_{2}^{\pi / x}\frac{(t - 1)^{-\alpha - 1} + (1 + t)^{-\alpha - 1} - 2t^{-\alpha - 1}}{1 + \alpha}t^{(1 - \alpha) / 2}\log(2e + 1 / (xt))\ dt,
\end{equation*}
giving
\begin{equation*}
  G_{\alpha,2}(x) = \frac{1 + \alpha}{1 - x^{1 + \alpha + (1 + \alpha)^{2} / 2}}(G_{\alpha,2,1}(x) + G_{\alpha,2,2}(x)).
\end{equation*}
The integral for \(G_{\alpha,2,2}\) is bounded on the interval of
integration and can be computed using numerical integration, as
discussed in Appendix~\ref{sec:rigorous-integration}. The only
problematic part is evaluating
\begin{equation*}
  \frac{(t - 1)^{-\alpha - 1} + (1 + t)^{-\alpha - 1} - 2t^{-\alpha - 1}}{1 + \alpha},
\end{equation*}
which has a removable singularity at \(\alpha = -1\), this can be
handled as described in Appendix~\ref{sec:removable-singularities}.
The integral for \(G_{\alpha,2,1}\) has a singularity at \(t = 1\) and
to compute it we use the following lemma.
\begin{lemma}
  \label{lemma:G21}
  We have
  \begin{equation*}
    G_{\alpha,2,1}(x) =
    \frac{C}{\log(1 / x)}\frac{1 + 2^{-\alpha} - 3^{-\alpha} + \alpha(-4 + 5 \cdot 2^{-\alpha} - 2 \cdot 3^{-\alpha})}{(\alpha - 1)\alpha(\alpha + 1)}
  \end{equation*}
  for some
  \begin{equation*}
    C \in [2^{-(1 + \alpha) / 2}\log(2e + 1 / (2x)), \log(2e + 1 / x)].
  \end{equation*}
\end{lemma}
\begin{proof}
  For \(t \in [1, 2]\) we have
  \begin{equation*}
    \log(2e + 1 / (xt)) \in [\log(2e + 1 / (2x)), \log(2e + 1 / x)].
  \end{equation*}
  Since the integrand is positive it follows that
  \begin{equation*}
    G_{\alpha,2,2}(x)
    = \frac{C_{1}}{\log(1 / x)}\int_{1}^{2}\frac{(t - 1)^{-\alpha - 1} + (1 + t)^{-\alpha - 1} - 2t^{-\alpha - 1}}{1 + \alpha}t^{(1 - \alpha) / 2}\ dt,
  \end{equation*}
  for some \(C_{1} \in [\log(2e + 1 / (2x)), \log(2e + 1 / x)]\).
  Furthermore we have
  \(t^{(1 - \alpha) / 2} = t t^{-(1 + \alpha) / 2}\) and
  \begin{equation*}
    t^{-(1 + \alpha) / 2} \in [2^{-(1 + \alpha) / 2}, 1].
  \end{equation*}
  This gives us
  \begin{equation*}
    G_{\alpha,2,2}(x)
    = \frac{C_{1}C_{2}}{\log(1 / x)}\int_{1}^{2}\frac{(t - 1)^{-\alpha - 1} + (1 + t)^{-\alpha - 1} - 2t^{-\alpha - 1}}{1 + \alpha}t\ dt,
  \end{equation*}
  for some \(C_{2} \in [2^{-(1 + \alpha) / 2}, 1]\). The integral can
  now be computed explicitly to be
  \begin{equation*}
    \int_{1}^{2} \frac{(t - 1)^{-\alpha - 1} + (1 + t)^{-\alpha - 1} - 2t^{-\alpha - 1}}{1 + \alpha}t\ dt
    = \frac{1 + 2^{-\alpha} - 3^{-\alpha} + \alpha(-4 + 5 \cdot 2^{-\alpha} - 2 \cdot 3^{-\alpha})}{(\alpha - 1)\alpha(\alpha + 1)}.
  \end{equation*}
  Finally we note that if we let \(C = C_{1}C_{2}\) then
  \begin{equation*}
    C \in [2^{-(1 + \alpha) / 2}\log(2e + 1 / (2x)), \log(2e + 1 / x)].
  \end{equation*}
  This gives us the result.
\end{proof}

\subsection{\(R_{\alpha}\)}
Recall that
\begin{equation*}
  R_{\alpha}(x) = 2\sum_{m = 1}^{\infty} (-1)^{m}\zeta(-\alpha - 2m)\frac{x^{2m}}{(2m)!}\sum_{k = 0}^{m - 1}\binom{2m}{2k}\int_{0}^{\pi / x}t^{2k + (1 - \alpha) / 2}\log(2e + 1 / (xt))\ dt.
\end{equation*}
To begin with we can note that since \(0 < -\alpha < 1\) we have
\((-1)^{m}\zeta(-\alpha - 2m) > 0\) for all \(m = 1, 2, \dots\) due to
the zeros of the zeta function on the real line being exactly the even
negative integers. As a consequence all terms in the sum are positive.
The following lemma gives a bound
\begin{lemma}
  For \(\alpha \in (-1, 0)\), \(x < 1\) and \(M \geq 1\),
  \(R_{\alpha}(x)\) satisfies the following bound
  \begin{multline*}
    R_{\alpha}(x) < 2\log(2e + 1/\pi) \Bigg(
    \sum_{m = 1}^{M - 1} (-1)^{m}\zeta(-\alpha - 2m)\frac{\pi^{2m}}{(2m)!}
    \sum_{k = 0}^{m - 1}\binom{2m}{2k}\frac{1}{2k + 1 + (1 - \alpha) / 2}\left(\frac{x}{\pi}\right)^{2(m - 1 - k)}\\
    + \pi^{2} 2(2\pi)^{1 - \alpha - 2M}\left|\sin\left(\frac{\pi}{2}\alpha\right)\right|\zeta(2M + 1 + \alpha) \frac{(\pi(1 + 1 / \pi))^{2M}}{4\pi^{2} - (\pi(1 + 1 / \pi))^{2}}
    \Bigg).
  \end{multline*}
\end{lemma}
\begin{proof}
  Looking at the integral in the inner sum we write
  \begin{equation*}
    \log(2e + 1 / (xt)) = \log(1 + 2ext) - \log(xt),
  \end{equation*}
  giving us
  \begin{equation*}
    \int_{0}^{\pi / x}t^{2k + (1 - \alpha) / 2}\log(2e + 1 / (xt))\ dt =
    \int_{0}^{\pi / x}t^{2k + (1 - \alpha) / 2}\log(1 + 2ext)\ dt
    -   \int_{0}^{\pi / x}t^{2k + (1 - \alpha) / 2}\log(xt)\ dt.
  \end{equation*}
  For the first integral we compute an upper bound using that
  \(\log(1 + 2ext)\) is bounded by \(\log(1 + 2e\pi)\), for the second
  integral we integrate directly.
  \begin{align*}
    \int_{0}^{\pi / x}t^{2k + (1 - \alpha) / 2}\log(1 + 2ext)\ dt
    &\leq \log(1 + 2e\pi) \int_{0}^{\pi / x}t^{2k + (1 - \alpha) / 2}\ dt\\
    &= \log(1 + 2e\pi)\frac{1}{2k + 1 + (1 - \alpha) / 2}\left(\frac{\pi}{x}\right)^{2k + 1 + (1 - \alpha) / 2}\\
    &\leq \frac{\log(1 + 2e\pi)}{2k + 1 + (1 - \alpha) / 2}\left(\frac{\pi}{x}\right)^{2(k + 1)}.
  \end{align*}
  The second integral can be computed explicitly to be
  \begin{equation*}
    \int_{0}^{\pi / x}t^{2k + (1 - \alpha) / 2}\log(xt)\ dt
    = \frac{(2k + 1 + (1 - \alpha) / 2)\log \pi - 1}{(2k + 1 + (1 - \alpha) / 2)^{2}}\left(\frac{\pi}{x}\right)^{2k + 1 + (1 - \alpha) / 2},
  \end{equation*}
  and we note that this is positive for \(k \geq 0\). An upper bound
  for the full integral is hence given by
  \begin{equation*}
    \int_{0}^{\pi / x}t^{2k + (1 - \alpha) / 2}\log(2e + 1 / (xt))\ dt
    \leq \frac{\log(2e + 1 / \pi)}{2k + 1 + (1 - \alpha) / 2}\left(\frac{\pi}{x}\right)^{2(k + 1)}
  \end{equation*}
  For the inner sum we thus get
  \begin{equation*}
    \sum_{k = 0}^{m - 1}\binom{2m}{2k}\int_{0}^{\pi / x}t^{2k + (1 - \alpha) / 2}\log(2e + 1 / (xt))\ dt
    \leq \log(2e + 1/\pi) \sum_{k = 0}^{m - 1}\binom{2m}{2k}\frac{1}{2k + 1 + (1 - \alpha) / 2}\left(\frac{\pi}{x}\right)^{2(k + 1)}.
  \end{equation*}
  Factoring out \(\left(\frac{\pi}{x}\right)^{2m}\) and inserting into
  \(R_{\alpha}\) we have
  \begin{equation*}
    R_{\alpha}(x) \leq 2\log(2e + 1/\pi) \sum_{m = 1}^{\infty} (-1)^{m}\zeta(-\alpha - 2m)\frac{\pi^{2m}}{(2m)!}
    \sum_{k = 0}^{m - 1}\binom{2m}{2k}\frac{1}{2k + 1 + (1 - \alpha) / 2}\left(\frac{x}{\pi}\right)^{2(m - 1 - k)}.
  \end{equation*}
  Taking \(M \geq 1\) we split this into one finite sum and one tail as
  \begin{multline*}
    R_{\alpha}(x) \leq 2\log(2e + 1/\pi) \Bigg(
    \sum_{m = 1}^{M - 1} (-1)^{m}\zeta(-\alpha - 2m)\frac{\pi^{2m}}{(2m)!}
    \sum_{k = 0}^{m - 1}\binom{2m}{2k}\frac{1}{2k + 1 + (1 - \alpha) / 2}\left(\frac{x}{\pi}\right)^{2(m - 1 - k)}\\
    + \sum_{m = M}^{\infty} (-1)^{m}\zeta(-\alpha - 2m)\frac{\pi^{2m}}{(2m)!}
    \sum_{k = 0}^{m - 1}\binom{2m}{2k}\frac{1}{2k + 1 + (1 - \alpha) / 2}\left(\frac{x}{\pi}\right)^{2(m - 1 - k)}
    \Bigg).
  \end{multline*}
  The finite sum can be enclosed directly, for the second sum we note
  that for \(x < 1\) we have
  \begin{align*}
    \sum_{k = 0}^{m - 1}\binom{2m}{2k}\frac{1}{2k + 1 + (1 - \alpha) / 2}\left(\frac{x}{\pi}\right)^{2(m - 1 - k)}
    &< \sum_{k = 0}^{m - 1}\binom{2m}{2k}\frac{1}{2k + 1}\left(\frac{1}{\pi}\right)^{2(m - 1 - k)}\\
    &< \sum_{k = 0}^{m}\binom{2m}{2k}\frac{1}{2k + 1}\left(\frac{1}{\pi}\right)^{2(m - 1 - k)}\\
    &= \pi^{1 - 2m}\frac{(\pi + 1)^{2m} - (\pi - 1)^{2m} + \pi(\pi + 1)^{2m} + \pi(\pi - 1)^{2m}}{2(1 + 2m)}\\
    &\leq \pi^{1 - 2m}(\pi + 1)^{2m}\frac{1 + 2\pi}{6}\\
    &= \pi^{2} \left(1 + \frac{1}{\pi}\right)^{2m}.
  \end{align*}
  and hence
  \begin{multline*}
    \sum_{m = M}^{\infty} (-1)^{m}\zeta(-\alpha - 2m)\frac{\pi^{2m}}{(2m)!}
    \sum_{k = 0}^{m - 1}\binom{2m}{2k}\frac{1}{2k + 1 + (1 - \alpha) / 2}\left(\frac{x}{\pi}\right)^{2(m - 1 - k)}\\
    \leq
    \pi^{2}\sum_{m = M}^{\infty} (-1)^{m}\zeta(-\alpha - 2m)\frac{(\pi(1 + 1 / \pi))^{2m}}{(2m)!}.
  \end{multline*}
  The last sum is the same as for the tail of the Clausen function
  occurring in Lemma~\ref{lemma:clausen-tails}, since
  \(\pi(1 + 1 / \pi) < 2\pi\) this gives us
  \begin{equation*}
    \sum_{m = M}^{\infty} (-1)^{m}\zeta(-\alpha - 2m)\frac{(\pi(1 + 1 / \pi))^{2m}}{(2m)!}
    \leq 2(2\pi)^{1 - \alpha - 2M}\left|\sin\left(\frac{\pi}{2}\alpha\right)\right|\zeta(2M + 1 + \alpha) \frac{(\pi(1 + 1 / \pi))^{2M}}{4\pi^{2} - (\pi(1 + 1 / \pi))^{2}}.
  \end{equation*}
  Finally we then get
  \begin{multline*}
    R_{\alpha}(x) < 2\log(2e + 1/\pi) \Bigg(
    \sum_{m = 1}^{M - 1} (-1)^{m}\zeta(-\alpha - 2m)\frac{\pi^{2m}}{(2m)!}
    \sum_{k = 0}^{m - 1}\binom{2m}{2k}\frac{1}{2k + 1 + (1 - \alpha) / 2}\left(\frac{x}{\pi}\right)^{2(m - 1 - k)}\\
    + \pi^{2} 2(2\pi)^{1 - \alpha - 2M}\left|\sin\left(\frac{\pi}{2}\alpha\right)\right|\zeta(2M + 1 + \alpha) \frac{(\pi(1 + 1 / \pi))^{2M}}{4\pi^{2} - (\pi(1 + 1 / \pi))^{2}}
    \Bigg),
  \end{multline*}
  which is what we wanted to prove.
\end{proof}

\section{Details for evaluating \(\mathcal{T}_{\alpha}(x)\) with \(\alpha\) near \(-1\) and \(x\) near zero}
\label{sec:evaluation-T-hybrid-asymptotic}
In this section we discuss how to bound
\begin{equation*}
  \frac{U_{\alpha}(x)}{\log(1 / x)x^{-\alpha + p}}
\end{equation*}
for small \(x\) in the hybrid case when \(\alpha\) is near \(-1\), see
Section~\ref{sec:evaluation-T-hybrid}.

In this case the weight is given by
\(w_{\alpha}(x) = |x|^{p}\log(2e + 1 / |x|)\), giving us
\begin{equation*}
  U_{\alpha}(x) = x^{1 + p}\int_{0}^{\pi / x} |\hat{I}_{\alpha}(x, t)|t^{p}\log(2e + 1 / (xt))\ dt.
\end{equation*}
Using that
\begin{equation*}
  \log(2e + 1 / (xt)) = \log(1 + 2ext) + log(1 / x) + log(1 / t)
\end{equation*}
we split \(U_{\alpha}\) into three parts
\begin{align*}
  U_{\alpha}^{1}(x) &= x^{1 + p}\int_{0}^{\pi / x} |\hat{I}_{\alpha}(x, t)|t^{p}\log(1 + 2ext)\ dt,\\
  U_{\alpha}^{2}(x) &= x^{1 + p}\int_{0}^{\pi / x} |\hat{I}_{\alpha}(x, t)|t^{p}\ dt,\\
  U_{\alpha}^{3}(x) &= x^{1 + p}\int_{0}^{\pi / x} |\hat{I}_{\alpha}(x, t)|t^{p}\log(1 / t)\ dt,
\end{align*}
satisfying
\begin{equation*}
  U_{\alpha}(x) = U_{\alpha}^{1}(x) +\log(1 / x) U_{\alpha}^{2}(x) + U_{\alpha}^{3}(x).
\end{equation*}

The integral \(U_{\alpha}^{2}\) is the same as what we would get with
the weight \(w_{\alpha}(x) = |x|^{p}\), and we can therefore use
Lemma~\ref{lemma:evaluation-U-I-2} to bound it.

For \(U_{\alpha}^{1}\) we can use that \(\log(1 + 2ext)\) is bounded
on the interval of integration to factor it out. We then recover the
same integral as for \(U_{\alpha}^{2}\), and we can bound it in the
same way. If \(x\) is very small this work well, but for \(x\) larger
than around \(10^{-5}\) this gives fairly poor bounds. In that case we
compute part of the integral using a rigorous integrator instead.

The computation of \(U_{\alpha}^{3}\) is the most complicate part.
Similar to in Lemma~\ref{lemma:evaluation-U-I-2} we expand the
integrand and split it into two main parts and one remainder part as
\begin{align*}
  U_{\alpha,M,1}^{3}(x) &= \Gamma(1 + \alpha)\sin\left(-\frac{\pi}{2}\alpha\right)x^{-\alpha + p}
  \int_{0}^{1}\left|(1 - t)^{-\alpha - 1} + (1 + t)^{-\alpha - 1} - 2t^{-\alpha - 1}\right|t^{p}\log(1 / t)\ dt\\
  U_{\alpha,M,2}^{3}(x) &= \Gamma(1 + \alpha)\sin\left(-\frac{\pi}{2}\alpha\right)x^{-\alpha + p}
  \int_{1}^{\pi / x}((t - 1)^{-\alpha - 1} + (1 + t)^{-\alpha - 1} - 2t^{-\alpha - 1})t^{p}\log(1 / t)\ dt\\
  U_{\alpha,R}^{3}(x) &= 2x^{1 + p}\sum_{m = 1}^{\infty} (-1)^{m}\zeta(-\alpha - 2m)\frac{x^{2m}}{(2m)!}
  \sum_{k = 0}^{m - 1}\binom{2m}{2k}\int_{0}^{\pi / x}t^{2k + p}\log(1 / t)\ dt,
\end{align*}
with
\begin{equation*}
  U_{\alpha}^{3}(x) \leq U_{\alpha,M,1}^{3}(x) + U_{\alpha,M,2}^{3}(x) + U_{\alpha,R}^{3}(x).
\end{equation*}

For \(U_{\alpha,M,1}^{3}\) the integral doesn't depend on \(x\). We
compute it by using a rigorous integrator from \(0\) to \(0.999\) and
from \(0.999\) to \(1\) we use that \(\log(1 / t)\) is bounded to
factor it out and integrate explicitly.

For \(U_{\alpha,M,2}^{3}\) we note that \(\log(1 / t)\) is negative
and the integral is hence decreasing in \(x\). To get an upper bound
it is thus enough to work with an upper bound of \(x\). We split the
interval of integration into several parts that are all treated
slightly different. We let
\begin{align*}
  U_{\alpha,M,2,1}^{3}(x)
  &= \int_{1}^{1.001}((t - 1)^{-\alpha - 1} + (1 + t)^{-\alpha - 1} - 2t^{-\alpha - 1})t^{p}\log(1 / t)\ dt,\\
  U_{\alpha,M,2,2}^{3}(x)
  &= \int_{1.001}^{10^{10}}((t - 1)^{-\alpha - 1} + (1 + t)^{-\alpha - 1} - 2t^{-\alpha - 1})t^{p}\log(1 / t)\ dt,\\
  U_{\alpha,M,2,3}^{3}(x)
  &= \int_{10^{10}}^{10^{50}}((t - 1)^{-\alpha - 1} + (1 + t)^{-\alpha - 1} - 2t^{-\alpha - 1})t^{p}\log(1 / t)\ dt,\\
  U_{\alpha,M,2,4}^{3}(x)
  &= \int_{10^{50}}^{\pi / x}((t - 1)^{-\alpha - 1} + (1 + t)^{-\alpha - 1} - 2t^{-\alpha - 1})t^{p}\log(1 / t)\ dt.
\end{align*}
If \(\pi / x \geq 10^{50}\) we then have
\begin{equation*}
  U_{\alpha,M,2}^{3} = U_{\alpha,M,2,1}^{3} + U_{\alpha,M,2,2}^{3} + U_{\alpha,M,2,3}^{3}+ U_{\alpha,M,2,4}^{3}.
\end{equation*}
If \(\pi / x\) is less than \(10^{50}\) we simply skip the integrals
above \(\pi / x\) and cut the last one off at \(\pi / x\). The
integral \(U_{\alpha,M,2,2}^{3}\) is computed using a rigorous
integrator. For the remaining ones we use that
\begin{equation*}
  \int_{a}^{b}((t - 1)^{-\alpha - 1} + (1 + t)^{-\alpha - 1} - 2t^{-\alpha - 1})t^{p}\log(1 / t)\ dt
  \geq \log(1 / a)\int_{a}^{b}((t - 1)^{-\alpha - 1} + (1 + t)^{-\alpha - 1} - 2t^{-\alpha - 1})t^{p}\ dt
\end{equation*}
and integrate explicitly. The primitive function is given by
\begin{multline*}
  \int ((t - 1)^{-\alpha - 1} + (1 + t)^{-\alpha - 1} - 2t^{-\alpha - 1})t^{p}\ dt\\
  = -\frac{t^{p - \alpha}}{\alpha - p}\left(
    {}_{2}F_{1}\left(1 + \alpha, \alpha - p; 1 + \alpha - p; \frac{1}{t}\right)
    + {}_{2}F_{1}\left(1 + \alpha, \alpha - p; 1 + \alpha - p; -\frac{1}{t}\right)
    - 2
  \right).
\end{multline*}
The \({}_{2}F_{1}\) functions are evaluated using the hypergeometric
series, valid when \(1 + \alpha - p\) is not an integer and \(t > 1\),
with bounds from~\cite{Johansson2019}.

For \(U_{\alpha,R}^{3}(x)\) we use the following lemma to compute a
bound.
\begin{lemma}
  Let \(0 < \epsilon < \frac{\pi}{2}\), for \(\alpha \in (-1, 0)\),
  \(x \in [0, \epsilon]\) and \(-\alpha < p < 1\) with
  \(1 + \alpha \not= p\) we have
  \begin{multline*}
    \frac{U_{\alpha,R}^{3}(x)}{x^{-\alpha + p}\log(1 / x)} \leq
    \frac{2x^{2 + \alpha - p}\pi^{p - 1}}{\log(1 / x)}\sum_{m = 1}^{M - 1} (-1)^{m}\zeta(-\alpha - 2m)\frac{\pi^{2m}}{(2m)!}
    \sum_{k = 0}^{m - 1}\binom{2m}{2k}\frac{1}{(2k + 1 + p)^{2}}\left(\frac{x}{\pi}\right)^{2(m - 1 - k)}\\
    - \frac{2x^{2 + \alpha - p}\pi^{p - 1}\log(\pi / x)}{\log(1 / x)}\sum_{m = 1}^{M - 1} (-1)^{m}\zeta(-\alpha - 2m)\frac{\pi^{2m}}{(2m)!}
    \sum_{k = 0}^{m - 1}\binom{2m}{2k}\frac{1}{2k + 1 + p}\left(\frac{x}{\pi}\right)^{2(m - 1 - k)}\\
    + \frac{6x^{2 + \alpha - p}\pi^{p - 1}}{\log(1 / x)}\sum_{m = M}^{\infty} (-1)^{m}\zeta(-\alpha - 2m)\frac{(\frac{3\pi}{2})^{2m}}{(2m)!}.
  \end{multline*}
\end{lemma}
\begin{proof}
  We have
  \begin{equation*}
    \int_{0}^{\pi / x}t^{2k + p}\log(1 / t)\ dt
    = \frac{1 - (2k + 1 + p)\log(\pi / x)}{(2k + 1 + p)^{2}}\left(\frac{\pi}{x}\right)^{2k + 1 + p},
  \end{equation*}
  giving us
  \begin{equation*}
    \frac{U_{\alpha,R}^{3}(x)}{x^{-\alpha + p}\log(1 / x)}
    = \frac{2x^{1 + \alpha}}{\log(1 / x)}\sum_{m = 1}^{\infty} (-1)^{m}\zeta(-\alpha - 2m)\frac{x^{2m}}{(2m)!}
    \sum_{k = 0}^{m - 1}\binom{2m}{2k}\frac{1 - (2k + 1 + p)\log(\pi / x)}{(2k + 1 + p)^{2}}\left(\frac{\pi}{x}\right)^{2k + 1 + p}.
  \end{equation*}
  Using a similar approach as in Lemma~\ref{lemma:evaluation-U-I-2} we
  can rewrite this as
  \begin{multline*}
    \frac{U_{\alpha,R}^{3}(x)}{x^{-\alpha + p}\log(1 / x)}\\
    = \frac{2x^{2 + \alpha - p}\pi^{p - 1}}{\log(1 / x)}\sum_{m = 1}^{\infty} (-1)^{m}\zeta(-\alpha - 2m)\frac{\pi^{2m}}{(2m)!}
    \sum_{k = 0}^{m - 1}\binom{2m}{2k}\frac{1 - (2k + 1 + p)\log(\pi / x)}{(2k + 1 + p)^{2}}\left(\frac{x}{\pi}\right)^{2(m - 1 - k)}.
  \end{multline*}
  Splitting it into two sums we can write it as
  \begin{multline*}
    \frac{U_{\alpha,R}^{3}(x)}{x^{-\alpha + p}\log(1 / x)}
    = \frac{2x^{2 + \alpha - p}\pi^{p - 1}}{\log(1 / x)}\sum_{m = 1}^{\infty} (-1)^{m}\zeta(-\alpha - 2m)\frac{\pi^{2m}}{(2m)!}
    \sum_{k = 0}^{m - 1}\binom{2m}{2k}\frac{1}{(2k + 1 + p)^{2}}\left(\frac{x}{\pi}\right)^{2(m - 1 - k)}\\
    - \frac{2x^{2 + \alpha - p}\pi^{p - 1}\log(\pi / x)}{\log(1 / x)}\sum_{m = 1}^{\infty} (-1)^{m}\zeta(-\alpha - 2m)\frac{\pi^{2m}}{(2m)!}
    \sum_{k = 0}^{m - 1}\binom{2m}{2k}\frac{1}{2k + 1 + p}\left(\frac{x}{\pi}\right)^{2(m - 1 - k)}.
  \end{multline*}
  We can treat the first few terms in the sums explicitly, what
  remains is bounding the tails
  \begin{equation}
    \label{eq:evaluation-T-hybrid-asymptotic-sum-1}
    \sum_{m = M}^{\infty} (-1)^{m}\zeta(-\alpha - 2m)\frac{\pi^{2m}}{(2m)!}
    \sum_{k = 0}^{m - 1}\binom{2m}{2k}\frac{1}{(2k + 1 + p)^{2}}\left(\frac{x}{\pi}\right)^{2(m - 1 - k)}
  \end{equation}
  and
  \begin{equation}
    \label{eq:evaluation-T-hybrid-asymptotic-sum-2}
    \sum_{m = M}^{\infty} (-1)^{m}\zeta(-\alpha - 2m)\frac{\pi^{2m}}{(2m)!}
    \sum_{k = 0}^{m - 1}\binom{2m}{2k}\frac{1}{2k + 1 + p}\left(\frac{x}{\pi}\right)^{2(m - 1 - k)}.
  \end{equation}
  For \eqref{eq:evaluation-T-hybrid-asymptotic-sum-2} the factor in
  front is negative, we hence only need to compute a lower bound.
  Since all terms in the sum are positive it is trivially lower
  bounded by zero.

  For \eqref{eq:evaluation-T-hybrid-asymptotic-sum-1} we need an upper
  bound. We can notice that
  \begin{equation*}
    \sum_{m = M}^{\infty} (-1)^{m}\zeta(-\alpha - 2m)\frac{\pi^{2m}}{(2m)!}
    \sum_{k = 0}^{m - 1}\binom{2m}{2k}\frac{1}{2k + 1 + p}\left(\frac{x}{\pi}\right)^{2(m - 1 - k)}
  \end{equation*}
  gives an upper bound and that this sum also occurs in the proof of
  Lemma~\ref{lemma:evaluation-U-I-2}. Following the same approach we
  therefore get that an upper bound is given by
  \begin{equation*}
    3\sum_{m = M}^{\infty} (-1)^{m}\zeta(-\alpha - 2m)\frac{(\frac{3\pi}{2})^{2m}}{(2m)!}.
  \end{equation*}
\end{proof}

\section*{Acknowledgments}

The author was partially supported by the ERC Starting Grant
ERC-StG-CAPA-852741 as well as MICINN (Spain) research grant number
PID2021- 125021NA-I00. This material is based upon work supported by
the National Science Foundation under Grant No. DMS-1929284 while the
author was in residence at the Institute for Computational and
Experimental Research in Mathematics in Providence, RI, during the
program “Hamiltonian Methods in Dispersive and Wave Evolution
Equations”. The author was partially supported by the Swedish-American
foundation for the visit. We are also thankful for the hospitality of
the Princeton Department of Mathematics and the Brown University
Department of Mathematics where parts of this paper were done. The
computations were enabled by resources provided by the National
Academic Infrastructure for Supercomputing in Sweden (NAISS) and the
Swedish National Infrastructure for Computing (SNIC) at the PDC Center
for High Performance Computing, KTH Royal Institute of Technology,
partially funded by the Swedish Research Council through grant
agreements no. 2022-06725 and no. 2018-05973. The author would like to
thank Javier Gómez-Serrano for his guidance and Erik Wahlén and
Gabriele Brüll for fruitful discussions about highest waves for
related equations.

\printbibliography

\begin{tabular}{l}
  \textbf{Joel Dahne} \\
  {Department of Mathematics} \\
  {Uppsala University} \\
  {L\"agerhyddsv\"agen 1, 752 37, Uppsala, Sweden} \\
  {Email: joel.dahne@math.uu.se} \\ \\
\end{tabular}

\end{document}